\documentclass[11pt,a4paper]{article}

\usepackage{enumitem}

\usepackage[margin=3cm]{geometry}
\usepackage{amsfonts}
\usepackage{amsthm}
\usepackage{amsmath}
\usepackage{amssymb}
\usepackage{centernot}
\usepackage[T1]{fontenc}
\usepackage{cite}
\usepackage{mathrsfs}
\usepackage{amscd}
\usepackage[utf8]{inputenc}
\usepackage{t1enc}
\usepackage{lmodern}
\usepackage{dsfont}
\usepackage{xcolor}
\usepackage{float}
\usepackage{hyperref}
\hypersetup{colorlinks=true, urlcolor= black, linkcolor=black, citecolor=blue}
\usepackage[mathscr]{eucal}
\usepackage{indentfirst}
\renewcommand{\epsilon}{\varepsilon}
\usepackage{bbm}
\usepackage{graphicx}
\usepackage{mathtools}

\mathtoolsset{showonlyrefs}

\newtheorem{theorem}{Theorem}[section]
\newtheorem{lemma}[theorem]{Lemma}
\newtheorem{proposition}[theorem]{Proposition}
\newtheorem{corollary}[theorem]{Corollary}

\theoremstyle{definition}
\newtheorem{remark}[theorem]{Remark}
\newtheorem{definition}[theorem]{Definition}

\numberwithin{equation}{section}

\newcommand\n{\mathbf{n}}
\newcommand\fg{\mathfrak{g}}
\newcommand\ft{\mathfrak{t}}
\newcommand\fs{\mathfrak{s}}
\newcommand\m{\mathbf{m}}
\newcommand\sn{\partial\mathbf{n}}

\newcommand\cR{\mathcal{R}}
\newcommand\R{\mathbb{R}}
\newcommand\Z{\mathbb{Z}}
\newcommand\N{\mathbb{N}}
\newcommand{\connect}{\xleftrightarrow}

\begin{document}

\title{The supercritical phase of the $\varphi^4$ model is well behaved}

\author{Trishen S. Gunaratnam\footnotemark[1]\footnote{TIFR Mumbai and ICTS Bengaluru, \url{trishen@math.tifr.res.in, trishen.gunaratnam@icts.res.in}}\:, Christoforos Panagiotis\footnotemark[2]\footnote{University of Bath, \url{cp2324@bath.ac.uk}}\:, \\ Romain Panis\footnotemark[3]\footnote{Institut Camille Jordan (Lyon 1), \url{panis@math.univ-lyon1.fr}, \url{severo@math.univ-lyon1.fr}}\:, Franco Severo\footnotemark[3]}

\maketitle

\begin{abstract}
In this article, we analyse the $\varphi^4$ model on $\mathbb Z^d$ in the supercritical regime $\beta > \beta_c$. We consider a random cluster representation of the $\varphi^4$ model, which corresponds to an Ising random cluster model on a random environment. We prove that the supercritical phase of this percolation model on $\mathbb Z^d$ ($d\geq 2$) is well behaved in the sense that, for every $\beta>\beta_c$, local uniqueness of macroscopic clusters occurs with high probability, uniformly in the boundary conditions. This result provides the basis for renormalisation techniques used to study several fine properties of the supercritical phase.  As applications, we prove surface order exponential bounds for the (lower) large deviations of the empirical magnetisation as well as for the spectral gaps of dynamical $\varphi^4$ models in the entire supercritical regime.
\end{abstract}

\tableofcontents

\section{Introduction}

The $\varphi^4$ model is a natural generalisation of the Ising model to a lattice spin model with unbounded spins. On a finite graph $G=(V,E)$, it is given by a probability measure on configurations $\varphi \in \mathbb R^V$ defined by its expectation values on bounded measurable functions $F:\mathbb R^\Lambda \rightarrow \mathbb R$,
\begin{equation}
\langle F(\varphi)\rangle = \frac{1}{Z} \int_{\mathbb R^V} F(\varphi)\exp \left( \beta \sum_{\{x,y\} \in E} \varphi_x\varphi_y -\sum_{x \in V}(g\varphi_x^4+a\varphi_x^2) \right) \prod_{x \in V} \mathrm{d}\varphi_x,
\end{equation}
where $Z$ is a normalisation constant, $\beta \geq 0$ is the inverse temperature, and $g>0$ and $a \in \mathbb R$ are coupling constants. It is well known that this model interpolates between the Ising model and the Gaussian free field. Furthermore, it arises in Euclidean quantum field theory, where it is used to construct (or destruct) the continuum $\varphi^4$ model \cite{Nelson1966PHI42d,GlimmJaffe1973PHI43d,Sokal1982Destructive,AizenmanGeometricAnalysis1982,FrohlichTriviality1982,AizenmanDuminilTriviality2021}.

The $\varphi^4$ model is a rich statistical physics model in its own right. On the hypercubic lattice $\mathbb Z^d$, it undergoes a \emph{phase transition} as $\beta$ is varied, provided the dimension satisfies $d \geq 2$. Namely, there is a \emph{critical point} $\beta_c$ separating a \emph{subcritical}/\emph{disordered} phase ($\beta < \beta_c$), in which two-point correlations tend to $0$ when the system size is large, and a \emph{supercritical}/\emph{ordered} phase ($\beta>\beta_c$), in which two-point correlations stay bounded away from $0$. The model's \emph{critical phase} is of significant interest since it is expected to be in the same \emph{universality class} as the Ising model in all dimensions, meaning that the two model's \emph{critical exponents} are expected to coincide. We refer to the introduction of \cite{GunaratnamPanagiotisPanisSeveroPhi42022} (see also \cite{A21,PanisThesis}) for a more complete discussion. On the mathematical side, \emph{subcritical sharpness} results \cite{AizenmanBarskyFernandezSharpnessIsing1987} imply an exponential decay of correlations for all $\beta<\beta_c$. In the critical phase, continuity of the phase transition in dimensions $d\geq 3$ \cite{GunaratnamPanagiotisPanisSeveroPhi42022}, triviality of the critical and near-critical scaling limits in $d\geq 4$ \cite{AizenmanGeometricAnalysis1982,FrohlichTriviality1982,AizenmanDuminilTriviality2021,PanisTriviality2023}, and bounds on the (near-)critical two-point function in dimensions $d\geq 3$ \cite{FrohlichSimonSpencerIRBounds1976,FrohlichIsraelLiebSimon1978,Sakai2015Phi4,DuminilPanis2024newLB} have been established.  Furthermore, the so-called \emph{weakly--coupled} $\varphi^4$ model (which corresponds to choosing $g$ small enough in \eqref{eq:single_site}) has been extensively studied in the literature through the \emph{renormalization group} approach \cite{GawedzkiKupiainen1985massless,Hara1987rigorous,FeldmanMagnenRivasseau1987construction,BauerschmidtBrydgesSlade2014Phi4fourdim,SladeTombergRGWeakPhi4in2016,BauerschmidtBrydgesSladeBOOKRG2019}. The model's critical behaviour in dimension $d=2$, weakly-coupled or not, remains a fascinating challenge.

In this paper, we prove that the supercritical phase of the $\varphi^4$ model is well behaved in the sense that properties which are typical of a robust supercritical behaviour are valid in the entire supercritical regime. 
In our Theorem~\ref{thm:ldp free}, we establish an important example of such a property: a surface order large deviation bound for the empirical magnetisation of the $\varphi^4$ model for all $\beta>\beta_c$. This result has implications for the decay of spectral gaps for dynamical $\varphi^4$ models in the supercritical regime, see Theorem \ref{theorem: dynamics}.

The main tool to prove Theorem \ref{thm:ldp free} is a \emph{random cluster}--- or Fortuin--Kasteleyn (FK)--- representation for the $\varphi^4$ model, which we define as an Ising random cluster model on a random environment. Although this object has been studied in the physics literature \cite{brower1989embedded}, in this article we give (to the best of our knowledge) the first mathematical treatment of this representation. Our main result is Theorem~\ref{thm: local uniqueness}, where we prove a \emph{supercritical sharpness} result for this percolation model in all dimensions $d\geq2$. The corresponding result for the Ising random cluster model was first obtained by Bodineau \cite{Bod05}. In words, this result says that, for every $\beta>\beta_c$, there exists a \emph{unique} macroscopic cluster in a large box with high probability, uniformly in boundary conditions. Supercritical sharpness is the essential input allowing to use non-perturbative renormalisation techniques to obtain a detailed description of the supercritical phase of statistical physics models. Indeed, we use a coarse graining argument, inspired by that of Pizstora \cite{Pis96}, to deduce Theorem~\ref{thm:ldp free} from Theorem~\ref{thm: local uniqueness}. 

Proving supercritical sharpness results for general models remains a significant challenge in percolation theory.
In dimension $d=2$, one can often rely on duality and crossing probabilities to prove such a result by relating the supercritical behaviour of the model with the subcritical behaviour of its dual model. 
In particular, supercritical sharpness is known for many percolation models in 2D, such as the random cluster model with $q\geq1$ \cite{BeffaraDuminilSelfDualPoint} (previously known for $q=1,2$ due to \cite{Kesten1980criticalproba,CCS87}), Voronoi percolation \cite{BR06} and Boolean percolation \cite{ATT18}. 
However, in dimensions $d\geq3$ there is no dual (percolation) model available in general, so a much more sophisticated approach is required. For instance, the random cluster model in dimensions $d\geq3$ has only been proved to satisfy supercritical sharpness in the special cases $q=1$ (corresponding to Bernoulli percolation) and $q=2$ (corresponding to the Ising model) in the celebrated works of Grimmett and Marstrand \cite{GM90} and Bodineau \cite{Bod05}, respectively--- see also \cite{CMT24} and \cite{Severo24} for new proofs. Although the problem remains open for the random cluster model with $q\neq 1,2$, several other percolation models were proved to satisfy supercritical sharpness in the last few years, such as level-set percolation for the (discrete) Gaussian free field \cite{DGRS19} and some continuous Gaussian fields \cite{Sev21}, the vacant set of random interlacements \cite{DGRST23a,DGRST23b,DGRST23c}, the occupied set of Boolean percolation \cite{DT22}, and Voronoi percolation \cite{DS23}.

Our proof coarsely follows the strategy of \cite{Severo24}. However, this approach requires handling boundary conditions for both the spin model and the associated random cluster measure. To address this, we develop new techniques, building on the random tangled current representation of the $\varphi^4$ model introduced in \cite{GunaratnamPanagiotisPanisSeveroPhi42022}. We stress that our result is new even in the two-dimensional case, where the random cluster representation does not seem to have a treatable dual model or any integrable feature. This is also one of the reasons why continuity of the phase transition remains open for $d=2$, despite being known for $d\geq3$ \cite{GunaratnamPanagiotisPanisSeveroPhi42022}.

\subsection{Surface order large deviations}\label{section:intro surface order large deviations}

We begin by introducing the model and relevant quantities to state the theorem. We fix once and for all constants $g>0$ and $a\in \R$ and let $\rho_{g,a}$ be the so-called single-site  measure on $\R$ defined by
\begin{equation}\label{eq:single_site}
    \textup{d}\rho_{g,a}(t)
    =
    \frac{1}{z_{g,a}}e^{-gt^4-at^2}\textup{d}t,
\end{equation}
where $z_{g,a}=\int_\R e^{-gt^4-at^2}\textup{d}t$. Let $\Lambda$ be a finite subset of $\mathbb Z^d$. Define $E(\Lambda):=\{\{x,y\}\subset \Lambda: x\sim y\}$ where $x\sim y$ means that $x$ and $y$ are neighbours in $\mathbb Z^d$. We often denote the elements of $E(\Lambda)$ as $xy=\{x,y\}$. Given an external magnetic field $h\in \mathbb R$, we consider the Hamiltonian 
\begin{equation}
    H_{\Lambda,h}(\varphi)
    :=
    -\sum_{\substack{xy\in E(\Lambda)}} \varphi_x\varphi_y - h\sum_{x\in \Lambda}\varphi_x, 	
    \qquad \forall \varphi\in \R^V.
\end{equation}
The $\varphi^4$ model on $\Lambda$ with external magnetic field $h$ and at inverse temperature $\beta\geq0$ is the probability measure $\nu_{\Lambda,\beta,h}$ (also denoted $\langle\cdot\rangle_{\Lambda,\beta,h})$ given by
\begin{equation}
    \mathrm{d}\nu_{\Lambda,\beta,h}(\varphi)
    \propto
    e^{-\beta H_{\Lambda,h}(\varphi)}\textup{d}\rho_{\Lambda,g,a}(\varphi),
\end{equation}
where $\textup{d}\rho_{\Lambda,g,a}(\varphi):=\prod_{x\in \Lambda}\textup{d}\rho_{g,a}(\varphi_x)$.
We may omit $h$ from the notation when $h=0$. We consider the infinite volume measures
\begin{equation}
    \nu_{\beta,h}=\lim_{\Lambda\uparrow \Z^d} \nu_{\Lambda,\beta,h}.
\end{equation}
It is classical that this limit exists and is monotonic increasing in $h$, see e.g.~\cite{GriffithsCorrelationsIsing1-1967}. 
We can then define the so-called \emph{plus measure} $\nu_\beta^+$ 
via the following monotonic weak limit:
\begin{equation}\label{eq:phi4_plus_def}
    \nu_\beta^+=\lim_{h\downarrow 0} \nu_{\beta,h}.
\end{equation}
The measure $\nu_\beta^+$ is translation invariant, ergodic and extremal. In fact, every translation invariant Gibbs measure for the $\varphi^4$ model is a linear combination of the two extremal measures $\nu_\beta^+$ and $\nu_\beta^-(\textup{d}\varphi):=\nu_\beta^+(-\textup{d}\varphi)$, see \cite{GunaratnamPanagiotisPanisSeveroPhi42022}.

The \emph{spontaneous magnetisation} at inverse temperature $\beta$ is defined as
\begin{equation}
    m^*(\beta)
    :=
    \langle \varphi_0 \rangle^+_\beta,
\end{equation}
where $\langle f\rangle^+_\beta$ denotes the expectation of $f$ with respect to $\nu^+_\beta$.
We can now define the \emph{critical point}
\begin{equation}
    \beta_c=\beta_c(g,a)
    :=
    \inf\lbrace \beta\geq0:\text{ } m^*(\beta)>0\rbrace.
\end{equation}
It is known that $\beta_c\in (0,\infty)$ for every $d\geq 2$ and all $(g,a)\in (0,\infty)\times \mathbb R$  \cite{GlimmJaffeSpencer1975PHI4,wells1977some} (alternatively, this can be proven by directly using the random cluster representation of Section \ref{sec:random_cluster}). 
Another quantity of interest is the \emph{empirical magnetisation} on $\Lambda$ given by
\begin{equation}
    m_{\Lambda}:= \frac{1}{|\Lambda|} \sum_{x\in\Lambda} \varphi_x.
\end{equation}
We will mostly be interested in the empirical magnetisation on boxes $\Lambda_n:=\{-n,\dots,n\}^d$ for $n \in \mathbb N^*$. Our first theorem is the following surface order large deviation estimate for the empirical magnetisation in the entire supercritical regime.

\begin{theorem}\label{thm:ldp free}
    For every $d\geq2$, $\beta > \beta_c$ and $\delta\in(0,m^*(\beta) )$, there exist constants $c,C\in(0,\infty)$ such that for every $n$ large enough,
    \begin{equation}\label{eq:ldp surface}
    e^{-Cn^{d-1}}\leq \nu_{\Lambda_n,\beta}\Big[ |m_{\Lambda_n}|\leq m^*(\beta)-\delta\Big]\leq e^{-cn^{d-1}}.
    \end{equation}
\end{theorem}

Surface order large deviations were first established for the two-dimensional Ising model \cite{schonmann1987second} at very low temperatures (i.e.\ $\beta\gg \beta_c$). The result was later extended to all $\beta > \beta_c$ \cite{CCS87}, and further generalised to higher dimensions \cite{Pis96}. Surface order large deviations are closely related to phase segregation phenomena and the emergence of Wulff shapes
\cite{dobrushin1992wulff,ioffe1995exact,cerf2000wulff,bodineau2000rigorous}. In the case of the $\varphi^4$ model, the techniques of \cite{CGW22} can be adapted to establish the surface order bounds in Theorem \ref{thm:ldp free} at very low temperatures. Our contribution is to extend them to every $\beta > \beta_c$.

\begin{remark} \label{rem: vol order}
The lower large deviation of \eqref{eq:ldp surface} differs from the upper large deviation, which is always of volume order. Indeed, it is not hard to prove (see Section~\textup{\ref{sec:coarse_graining}}) that for every $\beta\geq0$ and $\delta>0$, there exist $c,C>0$ such that for every $n$ large enough
\begin{equation}\label{eq:ldp volume}
     e^{-Cn^d}\leq \nu_{\Lambda_n,\beta}\Big[ |m_{\Lambda_n}|\geq m^*(\beta)+\delta \Big]\leq e^{-cn^d}.
     \end{equation}
Both large deviation estimates of \eqref{eq:ldp surface} and \eqref{eq:ldp volume} can be extended to the infinite volume measure. See also \cite{comets1986grandes, olla1988large}.
\end{remark}

\subsection{Application to dynamical phase transitions}

Dynamical spin models are a cornerstone of non-equilibrium statistical physics and theoretical computer science. Amongst these, \emph{Glauber dynamics} for the Ising model--- or \emph{stochastic Ising models}--- form an important class.
In finite volume, these dynamics are continuous-time (or discrete-time) Markov chains that converge in law to the Ising model. It is classical by the Perron--Frobenius theorem that these chains converge to equilibrium exponentially fast, with a volume-dependent rate governed by the \emph{spectral gap}--- which is the inverse of \emph{relaxation time} and closely related to the \emph{mixing time}. See \cite{martinelli99} for more information. 

A central question of the field is to understand the interplay between the phase transition of the equilibrium model and the behaviour of the spectral gaps of the dynamics as the volume tends to infinity. For example, the dynamical Ising model with free boundary conditions exhibits a phase transition at the critical point $\beta_c$ of the equilibrium model. In the subcritical phase $\beta<\beta_c$, the spectral gaps remain bounded away from $0$ as the system size grows. Moreover, the dynamics exhibit the \emph{cutoff phenomenon} (see \cite{LubSly13,LubSly16} and references therein). In the supercritical phase $\beta > \beta_c$, the spectral gaps decay exponentially with surface order rate \cite{cesi1996two}. At $\beta_c$, the dynamics is conjectured to exhibit a \emph{critical slowdown}, whereby spectral gaps decay polynomially fast. Polynomial bounds have been rigorously established in $d=2$ \cite{LubSly12}, but remain open in higher dimension (see, however, a conditional result in this direction \cite{Hu2025} which depends on the conjectured polynomial decay of the finite volume magnetisation at criticality).

Similar phase transitions are expected to occur for dynamical $\varphi^4$ models. We consider two natural dynamics (continuous-time Markov processes). The first is called \emph{Langevin dynamics}, also known as \emph{stochastic quantisation equations}, and corresponds to the Markov process generated by solutions of a system of stochastic differential equations. These arise as discretisations of $\varphi^4$ singular stochastic partial differential equations, which have received significant interest in recent years (see \cite{CGW22} and references therein). The second is the \emph{heat-bath dynamics}, a single-site update process more reminiscent of the Glauber dynamics for the Ising model. This corresponds to the process where the spins are resampled according to jump rates determined by the equilibrium measure conditional on the values of its neighbours. 

As in the case of the Ising model, both of these dynamics exhibit exponential convergence to equilibrium, with rate governed by their spectral gaps. We postpone formal definitions  to Section \ref{section:spectral gap}, and state our main theorem concerning the dynamics. 

\begin{theorem} \label{theorem: dynamics}
Let $d \geq 2$ and $\beta > \beta_c$. There exists $c>0$ such that for every $n$ sufficiently large,
\begin{equation}
\lambda(\Lambda_n,\beta)	\leq e^{-cn^{d-1}},
\end{equation}
where $\lambda(\Lambda_n,\beta)$ is the spectral gap of either Langevin dynamics or heat-bath dynamics of $\nu_{\Lambda_n,\beta}$.
\end{theorem}

Theorem \ref{theorem: dynamics} is a direct application of Theorem \ref{thm:ldp free} once a variational characterisation of $\lambda(\Lambda_n,\beta)$ is known. In the case of Langevin dynamics, our result was already known in the very low temperature regime $\beta \gg \beta_c$ \cite{CGW22} (strictly speaking \cite{CGW22} established the spectral gap decay with the continuum dynamics in mind, although it is straightforward to check that their argument extends to the lattice dynamics considered here), and our contribution is to extend this to every $\beta>\beta_c$. Furthermore, for $
\beta<\beta_c$, it was recently shown that the spectral gaps remain bounded away from $0$ as $n \rightarrow \infty$ for the Langevin dynamics \cite{BD24}. We expect the same to hold for heat-bath dynamics as well. Combining this with our result would then establish the existence of a (sharp) dynamical phase transition for both of these dynamics at $\beta_c$.

\subsection{Random cluster for $\varphi^4$ and supercritical sharpness}

The main tool used in this work is a percolation representation of the $\varphi^4$ model, which plays a role analogous to that of the random cluster representation of the Ising model. In Section~\ref{sec:random_cluster}, we define the model in detail and derive some of its basic properties, such as a coupling with the spin system, monotonicity in boundary conditions and correlation inequalities. 
We only give a rough description of the model here and defer the reader to Section~\ref{sec:random_cluster} for details.

Conditionally on the absolute value field $|\varphi|$, the sign field $\mathrm{sign}(\varphi)$ can be seen as an Ising model with random coupling constants $J(|\varphi|)_{x,y}=\beta |\varphi_x||\varphi_y|$. Therefore, we can consider the corresponding random cluster representation in a random environment. For every $\Lambda\subset \Z^d$, let $\Psi_{\Lambda,\beta}$ be the probability measure on configurations $(\omega,\mathsf{a})\in \{0,1\}^{E(\Lambda)} \times (\mathbb R^+)^\Lambda$, where $\mathsf{a}$ is distributed as $|\varphi|$ (under $\nu_{\Lambda,\beta}$), and $\omega$--- conditionally on $\mathsf{a}$--- is distributed as an Ising random cluster model with coupling constants $J_{x,y}=\beta \mathsf{a}_x \mathsf{a}_y$. Certain correlation functions of $\varphi$ can then be expressed in terms of connectivity properties for $\Psi_{\Lambda,\beta}$. For example, we have 
\begin{equation}
\langle \varphi_x \varphi_y\rangle_{\Lambda,\beta}= \Psi_{\Lambda,\beta}[ \mathsf{a}_x \mathsf{a}_y \mathbbm{1}\{x \overset{\omega}{\longleftrightarrow} y\}].    
\end{equation}
We can also construct random cluster measures $\Psi^\#_{\Lambda,\beta}$ with general \emph{boundary conditions} $\#$, see Section~\ref{sec:random_cluster} for details.

Similarly to the case of the Ising model, the crucial step to establish Theorem~\ref{thm:ldp free} is to prove that the percolation model $\Psi_{\Lambda,\beta}$ is ``well behaved'' for every $\beta>\beta_c$. In other words, we establish a \emph{supercritical sharpness} result for $\Psi_{\Lambda,\beta}$. There are several ways of defining supercritical sharpness. In dimensions $d\geq3$, a classical way is to consider the notion of \emph{slab percolation}, as done in \cite{GM90} for Bernoulli percolation and in \cite{Bod05} for the Ising model. However, this notion becomes meaningless in 2D, where classically the crossing of long rectangles is used instead. 
In order to state our result for all dimensions $d\geq 2$ in a unified way, we choose to consider the notion of \emph{local uniqueness}. Furthermore, this notion is more handy when performing renormalisation arguments to prove non-perturbative results.

For every scale $L\geq1$, consider the \emph{local uniqueness event} $U(L)$ of percolation configurations $\omega\in \{0,1\}^{E(\Lambda_{8L})}$ such that there exists at least one cluster of $\omega$ crossing the annulus $\Lambda_{8L}\setminus \Lambda_L$, and any two paths in $\omega$ crossing the annulus $\Lambda_{4L}\setminus \Lambda_{2L}$ are connected in $\omega\cap(\Lambda_{8L}\setminus \Lambda_L)$. Our main result is the following.

\begin{theorem}\label{thm: local uniqueness}
For every $d\geq 2$ and $\beta>\beta_c$, 
\begin{equation}
\lim_{L\to\infty}\inf_{\#}\Psi^{\#}_{\Lambda_{10L},\beta}[U(L)]= 1.    
\end{equation}
\end{theorem}

Theorem~\ref{thm: local uniqueness} implies that the percolation model $\Psi_{\Lambda,\beta}^{\#}$ is ``well behaved'' for every $\beta>\beta_c$. Indeed, it allows to compare the model with a highly supercritical Bernoulli percolation via standard renormalisation arguments. For instance, one can prove that the unique giant component is robust and ubiquitous, while the other components are tiny (namely logarithmically small). 
This is the key feature underlying the proof Theorem~\ref{thm:ldp free}, which combines Theorem~\ref{thm: local uniqueness} and a coarse graining argument. Another consequence of Theorem~\ref{thm: local uniqueness} is that the infinite volume plus measure $\nu_\beta^+$ can be obtained as the limit of $\varphi^4$ measures on $\Lambda_n$ with boundary condition $\eta_n\gg 1/n^{d-1}$, see Proposition \textup{\ref{prop: weak plus measure}}.
As mentioned above, we expect Theorem~\ref{thm: local uniqueness} to find several other applications concerning the supercritical behaviour of the $\varphi^4$ model.

\subsection{Outline of the proof}

We will now give a more detailed description of the proofs of our theorems, starting with Theorem~\ref{thm: local uniqueness}. At a global level, the strategy of proof is similar to that of \cite{Bod05,Severo24}.  However, several new ideas are necessary in our context in order to deal with extra difficulties, especially coming from the unboundedness of the spins. The main novelty is the use of the random tangled current representation of \cite{GunaratnamPanagiotisPanisSeveroPhi42022} in order to deal with some of these challenges.

The first step is to prove that a certain notion of surface tension for the $\varphi^4$ model is positive for every $\beta>\beta_c$. We adapt the corresponding argument of Lebowitz and Pfister \cite{lebowitz1981surface} for the Ising model. Roughly speaking, the surface tension measures the energetic cost of imposing a surface of disagreeing spins separating two halves of a box. For the Ising model, the surface tension can be rigorously defined as simply 
$$\tau_\beta:=-\lim_{L\to\infty}\frac{1}{L^{d-1}}\log \frac{Z^{+,-}_{\Lambda_L,\beta}}{Z^+_{\Lambda_L,\beta}},$$ where $Z^{+,-}_{\Lambda_L,\beta}$ and $Z^{+}_{\Lambda_L,\beta}$ denote, respectively, the Ising partition functions with Dobrushin (i.e.~$+1$ on the top and $-1$ on the bottom) and plus (i.e.\ $+1$ everywhere) boundary conditions. For the $\varphi^4$ model, while we could in principle define the surface tension in the same way, the proof of positivity from \cite{lebowitz1981surface} would only adapt if we knew that the $\varphi^4$ measure with boundary conditions $+1$ converges to the plus measure $\nu^+_\beta$ defined in \eqref{eq:phi4_plus_def}. This last fact is not easy to prove due to the unboundedness of the spins (however, this does follow, a posteriori, from our results, as proved in Proposition~\ref{prop: weak plus measure}). Even though one can obtain $\nu^+_\beta$ as the limit of finite volume measures with \emph{growing} boundary conditions (see \eqref{eq:conv_plus}), these measures are not ``regular'' (in the sense of Proposition~\ref{prop:regularity}) up to the boundary, which would cause us problems in the next step, where we compare the corresponding random cluster measure with its free counterpart. In order to solve this issue, we prove in Proposition~\ref{prop: thick conv} that the $\varphi^4$ measures with magnetic field equal to $+1$ on a \emph{thick boundary} of thickness $\log L$ (namely $\Lambda_L\setminus \Lambda_{L-\log L}$) converge to $\nu^+_\beta$ as $L$ tends to infinity. This result allows us to approximate $\nu^+_\beta$ by a sequence of measures that are ``regular'' up to the boundary since the magnetic field is bounded. The proof is based on a comparison between boundary conditions, which is derived through the use of the random tangled currents representation constructed in \cite{GunaratnamPanagiotisPanisSeveroPhi42022}.
We then use these ``thick-plus'' measures (and their Dobrushin counterpart) to define the surface tension $\tau_\beta$ of the $\varphi^4$ model, and to prove in Proposition~\ref{prop: surface tension positive} that indeed $\tau_\beta>0$ for every $\beta>\beta_c$. We stress that this proof, adapted from \cite{lebowitz1981surface}, crucially relies on the Ginibre inequality (see Proposition~\ref{prop:Ginibre}), which is specific to the Ising and $\varphi^4$ models.

The second step is to transfer the information that $\tau_\beta>0$ to the $\varphi^4$ random cluster model. It follows rather easily from the Edwards--Sokal coupling that the probability that the top and bottom of a box are disconnected from each other in the corresponding random cluster representation decays exponentially in the surface order (see Lemma~\ref{lem: ratio probablistic expression}). However, we need to prove that the same happens for the random cluster measure with \emph{free} boundary condition, which is in fact the minimal one. We therefore perform a comparison argument to go from the ``thick plus'' random cluster measure to the ``free'' one. In order to do so, we adapt an argument originally used by Bodineau \cite{Bod05} for the Ising model. However, this argument relies on the fact that there exists a unique random cluster measure on the half-space with any positive magnetic field on the boundary. For the Ising model, this follows from a work of Fröhlich and Pfister on the wetting transition \cite{FP87}, together with the Lee--Yang theorem. Instead of extending the results of \cite{FP87} to the $\varphi^4$ model (which remains an open problem), we give a completely different proof--- relying again on the random tangled currents--- of the uniqueness of half-space measures with positive boundary field (see Proposition~\ref{prop: unique half-space measure}). We remark that our proof can be easily adapted to give a new proof of this uniqueness result for the Ising model. 

The two steps described above imply that for every $\beta>\beta_c$, the probability that a box is not crossed by a cluster decays exponentially in the surface order for the $\varphi^4$ random cluster model with minimal (i.e.~free) boundary conditions at a macroscopic distance from the box. Finally, we prove that this implies that local uniqueness happens with high probability. We use two different approaches, depending on the dimension. In dimension $d=2$, we use the Russo--Seymour--Welsh theorem for FKG measures recently obtained by K{\"o}hler-Schindler and Tassion \cite{KohlerTassionGeneralRSW} to deduce that a macroscopic annulus is surrounded by a circuit with high probability, which easily implies the local uniqueness event $U(L)$. In dimensions $d\geq 3$, a different argument is required and we follow the approach of \cite{Severo24}, which is inspired by the work of Benjamini and Tassion \cite{BenjaminiTassion}. We first prove that the surface order exponential bound on disconnection implies local uniqueness with the help of an independent \emph{Bernoulli sprinkling}. By a standard static renormalisation, this implies that this Bernoulli-sprinkled random cluster model (with free boundary conditions) percolates on sufficiently thick slabs. Then, we prove in Proposition~\ref{prop: coupling} that the Bernoulli-sprinkled random cluster measure is stochastically dominated by the original random cluster measure at a slightly higher $\beta$, thus implying slab percolation for the latter. We stress that, although such a domination is trivial for the Ising random cluster model (see e.g.~\cite[Lemma 2.5]{Severo24}), this is not the case in our context since the absolute value field $|\varphi|$ can take arbitrarily small values. 
Once percolation on slabs is proved for the free random cluster measure for all $\beta>\beta_c$, we conclude that local uniqueness holds by a classical ``onion argument'', thus completing the proof of Theorem~\ref{thm: local uniqueness}.

As mentioned above, the surface order large deviation of Theorem~\ref{thm:ldp free} follows from Theorem~\ref{thm: local uniqueness} via a coarse graining argument, very similar to the one that Pisztora \cite{Pis96} developed for the Ising model. Another essential ingredient in the coarse graining argument is the uniqueness of the infinite volume measure for the $\varphi^4$ random cluster model (see Lemma \ref{lem: volume ldp small average}). This is proved in Proposition~\ref{prop: unique gibbs measure} as a consequence of the characterization of translation invariant Gibbs measures for the $\varphi^4$ model obtained in \cite{GunaratnamPanagiotisPanisSeveroPhi42022}.
Since the spins are continuous and unbounded, some extra care is required when running the coarse graining argument in our context, but this is dealt with in a rather simple way by relying on regularity estimates.

\paragraph{Organisation of the paper.} In Section~\ref{sec:preliminaries} we recall the main properties of the $\varphi^4$ model, including some correlation inequalities, a regularity theorem and the switching principle for the random tangled current representation. In Section~\ref{sec:magnetisation} we prove the aforementioned Propositions~\ref{prop: thick conv} and \ref{prop: unique half-space measure} on the convergence of the thick-plus measures to $\nu^+_\beta$ and the uniqueness of half-space measure with positive magnetic field. This section is the only one that uses the random tangled currents. In Section~\ref{sec:random_cluster} we properly define the random cluster representation for the $\varphi^4$ model and prove its main basic properties. Section~\ref{sec:surface_tension} is devoted to the first two steps in the proof of Theorem~\ref{thm: local uniqueness} described above, while Section~\ref{sec:uniqueness} is concerned with the final step. In Section~\ref{sec:coarse_graining} we deduce Theorem~\ref{thm:ldp free} from Theorem~\ref{thm: local uniqueness}. Finally, in Section~\ref{section:spectral gap} we use Theorem \ref{thm:ldp free} to prove Theorem \ref{theorem: dynamics}.

\paragraph{Acknowledgements.} We thank the University of Geneva, where the project was initiated, for its hospitality. TSG is supported by the Department of Atomic Energy, Government of India, under project no.12-R\&D-
TFR-5.01-0500, and supported in part by an endowment of the Infosys Foundation. RP acknowledges the support of the Swiss National Science Foundation through a Postdoc.Mobility grant, the NCCR SwissMAP, and the European Research Council (ERC) under the European Union’s Horizon 2020 research and innovation programme (grant agreement No.\ 757296). FS was supported by the ERC grants CriSP (No.\ 851565) and Vortex (No.\ 101043450).

\section{Preliminaries}\label{sec:preliminaries}
In this section, we recall some classical definitions and properties related to the $\varphi^4$ model. For further details and generalisations, see \cite{GunaratnamPanagiotisPanisSeveroPhi42022}.

\subsection{The $\varphi^4$ model on $\mathbb Z^d$}\label{sec:definition_phi4}

Let $\Lambda$ be a finite subset of $\mathbb Z^d$. Recall that $E(\Lambda)=\{\{x,y\}\subset \Lambda: x\sim y\}$ and that we write the elements of $E(\Lambda)$ as $xy$. Let $g>0$ and $a\in \mathbb R$. Since $g,a$ will be fixed for the entire article, we drop them from the notation. The ferromagnetic $\varphi^4$ model on $\Lambda$ (with parameters $(g,a)$) at inverse temperature $\beta\geq0$ and coupling strengths $J=(J_e)_{e\in E(\Lambda)} \in(\mathbb R^+)^{E(\Lambda)}$, and with external magnetic field $\mathsf{h}=(\mathsf{h}_x)_{x\in \Lambda}\in \mathbb R^\Lambda$ is the finite volume probability measure $\nu_{\Lambda,\beta,\mathsf{h},J}$ whose expectation values are defined for $F:\mathbb R^
\Lambda \rightarrow \mathbb R$ bounded and measurable by
\begin{equation}
    \langle F(\varphi)
    \rangle_{\Lambda,\beta,\mathsf{h},J}:=\frac{1}{Z^{\varphi^4}_{\Lambda,\beta,\mathsf{h},J}}\int F(\varphi)\exp\left(-\beta H_{\Lambda,\mathsf{h},J}(\varphi)\right) \textup{d}\rho_{\Lambda,g,a}(\varphi),
\end{equation}
where $Z^{\varphi^4}_{\Lambda,\beta,\mathsf{h},J}$ is the partition function which guarantees that $\langle 1\rangle_{\Lambda,\beta,\mathsf h,J}=1$, and the Hamiltonian is given by 
\begin{equation}
    H_{\Lambda,\mathsf{h},J}(\varphi):=-\sum_{xy\in E(\Lambda)}J_{xy}\varphi_x\varphi_y-\sum_{x\in \Lambda}\mathsf{h}_x\varphi_x.
\end{equation}

We let $\langle \cdot\rangle_0$ be the expectation with respect to the measure $\rho_{g,a}$. Unless specified otherwise, we will consider the homogeneous case $J\equiv 1$ (i.e.\ $J_e=1$ for every $e\in E(\Lambda)$) and drop $J$ from the notation.

The above definition includes the possibility of having \emph{boundary conditions}. Let us first introduce a useful notation. The (inner) vertex boundary of $\Lambda$ is the set $\partial \Lambda := \{ x \in \Lambda: \exists y \in \Lambda^c, x\sim y\}$.  Boundary conditions are interpreted as (and are synonymous in this paper with) inhomogeneous external magnetic fields $\mathsf h$ that satisfy $\mathsf{h}_x=0$ for $x\in {\rm Int}(\Lambda):=\Lambda\setminus \partial \Lambda$. 

Let us now discuss an important property satisfied by these measures: the domain Markov property. We begin by introducing some useful notation. Let $\Delta\subset \Lambda$ be finite subsets of $\mathbb Z^d$. If $\varphi\in \mathbb R^\Lambda$, we define a magnetic field $\mathsf{h}^{\Lambda,\Delta}_\varphi$ on $\Delta$ (or, the boundary condition on $\partial \Delta$ induced by $\varphi$) as follows: for all $x\in \Delta$,
\begin{equation}\label{eq:bc->mag_field}
\mathsf{h}^{\Lambda,\Delta}_\varphi(x):=\mathbbm{1}_{x\in \partial \Delta}\sum_{\substack{y\in \Lambda\setminus \Delta\\xy\in E(\Lambda)}}\varphi_y.
\end{equation}
We denote by $\nu_{\Lambda,\beta,\mathsf{h}}^{(\Lambda\setminus\Delta)}$ the push-forward of $\nu_{\Lambda,\beta,\mathsf{h}}$ under the restriction map $(\varphi_x)_{x\in \Lambda}\mapsto (\varphi_x)_{x\in \Lambda\setminus \Delta}$.
If $\mathsf{h}\in \mathbb R^\Lambda$ and $\beta\geq0$, the measure $\nu_{\Lambda,\beta,\mathsf{h}}$ introduced above satisfies, for every $f:\mathbb R^{\Delta}\rightarrow \mathbb R^+$,
\begin{equation}\label{eq:dlr}
    \nu_{\Lambda,\beta,\mathsf{h}}[f]=\int_{\mathbb R^{\Lambda\setminus \Delta}}\nu_{\Delta,\beta,\mathsf{h}+\mathsf{h}^{\Lambda,\Delta}_\varphi}[f]\mathrm{d}\nu_{\Lambda,\beta,\mathsf{h}}^{(\Lambda\setminus\Delta)}(\varphi).
\end{equation}

\subsection{Correlation inequalities}\label{sec:correlation_ineq}

We now present various classical correlation inequalities that will be used in this paper (see \cite{GunaratnamPanagiotisPanisSeveroPhi42022} and references therein for proofs and background). We begin with correlations of increasing functions.

\begin{proposition}[FKG inequalities, {\cite[Propositions~A.1--A.2]{GunaratnamPanagiotisPanisSeveroPhi42022} and \cite[Corollary~6.4]{LammersOtt2021}}] \label{prop: FKG phi4} Let $\Lambda \subset \mathbb Z^d$ be finite, $\beta\geq0$, $J\in (\mathbb R^+)^{E(\Lambda)}$, and $\mathsf{h}\in \mathbb R^\Lambda$. Then, for any increasing and bounded functions $F,G:\R^\Lambda\to \R$,
\begin{equation}
    \langle F(\varphi)G(\varphi)\rangle_{\Lambda,\beta,\mathsf{h},J}\geq \langle F(\varphi)\rangle_{\Lambda,\beta,\mathsf{h},J}\langle G(\varphi)\rangle_{\Lambda,\beta,\mathsf{h},J}.
\end{equation}
The same holds for the absolute value field if $\mathsf h_x \geq 0$ for every $x \in \Lambda$, namely
\begin{equation}
    \langle F(|\varphi|)G(|\varphi|) \rangle_{\Lambda,\beta,\mathsf{h},J}\geq \langle F(|\varphi|) \rangle_{\Lambda,\beta,\mathsf{h},J} \langle G(|\varphi|) \rangle_{\Lambda,\beta,\mathsf{h},J}.
\end{equation}
\end{proposition}

As a consequence of the FKG inequalities above, we obtain the following monotonicity properties of correlations: the first inequality is standard, and the second inequality follows from a straightforward adaptation of the proof of \cite[Lemma 2.13]{GunaratnamPanagiotisPanisSeveroPhi42022}. 

\begin{proposition}\label{prop:stoc_monotonicity}
Let $\Lambda \subset \mathbb Z^d$ be finite, and $\beta\geq0$. Let also $J_1,J_2\in (\mathbb R^+)^{E(\Lambda)}$ be such that $J_1\leq J_2$, and $\mathsf{h}_1,\mathsf{h}_2\in \mathbb R^\Lambda$ be such that $\mathsf{h}_1\leq \mathsf{h}_2$. Then, for any increasing function $F:\mathbb{R}^{\Lambda}\to\mathbb{R}$ we have 
\begin{equation}
\langle F(\varphi) \rangle_{\Lambda,\beta,\mathsf{h}_1,J_1}\leq \langle F(\varphi) \rangle_{\Lambda,\beta,\mathsf{h}_2,J_2}.  
\end{equation}
The same holds for the absolute value field if $|\mathsf{h}_1| \leq \mathsf{h}_2$, namely for any increasing function $F: (\mathbb{R}^+)^{\Lambda}\to\mathbb{R}$ we have 
\begin{equation}
\langle F(|\varphi|) \rangle_{\Lambda,\beta,\mathsf{h}_1,J_1}\leq \langle F(|\varphi|) \rangle_{\Lambda,\beta,\mathsf{h}_2,J_2}.  
\end{equation}
\end{proposition}

We now turn to spin correlations. In what follows, given $\Lambda\subset\mathbb{Z}^d$ finite and $A:\Lambda\rightarrow \mathbb{N}$, we write $\varphi_A\coloneqq \prod_{x\in \Lambda} \varphi_x^{A_x}$.

\begin{proposition}[Monotonicity in coupling strength, {\cite[Proposition 3.18 and Remark 3.19]{GunaratnamPanagiotisPanisSeveroPhi42022}}] \label{prop:monotonicity}
Let $\Lambda \subset \mathbb Z^d$ be finite, $\beta\geq0$, $J\in (\mathbb R^+)^{E(\Lambda)}$, $\mathsf{h}\in \mathbb (\mathbb R^+)^\Lambda$, and $A: \Lambda\to \N$. Then, for every $e\in E(\Lambda)$, the function
$J_e\mapsto \langle \varphi_A\rangle_{\Lambda,\beta,\mathsf{h},J}$ is increasing.
\end{proposition}

\begin{proposition}[Ginibre inequality, {\cite[Proposition~2.11]{GunaratnamPanagiotisPanisSeveroPhi42022}}]\label{prop:Ginibre} Let $A,B: \Lambda\rightarrow \N$. Then, for any $\mathsf{h}_1, \mathsf{h}_2\in (\mathbb R^+)^\Lambda$ such that $|\mathsf{h}_1|\leq \mathsf{h}_2$,
\begin{equation}\label{eq:Ginibre}
    \langle \varphi_A\varphi_B\rangle_{\Lambda,\beta,\mathsf{h}_2,J}-\langle \varphi_A\varphi_B\rangle_{\Lambda,\beta,\mathsf{h}_1,J}\geq \left|\langle \varphi_A\rangle_{\Lambda,\beta,\mathsf{h}_2,J}\langle \varphi_B\rangle_{\Lambda,\beta,\mathsf{h}_1,J}-\langle \varphi_A\rangle_{\Lambda,\beta,\mathsf{h}_1,J}\langle\varphi_B\rangle_{\Lambda,\beta,\mathsf{h}_2,J}\right|.
\end{equation}
\end{proposition}

\begin{remark} 
In \textup{\cite{GunaratnamPanagiotisPanisSeveroPhi42022}}, we only proved Proposition \textup{\ref{prop:Ginibre}} when $\mathsf{h}_1,\mathsf{h}_2\equiv 0$ in $\Lambda\setminus \partial \Lambda$, i.e.\ when $\mathsf{h}_1,\mathsf{h}_2$ can be seen as boundary conditions, but the proof adapts mutatis mutandis. 
\end{remark}

\subsection{Quartic tail bounds}\label{sec:regularity}

For the $\varphi^4$ model, since spins are unbounded, there is a priori no notion of a maximal or minimal boundary condition. However, it is natural to expect that, similarly to the single site measure $\rho_{g,a}$, the finite volume measures $\nu_{\Lambda,\beta,\mathsf{h},J}$ also have tails of the type $e^{-c\varphi^4}$. In this case, one can deduce that, with high probability, every spin in a domain $\Lambda$ is smaller than $C(\log |\Lambda|)^{1/4}$ in absolute value, for $C$ large enough. This suggests that $\pm C(\log |\Lambda|)^{1/4}$ are good candidates for the maximal and minimal boundary condition. This is indeed the case and follows from a regularity property encapsulated in Proposition~\ref{prop:regularity} below. Recall that for $\Delta\subset \Lambda$, $\nu_{\Lambda,\beta,\mathsf{h}}^{(\Delta)}$ denotes the push-forward of $\nu_{\Lambda,\beta,\mathsf{h}}$ under the restriction map $(\varphi_x)_{x\in \Lambda}\mapsto (\varphi_x)_{x\in \Delta}$. Additionally, we let $\mathrm{d}\varphi_\Delta:=\prod_{x\in \Delta}\mathrm{d}\varphi_x$. The following proposition states that the measures $\nu_{\Lambda_L,\beta,\mathsf h}$ have (uniform) tails of the form $e^{-c\varphi^4}$ at vertices that are sufficiently far away from locations where $\mathsf{h}$ takes large values. Following the terminology used in \cite{GunaratnamPanagiotisPanisSeveroPhi42022}, we will refer to this result as a \emph{regularity} estimate. See Figure \ref{fig:regularity} for an illustration.

Recall that $|\cdot|$ denotes the infinite norm on $\mathbb R^d$, and that for every $L \in \mathbb N^*$, $\Lambda_L= \{-L,\dots,L\}^d = \{ x \in \mathbb Z^d : |x|\leq L\}$, and for every $z \in \mathbb Z^d$, $\Lambda_L(z)=z+\Lambda_L$. Furthermore, for every $z \in \mathbb Z^d$ and $A \subset \mathbb Z^d$, we write ${\rm dist}^\infty(z,A)= \inf \{ |z-a| : a \in A \}$ with the convention that ${\rm dist}^\infty(z,\emptyset) = \infty$.

\begin{proposition}\label{prop:regularity}
For every $\beta,h\geq0$, there exists a constant $C=C(\beta,h)\in(0,\infty)$, depending continuously on the parameters,
such that the following holds.
Let $L\geq1$, and $\mathsf h \in \mathbb R^{\Lambda_L}$ be such that $|\mathsf h_x|\leq h$ for every $x \in \Lambda_L\setminus \partial \Lambda_L$ and $\mathsf |\mathsf h_x| \leq \sqrt{\log L}$ for every $x \in \partial \Lambda_L$. Then, for every $z\in \Lambda_L$ such that $M(z):=\mathrm{dist}^\infty(z,\{x\in \partial\Lambda_L:~|\mathsf{h}_x|>h\})\geq L$, 

\begin{equation}
\label{eq:regularity}
\nu^{(\Delta)}_{\Lambda_L, \beta, \mathsf h} (|\cdot|)\preccurlyeq \rho_{\Delta,g/2,0}(C+|\cdot|),
\end{equation}
where $\preccurlyeq$ refers to stochastic domination of measures, $\Delta\subset \Lambda_{m-\sqrt{m}}(z)\cap\Lambda_L$ for $m:=\min\{M(z),2L\}$, and 
$\mu(f(\cdot))$ denotes the push-forward of the measure $\mu$ with respect to the function $f$. Moreover, $\langle \varphi^2_x \rangle_{\Lambda_L,\beta,\mathsf{h}}\leq C\sqrt{\log L}$ for every $x\in \Lambda_L$. 
\end{proposition}
\begin{figure}[htb]
    \centering
    \includegraphics[width=0.7\linewidth]{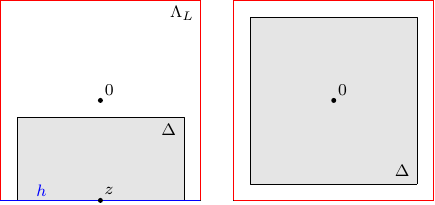}
    \put(-335,30){\color{red}$\sqrt{\log L}$}
    \caption{An illustration of two applications of Proposition \ref{prop:regularity}. On both sides, the magnetic field $\mathsf{h}$ is equal to $0$ in $\Lambda_L\setminus\partial \Lambda_L$. On the left: the part of $\partial \Lambda_L$ in red carries a ``large'' magnetic field $\mathsf{h}_x=\sqrt{\log L}$, while the blue part carries a constant magnetic field $\mathsf{h}_x=h$. With the notations of Proposition \ref{prop:regularity}, the point $z$ satisfies $M(z)\geq L$. The restriction of $\nu_{\Lambda_L,\beta,\mathsf{h}}$ to the shaded grey region $\Delta$ has quartic tails. On the right: $\mathsf{h}_x=\sqrt{\log L}$ for $x \in \partial \Lambda_L$. Here, $M(0)\geq L$ and the restriction of $\nu_{\Lambda_L,\beta,\mathsf{h}}$ to the shaded grey region $\Delta$ has quartic tails.}
    \label{fig:regularity}
\end{figure}

\begin{proof}
Let $z$ and $\Delta\subset\Lambda_{m-\sqrt{m}}(z)\cap \Lambda_L$ be as in the statement, and define $\Lambda=\Lambda_L\cup\Lambda_m(z)$. By Proposition~\ref{prop:stoc_monotonicity}, 
\begin{equation}
\nu^{(\Delta)}_{\Lambda_L, \beta, \mathsf h} (|\cdot|)\preccurlyeq \nu^{(\Delta)}_{\Lambda, \beta, \tilde{\mathsf h}} (|\cdot|),    
\end{equation}
where $\tilde{\mathsf{h}}_x=\max\{\mathsf{h}_x,h\}$ for $x\in \Lambda_L$, and $\tilde{\mathsf{h}}_x=h$ for $x\in \Lambda\setminus \Lambda_L$. Note that the only vertices where $\tilde{\mathsf{h}}>h$ lie at $\partial \Lambda$, and $|\mathsf{h}_x|\leq \sqrt{\log (\mathrm{dist}^\infty(z,\partial \Lambda))}$ for every $x\in \partial \Lambda$.
We can thus apply \cite[Corollary D.12]{GunaratnamPanagiotisPanisSeveroPhi42022} (for the translation of $\Lambda$ by $-z$) to deduce that the density of $\nu^{(\Delta)}_{\Lambda, \beta, \tilde{\mathsf h}}$ (with respect to the product Lebesgue measure) is upper bounded by $e^{-\sum_{x\in \Delta}(\frac{g}{2} \varphi^4_x-A)}$ for some constant $A>0$. 
Note that there exists $C>0$ large enough such that for every $u=(u_x)_{x \in \Delta} \in (\mathbb R^+)^\Delta$, letting $I(u)=\prod_{x \in \Delta}[C+u_x,\infty)$, we have that 
\begin{equation}
\begin{aligned}
\int_{I(u)}e^{-\sum_{x \in \Delta}(\frac{g}{2} \varphi^4_x-A)} \prod_{x\in \Delta}\mathrm{d}\varphi_x&\leq \rho_{\Delta,g/2,0}(C+|\varphi_x|\geq C+u_x ,\, \forall x\in \Delta) \\ &= \rho_{\Delta,g/2,0}(|\varphi_x|\geq u_x ,\, \forall x\in \Delta).  
\end{aligned}
\end{equation}
This implies the desired stochastic domination \eqref{eq:regularity}. The second part of the statement follows from \cite[Remark D.13]{GunaratnamPanagiotisPanisSeveroPhi42022}.
\end{proof}

\begin{remark}Let us remark that the argument of Lemma~\textup{\ref{lem: cond regularity}} can be adapted to give an alternative proof of the statement in the case where $\Delta$ consists of a single vertex.
\end{remark}

\begin{remark}\label{rem:boundary regularity}
 Proposition~\textup{\ref{prop:regularity}} implies that regularity estimates hold far away from the boundary, and even at boundary vertices that are sufficiently far away from vertices where $\mathsf{h}$ is large. In the particular case where $|\mathsf{h}_x|\leq h$ for every $x\in \Lambda_L$, the regularity estimates hold everywhere in $\Lambda_L$. 
\end{remark}
\begin{remark}\label{rem:tightness}
 Proposition~\textup{\ref{prop:regularity}} implies that the family of measures considered is tight. We will use this several times (together with monotonicity arguments) to prove that certain sequences of measures converge to an infinite volume measure. Moreover, it will follow immediately that in these cases the corresponding infinite volume measure inherits the bound \eqref{eq:regularity}.
\end{remark}

As mentioned above, a consequence of Proposition~\ref{prop:regularity} is that we can make sense of maximal (and minimal) boundary conditions as follows. Let us write $\partial^{\rm ext}\Lambda := \{ x \in \Lambda^c: \exists \; y \in \Lambda, x \sim y \}$ to denote the exterior boundary of $\Lambda$ and $\overline \Lambda:= \Lambda \sqcup \partial^{\rm ext}\Lambda$.
It follows from a union bound that there exists a large enough constant $C_{0}=C_0(h)\in(0,\infty)$ such that, setting $\mathfrak{M}_\Delta:=C_{0} (1\vee\log |\Delta|)^{1/4}$, one has  
\begin{equation}\label{eq:maximum_reg}
    \nu_{\Lambda_L,\beta,\mathsf{h}} \left( \max_{x\in\overline \Delta} |\varphi_x| \leq \mathfrak{M}_\Delta \right) = 1-o(1),
\end{equation}
where $o(1)$ tends to $0$ uniformly in $\Delta$ of the form $\overline \Delta\subset \Lambda_{m-\sqrt{m}}(z)\cap\Lambda_L$ as $|\Delta|\rightarrow\infty$, with $m=\min\{M(z),2L\}$ and $z$ satisfying $M(z)=\mathrm{dist}^\infty(z,\{x\in \partial\Lambda_L:~|\mathsf{h}(x)|>h\})\geq L$. 
Having the Markov property \eqref{eq:dlr} in mind, we can therefore think of $\mathfrak{M}_\Delta$ as the effectively maximal boundary condition in $\Delta$.
It follows from \eqref{eq:dlr}, \eqref{eq:maximum_reg} and monotonicity in $\mathsf{h}$ (see Proposition~\ref{prop:stoc_monotonicity}), that for every increasing event $E$,
\begin{equation}
    \nu_{\Lambda_L,\beta,\mathsf{h}}(E)\leq \nu_{\Delta,\beta,\mathfrak{p}_\Delta}(E) + o(1),
\end{equation}
where, again, $o(1)$ tends to $0$ uniformly in $\Delta$ chosen as above as $|\Delta|\to\infty$, and where 
\begin{equation}
    \mathfrak{p}_{\Delta}(x):= \mathsf{h}^{\overline{\Delta},\Delta}_{\mathfrak M_\Delta}(x),
\end{equation}
with $\mathsf{h}_{\varphi}^{\overline{\Delta},\Delta}$ defined in \eqref{eq:bc->mag_field}.
When clear from context, we omit the subscript $\Delta$ in the notation $\mathfrak{p}_\Delta$.
Then, as shown in \cite{Ruelle1970,LebowitzPresutti1976} (see also \cite{GunaratnamPanagiotisPanisSeveroPhi42022} in the context of general graphs of polynomial growth), one has 
\begin{align}\label{eq:conv_plus}
\nu_{\Lambda_L,\beta,\mathfrak{p}_{\Lambda_L}}&\underset{L\to\infty}{\longrightarrow} \nu^+_\beta,
\end{align}
where $\nu^+_\beta$ is the translation-invariant, extremal probability measure $\nu^+_\beta$ introduced in Section \ref{section:intro surface order large deviations}. Analogously, we have $\nu_{\Lambda_L,\beta,\mathfrak{m}_{\Lambda_L}}\underset{L\to\infty}{\longrightarrow} \nu_\beta^-$, where for a finite set $\Lambda\subset \mathbb Z^d$,  $\mathfrak{m}_\Lambda:=-\mathfrak{p}_\Lambda$.

In \cite{GunaratnamPanagiotisPanisSeveroPhi42022}, we additionally proved that for every $\beta\geq0$, the set of translation invariant Gibbs measures at $\beta$ (see \cite[Definition 1.1]{GunaratnamPanagiotisPanisSeveroPhi42022}) consists of convex combinations of the extremal measures $\nu^+_\beta$ and $\nu^-_\beta$. In particular, one has
\begin{equation}\label{eq:free=plus+minus/2}
    \nu_\beta=\nu_{\beta,0}=\frac{1}{2}\Big(\nu^+_\beta+\nu_\beta^-\Big).
\end{equation}

\begin{remark}\label{rem:def_p}
A question that arises naturally is the following: how big or small can a boundary field $\mathsf{h}_L$ be so that $\langle \cdot \rangle_{\Lambda_L,\beta,\mathsf{h}_L}\underset{L\to\infty}{\longrightarrow} \langle \cdot \rangle^+_\beta$? As a consequence of our results, for $\beta>\beta_c$, the convergence holds for boundary fields satisfying $\mathsf{h}_L\leq \mathfrak{p}_L$ and $L^{d-1}\mathsf{h}_L\to\infty$. This is established in Proposition~\textup{\ref{prop: weak plus measure}}. However, the convergence may also hold for boundary fields that grow faster than $\mathfrak{p}_{\Lambda}$, provided the regularity estimates of Proposition~\textup{\ref{prop:regularity}} remain valid deep in the bulk. Based on the argument in Lemma~\textup{\ref{lem: cond regularity}}, we expect that regularity estimates hold even for exponentially growing boundary fields. 
\end{remark}

\subsection{The random tangled current representation}\label{sec:tangled_currents}

We recall the random tangled current representation of the $\varphi^4$ model introduced in \cite{GunaratnamPanagiotisPanisSeveroPhi42022}. We begin by defining what a tangled current is and then state our main combinatorial tool: the \emph{switching principle}. This object will only be used in Section~\ref{sec:magnetisation} in order to compare measures with different boundary conditions.

\subsubsection{Tangled currents}

We begin by recalling the definition of a current. Below, $\mathfrak{g}$ refers to a ghost vertex and we let $\varphi_\mathfrak g\equiv 1$. Let $(\Lambda,E(\Lambda))$ be a finite graph. If $\mathsf{h}\in \mathbb R^{\Lambda}$, we let $\Lambda[\mathsf{h}]$ be the graph with vertex set $\Lambda^{\mathfrak{g}}:=\Lambda\cup\{\mathfrak g\}$, and edge set $E(\Lambda[\mathsf{h}]):=E(\Lambda)\cup\{\{x,\fg\}:x\in \Lambda, \: \mathsf{h}_x\neq 0\}$. 
We will sometimes make a slight abuse of notation and additionally view $\Lambda[\mathsf{h}]$ as a weighted graph. This means that the notation $\Lambda[\mathsf{h}]$ also carries the information of the value of $\mathsf{h}$. When $\mathsf{h}\equiv 0$, the graph $\Lambda[\mathsf{h}]$ is equal to $(\Lambda,E(\Lambda))$.

\paragraph{Single currents.} Let $\mathsf{h}\in \mathbb R^\Lambda$. A \emph{current} $\n$ on $\Lambda[\mathsf{h}]$ is a function $\n: E(\Lambda[\mathsf{h}])\rightarrow \mathbb N$. We let $\Omega_{\Lambda[\mathsf{h}]}$ be the set of currents on $\Lambda[\mathsf{h}]$. We write $\n_{x,y}=\n_{xy}=\n(x,y)$ for the value of $\n$ on $xy\in E(\Lambda[\mathsf{h}])$. Given $\n \in \Omega_{\Lambda[\mathsf{h}]}$, let $\Delta\n(x):=\sum_{e\in E(\Lambda[\mathsf{h}]), \: e\ni x} \n_{e}$ be the $\n$-degree of $x$. The set of \emph{sources} of $\n$ is defined by
\begin{equation}
    \sn:=\{x\in \Lambda^\fg: \Delta\n(x)\textup{ is odd}\}.
\end{equation}
Additionally, we define
\begin{equation}
\mathcal{M}(\Lambda^\fg)
:=
\Big\lbrace A \in \mathbb N^{\Lambda^\fg}: A_\fg\leq 1,  \: \sum_{x\in \Lambda^\fg}A_x \textup{ is even}\Big\rbrace
\end{equation}
to be the set of \emph{admissible moments} or \emph{source functions}  on $\Lambda^\fg$, and write $\partial A:=\{x \in \Lambda^\fg : A_x \textup{ is odd}\}$. We write $A=\emptyset$ if $A_x=0$ for all $x\in \Lambda^\fg$. Additionally, if $A=\mathds{1}_x+\mathds{1}_y$, we write $A=xy$. Given $A,B\in \mathcal{M}(\Lambda^\fg)$, define $A+B \in \mathcal{M}(\Lambda^\fg)$ by:
\begin{equation}
(A+B)_x
:=
A_x+B_x
\end{equation}
for all $x \in \Lambda$, and
\begin{equation}
    (A+B)_\fg:=A_\fg+B_\fg \textup{ mod }2.
\end{equation}
Finally, for $\n\in \Omega_{\Lambda[\mathsf{h}]}$ and $A\in \mathcal{M}(\Lambda^\fg)$, we introduce the weight
\begin{equation}
    w^{A}_{\beta,\mathsf{h}}(\n):=\prod_{xy\in E(\Lambda)}\frac{\beta^{\n_{xy}}}{\n_{xy}!}\prod_{x\in \Lambda}\frac{(\beta\mathsf{h}_x)^{\n_{x\fg}}}{\n_{x\fg}!}\langle\varphi^{A_x+\Delta\n(x)}\rangle_0.
\end{equation}
Since $\rho_{g,a}$ is an even measure, the above weight is equal to zero unless $\sn=\partial A$.

It is possible to expand the correlation functions of the $\varphi^4$ model to relate them to currents. Recall that, for $A\in \mathcal{M}(\Lambda^\fg)$, $\varphi_A=\prod_{x\in \Lambda}\varphi_x^{A_x}$.
\begin{lemma}[\hspace{1pt}{\cite[Proposition~3.1]{GunaratnamPanagiotisPanisSeveroPhi42022}}]\label{lem: current expansion} Let $\beta>0$ and $\mathsf{h}\in \mathbb R^\Lambda$. Then, for all $A\in \mathcal{M}(\Lambda^\fg)$,
\begin{equation}
    \langle \varphi_A\rangle_{\Lambda,\beta,\mathsf{h}}=\frac{\sum_{\sn=\partial A}w^A_{\beta,\mathsf{h}}(\n)}{\sum_{\sn=\emptyset}w^\emptyset_{\beta,\mathsf{h}}(\n)}.
\end{equation}
\end{lemma}
We now turn to the definition of \emph{tangled} currents. We will introduce two useful notions: the \emph{single} tangled current, and the \emph{double} tangled current. The latter will be useful to state the switching principle below. We fix $\mathsf{h}\in (\mathbb R^+)^\Lambda$. 

\paragraph{Single tangled current.}
We fix $A \in \mathcal{M}(\Lambda^\fg)$, and let $\n \in 
\Omega_{\Lambda[\mathsf{h}]}$ be such that $\partial \n = \partial 
A$. For each  $z\in \Lambda$, we define the block $
\mathcal{B}^A_z(\n)$ as follows: for each $y\in \Lambda^\fg$ such that $yz\in E(\Lambda[\mathsf{h}])$, it 
contains $\n_{zy}$ points labelled $(zy(k))_{1\leq k \leq 
\n_{zy}}$, and it also contains $A_z$ points labelled $(za(k))_{1\leq k \leq 
A_z}$. Note that $\mathcal{B}^A_z(\n)$ has cardinality $\Delta \n(z)
+A_z$, which is an even number by definition. We also write $
\mathcal{B}^A_\fg(\n)=\lbrace \fg\rbrace$, and when $A_\fg=1$ we may 
also write $\fg a(1)=\fg$. For each $z \in \Lambda$, let $
\mathcal{T}_{\n}^A(z)$ be the set of {\it even partitions} 
of the block $\mathcal{B}^A_z(\n)$, i.e.\ for every $P=\{P_1,P_2,\ldots,P_k\}\in \mathcal{T}_{\n}^A(z)$, each $|P_i|$ is even.
Define
\begin{equation}
\mathcal{T}_{\n}^A
=
\underset{z \in \Lambda}{\bigotimes} \,\mathcal{T}_{\n}^A(z).
\end{equation}
An element $\mathfrak{t} \in \mathcal{T}_{\n}^A$ is called a \textit{tangling}. We sometimes refer to an element $\mathfrak{t}_z$ of $\mathcal{T}_{\n}^A(z)$ as a tangling (of $\mathcal{B}_z^A(\n)$), this will be clear from the context. We call the pair $(\n, \mathfrak{t})$ a \textit{tangled current}. 

Let $\mathcal{H}^A(\n)$ be the graph consisting of vertex set 
\begin{equation}
\bigcup_{z \in \Lambda^\fg} \mathcal{B}^A_z(\n),
\end{equation}
and edge set 
\begin{equation}
\left(\bigcup_{\lbrace x,y\rbrace\subset \Lambda}\left\lbrace \lbrace xy(k),yx(k) \rbrace: \: 1\leq k \leq \n_{xy}\right\rbrace\right)\cup\left(\bigcup_{x\in \Lambda : \:h_x\neq 0}\lbrace \lbrace x\fg(k),\fg\rbrace: \: 1\leq k\leq \n_{x\fg}\rbrace\right).	
\end{equation}
Given any tangling $\ft \in \mathcal{T}_{\n}^A$, the graph $\mathcal{H}^A(\n)$  naturally induces a graph $\mathcal{H}^A(\n,\ft)$ defined as follows. For each block $\mathcal{B}_z^A(\n)$, if $\ft_z=\{P_1,\ldots,P_k\}\in \mathcal{T}_\n^A(z)$, for every $1\leq i \leq k$, we add $\binom{|P_i|}{2}$ edges connecting pairwise the vertices lying in $P_i$. In words, we add to $\mathcal H^A(\n)$ the complete graph on the elements of each partition induced by $\ft$.
 This gives a canonical notion of connectivity in a tangled current $(\n,\ft)$.
In the case where $A=\emptyset$, we may drop the superscript $A$ from the notation.

We now move to the definition of double tangled currents. We will need the following terminology. 
\begin{definition} Let $\Lambda'\subset\Lambda$ be finite subsets of $\mathbb Z^d$. If $(\mathsf{h},\mathsf{h}')\in (\mathbb R^+)^\Lambda\times (\mathbb R^+)^{\Lambda'}$, we say that $\mathsf{h}'$ is a \emph{restriction} of $\mathsf{h}$ if 
\begin{equation}
    \{ x \in \Lambda' : \mathsf{h}_x' \neq \mathsf{h}_x\}\subset \{x \in \Lambda': \mathsf{h}'_x=0\}.
\end{equation}
 Note that by definition, $\Lambda'[\mathsf{h}']$ is a subgraph of $\Lambda[\mathsf{h}]$.
\end{definition}

\paragraph{Double tangled current.} Let $\Lambda'\subset \Lambda$ and $\mathsf{h}'\in(\mathbb R^+)^{\Lambda'}$ be a restriction of $\mathsf{h}$.
Fix $(A,B)\in \mathcal{M}(\Lambda^\fg)\times \mathcal{M}((\Lambda')^\fg)$. Let $\n_1 \in \Omega_{\Lambda[\mathsf{h}]}$ and $\n_2 \in \Omega_{\Lambda'[\mathsf{h}']}$ 
satisfy $\partial \n_1 = \partial A$ and $\partial \n_2 = \partial B$. We trivially view $\n_2$ as an element of 
$\Omega_{\Lambda[\mathsf{h}]}$ by setting $\n_2(x,y) = 0$ for all edges $xy\in E(\Lambda[\mathsf{h}])\setminus E(\Lambda'[\mathsf{h}'])$.  
We define $\mathcal{B}_z^{A,B}(\n_1,\n_2)$ as the disjoint union of $\mathcal{B}_z^{A}(\n_1)$ and $\mathcal{B}_z^{B}(\n_2)$.   
We write 
$\mathcal{T}_{\n_1,\n_2}^{A,B}(z)$ for the set of even partitions of the set $\mathcal{B}_z^{A,B}(\n_1,\n_2)$ whose restriction 
to $\mathcal{B}_z^A(\n_1)$ and $\mathcal{B}_z^B(\n_2)$ is also an even partition, i.e.\ for every $P=\{P_1,P_2,\ldots, P_k\}\in \mathcal{T}_{\n_1,\n_2}^{A,B}(z)$, both $|P_i\cap \mathcal{B}_z^A(\n_1)|$ and $|P_i \cap \mathcal{B}_z^B(\n_2)|$ are even numbers. The elements of $\mathcal{T}_{\n_1,\n_2}^{A,B}(z)$ are called {\it admissible} partitions.
As above, we define
\begin{equation}
\mathcal{T}^{A,B}_{\n_1,\n_2}
=
\underset{z \in \Lambda}{\bigotimes} \,\mathcal{T}^{A,B}_{\n_1,\n_2}(z).
\end{equation}
As above, for any $\ft$ in $\mathcal{T}_{\n_1,\n_2}
^{A,B}$, we can define a graph 
$\mathcal{H}^{A,B}(\n_1,\n_2,\ft)$ by the 
following procedure: for each block
$\mathcal{B}_z^{A,B}(\n_1,\n_2)$, if $\ft_z
=\{P_1,\ldots,P_k\}\in \mathcal{T}_{\n_1,\n_2}^{A,B}(z)$,
for every $1\leq i \leq k$, we add $\binom{|P_i|}{2}$ edges connecting pairwise the vertices lying in $P_i$. This provides a natural notion of connectivity in $\mathcal{H}^{A,B}(\n_1,\n_2)$.
If $G=(V(G),E(G))$ is a subgraph of $\Lambda[\mathsf{h}]$, we let $\mathcal{H}^{A,B}_{G}(\n_1,\n_2,\ft)$ be the induced subgraph of $\mathcal{H}^{A,B}(\n_1,\n_2,\ft)$, where we restrict to blocks labelled by $V(G)$ and edges which projects to elements of $E(G)$.

\subsubsection{Switching principle for tangled currents}

We now state the version of the switching principle which will be used in this paper. A more general statement can be found in \cite{GunaratnamPanagiotisPanisSeveroPhi42022}.

Let $\Lambda'\subset \Lambda$ be two finite subsets of $\mathbb Z^d$ and $\beta>0$. Let $(\mathsf{h},\mathsf{h}')\in (\mathbb R^+)^\Lambda\times (\mathbb R^+)^{\Lambda'}$ with $\mathsf{h}'$ a restriction of $\mathsf{h}$.
For $(A,B)\in \mathcal M(\Lambda^\fg)\times \mathcal{M}((\Lambda')^\fg)$, we define $ \Omega_{\Lambda[\mathsf{h}],\Lambda'[\mathsf{h}']}^{\mathcal{T},A,B}$ to be the set of double tangled currents on $\Lambda[\mathsf{h}]$ with source functions $A,B$, i.e the set of triples $(\n_1,\n_2,\ft)$ where: $(\sn_1,\sn_2)=(\partial A,\partial B)$, and $\ft \in \mathcal{T}^{A,B}_{\n_1,\n_2}$.  We sometimes identify such a triple with its associated graph $\mathcal H^{A,B}(\n_1,\n_2,\ft)$. We set
 \begin{equation}
 	\Omega^\mathcal{T}_{\Lambda[\mathsf{h}],\Lambda'[\mathsf{h}']}:=\bigcup_{(A,B)\in \mathcal M(\Lambda^\fg)\times \mathcal{M}((\Lambda')^\fg)}\Omega_{\Lambda[\mathsf{h}],\Lambda'[\mathsf{h}']}^{\mathcal{T},A,B}.
 \end{equation}
 
Before stating the switching principle, we need a definition.

\paragraph{Pairing events.} Let $(B, C)\in \mathcal M(\Lambda^\fg)\times \mathcal{M}((\Lambda')^\fg)$. Assume that $B\leq C$, i.e.\ $B_x\leq C_x$ for every $x\in (\Lambda')^\fg$. Define $\mathcal{F}_{\Lambda'[\mathsf{h}']}^{C}(B)$ to be the subset of $\Omega_{\Lambda[\mathsf{h}],\Lambda'[\mathsf{h}']}^{\mathcal{T},C,\emptyset}$ consisting of all double tangled currents $(\n_1,\n_2,\ft)$ satisfying the following condition: each connected component of $\mathcal H_{\Lambda'[\mathsf{h}']}^{C,\emptyset}(\n_1,\n_2,\ft)$ intersects $\{zc(k): z\in (\Lambda')^\fg, \: C_z-B_z+1\leq k \leq C_z\}$ an even number of times.

\begin{theorem}[\hspace{1pt}{\cite[Theorem~3.11]{GunaratnamPanagiotisPanisSeveroPhi42022}}]\label{thm: switching lemma rtc}
Let $\Lambda'\subset\Lambda$ be finite subsets of $\mathbb Z^d$. Let $(\mathsf{h},\mathsf{h}')\in (\mathbb R^+)^\Lambda\times (\mathbb R^+)^{\Lambda'}$ with $\mathsf{h}'$ a restriction of $\mathsf{h}$. For every $z\in \Lambda$, every $(C,D)\in \mathcal M(\Lambda^\fg)\times  \mathcal M((\Lambda')^\fg)$, every $(\n_1,\n_2)\in \Omega_{\Lambda[\mathsf{h}]}\times \Omega_{\Lambda'[\mathsf{h}']}$ with $(\sn_1,\sn_2)=(\partial C,\partial D)$, there exists a probability measure $\rho_{z,\n_1,\n_2}^{C,D}$ on $\mathcal{T}_{\n_1,\n_2}^{C,D}(z)$ such that the following holds. For every $(A,B)\in \mathcal{M}(\Lambda^\fg)\times \mathcal M((\Lambda')^\fg)$, 
\begin{equation}
\sum_{\substack{\partial \n_1 = \partial A \\ \partial \n_2 = \partial B}} 
w^A_{\beta,\mathsf{h}}(\n_1) w^B_{\beta,\mathsf{h}'}(\n_2) =\sum_{\substack{\partial \n_1 = \partial(A+B) \\ \partial \n_2 = \emptyset}} 
w^{A+B}_{\beta,\mathsf{h}}(\n_1)w^{\emptyset}_{\beta,\mathsf{h}'}(\n_2)  \rho^{A+B,\emptyset}_{\n_1,\n_2} \big[ \mathcal{F}_{\Lambda'[\mathsf{h}']}^{A+B,\emptyset}(B) \big],
\end{equation}
where 
\begin{equation}
\rho_{\n_1,\n_2}^{A+B,\emptyset}:=\bigotimes_{z \in \Lambda} \rho_{z,\n_1, \n_2}^{A+B,\emptyset},
\end{equation}
is a probability measure on $\mathcal{T}_{\n_1,\n_2}^{A+B,\emptyset}$, and $\mathcal{F}_{\Lambda'[\mathsf{h}']}^{A+B,\emptyset}(B)$ is the event defined above.
\end{theorem}
\begin{remark}\label{rem:dependencyoftanglingmeasures}
As it turns out (see the construction in \textup{\cite{GunaratnamPanagiotisPanisSeveroPhi42022}}), for any $z$, the measures $\rho^{A,B}_{z,\n_1,\n_2}$ depend only on 
$(g,a,\Delta \n_1(z)+A_z,\Delta\n_2(z)+B_z)$, but they do not depend on $\Lambda$, $\beta$, $\mathsf{h}$ or the rest of $\n_1,\n_2$.
Sometimes, for $\Delta{\n}_1(z)+A_z=2k$ and $\Delta{\n}_2(z)+B_z=2k'$ ($k,k'\geq 0$), we will denote the tangling measure by $\rho^{2k,2k'}:=\rho^{A,B}_{z,\n_1,\n_2}$.
\end{remark}

We can also state the switching lemma in a probabilistic way. We start by introducing the random tangled current measures of interest. We define $\rho^{A}_{z,\n}:=\rho_{z,\n,0}^{A,\emptyset}$ and
\begin{equation}
    \rho_{\n}^A:=\bigotimes_{z\in \Lambda}\rho^{A}_{z,\n}.
\end{equation}
We view $\rho^A_{\n}$ as a measure on single currents. 
In view of Remark~\ref{rem:dependencyoftanglingmeasures}, we may sometimes write $\rho^{2k}:=\rho^{A}_{z,\n}$, for $\Delta{\n}(z)+A_z=2k$.
With a small abuse of notation, we view the set of single tangled currents on $\Lambda[\mathsf{h}]$ as a subset of $\Omega^\mathcal{T}_{\Lambda[\mathsf{h}],\Lambda'[\mathsf{h}']}$ that we denote by $\Omega_{\Lambda[\mathsf{h}]}^{\mathcal{T}}$. Similarly, we define for $A\in \mathcal M(\Lambda^\fg)$ the set $\Omega_{\Lambda[\mathsf{h}]}^{\mathcal{T},A}\subset \Omega_{\Lambda[\mathsf{h}],\Lambda'[\mathsf{h}']}^{\mathcal{T},A,\emptyset}$.

\paragraph{Tangled current measures.}

Let $\Lambda'\subset\Lambda$ be finite subsets of $\mathbb Z^d$ and $(A,B) \in \mathcal{M}(\Lambda^\fg)\times \mathcal{M}((\Lambda')^\fg)$. Let $(\mathsf{h},\mathsf{h}')\in (\mathbb R^+)^\Lambda\times (\mathbb R^+)^{\Lambda'}$ with $\mathsf{h}'$ a restriction of $\mathsf{h}$. We define a measure $\mathbf P^{A}_{\Lambda[\mathsf{h}],\beta}$ on $\Omega_{\Lambda[\mathsf{h}]}^{\mathcal{T},A}$ by
\begin{equation}
    \mathbf P^{A}_{\Lambda[\mathsf{h}],\beta}[(\n,\mathfrak{t})]=\frac{w_{\beta,\mathsf{h}}^A(\n)\rho^A_{\n}(\mathfrak{t})}{\sum_{\partial \m=\partial A}w_{\beta,\mathsf{h}}^A(\m)}.
\end{equation}
We also define a measure $\mathbf P_{\Lambda[\mathsf{h}],\Lambda'[\mathsf{h}'],\beta}^{A,B}$ on $\Omega^{\mathcal{T},A,B}_{\Lambda[\mathsf{h}],\Lambda'[\mathsf{h}']}$ by
\begin{equation}
    \mathbf P^{A,B}_{\Lambda[\mathsf{h}],\Lambda[\mathsf{h}'],\beta}[(\n_1,\n_2,\ft)]=\frac{w_{\beta,\mathsf{h}}^A(\n_1)w_{\beta,\mathsf{h}'}^B(\n_2)\rho_{\n_1,\n_2}^{A,B}(\mathfrak{t})}{\sum_{\substack{\partial \n_1=\partial A\\ \partial \n_2=\partial B}}w_{\beta,\mathsf{h}}^A(\n_1)w_{\beta,\mathsf{h}'}^B(\n_2)}.
\end{equation}
We write $\mathbf{E}^A_{\Lambda[\mathsf{h}],\beta}$ (resp.\ $\mathbf{E}^{A,B}_{\Lambda[\mathsf{h}],\Lambda'[\mathsf{h}'],\beta}$) for the expectation with respect to the measure $\mathbf{P}^A_{\Lambda[\mathsf{h}],\beta}$ (resp.\ $\mathbf{P}^{A,B}_{\Lambda[\mathsf{h}],\Lambda'[\mathsf{h}'],\beta}$). The measures introduced above also depend on $\mathsf{h}$ (resp.\ $\mathsf{h}'$). However, as explained in the beginning of the section, this information is contained in the notation $\Lambda[\mathsf{h}]$ (resp.\ $\Lambda[\mathsf{h}']$). 
Using these definitions, we can reformulate the switching lemma in a more probabilistic way.

\begin{corollary}[Probabilistic version of the switching lemma]\label{thm: probabilistic switching rtc}
For every $(A,B)\in \mathcal{M}(\Lambda^\fg)\times \mathcal{M}((\Lambda')^\fg)$,
\begin{equation}
    \frac{\langle \varphi_A\rangle_{\Lambda,\beta,\mathsf{h}}\langle \varphi_B\rangle_{\Lambda',\beta,\mathsf{h}'}}{\langle \varphi_{A+B}\rangle_{\Lambda,\beta,\mathsf{h}}}=\mathbf{P}^{A+B,\emptyset}_{\Lambda[\mathsf{h}],\Lambda'[\mathsf{h}'],\beta}[\mathcal{F}_{\Lambda'[\mathsf{h}']}^{A+B}(B)].
\end{equation}
\end{corollary}

In what follows, $\beta$ is often fixed and to lighten the notations we omit it in the notation of the tangled current measures.

\paragraph{Coupling between double and single tangled currents measures.}
Note that by definition, the marginal of $\n_1$ and $\n_2$ in $\mathbf{P}_{\Lambda[\mathsf{h}],\Lambda'[\mathsf{h}'],\beta}^{A,B}$ are independent and distributed according to the marginals of $\n$ in $\mathbf P^{A}_{\Lambda[\mathsf{h}],\beta}$ and $\mathbf P^{B}_{\Lambda'[\mathsf{h'}],\beta}$ respectively. However, the relation between $\ft$ and its marginals in these measures is less clear. In the following proposition, we establish a natural stochastic domination.

If $P,Q$ are two partitions of a set $S$, we say that $P$ is coarser than $Q$, and write $P\succ Q$, if any element of $P$ can be written as a union of elements of $Q$. Recall the shorthand notation for the tangling measures introduced in Remark \ref{rem:dependencyoftanglingmeasures}.

\begin{proposition}[\hspace{1pt}{\cite[Proposition~3.36]{GunaratnamPanagiotisPanisSeveroPhi42022}}]\label{prop:tanglings_domination}
Let $k,k'\in\mathbb N$. 
There exists a coupling of $\rho^{2k,2k'}$, $\rho^{2k}$ and $\rho^{2k'}$ such that if $(\ft,\ft_1,\ft_2)\sim (\rho^{2k,2k'},\rho^{2k},\rho^{2k'})$ and if $\ft_1\sqcup \ft_2$ is the partition whose partition classes are the partition classes of $\ft_1$ and $\ft_2$,
\begin{equation}
    \ft \succ \ft_1\sqcup \ft_2, \qquad \text{almost surely.}
\end{equation}
We write, 
\begin{equation}\label{eq: domination}
    \rho^{2k,2k'} 
    \succ
    \rho^{2k}\sqcup\rho^{2k'}.
\end{equation}
We may abuse the notation and write $\rho^{2k,2k'}$ for the associated (extended) measure on triplets $(\ft,\ft_1,\ft_2)$.
\end{proposition}

\begin{remark}\label{rem:extended_measure}
    Since the tangling measures on each block are independent (conditionally on the currents $\n_1$ and $\n_2$), we can apply Proposition~\textup{\ref{prop:tanglings_domination}} to each block to construct a random quintuple $(\n_1,\n_2,\ft,\ft_1,\ft_2)$ such that
    \begin{enumerate}[label=(\alph*)]
        \item $(\n_1,\ft_1)\sim \mathbf{P}^A_{\Lambda[\mathsf{h}],\beta}$, 
        \item $(\n_2,\ft_2)\sim \mathbf{P}^B_{\Lambda'[\mathsf{h}']\beta}$,
        \item $(\n_1,\n_2,\ft)\sim \mathbf{P}^{A,B}_{\Lambda[\mathsf{h}],\Lambda'[\mathsf{h}'],\beta}$,
        \item $\ft(z) \succ \ft_1(z) \sqcup \ft_2(z)$ for every $z\in\Lambda$ almost surely.
    \end{enumerate} 
    When clear from context, we abuse notation and write $\mathbf{P}^{A,B}_{\Lambda[\mathsf{h}],\Lambda'[\mathsf{h}'],\beta}$ for this (extended) measure on quintuples.
\end{remark}

In fact, the above stochastic domination is strict in the sense of the following proposition, which is proved in Appendix \ref{appendix:strict stochastic dom tanglings}.

\begin{proposition}\label{prop: tanglings ECT}
For any $k_1>0$, $k_2>0$ there exists $\varepsilon=\varepsilon(k_1,k_2)>0$ such that for every even partition $P_1$ of $\{1,2,\ldots,2k_1\}$ and every even partition $P_2$ of $\{1,2,\ldots,2k_2\}$ we have
\begin{equation}
\rho^{2k_1,2k_2}[\mathfrak t = \mathrm{ECT} \mid \mathfrak t_1=P_1,\mathfrak t_2=P_2]\geq \varepsilon,    
\end{equation}
where $\mathrm{ECT}$ is the partition consisting of one element, i.e.\ the partition in which ``Everybody is Connected Together''.
\end{proposition}

\section{Bulk and boundary magnetisation}
\label{sec:magnetisation}

In this section, we prove two fundamental results about the magnetisation in the full-space and the half-space settings. The first result shows that one may replace the $+$ boundary condition $\mathfrak{p}$ by a sufficiently thick layer of external magnetic field at the boundary taking the value $+1$. The second result says that there is a unique infinite volume Gibbs measure on the half space with positive field on the boundary. The corresponding result for the Ising model was obtained by Bodineau \cite{Bod05} and relied on a work of Fröhlich and Pfister on the wetting transition \cite{FP87}, and a consequence of the Lee--Yang theorem obtained in \cite{MMP84}. Our proof for the $\varphi^4$ model has a more probabilistic flavour and can easily be adapted to the Ising model. For both statements, we strongly rely on the switching principle for tangled currents stated in Theorem \ref{thm: switching lemma rtc}.

\subsection{Bulk magnetisation}

Let $L\in \mathbb N$. We define $\mathsf{h}_L$ to be the external magnetic field taking value $+1$ on 
$\Lambda_{L}\setminus \Lambda_{L-\log L}$
and $0$ elsewhere.

\begin{proposition}\label{prop: thick conv} Let $d\geq 2$. For all $\beta >0$, we have
\begin{equation}
    \lim_{L\rightarrow \infty} \langle \varphi_0\rangle_{\Lambda_{L},\beta,\mathsf{h}_L}=\langle \varphi_0\rangle^+_\beta.
\end{equation}
\end{proposition}

\begin{proof} We fix $\beta>0$ and drop it from the notations. We will compare $\langle \varphi_0\rangle_{\Lambda_L,\mathsf{h}_L}$ to a finite-volume approximation of $\langle \varphi_0\rangle^+$. Let $\mathfrak{p}_k:=\mathfrak{p}_{\Lambda_{k}}$ for $k\geq 1$. 
We first claim that the following convergence result holds: 
\begin{equation}\label{eq:cv h_l+p_l to +}
    \lim_{L\rightarrow \infty}\langle\varphi_0\rangle_{\Lambda_{2L},\mathsf{h}_L+\mathfrak{p}_{2L}}=\langle \varphi_0\rangle^+.
\end{equation}
Indeed, by monotonicity (see Proposition \ref{prop:Ginibre}) one has $\langle \varphi_0\rangle_{\Lambda_{2L},\mathfrak{p}_{2L}}\leq \langle \varphi_0\rangle_{\Lambda_{2L},\mathsf{h}_L+\mathfrak{p}_{2L}}$. Then, using \eqref{eq:dlr}, Proposition \ref{prop:stoc_monotonicity}, the Cauchy--Schwarz inequality, and Proposition \ref{prop:regularity},
\begin{align}
    \langle \varphi_0\rangle_{\Lambda_{2L},\mathsf{h}_L+\mathfrak{p}_{2L}}&\leq\langle \varphi_0\rangle_{\Lambda_{L/2},\beta,\mathfrak{p}_{L/2}}+\sqrt{\langle \varphi_0^2\rangle_{\Lambda_{2L},\beta,\mathsf{h}_L+\mathfrak{p}_{2L}}}\sqrt{\nu_{\Lambda_{2L},\beta,\mathsf{h}_L+\mathfrak{p}_{2L}}[\max_{x \in \partial^{\rm ext} \Lambda_{L/2}} \varphi_x>\mathfrak{M}_{\Lambda_{L/2}}]}\notag
    \\&\leq \langle \varphi_0\rangle_{\Lambda_{L/2},\beta,\mathfrak{p}_{L/2}}+o(1),
\end{align}
where $o(1)$ tends to $0$ as $L$ tends to infinity, and where $\mathfrak{M}$ was defined above \eqref{eq:maximum_reg}. These two inequalities together with \eqref{eq:conv_plus} readily imply \eqref{eq:cv h_l+p_l to +}.

Since $\mathsf{h}_L$ is a restriction of $\mathsf{h}_L+\mathfrak{p}_{2L}$, the switching principle of Theorem \ref{thm: switching lemma rtc} gives
\begin{equation}\label{eq: thick switching}
     \langle \varphi_0\rangle_{\Lambda_{2L},\mathsf{h}_L+\mathfrak{p}_{2L}} - \langle \varphi_0\rangle_{\Lambda_L,\mathsf{h}_L}= \langle \varphi_0\rangle_{\Lambda_{2L},\mathsf{h}_L+\mathfrak{p}_{2L}}\mathbf{P}_{\Lambda_{2L}[\mathsf{h}_L+\mathfrak{p}_{2L}],\Lambda_L[\mathsf{h}_L]}^{0\fg,\emptyset} [\tilde 0 \centernot\longleftrightarrow \fg \text{ in } {\mathcal{H}_{\Lambda_L[\mathsf{h}_L]}^{0\fg,\emptyset}(\n_1,\n_2,\ft)}],
\end{equation}
where we recall that $\tilde{0}=0a(1)$ denotes the extra vertex in $\mathcal{B}_0$ given by the source constraint. For the rest of the proof, we will use the extended measure on quintuples $(\n_1,\n_2,\ft,\ft_1,\ft_2)$ introduced in the end of Section~\ref{sec:tangled_currents} and write $\mathbf{P}=\mathbf{P}_{\Lambda_{2L}[\mathsf{h}_L+\mathfrak{p}_{2L}],\Lambda_L[\mathsf{h}_L]}^{0\fg,\emptyset}$ for short. 

Under $\mathbf{P}$, there exists a path in $\mathcal{H}^{0\fg}(\n_1,\ft_1)$ that connects $\tilde 0$ to $\fg$ because of the source constraints. When $\tilde 0$ is not connected to $\fg$ in ${\Lambda_L[\mathsf{h}_L]}$, there must exist a path in $(\n_1,\ft_1)$ that crosses the annulus $\Lambda_{L}\setminus \Lambda_{L-\log L}$. Our aim is to show that any path that crosses the annulus $\Lambda_{L}\setminus \Lambda_{L-\log L}$ connects with high probability to $\fg$ before reaching $\partial \Lambda_L$, which implies that the right-hand side of \eqref{eq: thick switching} goes to $0$ as $L$ goes to infinity. See Figure \ref{fig: thick} for an illustration. 

\begin{figure}[htb]
    \centering
    \includegraphics[width=0.5\linewidth]{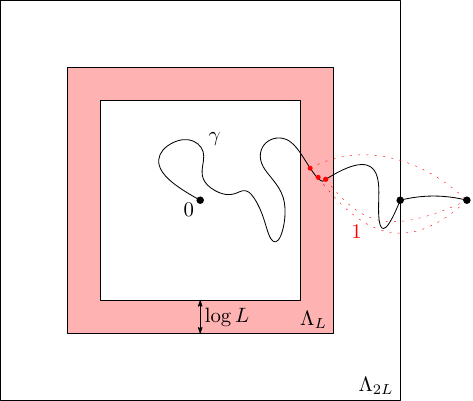}
    \put(2,92){\footnotesize$\mathfrak{g}$}
    \put(-29,97){\footnotesize$\mathfrak{p}_{2L}$}
    \caption{An illustration of the proof of Proposition \ref{prop: thick conv}. The magnetic field $\mathsf{h}_L$ is supported on the red region, while $\mathfrak{p}_{2L}$ is supported on $\partial \Lambda_{2L}$. A path $\gamma$ connects $\tilde{0}$ to $\fg$ in $\mathcal{H}^{0\fg}(\n_1,\ft_1)$. Red vertices correspond to blocks $\mathcal{B}_x$ along $\gamma$ in which the tangling is $\mathrm{ECT}(\mathcal{B}_x)$ and the degree $\Delta \n_1(x)$ is not too large. There are typically order $\log L$ many such vertices. Each red dotted edge represents a chance for $\tilde{0}$ to connect to the ghost within $\mathcal{H}^{0\fg}_{\Lambda_L[\mathsf{h}_L]}(\n_1,\ft_1)$.  The thickness of the red region is chosen in such a way that (at least) one of these edges is open with high probability.}
    \label{fig: thick}
\end{figure}

Given a constant $M\geq1$ to be specified later, consider the set of ``good vertices''
$$\mathcal{V}:=\{x\in \Lambda_L\setminus \Lambda_{L-\log L}:~\Delta \n_1(x)+\Delta_{\Lambda_L} \n_2(x)\leq M \text{ and } \Delta \n_1(x)+\Delta \n_2(x)\leq M+2 \},$$
where $\Delta_{\Lambda_L} \n_2(x):=\sum_{y\in \Lambda_L} \n_2(x,y)$. Let $\gamma$ be a path in $\Lambda_L\setminus \Lambda_{L-\log L}$ that crosses the annulus $\Lambda_L\setminus \Lambda_{L-\log L}$, and notice that by definition $\sum_{\substack{e: e \cap \gamma\neq \emptyset}}\n_1(e)+\n_2(e)\geq \frac{1}{2}\sum_{x\in \gamma} \Delta\n_1(x)+\Delta\n_2(x)\geq \frac{1}{2}|\gamma\cap \mathcal{V}^c| M$.
Therefore, by Markov's inequality for $2^{\sum_{\substack{e: e \cap \gamma\neq \emptyset}}\n_1(e)+\n_2(e)}$, Lemma~\ref{lem:degree_exp.moments} and Proposition~\ref{prop:regularity}, 
\begin{equation}\label{eq:good_path}
\langle \varphi_0\rangle_{\Lambda_{L},\mathsf{h}_L+\mathfrak{p}_{2L}}\mathbf{P}[|\gamma\cap\mathcal{V}^c|> |\gamma|/2]\leq e^{2(2d+1)C'|\gamma|} \cdot 2^{-M |\gamma|/4},
\end{equation}
where $C'>0$ is the constant in Lemma~\ref{lem:degree_exp.moments} (which in turn depends on the constant $C>0$ of Proposition~\ref{prop:regularity}). 
Consider the following event (which only depends on $\mathcal{V}$)
$$\mathcal{A}:= \{|\gamma\cap\mathcal{V}| \geq |\gamma|/2, ~~ \forall\, \gamma \text{ crossing } \Lambda_{L}\setminus \Lambda_{L-\log L} \}.$$ 
Note that the number of paths of length $k$ in $\mathbb{Z}^d$ crossing the annulus $\Lambda_L\setminus \Lambda_{L-\log L}$ is at most
$|\partial \Lambda_L|(2d)^k\leq C_1L^{d-1} (2d)^k$.
Choosing $M$ large enough so that $a:=2d e^{2(2d+1)C'}2^{-M/4}<e^{-d}$, a union bound and \eqref{eq:good_path} gives (note that $a^{\log L}< L^{-d}$) 
\begin{equation}\label{eq: high degree points}
\langle \varphi_0\rangle_{\Lambda_{2L},\mathsf{h}_L+\mathfrak{p}_{2L}}\mathbf{P}[\mathcal{A}^c]\leq C_1 L^{d-1}
\sum_{k\geq \log L} a^k\leq \frac{C_2}{L}.
\end{equation}

We will show that (the block of) each vertex $x\in\mathcal{V}$ has a uniformly positive probability (independently from each other) of being fully tangled and connected to $\fg$, thus implying that the event $\{\tilde 0\longleftrightarrow \fg\text{ in }\mathcal{H}_{\Lambda_L[\mathsf{h}_L]}^{0\fg,\emptyset}(\n_1,\n_2,\ft)\}$ occurs with high probability conditionally on $\mathcal{A}$.
Indeed, consider a configuration $(\m_1,\m_2)$ (with $\partial \m_1= \{0,\fg\}$ and $\partial \m_2=\emptyset
$) and let $x\in \Lambda_L$ be a vertex satisfying $\Delta \m_1(x)+\Delta_{\Lambda_L} \m_2(x)\leq M$ and $\m_2(x\fg)=0$.
Define the current $\tilde \m_2=\tilde \m_2^{x}$ as $\tilde \m_2(x\fg)=2$ and $\tilde \m_2(e)=\m_2(e)$ otherwise. Note that $\tilde \m_2$ has no sources, $\Delta \m_1(x)+\Delta_{\Lambda_L} \tilde \m_2(x)\leq M$, $\Delta \m_1(x)+\Delta \tilde \m_2(x)\leq M+2$ and that 
\begin{equation}
Cw_{\beta,\mathsf{h}_L}^\emptyset(\tilde \m_2)\geq w_{\beta,\mathsf{h}_L}^\emptyset(\m_2),
\end{equation}
where $C:=2\beta^{-2}\max\big\{\frac{\langle \varphi^{2k} \rangle_0}{\langle \varphi^{2k+2} \rangle_0} : 0\leq k\leq M/2\big\}.$
This implies that, for $(\m_1,\m_2)$ and $x$ as above,
\begin{equation}\label{eq:proof thick layer 1}
\mathbf{P}[\n_2(x\fg)=0 \mid \mathcal{C}(\m_1,\m_2,x)]\leq C \mathbf{P}[\n_2(x\fg)>0 \mid \mathcal{C}(\m_1,\m_2,x)],   
\end{equation}
where, 
\begin{align*}
\mathcal{C}(\m_1,\m_2,x):=\Big\{(\n_1,\n_2):\n_1=\m_1,\: \n_2|_{E(\Lambda_L[\mathsf{h}_L])\setminus\{x\fg\}}&=\m_2|_{E(\Lambda_L[\mathsf{h}_L])\setminus\{x\fg\}},\\&  \Delta \n_1(x)+\Delta \n_2(x)\leq M+2\Big\}.
\end{align*}
Hence, from \eqref{eq:proof thick layer 1} we find
\begin{equation}
\mathbf{P}[\n_2(x\fg)>0 \mid \mathcal C(\m_1,\m_2,x)] \geq \frac{1}{1+C}.    
\end{equation}
Furthermore, using Proposition \ref{prop: tanglings ECT}, there exists $\varepsilon=\varepsilon(M)$ such that, for $(\n_1,\n_2)$ satisfying $\Delta \n_1(x)+\Delta\n_2(x)\leq M+2$ and any tangling configuration $\fs_1$, 
\begin{equation}
\rho_{\n_1,\n_2}^{0\fg,\emptyset}[\ft_x=\mathrm{ECT}\:\mid \ft_1=\fs_1, \ft_{|{\Lambda_{L}\setminus\{x\}}}]\geq \varepsilon.
\end{equation}
The last two displayed equations easily imply that, conditionally on $(\n_1,\ft_1)$ and $\mathcal{V}$, the sequence of random variables $(\omega_x)_{x\in \mathcal{V}}$, where $\omega_x:=\mathbbm{1}_{\{\n_2(x\fg)>0,\: \ft_x=\mathrm{ECT}\}}$, stochastically dominates a sequence of independent Bernoulli random variables $(Z_x)_{x\in\mathcal{V}}$ with $\mathbb{P}[Z_x=1]=\varepsilon/(1+C)$. 

Now, notice that on the event $\{\tilde 0\centernot\longleftrightarrow \fg\text{ in }\mathcal{H}_{\Lambda_L[\mathsf{h}_L]}^{0\fg,\emptyset}(\n_1,\n_2,\ft)\}$, the cluster of $\tilde{0}$ in $\mathcal{H}^{0\fg}(\n_1,\ft_1)$ induces a path $\gamma$ in $\Z^d$ crossing the annulus $\Lambda_L\setminus\Lambda_{L-\log L}$, for which $\omega_x=0$ for every $x\in\gamma$. Furthermore, by definition, on the event $\mathcal{A}$, the set $\gamma'=\gamma\cap\mathcal{V}$ satisfies $|\gamma'|\geq |\gamma|/2\geq (\log L)/2$. Therefore 
\begin{equation}\label{eq: conditioning}
\mathbf{P}[\tilde 0\centernot\longleftrightarrow \fg\text{ in }\mathcal{H}_{\Lambda_L[\mathsf{h}_L]}^{0\fg,\emptyset}(\n_1,\n_2,\ft)\mid \mathcal{A}]\leq \mathbf{E}[\mathbf{P}[\omega_x=0 \; \forall x\in\gamma' \mid  (\n_1,\ft_1), \mathcal{V}]\mathbbm{1}_{\mathcal{A}}]  \leq L^{-c'} 
\end{equation}
for some $c'>0$, where in the last inequality we used the stochastic domination proved in the previous paragraph.
Combining \eqref{eq: conditioning} with \eqref{eq: high degree points} we obtain
\begin{align}
\langle \varphi_0\rangle_{\Lambda_{2L},\mathsf{h}_L+\mathfrak{p}_{2L}}\mathbf{P}[\tilde 0\centernot\longleftrightarrow &\fg\text{ in }\mathcal{H}_{\Lambda_L[\mathsf{h}_L]}^{0\fg,\emptyset}(\n_1,\n_2,\ft)]\notag
\\&\leq \langle\varphi_0\rangle_{\Lambda_{2L},\mathsf{h}_L+\mathfrak{p}_{2L}}\mathbf{P}[\tilde 0\centernot\longleftrightarrow \fg\text{ in }\mathcal{H}_{\Lambda_L[\mathsf{h}_L]}^{0\fg,\emptyset}(\n_1,\n_2,\ft),\: \mathcal{A}]+\frac{C_2}{L}\notag
\\&\leq \langle\varphi_0\rangle_{\Lambda_{2L},\mathsf{h}_L+\mathfrak{p}_{2L}}L^{-c'}+\frac{C_2}{L}.\label{eq: discon unlikely}
\end{align}
Combining \eqref{eq: discon unlikely}, \eqref{eq: thick switching}, and \eqref{eq:cv h_l+p_l to +}, we deduce that 
\begin{equation}
\lim_{L\rightarrow \infty} \langle \varphi_0\rangle_{\Lambda_L,\mathsf{h}_L}=\lim_{L\rightarrow \infty} \langle \varphi_0\rangle_{\Lambda_{2L},\mathsf{h}_L+\mathfrak{p}_{2L}}=\langle \varphi_0\rangle^+.
\end{equation}
This concludes the proof.
\end{proof}

\begin{remark}\label{rem: equality mag implies equality meas} It is straightforwad to extend Proposition~\textup{\ref{prop: thick conv}} to any $x \in \mathbb Z^d$. Since the sequence $(\langle \cdot \rangle_{\Lambda_L,\beta,\mathsf{h}_L})_{L\geq 1}$ is tight and $\langle \cdot \rangle_{\Lambda_L,\beta,\mathsf{h}_L}$ is stochastically dominated by $\langle \cdot \rangle_{\Lambda_{2L},\beta,\mathsf{h}_L+\mathfrak{p}_{2L}}$, it follows that any weak limit $\langle \cdot \rangle_{\beta}$ of $(\langle \cdot \rangle_{\Lambda_L,\beta,{\mathsf{h}_L}})_{L\geq 1}$ is stochastically dominated by $\langle \cdot \rangle^+_{\beta}$. In particular, there exists a monotone coupling $(\varphi,\varphi^+)$, $\varphi\sim \langle \cdot \rangle_{\beta}$, $\varphi^+\sim \langle \cdot \rangle^+_{\beta}$, where $\varphi_x\leq \varphi^+_x$ almost surely for every $x\in \mathbb{Z}^d$ by Strassen's theorem \textup{\cite{Strassen1965}}. Thus, by Proposition~\textup{\ref{prop: thick conv}} and its aforementioned extension, this implies that $\langle \cdot \rangle_{\beta}=\langle \cdot \rangle^+_{\beta}$, hence the full measure converges, namely $\lim_{L\rightarrow \infty} \langle \cdot \rangle_{\Lambda_L,\beta,h_L}=\langle \cdot \rangle^+_\beta$. 
\end{remark}

Let $\tilde{\mathsf{h}}_L$ be the external magnetic field which takes the value $+1$ for all $x\in \{-L,\ldots,-L+\log L-1\}\times \{-L,\ldots,L\}^{d-1}$ and $0$ elsewhere.
\begin{proposition}[One side is enough]\label{prop: one side is enough}
Let $d\geq 2$ and $\beta>0$. Then,
 \begin{equation}
     \liminf_{L\rightarrow\infty}\langle \varphi_0\rangle_{\Lambda_L,\beta,\tilde{\mathsf{h}}_L}\geq \frac{1}{2d}\langle \varphi_0\rangle^+_{\beta}.
 \end{equation}
\end{proposition}

\begin{proof} Observe that $\tilde{\mathsf{h}}_L$ is a restriction of $\mathsf{h}_L$.
By the switching principle,
\begin{equation}\label{eq: switching one side is enough}
     \langle \varphi_0\rangle_{\Lambda_L,\beta,\mathsf{h}_L}- \langle \varphi_0\rangle_{\Lambda_L,\beta,\tilde{\mathsf{h}}_L}= \langle \varphi_0\rangle_{\Lambda_L,\beta, \mathsf{h}_L}\mathbf{P}_{\Lambda_L[\mathsf{h}_L],\Lambda_L[\tilde{\mathsf{h}}_L]}^{0\fg,\emptyset}[\tilde 0 \centernot\longleftrightarrow \fg \text{ in }\mathcal{H}_{\Lambda_L[\tilde{\mathsf{h}}_L]}^{0\fg,\emptyset}(\n_1,\n_2,\ft)].
\end{equation}
However, using the $\pi/2$ rotational symmetries of $\mathbf{P}_{\Lambda_L[\mathsf{h}_L]}^{0\fg}$, together with the fact that the event $\{\tilde{0} \longleftrightarrow \fg \text{ in } \mathcal H^{0\fg}_{\Lambda_L[\tilde{\mathsf{h}}_L]}(\n_1,\ft_1)\}$ happens by the source constraints, and the stochastic domination of Proposition \ref{prop:tanglings_domination}, we obtain
\begin{align}
\mathbf{P}_{\Lambda_L[\mathsf{h}_L],\Lambda_L[\tilde{\mathsf{h}}_L]}^{0\fg,\emptyset}[\tilde 0 \longleftrightarrow \fg \text{ in }\mathcal{H}^{0\fg,\emptyset}_{\Lambda_L[\tilde{\mathsf{h}}_L]}(\n_1,\n_2,\ft)]&\geq \mathbf{P}_{\Lambda_L[\mathsf{h}_L]}^{0\fg}[\tilde 0 \longleftrightarrow \fg \text{ in }\mathcal{H}^{0\fg}_{\Lambda_L[\tilde{\mathsf{h}}_L]}(\n_1,\ft_1)]&\geq \frac{1}{2d}.\label{eq:proof one side enough2}
\end{align}
Combining \eqref{eq:proof one side enough2} with \eqref{eq: switching one side is enough} implies that
\begin{equation}
     \liminf_{L\rightarrow\infty}\langle \varphi_0\rangle_{\Lambda_L,\beta,\tilde{\mathsf{h}}_L}\geq \frac{1}{2d} \liminf_{L\rightarrow\infty}\langle \varphi_0\rangle_{\Lambda_L,\beta, \mathsf{h}_L}=\frac{1}{2d}\langle \varphi_0\rangle^+_\beta,
\end{equation}
where we used Proposition~\ref{prop: thick conv} in the equality.
\end{proof}

\subsection{Boundary magnetisation}

Let $\mathbb{H}=\mathbb{Z}^{d-1}\times \mathbb{N}$, and write $\Lambda_L^+=\Lambda_L(L\mathbf{e}_d)$, where $\mathbf{e}_d$ is the unit vector of last coordinate equal to $1$. Given $h>0$,
we define the measures 
\begin{equation}\label{eq:half-space_measures}
\langle \cdot \rangle^{0,h}_{\Lambda_L^+,\beta}:= \langle \cdot \rangle_{\Lambda_L^+,\beta, \mathsf{h}} ~~~~~\text{ and }~~~~~ \langle \cdot \rangle^{\mathfrak{p},h}_{\Lambda_L^+,\beta}:= \langle \cdot \rangle_{\Lambda_L^+,\beta, \mathfrak{p}+\mathsf{h}},
\end{equation}
where $\mathfrak{p}$ denotes the magnetic field which is equal to $\mathfrak{p}_{\Lambda^+_L}$ on $\partial \Lambda_L^+\setminus \partial \mathbb H$ and $0$ elsewhere, and $\mathsf{h}$ denotes the magnetic field which is equal to $h$ on $\Lambda_L^+\cap \partial \mathbb H$ and $0$ elsewhere. (Although both $\mathfrak{p}$ and $\mathsf{h}$ depend on $L$ and $h>0$, we choose not to stress this dependency in the notation throughout this section as it will be clear in the context.)

\begin{lemma}\label{lem: existence of phi4 halspace measures} Let $d\geq 2$ and $\beta,h>0$. Then, as $L$ tends to infinity, $\langle \cdot \rangle^{\mathfrak p,h}_{\Lambda_L^+,\beta}$ and $\langle \cdot \rangle^{0,h}_{\Lambda_L^+,\beta}$ converge weakly to probability measures $\langle \cdot \rangle^{+,h}_{\mathbb{H},\beta}$ and $\langle \cdot \rangle^{0,h}_{\mathbb{H},\beta}$. Furthermore, both $\langle \cdot \rangle^{+,h}_{\mathbb{H},\beta}$ and $\langle \cdot \rangle^{0,h}_{\mathbb{H},\beta}$ are regular and invariant under the isometries of $\mathbb H$.
\end{lemma}
\begin{proof} We fix $h>0$ and $\beta>0$ and drop the latter from the notations.
We begin with the study of $(\langle \cdot\rangle_{\Lambda_L^+}^{0,h})_{L\geq 1}$. Proposition \ref{prop:stoc_monotonicity} implies the following (stochastic) monotonicity: for every $L\geq 1$, $\langle \cdot \rangle^{0,h}_{\Lambda_L^+} \preceq \langle \cdot \rangle^{0,h}_{\Lambda_{L+1}^+}$. Combined with Proposition \ref{prop:regularity}, this implies that $\langle \cdot \rangle^{0,h}_{\Lambda_L^+}$ converges weakly. 
We now turn to $(\langle \cdot\rangle_{\Lambda_L^+}^{\mathfrak p,h})_{L\geq 1}$. We follow the argument used in the proof of Proposition \ref{prop: thick conv}. Consider a bounded local increasing function $f$. Fix some $m\geq 1$ such that $\mathrm{sup}(f)\subset \Lambda_m^+$, and consider $L$ large enough such that $L\geq 3m$. Then, by Proposition \ref{prop:regularity}, the measures $\langle \cdot \rangle^{\mathfrak p,h}_{\Lambda_L^+}$ (with $L\geq 3m$) are (uniformly) regular in $\Lambda_m^+$. Combining this observation with the Cauchy--Schwarz inequality gives
\begin{equation}
\langle f \mathbbm{1}_{\exists x \in \partial^{\rm ext} \Lambda_m^+\cap\mathbb H : \varphi_x> \mathfrak{M}_{\Lambda_m^+}} \rangle^{\mathfrak p,h}_{\Lambda_L^+} \leq \sqrt{\langle f^2 \rangle^{\mathfrak p,h}_{\Lambda_L^+}} \sqrt{\langle \mathbbm{1}_{\exists x \in \partial^{\rm ext} \Lambda_m^+\cap\mathbb H : \varphi_x> \mathfrak{M}_{\Lambda_m^+}} \rangle^{\mathfrak p,h}_{\Lambda_L^+}}=o(1),
\end{equation}
where $o(1)$ tends to $0$ as $m$ tends to infinity.
By the monotonicity in the boundary conditions (see Proposition \ref{prop:stoc_monotonicity}),
\begin{equation}
\langle f \rangle^{\mathfrak p,h}_{\Lambda_L^+}=\langle f \mathbbm{1}_{\varphi_x\leq \mathfrak{M}_{\Lambda_m},\: \forall x \in \partial^{\rm ext} \Lambda_m^+\cap\mathbb H} \rangle^{\mathfrak p,h}_{\Lambda_L^+}+\langle f \mathbbm{1}_{\exists x \in \partial^{\rm ext} \Lambda_m^+\cap\mathbb H: \varphi_x> \mathfrak{M}_{\Lambda_m}} \rangle^{\mathfrak p,h}_{\Lambda_L^+}\leq \langle f \rangle^{\mathfrak p,h}_{\Lambda_m^+}+o(1).   
\end{equation}
Sending first $L$ to infinity and then $m$ to infinity we obtain 
\begin{equation}
\limsup_{L\to\infty} \langle f \rangle^{\mathfrak p,h}_{\Lambda_L^+} \leq \liminf_{L\to\infty} \langle f \rangle^{\mathfrak p,h}_{\Lambda_L^+},     
\end{equation}
which implies that the limit $\lim_{L\to\infty} \langle f \rangle^{\mathfrak p,h}_{\Lambda_L^+}$ exists. Since this holds for all increasing functions, by the inclusion-exclusion principle, it holds for all bounded local functions. Thus, the sequence $(\langle \cdot \rangle^{\mathfrak p,h}_{\Lambda_L^+})_{L\geq 1}$ converges weakly. 
\end{proof}

Combined with the methods of \cite[Section~5.1]{GunaratnamPanagiotisPanisSeveroPhi42022}, Lemma \ref{lem: existence of phi4 halspace measures} allows to construct the associated infinite volume (double) tangled current measures. Therefore, we obtain the following (the proof is omitted).

\begin{corollary}\label{cor: existence infinite halspace tangled measures} Let $d\geq 2$ and $\beta,h>0$. Then, as $L$ tends to infinity, the measures $\mathbf{P}^{\emptyset}_{\Lambda_L^+[\mathfrak{p}+\mathsf{h}],\beta}$, $\mathbf{P}^{\emptyset}_{\Lambda_L^+[\mathsf{h}],\beta}$, and $\mathbf{P}^{\emptyset,\emptyset}_{\Lambda_L^+[\mathfrak{p}+\mathsf{h}],\Lambda_L^+[\mathsf{h}],\beta}$ converge weakly to probability measures that we denote by $\mathbf{P}^\emptyset_{\mathbb H(+,h),\beta}$, $\mathbf{P}^\emptyset_{\mathbb H(0,h),\beta}$, and $\mathbf{P}^{\emptyset,\emptyset}_{\mathbb H(+,h),\mathbb H(0,h),\beta}$ respectively.
\end{corollary}

In fact, the two measures constructed in Lemma \ref{lem: existence of phi4 halspace measures} coincide for every $\beta,h>0$. Such a result was already derived for the Ising model in the work of Bodineau (see Step 2 in \cite[Lemma~3.1]{Bod05}). The argument used there relies on both the Lee--Yang theorem and a result of Fröhlich and Pfister on the wetting transition \cite{FP87}. Although the Lee--Yang property still holds in our setup, extending the results of Fröhlich and Pfister to the $\varphi^4$ model remains an open problem and is beyond the scope of this paper. Instead, we take a different approach and use the switching principle to obtain a more probabilistic proof of the desired equality. Let us mention that the switching principle was also used to analyse half-space measures in the context of the Book-Ising model, see \cite{duminil2023long}.

\begin{proposition}\label{prop: unique half-space measure} Let $d\geq 2$ and $\beta,h>0$. Then, for every $x \in \mathbb H$,
\begin{equation}\label{eq:equality mag half space}
    \langle \varphi_x\rangle_{\mathbb H,\beta}^{+,h}=\langle \varphi_x\rangle_{\mathbb H,\beta}^{0,h}.
\end{equation}
In particular, $\langle \cdot\rangle^{+,h}_{\mathbb H,\beta}=\langle \cdot\rangle^{0,h}_{\mathbb H,\beta}$.
\end{proposition}

\begin{proof} 
Fix $h>0$ and $\beta>0$, and drop the latter from the notations. Note that the second part of the proposition follows from \eqref{eq:equality mag half space} by a similar argument to the one described in Remark \ref{rem: equality mag implies equality meas}.  We thus focus our attention on \eqref{eq:equality mag half space} and we do the case $x=0$. The same argument can be adapted to the general case.

Recall the definitions of $\mathfrak{p}$ and $\mathsf{h}$ below \eqref{eq:half-space_measures}. 
Notice that $\mathsf{h}$ is a restriction of $\mathfrak{p}+\mathsf{h}$. Hence, the switching principle of Theorem \ref{thm: switching lemma rtc} gives
\begin{equation}\label{eq:proof halfspace 1}
    \langle \varphi_0\rangle_{\Lambda_L^+}^{\mathfrak p,h}-\langle \varphi_0\rangle_{\Lambda_L^+}^{0,h}=\langle \varphi_0\rangle_{\Lambda_n^+}^{\mathfrak{p},h}\mathbf P_{\Lambda_L^+[\mathfrak{p}+\mathsf{h}],\Lambda_L^+[\mathsf{h}]}^{0\fg,\emptyset}[\tilde 0 \centernot\longleftrightarrow \fg \text{ in } \mathcal{H}^{0\fg,\emptyset}_{\Lambda_L^+[\mathsf{h}]}(\n_1,\n_2,\ft)]
\end{equation}
for every $L\geq1$. It is more convenient to work with a sourceless measure. To turn the measure in \eqref{eq:proof halfspace 1} into a sourceless one, we will add an edge between $\tilde{0}$ and $\fg$ in $\n_1$. This is performed rigorously in the next paragraph.

To each triplet $(\n_1,\n_2,\ft)$ with $\partial\n_1=0\fg$ and $\partial\n_2=\emptyset$, we associate a new triplet $(\n_1',\n_2,\ft')$ with $\partial\n_1=\partial\n_2=\emptyset$ defined as follows. First, let $\n_1'(e)=\n_1(e)+\mathds{1}_{e=0\fg}$ and $\ft'_x=\ft_x$ for every $x\neq0$. It remains to define $\ft'_0$, which is basically set to be the same as $\ft_0$, except for the different labelling of $\mathcal{B}^{0\fg,\emptyset}_0(\n_1,\n_2)$ and $\mathcal{B}^{\emptyset,\emptyset}_0(\n'_1,\n_2)$. In order to formally define $\ft$, recall from Section~\ref{sec:tangled_currents} the convention for labelling the blocks and that $\tilde{0}=0a(1)$ stands for the extra point in $\mathcal{B}^{0\fg,\emptyset}_0(\n_1,\n_2)$ due to the source constraint. Additionally, we denote $\overline{0}:=0\fg(\n'_1(0\fg))$ the ``last point'' in $\mathcal{B}^{\emptyset,\emptyset}_0(\n'_1,\n_2)$ associated to the edges from $0$ to $\fg$ in $\n'_1$. This way, we can write $\mathcal{B}^{0\fg,\emptyset}_0(\n_1,\n_2)\setminus\{\tilde{0}\}=\mathcal{B}^{\emptyset,\emptyset}_0(\n'_1,\n_2)\setminus\{\overline{0}\}$. Now, if $\ft=\{P_1,\ldots, P_k\}$ and $P_k$ is the subset of the partition containing $\tilde 0$, we let $\ft':=\{P_1,\ldots,P_{k-1},P_k'\}$ where $P_k':=P_k\cup\{\overline{0}\}\setminus \{\tilde 0\}$. 
Note that the map $(\n_1,\n_2,\ft)\mapsto (\n_1',\n_2,\ft')$ is one-to-one and that $\rho_{\n'_1,\n_2}^{\emptyset,\emptyset}(\ft')=\rho_{\n_1,\n_2}^{0\fg,\emptyset}(\ft)$. Additionally, observe that 
\begin{equation}
    w^{0\fg}_{\mathfrak{p}+\mathsf{h}}(\n_1)w^\emptyset_{\mathsf{h}}(\n_2)=(\beta h)^{-1}\n_1'(0,\fg)w^\emptyset_{\mathfrak{p}+\mathsf{h}}(\n_1')w^\emptyset_{\mathsf{h}}(\n_2).
\end{equation}
Recall the extended measure on quintuples $(\n_1,\n_2,\ft,\ft_1,\ft_2)$ introduced in the end of Section~\ref{sec:tangled_currents}. 
Using the above, 
\begin{multline}\label{eq:proof halfspace 2}
    \langle \varphi_0\rangle_{\Lambda_L^+}^{\mathfrak{p},h}\mathbf P_{\Lambda_L^+[\mathfrak{p}+\mathsf{h}],\Lambda_L^+[\mathsf{h}]}^{0\fg,\emptyset}[\tilde 0 \centernot\longleftrightarrow \fg \text{ in } \mathcal{H}^{0\fg,\emptyset}_{\Lambda_L^+[\mathsf{h}]}(\n_1,\n_2,\ft)]
    \\\leq 
    (\beta h)^{-1}\mathbf E_{\Lambda_L^+[\mathfrak{p}+\mathsf{h}],\Lambda_L^+[\mathsf{h}]}^{\emptyset,\emptyset}\Big[\n_1(0,\fg)\cdot\mathds{1}\Big\{ \overline{0}\longleftrightarrow \fg \text{ in }\mathcal{H}^\emptyset_{\Lambda_L^+[\mathfrak{p}]}(\n_1,\ft_1) ,\: \overline{0}\centernot\longleftrightarrow \fg \text{ in } \mathcal{H}_{\Lambda^+_L[\mathsf{h}]}^{\setminus \{\overline{0}\fg\}}\Big\}\Big],
\end{multline}
where $\mathcal{H}^{\setminus \{\overline{0}\fg\}}_{\Lambda^+_L[\mathsf{h}]}\coloneqq \mathcal{H}_{\Lambda^+_L[\mathsf{h}]}^{\emptyset,\emptyset}(\n_1,\n_2,\ft)\setminus \{\{\overline{0},\fg\}\}$. See Figure \ref{fig:halfspace1} for an illustration of the event appearing on the right-hand side of \eqref{eq:proof halfspace 2}. 

\begin{figure}[htb]
    \centering
    \includegraphics[width=0.7\linewidth]{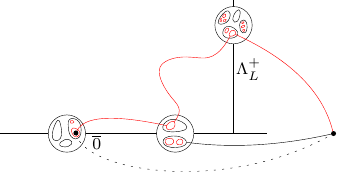}
    \put(2,42){$\mathfrak{g}$}
    \put(-310,38){\Large$\partial \mathbb H$}
    \caption{An illustration of the event appearing on the right-hand side of \eqref{eq:proof halfspace 2}. The red path is a subgraph of $\mathcal{H}^\emptyset_{\Lambda_L^+[\mathfrak{p}]}(\n_1,\ft_1)$ which connects $\overline{0}$ to $\mathfrak{g}$. Red (resp. black) bubbles correspond to elements of the tanglings of $\ft_1$ (resp. $\ft$). Note that by the increasing coupling described in Remark \ref{rem:extended_measure}, every partition class of $\ft_1$ is included in a partition class of $\ft$. The dotted line represents an open edge that we do not consider in the connection events. Note that in this picture, $\overline{0}$ does not connect to $\mathfrak{g}$ in $\Lambda_L^+[\mathsf{h}]$ when the dotted edge is removed.}
    \label{fig:halfspace1}
\end{figure}

Due to Lemma~\ref{lem:degree_exp.moments} and Proposition~\ref{prop:regularity}, $\n_1(0,\fg)$ is uniformly integrable under the sequence $\mathbf P_{\Lambda_L^+[\mathfrak{p}+\mathsf{h}],\Lambda_L^+[\mathsf{h}]}^{\emptyset,\emptyset}$, which in turn converges weakly to $\mathbf P_{\mathbb{H}(+,h),\mathbb H(0,h)}^{\emptyset,\emptyset}$ (by Corollary~\ref{cor: existence infinite halspace tangled measures}). Thus, plugging \eqref{eq:proof halfspace 2} in \eqref{eq:proof halfspace 1} and taking the limit $L\rightarrow \infty$, we obtain
\begin{equation}
    \langle \varphi_0\rangle_{\mathbb H,\beta}^{+,h}-\langle \varphi_0\rangle_{\mathbb H,\beta}^{0,h}\leq (\beta h)^{-1}\mathbf E_{\mathbb{H}(+,h),\mathbb H(0,h)}^{\emptyset,\emptyset}\Big[\n_1(0,\fg)\cdot\mathds{1}\Big\{\overline 0 \longleftrightarrow \infty \text{ in } \mathcal{H}_1, \: \overline 0\centernot\longleftrightarrow \fg \text{ in } \mathcal{H}^{\setminus \{\overline{0}\fg\}} \Big\}\Big],
\end{equation}
where $\mathcal{H}^{\setminus \{\overline{0}\fg\}}\coloneqq \mathcal{H}^{\emptyset,\emptyset}(\n_1,\n_2,\ft)\setminus \{\{\overline{0},\fg\}\}$, and $\mathcal{H}_1:=\mathcal{H}^\emptyset_{\mathbb H}(\n_1,\ft_1)$.
As a conclusion, it is sufficient to show that 
\begin{equation}\label{eq:goal_boundary-magnetisation}
    \mathbf{P}[\overline 0 \longleftrightarrow \infty \text{ in } \mathcal{H}_1, \: \overline 0\centernot\longleftrightarrow \fg \text{ in } \mathcal{H}^{\setminus \{\overline{0}\fg\}} ]=0,
\end{equation}
where we write $\mathbf{P}:=\mathbf{P}_{\mathbb{H}(+,h),\mathbb H(0,h)}^{\emptyset,\emptyset}$ to shorten the notation.

We now proceed to prove \eqref{eq:goal_boundary-magnetisation}.
This follows from the following lemma, which says that when the cluster of $\overline{0}$ in $\mathcal{H}_1$ is infinite, it necessarily touches $\partial \mathbb H$ at infinitely many ``good sites'', i.e.~sites with bounded $\n_1$ degree, see Figure \ref{fig:halfspace2} for an illustration.

\begin{lemma}\label{lem:infinite_touch-points}
For every $M\geq1$, let $\mathcal{G}_M\coloneqq \{x\in\partial \mathbb{H}:~\Delta{\n_1}(x)\leq M\}$. Then,
    \begin{equation}\label{eq:infinite_touch-points}
    \mathbf{P}
    [ |\mathcal{C}_1(0)|=\infty,\, |\pi(\mathcal{C}_1(0))\cap \mathcal{G}_M|<\infty~~ \forall\, M\geq1]=0,
    \end{equation}
where $\mathcal{C}_1(0)$ denotes cluster of $\overline 0$ in $\mathcal{H}_1$ and $\pi:\mathcal{H}_1\to \mathbb{H}$ is the natural projection.
\end{lemma}

Before proving the Lemma~\ref{lem:infinite_touch-points}, let us conclude the proof of Proposition~\ref{prop: unique half-space measure}. The idea 
is that at the ``good sites'' in $\mathcal{G}_M$, there is a uniformly positive conditional probability of connecting to the ghost $\fg$ in the double tangled current $\mathcal{H}^{\setminus \{\overline{0}\fg\}}$. The argument is similar to the one used in the proof of Proposition~\ref{prop: thick conv}. 

For each $M\geq1$, let $A_{M}:=\{|\pi(\mathcal{C}_1(0))\cap \mathcal{G}_{M}|=\infty\}$. We will prove that 
\begin{equation}\label{eq: intermediate}
    \mathbf{P}[\overline{0}\longleftrightarrow \fg \text{ in } \mathcal{H}^{\setminus \{\overline{0}\fg\}} \,|\, A_M]=1 ~~~~\forall M\geq1,
\end{equation}
which readily implies \eqref{eq:goal_boundary-magnetisation} by Lemma~\ref{lem:infinite_touch-points}. We condition on the $(\n_1,\ft_1)$-measurable set $\mathcal{U}_M:=\pi(\mathcal{C}_1(0))\cap \mathcal{G}_{M}$, which is infinite on the event $A_M$.
Since $\n_2$ is independent of $\n_1$, we can apply Lemma~\ref{lem:degree_exp.moments},  Proposition~\ref{prop:regularity}, and the exponential Markov inequality to deduce that for some $D>0$ we have
\begin{equation}
     \mathbf{P}[|\mathcal{V}_M|=\infty \,|\, A_M]=1,
\end{equation}
where $\mathcal{V}_M:=\{x\in \mathcal{U}_M: \Delta_{\mathbb H}\n_2(x)\leq D,~\Delta\n_2(x)\leq D+2\}$ and $\Delta_\mathbb{H}\n_2(x)=\sum_{y\in \mathbb H}\n_2(x,y)$.
Note that $\Delta\n_1(x)+\Delta_{\mathbb H}\n_2(x)\leq M'$ and $\Delta\n_1(x)+\Delta\n_2(x)\leq M'+2$ for every $x\in \mathcal{V}_M$, where $M':=M+D$.
Following the same lines as in the proof of Proposition~\ref{prop: thick conv}, one can prove that the random variables $\omega_x:=\mathbbm{1}_{\{\n_2(x\fg)>0,\, \ft_x=\mathrm{ECT}(\mathcal{B}_x)\}}$ (for $x\in \mathcal{V}_M$) stochastically dominate a sequence of i.i.d.~Bernoulli random variables of some parameter $\delta=\delta(M)>0$.
However, the event $\{\overline{0}\longleftrightarrow \fg \text{ in } \mathcal{H}^{\setminus \{\overline{0}\fg\}}\}$ happens as long as $\omega_x=1$ for some $x\in \mathcal{V}_M$. Thus,
\begin{equation}
    \mathbf{P}[\overline{0}\longleftrightarrow \fg \text{ in } \mathcal{H}^{\setminus \{\overline{0}\fg\}} \,|\, |\mathcal{V}_M|=\infty]=1,
\end{equation}
which implies \eqref{eq: intermediate}. This concludes the proof.
\end{proof}

\begin{figure}[htb]
    \centering
    \includegraphics[width=0.7\linewidth]{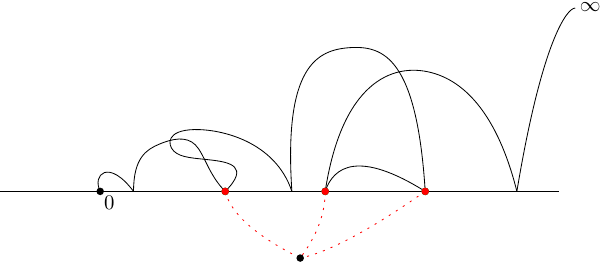}
    \put(-150,-7){$\mathfrak{g}$}
    \caption{An illustration of the strategy used in the proof of \eqref{eq:goal_boundary-magnetisation}. The infinite cluster lying in $\mathcal{H}_1$ has to touch the boundary infinitely many times. Even better, as shown in Lemma \ref{lem:infinite_touch-points}, it has to touch infinitely many ``good'' blocks where the total degree is uniformly bounded. These blocks are pictured in red. In each of them, there is a positive probability of connecting to the ghost. It is hence a zero probability event to observe an infinite cluster which manages to avoid $\mathfrak{g}$. }
    \label{fig:halfspace2}
\end{figure}

\begin{proof}[Proof of Lemma~\textup{\ref{lem:infinite_touch-points}}]
For every $L\geq 1$, let $\mathcal{N}_L$ be the number of distinct infinite clusters of $\mathcal{H}^1$ intersecting $\Lambda^+_L$. We will prove that, 
\begin{equation}\label{eq:goal_infinite_touch-points}
    \mathbf{E}[\mathcal{N}_L]=o(L^{d-1}).
\end{equation}
The lemma is a direct consequence of \eqref{eq:goal_infinite_touch-points}. Indeed, assume by contradiction that \eqref{eq:infinite_touch-points} does not hold. Since $\mathcal{G}_M\uparrow\partial \mathbb{H}$ almost surely as $M\to \infty$, there exist $1\leq M,K<\infty$ and $\delta>0$ such that 
\begin{equation}
\mathbf{P}[|\mathcal{C}_1(0)|=\infty,\, 0\in \mathcal{G}_M,\, |\mathcal{C}_1(0)\cap\mathcal{G}_M|\leq K]\geq \delta.
\end{equation}
Consider $\mathcal{U}\coloneqq \{x\in \partial\mathbb{H}:|\mathcal{C}_1(x)|=\infty,\, x\in \mathcal{G}_M,\, |\mathcal{C}_1(x)\cap\mathcal{G}_M|\leq K\}$, where $\mathcal{C}_1(x)$ is the cluster of $\overline{x}\coloneqq x\fg(\n_1(x\fg))$ in $\mathcal{H}_1$. Notice that by definition, each infinite cluster $\mathcal{C}$ coincides with at most $K$-many $\mathcal{C}(x)$, $x\in \mathcal{U}$. Therefore, by translation invariance one obtains
\begin{equation}
\mathbf{E}[\mathcal{N}_L]\geq \frac{1}{K}\mathbf{E}[|\mathcal{U}\cap \Lambda^+_L|]=\frac{1}{K}\sum_{x\in \partial\mathbb{H}\cap\Lambda^+_L} \mathbf{P}[x\in \mathcal{U}]
\geq \frac{\delta}{K}|\partial\mathbb{H}\cap\Lambda^+_L|\geq c(\delta,K)L^{d-1},
\end{equation}
which is in contradiction with \eqref{eq:goal_infinite_touch-points}.

It remains to prove \eqref{eq:goal_infinite_touch-points}, which we do now. We first show that infinite clusters are unlikely to be close to each other deep in the bulk. More precisely, let $\mathcal{N}_m(x)$ be the number of infinite clusters intersecting the box  $\Lambda^+_m(x):=x+\Lambda_m^+$ and consider
$$p_m(k)\coloneqq \mathbf{P} [\mathcal{N}_m(k\mathbf{e}_d)\geq 3],$$
where we recall that $\mathbf{e}_d=(0,\dots,0,1)$. We claim that for every $m\geq0$,
\begin{equation}\label{eq:pk_to_0}
    p_m(k)\longrightarrow 0 \text{ as } k\to\infty.
\end{equation}
The proof of \eqref{eq:pk_to_0} follows a Burton--Keane type argument (see \cite{BurtonKeane1989density}). First, fix $m\geq0$, $\delta>0$ and let $\mathcal{K}\subset \mathbb{N}$ be the set of $k\geq0$ such that $p_m(k)\geq \delta$. We will prove that $\mathcal{K}$ is finite, thus implying \eqref{eq:pk_to_0}. 
We say that an $m$-box $\Lambda^+_m(z)$ is a \emph{coarse-trifurcation} if every vertex in $\Lambda^+_m(z)$ belongs to the same infinite cluster $\mathcal{C}_z$, which in turn breaks into at least $3$ distinct infinite clusters if all vertices of $\Lambda^+_m(z)$ are removed. 
Notice that if $\mathcal{N}_m(z)\geq 3$, $\n_1(x,y)\geq1$ for all neighbours $x,y\in \Lambda^+_m(z)\cap \mathbb{H}$ and $\ft_x=\mathrm{ECT}$ for all vertices $x\in \Lambda^+_m(z)\cap \mathbb{H}$, then $\Lambda^+_m(z)$ is a coarse-trifurcation.
By a simple insertion-tolerance property, as in \cite[Proposition~5.12]{GunaratnamPanagiotisPanisSeveroPhi42022},
we deduce that there exists $\delta'=\delta'(m,\delta)>0$ such that for every $k\in \mathcal K$,
\begin{equation}\label{eq:qm>delta'}
    q_m(k)\coloneq \mathbf{P}[\Lambda^+_m(k\mathbf{e}_d) \text{ is a coarse-trifurcation}]\geq \delta'.
\end{equation} 
For full disclosure, here we used the following statement: for every $m\geq0$ and $\delta>0$ there exists $\delta'=\delta'(m,\delta)$>0 such that if an event $\mathcal{A}$ satisfies $\mathbf{P}[\mathcal{A}]\geq \delta$, then for any $m$-box $\Lambda^+_m(z)$ lying in $\mathbb H$, one has $\mathbf{P}[\Phi_{\Lambda^+_m(z)}^\mathcal{T}(A)]\geq \delta'$, where $\Phi^\mathcal{T}_{\Lambda_m(z)}$ is as in \cite[Definition~5.11]{GunaratnamPanagiotisPanisSeveroPhi42022}. Although this is not exactly the same statement as \cite[Proposition~5.12]{GunaratnamPanagiotisPanisSeveroPhi42022}, the proof is precisely the same.

For $L\geq 8m$, let $\mathcal{T}_L:=|\{z\in 8m\mathbb{H} \cap \Lambda^+_{L-8m}: \Lambda^+_m(z) \text{ is a coarse-trifurcation}\}|$ be the number of well spaced coarse-trifurcations of size $m$ deep inside $\Lambda^+_L$. By a classical deterministic argument--- see e.g.~the proof of \cite[Proposition 5.13]{GunaratnamPanagiotisPanisSeveroPhi42022}--- 
we have $\mathcal{T}_L\leq \mathcal{N}_L$. On the one hand, by translation invariance we have 
\begin{align*} 
\mathbf{E}[\mathcal{T}_L] &= \sum_{x\in (8m\mathbb{H})\cap\Lambda^+_{L-8m}} \mathbf{P}[\Lambda^+_m(x) \text{ is a coarse-trifurcation}]
\geq c\delta'|\mathcal{K}\cap [0,L-8m]| L^{d-1}.
\end{align*}
On the other hand, $\mathcal{N}_L$ is smaller than the sum of all the degrees over $\partial \Lambda^+_L$, hence $\mathbf{E}[\mathcal{N}_L]\leq CL^{d-1}$ by Lemma~\ref{lem:degree_exp.moments} and Proposition~\ref{prop:regularity}. We conclude that $|\mathcal{K}\cap [0,L-8m]|\leq C/c\delta'$ for every $L\geq 8m$, thus $\mathcal{K}$ is finite and \eqref{eq:pk_to_0} follows.

Finally, we show how to deduce \eqref{eq:goal_infinite_touch-points} from \eqref{eq:pk_to_0}. First, by applying Lemma~\ref{lem:degree_exp.moments} and Proposition~\ref{prop:regularity} again, we have $\mathbf{E}[\mathcal{N}_m(x)^2]\leq C(m^{d-1})^2$. Therefore, for $x\in \mathbb{Z}^{d-1}\times \{k\}$, $k\geq0$, by the Cauchy--Schwarz inequality, we have
\begin{align} 
\mathbf{E}[\mathcal{N}_m(x)]&\leq 2 + \mathbf{E}[\mathcal{N}_m(x) \mathbbm{1}_{\mathcal{N}_m(x)\geq3}] 
\notag\\&
\leq 2+\sqrt{p_m(k)\mathbf{E}[\mathcal{N}_m(x)^2]}
\leq 2+C_1\sqrt{p_m(k)} m^{d-1}.\label{eq:bound exp nm(x)}
\end{align}
Let $I\coloneqq \lceil \frac{L}{m} \rceil$ and, for every $i \in \{1,\dots, I \}$, cover $\partial \Lambda^+_L \cap (\mathbb{Z}^{d-1}\times [(i-1)m, im])$ with boxes $\Lambda^+_m(x^i_j)$, $1\leq j\leq J_i$, where $x^i_j\in \mathbb{Z}^{d-1}\times \{im\}$. Notice that we can take $J_i\leq C_2 (L/m)^{d-2}$ for every $i\leq I-1$, and $J_I\leq C_2 (L/m)^{d-1}$. Now, for every $L$ large enough,
\begin{align*} 
\mathbf{E}[\mathcal{N}_L] &\leq \sum_{\substack{i=1}}^I\sum_{j=1}^{J_i} \mathbf{E}[\mathcal{N}_m(x^i_j)]
\\&
\leq \sum_{i=1}^{I} |J_i| \left( 2+ C_1\sqrt{p_m(im)}\, m^{d-1}\right) \\
&\leq 4C_2 \left(\frac{L}{m}\right)^{d-1}+ C_1C_2\left(\frac{L}{m}\right)^{d-2}m^{d-1} \sum_{i=1}^{I-1}\sqrt{p_m(im)}+C_1C_2 L^{d-1}\sqrt{p_m(Im)}\\&= C_2L^{d-1} \left( \frac{4}{m^{d-1}} + \frac{C_1}{I-1} \sum_{i=1}^{I-1} \sqrt{p_m(im)}+C_1\sqrt{p_m(Im)}  \right)
\\&
\leq C_3 \left(\frac{L}{m}\right)^{d-1},
\end{align*}
where we used \eqref{eq:bound exp nm(x)} in the second line, and that due to \eqref{eq:pk_to_0}, one has $\sqrt{p_m(Im)}\to 0$ and $\frac{1}{I-1} \sum_{i=1}^{I-1} \sqrt{p_m(im)}\to 0$ as $L\to\infty$.
Since $m$ is arbitrary, \eqref{eq:goal_infinite_touch-points} follows.
\end{proof}

\section{Random cluster for the $\varphi^4$ model}\label{sec:random_cluster}

In this section, we introduce the random cluster representation of the $\varphi^4$ model. This representation arises by viewing $\varphi^4$ as an Ising model with random coupling constants coming from the absolute value field, and considering the random cluster representation of the resulting (random) Ising model. After we define the model, we derive some basic properties and correlation inequalities. We stress from the outset that the definitions and properties are true in greater generality: we could replace the product of $\varphi^4$ single-site measures by any product of single-site even measures $\mu$ on $\mathbb{R}$ having super-Gaussian tails, i.e.\ $\int e^{cx^2} \mathrm{d}\mu(x) <\infty$ for every $c>0$, so that the corresponding spin model is well defined for all values of $\beta\geq 0$.

\subsection{Definition of the $\varphi^4$ random cluster model}

Let $\Lambda\subset \mathbb Z^d$ be finite. Define $\partial^{\rm ext} \Lambda:=\{x\not \in \Lambda: \exists y\in \Lambda, \: y\sim x\}$ and set $\overline \Lambda := \Lambda \sqcup \partial^{\rm ext} \Lambda$. Also, recall that $E(\Lambda )= \{ xy \in  E(\mathbb Z^d) : x,y \in \Lambda\}$, and define $\overline E(\Lambda) = \{ xy \in E(\mathbb Z^d) : x\text{ or }y \in \Lambda \}$. 
Given an external magnetic field $\mathsf h \in (\mathbb R^+)^\Lambda$, recall the definition of the (weighted) graph $\Lambda[\mathsf h]=(\Lambda^\fg,E(\Lambda[\mathsf h]))$ from Section~\ref{sec:tangled_currents}, and also define $\overline E(\Lambda[\mathsf h]):=E(\Lambda[\mathsf h])\cup \overline{E}(\Lambda)$.
A boundary condition on $\Lambda[\mathsf h]$ is a pair $(\xi, \mathsf b)$, where $\xi=\{\xi_1,\ldots,\xi_{|\xi|}\}$ is a partition of $\partial^{\rm ext} \Lambda$, and $\mathsf{b}\in (\mathbb R^+)^{\partial^{\rm ext} \Lambda}$. Given a percolation configuration $\omega \in \{0,1\}^{\overline E(\Lambda[\mathsf h])}$, we denote by $k^\xi(\omega)$ the number of connected components in the graph obtained from  $(\overline{\Lambda}^{\mathfrak g}, \{ e \in\overline E(\Lambda[\mathsf h]):\omega_e=1\})$ by identifying the vertices in the elements of the partition $\xi$. In what follows, we abuse the notation and write $\mathrm{d}\rho_{g,a}(\mathsf a_x)$ to denote the pushforward $\mathrm d \rho_{g,a}(\varphi_x)$ under the map $\varphi_x\mapsto |\varphi_x|=\mathsf a_x$.

\begin{definition} \label{def: random cluster phi4}
Let $\Lambda \subset \mathbb Z^d$ be finite, $(\xi,\mathsf{b})$ be a boundary condition, and $\mathsf h\in (\mathbb R^+)^{\Lambda}$ be an external field. The $\varphi^4$ random cluster measure $\Psi_{\Lambda,\beta,\mathsf{h}}^{(\xi,\mathsf{b})}$ on $\Lambda$ with boundary condition $(\xi,\mathsf{b})$, at inverse temperature $\beta\geq 0$, and with external magnetic field $\mathsf{h}$ is the probability measure on pairs $(\omega,\mathsf{a})\in \{0,1\}^{\overline E(\Lambda[\mathsf{h}])}\times (\mathbb R^+)^{\overline{\Lambda}}$ satisfying $\mathsf{a}_x=\mathsf{b}_x$ for all $x\in \partial^{\rm ext} \Lambda$ defined by
\begin{equation}\label{eq: random cluster explicit density}
\begin{aligned}
\mathrm{d}\Psi_{\Lambda,\beta,\mathsf{h}}^{(\xi,\mathsf{b})}[(\omega,\mathsf{a})]=&\frac{\mathbbm 1_{\mathsf{a}|_{\partial^{{\rm ext}}\Lambda} = \mathsf b}}{Z^{(\xi,\mathsf{b})}_{\Lambda,\beta,\mathsf{h}}}\prod_{xy \in \overline{E}(\Lambda)} \sqrt{1-p(\beta,\mathsf a)_{xy}}\Big(\frac{p(\beta,\mathsf a)_{xy}}{1-p(\beta,\mathsf a)_{xy}}\Big)^{\omega_{xy}} \, \\
& \times \prod_{x\in \Lambda} \sqrt{1-p(\beta, \mathsf h, \mathsf a)_{x\fg}}\Big(\frac{p(\beta, \mathsf{h}, \mathsf a)_{x\fg}}{1-p(\beta,\mathsf{h}, \mathsf a)_{x\fg}}\Big)^{\omega_{x\fg}} 2^{k^\xi(\omega)} 
\prod_{x\in \Lambda}\mathrm{d} \rho_{g,a}(\mathsf{a}_x),
\end{aligned}
\end{equation}
where $p(\beta, \mathsf a)_{xy}:=1-e^{-2\beta \mathsf a_x \mathsf a_y}$, $p(\beta,\mathsf{h}, \mathsf a)_{x\fg}:=1-e^{-2\beta \mathsf h_x \mathsf a_x}$, and $Z^{(\xi,\mathsf{b})}_{\Lambda,\beta,\mathsf{h}}$ is a normalisation constant. We write $\Phi^{(\xi,\mathsf b)}_{\Lambda,\beta,\mathsf{h}}$ for the marginal on $\omega$ and $\mu_{\Lambda,\beta,\mathsf h}^{(\xi,\mathsf b)}$ for the marginal on $\mathsf a$.

If $\xi=\{\{x\}: x\in \partial^{\rm ext} \Lambda\}$, i.e.\ each partition class consists of a single vertex,
and $\mathsf{b}\equiv 0$, we call it the {free boundary condition}, and simply write $\Psi_{\Lambda,\beta,\mathsf{h}}^{0}$, $\Phi^{0}_{\Lambda,\beta,\mathsf{h}}$,  and $\mu_{\Lambda,\beta,\mathsf{h}}^{0}$. If $\xi=\{ \partial^{\rm ext} \Lambda\}$, i.e.\ all vertices are in the same partition,
and $\mathsf b \equiv \mathfrak M_{\Lambda}$, we call it the {wired plus boundary condition}, and we simply write $\Psi_{\Lambda,\beta,\mathsf{h}}^{(w,\mathfrak{p})}$, $\Phi^{(w,\mathfrak{p})}_{\Lambda,\beta,\mathsf{h}}$, and $\mu_{\Lambda,\beta,\mathsf{h}}^{(w,\mathfrak{p})}$. When $\mathsf{h}=0$, we may drop it from the notation.
\end{definition}

\begin{remark}
Unlike in the definition of the $\varphi^4$ measure from Section~\textup{\ref{sec:preliminaries}}, here we consider both boundary conditions and a magnetic field. We chose to define the random cluster measure in this way so that we can state its domain Markov property more easily. This makes it possible to define the same random cluster measure in different ways, by considering either a wired (i.e.~$\xi=\{\partial^{\rm ext}\}$) boundary condition (with $\mathsf{h}=0$) or a magnetic field supported on the boundary.  For instance, the measures $\Psi_{\Lambda,\beta,\mathfrak{p}}^{0}$ can be obtained from $\Psi_{\Lambda,\beta}^{(w,\mathfrak{p})}$ by identifying all vertices of $\partial^{\rm ext}\Lambda$ as a single vertex that we call $\fg$. 
\end{remark}

\begin{remark}
In \eqref{eq: random cluster explicit density}, we could replace $\rho_{g,a}$ with any even single-site measure $\mu$ with super-Gaussian tails. In particular, choosing $\mu$ to be an even measure which is a mixture of $\delta_0$ and $\delta_{\pm 1}$ we recover the dilute random cluster measure considered in \textup{\cite{Graham2006random, gunaratnam2024blume}} in the context of the Blume--Capel model.
\end{remark}

It is straightforward to check from \eqref{eq: random cluster explicit density} that a \emph{domain Markov property} holds for $\varphi^4$ random cluster measures, as stated in the following proposition.

\begin{proposition}[Domain Markov property]\label{prop: domain markov}
Let $\Lambda'\subset \Lambda$ be finite subsets of $\mathbb Z^d$, and $\beta\geq 0$. Consider a boundary condition $(\xi,\mathsf{b})$ on $\Lambda$ and let $\mathsf h\in (\mathbb{R}^+)^{\Lambda}$. For every $\theta\in \{0,1\}^{\overline E(\Lambda[\mathsf{h}]) \setminus \overline E(\Lambda'[\mathsf{h}])}$ and every $\mathsf{s}\in (\mathbb{R}^+)^{\Lambda\setminus \Lambda'}$, we have
\begin{equation}
\Psi^{(\xi,\mathsf{b})}_{\Lambda,\beta,\mathsf{h}}[(\omega,\mathsf{a}) \mid (\omega|_{\overline E(\Lambda[\mathsf{h}])\setminus \overline E(\Lambda'[\mathsf{h}])},\mathsf{a}|_{\Lambda\setminus \Lambda'})=(\theta,\mathsf{s})]=\Psi^{(\xi^{\theta},\mathsf{s})}_{\Lambda',\beta,\mathsf{h}}[(\omega,\mathsf{a})],
\end{equation}
where $\xi^{\theta}$ is the partition of $\partial^{\rm ext} \Lambda'$ where $x,y$ are in the same partition if and only if they are in the same connected component in the graph obtained from  $((\overline{\Lambda}\setminus \Lambda')^{\mathfrak g}, \{ e \in \overline E(\Lambda[\mathsf h])\setminus \overline E(\Lambda'[\mathsf h]):\theta_e=1\})$ by identifying the vertices in the elements of the partition $\xi$.
\end{proposition}

\begin{remark}
In words, the domain Markov property holds whenever we condition on the value of $\omega$ on a set of edges $E$, and the value of $\mathsf{a}$ on the endpoints of the edges in $E$. The reason why we define the configuration $\omega$ on $\overline E(\Lambda)$ in the first place is so that the domain Markov property holds.
\end{remark}

\subsection{Coupling to the spin model}

We now couple the $\varphi^4$ model with its random cluster measure introduced above in the same spirit as the Edwards--Sokal coupling between the Ising model and the random cluster model with cluster weight $q=2$ (see \cite[Chapter 1]{Grimmett2006RCM}). We will prove that the marginal $\mathsf{a}$ under $\Psi^0_{\Lambda,\beta,\mathsf{h}}$ is distributed according to the absolute value field under $\nu_{\Lambda,\beta,\mathsf{h}}$, and that the percolation marginal $\omega$--- conditionally on $\mathsf{a}$--- has the law of a FK-Ising model in a random environment coming from $\mathsf{a}$. 

We begin by introducing the FK-Ising random cluster model with inhomogeneous coupling constants. Let $\Lambda \subset \mathbb Z^d$, $\mathsf a \in (\mathbb R^+)^{\overline \Lambda}$, $
\mathsf h \in (\mathbb R^+)^\Lambda$, and $\xi$ be a partition of $\partial^{\rm ext}\Lambda$. Recall that, given $\beta \geq 0$, $p(\beta, \mathsf a)_{xy}=1-e^{-2\beta \mathsf a_x \mathsf a_y}$ for every edge $xy \in \overline E(\Lambda)$ and $p(\beta, \mathsf h, \mathsf a)_{x\mathfrak g}=1-e^{-2\beta \mathsf h_x \mathsf a_x}$ for every $x \in \Lambda$. Denote by $\boldsymbol \phi^\xi_{\Lambda,\beta,\mathsf h, \mathsf a}$ the probability measure on $\omega\in \{0,1\}^{\overline E(\Lambda[\mathsf h])}$ defined by
\begin{equation} \label{def: random weight rc}
\boldsymbol{\phi}^\xi_{\Lambda,\beta,\mathsf h, \mathsf a}[\omega]
=
\frac{1}{{\bf Z}^{\xi}_{\Lambda,\beta,\mathsf h, \mathsf a}}\prod_{xy \in \overline E(\Lambda)} \Big(\frac{p(\beta, \mathsf a)_{xy}}{1-p(\beta, \mathsf a)_{xy}}\Big)^{\omega_{xy}} \prod_{x \in \Lambda} \Big(\frac{p(\beta, \mathsf{h},\mathsf a)_{x\fg}}{1-p(\beta,\mathsf{h}, \mathsf a)_{x\fg}}\Big)^{\omega_{x\fg}} \, 2^{k^\xi(\omega)},
\end{equation} 
where ${\bf Z}^{\xi}_{\Lambda,\beta,\mathsf h, \mathsf a}$ is a normalisation constant.

 We now turn to definition of the Ising model (with inhomogeneous coupling constants) that arises in the Edwards--Sokal coupling of $\boldsymbol \phi^\xi_{\Lambda,\beta,\mathsf h, \mathsf a}$. For this definition, we may also consider $\mathsf{h}\in \mathbb R^\Lambda$. We say that a configuration $\sigma \in \{\pm 1\}^{\overline \Lambda}$ is $\xi$-admissible, and write $\sigma \sim_{\rm ext} \xi$, if $\sigma_x$ is constant for every $x$ in the same element of the partition $\xi$. Let $\langle \cdot \rangle^{{\rm Ising}, \xi}_{\Lambda,\beta,\mathsf h, \mathsf a}$ to the probability measure on $\{\pm 1\}^{\overline \Lambda}$ which assigns the following weight to every configuration $\sigma$:
\begin{equation}\label{eq:Ising_def}
\langle \sigma \rangle^{\textup{Ising},\xi}_{\Lambda,\beta,\mathsf h, \mathsf{a}}=\frac{ \mathbbm 1_{\sigma\sim_{\rm ext}\xi}
}{Z^{\textup{Ising},\xi}_{\Lambda,\beta,\mathsf h, \mathsf{a}}} \exp\Bigg( \beta\sum_{xy\in \overline{E}(\Lambda)}  \mathsf a_x \mathsf{a}_y \sigma_x \sigma_y + \beta\sum_{x\in \Lambda} \mathsf{h}_x  \mathsf{a}_x \sigma_x \Bigg),
\end{equation}
where $Z^{\textup{Ising},\xi}_{\Lambda,\beta,\mathsf h, \mathsf{a}}$ is a normalisation constant. If $\xi=\{\{x\}: x\in \partial^{\rm ext}\Lambda\}$, we let $Z^{\textup{Ising},\xi}_{\Lambda,\beta,\mathsf h, \mathsf{a}}=Z^{\textup{Ising},0}_{\Lambda,\beta,\mathsf h, \mathsf{a}}$, and if $\xi=\{\partial^{\rm ext}\Lambda\}$, we let $Z^{\textup{Ising},\xi}_{\Lambda,\beta,\mathsf h, \mathsf{a}}=Z^{\textup{Ising},w}_{\Lambda,\beta,\mathsf h, \mathsf{a}}$. Note the following fundamental relation,
\begin{equation}\label{eq: relation phi4 Ising partition functions}
    Z^{\varphi^4}_{\Lambda,\beta,\mathsf{h}}=\int_{(\mathbb R^+)^\Lambda}Z^{\rm Ising, 0}_{\Lambda,\beta,\mathsf{h},\mathsf{a}}\prod_{x\in \Lambda}\mathrm{d}\rho_{g,a}(\mathsf{a}_x).
\end{equation}

\begin{proposition} \label{prop: phi4rc is annealed rndisingrc}
Let $\Lambda \subset\mathbb Z^d$ be finite. For every boundary condition $(\xi,\mathsf{b})$ on $\Lambda$, every external magnetic field $\mathsf h\in (\mathbb R^+)^{\Lambda}$, and every inverse temperature $\beta\geq 0$ we have
\begin{equation} \label{eq: annealed-quenched rc}
\mathrm{d}\Psi_{\Lambda,\beta,\mathsf h}^{(\xi,\mathsf{b})}[(\omega,\mathsf{a})]=\boldsymbol{\phi}^{\xi}_{\Lambda,\beta,\mathsf h, \mathsf a}[\omega] \mathrm{d}\mu_{\Lambda,\beta,\mathsf{h}}^{(\xi,\mathsf{b})}(\mathsf{a}),
\end{equation}
and 
\begin{equation}\label{eq: absolute value density}
\mathrm{d}\mu_{\Lambda,\beta,\mathsf{h}}^{(\xi,\mathsf{b})}(\mathsf{a}) = 2\mathbbm 1_{\mathsf a|_{\partial^{\rm ext}\Lambda}=\mathsf b} \frac{Z^{{\rm Ising}, \xi}_{\Lambda,\beta,\mathsf h, \mathsf a}}{Z^{(\xi,\mathsf b)}_{\Lambda,\beta,\mathsf h}} \prod_{x \in \Lambda} \mathrm{d}\rho_{g,a}(\mathsf a_x).
\end{equation}
Moreover, $\mu^{0}_{\Lambda,\beta,\mathsf h}$ is the law of the absolute value field of the $\varphi^4$ spin measure $\nu_{\Lambda,\beta,\mathsf h}$.
\end{proposition}

\begin{proof}
Observe that \eqref{eq: annealed-quenched rc} holds with
\begin{equation}\label{eq: expression dmu section 6}
\mathrm{d}\mu_{\Lambda,\beta,\mathsf{h}}^{(\xi,\mathsf{b})}(\mathsf{a})=\mathbbm 1_{\mathsf a|_{\partial^{\rm ext}\Lambda}=\mathsf b} \frac{{\bf Z}^{\xi}_{\Lambda,\beta,\mathsf h, \mathsf a}}{Z^{(\xi,\mathsf{b})}_{\Lambda,\beta,\mathsf{h}}}\prod_{xy \in \overline{E}(\Lambda)} \sqrt{1-p(\beta,\mathsf a)_{xy}} \prod_{x\in \Lambda} \sqrt{1-p(\beta, \mathsf h, \mathsf a)_{x\fg}}
\prod_{x\in \Lambda}\mathrm{d} \rho_{g,a}(\mathsf{a}_x).  \end{equation}
The equality \eqref{eq: annealed-quenched rc} then follows if we can establish
\begin{equation}\label{eq: partition function equality}
Z^{\textup{Ising},\xi}_{\Lambda,\beta,\mathsf h, \mathsf{a}}=\frac{1}{2}{\bf Z}^{\xi}_{\Lambda,\beta, \mathsf h, \mathsf{a}}\prod_{xy \in \overline{E}(\Lambda)} \sqrt{1-p(\beta,\mathsf a)_{xy}} \prod_{x\in \Lambda} \sqrt{1-p(\beta, \mathsf h, \mathsf a)_{x\fg}}.
\end{equation}
Here, the factor $1/2$ comes from the fact that we have (implicitly) fixed $\sigma_{\fg}=1$.

Let us now show \eqref{eq: partition function equality}. We say that a configuration $\omega\in \{0,1\}^{\overline{E}(\Lambda[\mathsf{h}])}$ is \emph{compatible} with some $\sigma\in \{\pm 1\}^{\overline\Lambda}$ and write $\omega\sim \sigma$ if $\omega_{xy}=0$ for every $xy\in \overline{E}(\Lambda[\mathsf{h}])$ such that $\sigma_x\neq \sigma_y$. Note that, if additionally $\xi \sim_{\rm ext} \sigma$ (which means that $x,y\in \partial^{\rm ext}\Lambda$ such that $\sigma_x\neq \sigma_y$ cannot lie in the same partition class of $\xi$), each $\omega$ is compatible with $2^{k^\xi(\omega)-1}$ configurations $\sigma$, since $\sigma$ needs to be constant on each cluster of $\omega$ and there are $2$ possibilities for the sign of each cluster of $\omega$ other than the cluster of $\fg$, where we implicitly extend $\sigma$ by setting $\sigma_{\mathfrak g}=1$. Thus,
\begin{equation}
{\bf Z}^{\xi}_{\Lambda,\beta, \mathsf h, \mathsf{a}}=2\sum_{\sigma\in \{\pm 1\}^{\overline \Lambda}}\mathbbm 1_{\xi\sim_{\rm ext}\sigma}\sum_{\omega\sim \sigma}\prod_{xy \in \overline{E}(\Lambda)} \Big(\frac{p(\beta,\mathsf{a})_{xy}}{1-p(\beta,\mathsf{a})_{xy}}\Big)^{\omega_{xy}} \prod_{x \in \Lambda} \Big(\frac{p(\beta,\mathsf{h},\mathsf{a})_{x\fg}}{1-p(\beta,\mathsf{h},\mathsf{a})_{x\fg}}\Big)^{\omega_{x\fg}}.    
\end{equation}
Furthermore, since $\omega_{xy}$ needs to be $0$ whenever $\sigma_x\neq \sigma_y$, and both values $0$ and $1$ are allowed whenever $\sigma_x=\sigma_y$, we get that ${\bf Z}^{\xi}_{\Lambda,\beta, \mathsf h, \mathsf{a}}$ is equal to
\begin{equation}
\begin{aligned}
2\sum_{\sigma\in \{\pm 1\}^{\overline \Lambda}}\mathbbm 1_{\xi\sim_{\rm ext}\sigma}\prod_{\substack{xy \in \overline{E}(\Lambda)\\ \sigma_x=\sigma_y}}&\sum_{\omega_{xy}\in \{0,1\}} \Big(\frac{p(\beta,\mathsf{a})_{xy}}{1-p(\beta,\mathsf{a})_{xy}}\Big)^{\omega_{xy}} \prod_{\substack{x \in \Lambda \\ \sigma_x=1}}\sum_{\omega_{x\fg}\in \{0,1\}} \Big(\frac{p(\beta,\mathsf{h},\mathsf{a})_{x\fg}}{1-p(\beta,\mathsf{h},\mathsf{a})_{x\fg}}\Big)^{\omega_{x\fg}}\\
&=2\sum_{\sigma\in \{\pm 1\}^{\overline \Lambda}}\mathbbm 1_{\xi\sim_{\rm ext}\sigma}\prod_{\substack{xy \in \overline{E}(\Lambda)\\ \sigma_x=\sigma_y}}e^{2\beta \mathsf{a}_x \mathsf{a}_y} \prod_{\substack{x \in \Lambda \\ \sigma_x=1}} e^{2\beta \mathsf{h}_x \mathsf{a}_{x}}.
\end{aligned}
\end{equation}
Since $\sigma_x\sigma_y=1$ when $\sigma_x=\sigma_y$ and $\sigma_x\sigma_y=-1$ when $\sigma_x\neq \sigma_y$, we can rewrite the latter to obtain \eqref{eq: partition function equality}.

It follows from \eqref{eq: absolute value density} that
\begin{equation}
Z^0_{\Lambda,\beta,\mathsf{h}}=2\int Z^{\rm Ising,0}_{\Lambda,\beta, \mathsf{h},\mathsf{a}}
\prod_{x\in \Lambda}\mathrm{d} \rho_{g,a}(\mathsf{a}_x).  
\end{equation}
The latter is equal to 2$Z^{\varphi^4}_{\Lambda,\beta,\mathsf{h}}$ thanks to \eqref{eq: relation phi4 Ising partition functions}, hence 
\begin{equation}\label{eq: phi 4 partition functionsequality} Z^0_{\Lambda,\beta,\mathsf{h}}=2Z^{\varphi^4}_{\Lambda,\beta,\mathsf{h}}.
\end{equation}
We can now combine equations \eqref{eq: relation phi4 Ising partition functions}, \eqref{eq: absolute value density} and \eqref{eq: phi 4 partition functionsequality} to deduce that $\mu^{0}_{\Lambda,\beta,\mathsf h}$ is the law of the absolute value field of $\nu_{\Lambda,\beta,\mathsf h}$.
\end{proof}

Proposition \ref{prop: phi4rc is annealed rndisingrc} together with the Edwards--Sokal coupling for the Ising model--- see \cite[Chapter 1]{Grimmett2006RCM}--- implies the following.

\paragraph{The Edwards--Sokal coupling for $\varphi^4$.}
\
\vspace{1mm}
\

\noindent\textit{From percolation to spin model.} 
Given a pair $(\omega,\mathsf{a})\sim \Psi^0_{\Lambda,\beta,\mathsf{h}}$, we can sample a field $\varphi\sim \langle \cdot \rangle_{\Lambda,\beta,\mathsf{h}}$ as follows.
\begin{enumerate}
\item[(i)] Consider a sequence of independent random variables $\sigma_{\mathcal{C}}\in \{\pm 1\}$, indexed by the connected components $\mathcal{C}$ of $\omega$, such that $\mathbb{P}[\sigma_{\mathcal{C}}=1]=1$ if $\mathcal{C}$ is the cluster of $\fg$, and $\mathbb{P}[\sigma_{\mathcal{C}}=1]=1/2$ otherwise.
\item[(ii)] For every cluster $\mathcal{C}$ and $x\in \mathcal{C}$, set $\varphi_x = \mathsf{a}_x \sigma_{\mathcal{C}}$.
\end{enumerate}

\noindent \textit{From spin model to percolation.} Conversely, given a field $\varphi\sim \langle \cdot \rangle_{\Lambda,\beta,\mathsf{h}}$, we can sample a pair $(\omega, \mathsf{a})\sim \Psi^0_{\Lambda,\beta,\mathsf{h}}$ as follows. (Recall that $\varphi_{\fg}=1$, and note that almost surely we have $\varphi_x\neq 0$ for all $x\in \Lambda$, so that $\textup{sgn}(\varphi_x):= \varphi_x/|\varphi_x|\in \{\pm 1\}$ exists almost surely.) 
\begin{enumerate}
\item[(i)] For each $x\in \Lambda$, set $\mathsf{a}_x=|\varphi_x|$.
\item[(ii)] For each edge $xy\in E(\Lambda[\mathsf{h}])$ such that $\textup{sgn}(\varphi_x) \neq \textup{sgn}(\varphi_y)$, set $\omega_{xy}=0$. 
\item[(iii)] For each edge $xy\in E(\Lambda)$ such that $\textup{sgn}(\varphi_x) = \textup{sgn}(\varphi_y)$, let $\omega_{xy}=1$ with probability equal to $p(\mathsf{a},\beta)_{xy}$, independently of  the other edges.
\item[(iv)] For each vertex $x\in \Lambda$ such that $\textup{sgn}(\varphi_x) = 1$, let $\omega_{x\mathfrak g}=1$
with probability equal to $p(\mathsf{a},\beta,\mathsf{h})_{x\fg}$, independently of the other edges.
\end{enumerate}

As a direct consequence of this coupling (see \cite[Theorem 1.16]{Grimmett2006RCM}) we obtain the following result for $\Psi^0_{\Lambda,\beta,\mathsf{h}}$.
We note that this result also holds for $\Psi^{(w,\mathfrak{p})}_{\Lambda,\beta}$, since $\Psi^{(w,\mathfrak{p})}_{\Lambda,\beta}=\Psi^0_{\Lambda,\beta,\mathfrak{p}_{\Lambda}}$. Below, if $x,y\in \Lambda^\fg$, we write $\{x\leftrightarrow y\}$ to denote the event that $x$ and $y$ lie in the same connected component of $\omega$.

\begin{corollary} \label{cor: ES correlations}
Let $\beta\geq 0$ and $\Lambda\subset \mathbb{Z}^d$ finite. For every $\mathsf h \in (\mathbb R^+)^\Lambda$, and every $x,y\in \Lambda$, 
\begin{align}
\label{eq: es cor: 1}
\langle \varphi_x \rangle_{\Lambda,\beta,\mathsf{h}}
&=
\Psi^0_{\Lambda,\beta,\mathsf{h}}[\mathsf{a}_x \mathbbm{1}_{x\longleftrightarrow \fg}], 
\\
\label{eq: es cor: 2}
\langle \textup{sgn}(\varphi_x) \rangle_{\Lambda,\beta,\mathsf{h}}
&=
\Psi^0_{\Lambda,\beta,\mathsf{h}}[x\longleftrightarrow \fg], 
\\ \label{eq: es cor: 3}
\langle \varphi_x \varphi_y \rangle_{\Lambda,\beta,\mathsf{h}}
&=
\Psi^0_{\Lambda,\beta,\mathsf{h}}[\mathsf{a}_x \mathsf{a}_y \mathbbm{1}_{x\longleftrightarrow y}], 
\\ \label{eq: es cor: 4}
\langle \textup{sgn}(\varphi_x) \textup{sgn}(\varphi_y) \rangle_{\Lambda,\beta,\mathsf{h}}
&=
\Psi^0_{\Lambda,\beta,\mathsf{h}}[x\longleftrightarrow y]. 
\end{align}   
\end{corollary}

\subsection{Correlation inequalities and monotonicity}

In this section, we derive some basic correlation inequalities and monotonicity properties for $\Psi_{\Lambda,\beta,\mathsf{h}}^{(\xi,\mathsf{b})}$. We start by proving that $\Psi_{\Lambda,\beta,\mathsf{h}}^{(\xi,\mathsf{b})}$ satisfies the absolute value FKG property, which is equivalent to saying that the marginals $\mu^{(\xi,\mathsf b)}_{\Lambda, \beta, \mathsf h}$ satisfy the FKG inequality. We will in fact state a stronger inequality that holds for conditional probabilities. The proof is a modification of the classical argument in \cite[Theorem 4.4.1]{GlimmJaffeQuantumBOOK}.

\begin{proposition}[Absolute value FKG, conditional version]\label{prop: absolute value FKG}
Let $\Lambda\subset \mathbb{Z}^d$ be finite, let $(\xi,\mathsf{b})$ be a boundary condition on $\Lambda$, and let $\mathsf h\in (\mathbb{R}^+)^{\Lambda}$. For every  $A \subset \Lambda$, $\eta \in (\mathbb R^+)^A$, and increasing, square integrable functions $F,G: (\mathbb{R}^+)^{\overline{\Lambda} \setminus A}\to\mathbb{R}$, 
\begin{equation} \label{eq: conditional abs fkg}
\mu_{\Lambda,\beta,\mathsf{h}}^{(\xi,\mathsf{b})}[F\cdot G \mid \mathsf a|_A=\eta]\geq \mu_{\Lambda,\beta,\mathsf{h}}^{(\xi,\mathsf{b})}[F \mid \mathsf a|_A=\eta]\mu_{\Lambda,\beta,\mathsf{h}}^{(\xi,\mathsf{b})}[G \mid \mathsf a|_A=\eta].
\end{equation}
\end{proposition}

\begin{proof}
For ease of notation, let us write for every $A \subset \Lambda$ and $\eta \in (\mathbb R^+)^A$
\begin{equation}
\mu^{A,\eta}[\:\cdot\:]:= \mu^{(\xi,\mathsf b)}_{\Lambda,\beta,\mathsf h} [\:\cdot \mid \mathsf a|_A=\eta].
\end{equation}
Let us first observe that, sampling two independent fields $\mathsf a, \mathsf a' \sim \mu^{A,\eta}$, for every square integrable $F,G : (\mathbb R
^+)^{\overline{\Lambda} \setminus A} \rightarrow \mathbb R$ we have 
\begin{equation} \label{eq: cond abs fkg 1}
\mu^{A,\eta}\otimes \mu^{A,\eta}\left[ (F(\mathsf a)-F(\mathsf a'))(G(\mathsf a)-G(\mathsf a')) \right]  = 2\left( \mu^{A,\eta}[FG] - \mu^{A,\eta}[F]\mu^{A,\eta}[G]\right)
\end{equation}
We will prove that $\mu^{A,\eta}$ satisfies \eqref{eq: conditional abs fkg}, or equivalently that the left-hand side of \eqref{eq: cond abs fkg 1} is positive, by induction on the number of sites $n$ we do \emph{not} condition on, i.e.\ $n=|\Lambda\setminus A|$. Notice when $n=1$, the desired FKG inequality follows (from \eqref{eq: cond abs fkg 1}) since $F(\mathsf a)-F(\mathsf a')$ and $G(\mathsf a)-G(\mathsf a')$ have the same sign if they are both increasing. Assume now that for some $n \geq 1$, we have
\begin{equation}
\mu^{A,\eta}[FG] \geq \mu^{A,\eta}[F]\mu^{A,\eta}[G]
\end{equation}
for every $A \subset \Lambda$ such that $|\Lambda \setminus A|\leq n$ and any $\eta \in (\mathbb R^+)^A$. Fix any such $A$ and $\eta$. We claim that for every $x \in A$, every $s,s' \geq 0$, and every increasing square integrable functions $F,G:(\mathbb R^+)^{\overline{\Lambda}\setminus A} \rightarrow \mathbb R$,
\begin{equation} \label{eq: cond abs fkg 2}
\mu^{A\setminus \{x\}, \eta|_{A\setminus \{x\}}} \otimes \mu^{A\setminus \{x\}, \eta|_{A\setminus \{x\}}} \left[ (F(\mathsf a)-F(\mathsf a'))(G(\mathsf a)-G(\mathsf a'))  \mid \mathsf a_x=s, \, \mathsf a'_x = s'\right] \geq 0.
\end{equation}
Upon taking expectation of the left-hand side of \eqref{eq: cond abs fkg 2} with respect to $\mu^{A\setminus \{x\},\eta|_{A\setminus \{x\}}}\otimes \mu^{A\setminus \{x\},\eta|_{A\setminus \{x\}}}$, using \eqref{eq: cond abs fkg 1}, and using the fact that $A,\eta,x$ were arbitrary, we obtain the desired induction hypothesis in the case $n+1$, thus completing the proof. 

It remains to prove the claim \eqref{eq: cond abs fkg 2}. For shorthand, for every $s \in \mathbb R^+$ let us write
\begin{equation}
    \mu^{A, \eta \Delta s}[\:\cdot\:]:= \mu^{A\setminus \{x\}, \eta|_{A\setminus \{x\}}}[\:\cdot \mid \mathsf a_x=s].
\end{equation}
Observe that $\mu^{A, \eta\Delta s}=\mu^{A,\tilde\eta}$, where $\tilde \eta \in (\mathbb R^+)^A$ is such that $\tilde\eta|_{A \setminus \{x\}} = \eta|_{A \setminus \{x\}}$ and $\tilde\eta_x~=~s$, and hence $\mu^{A,\eta\Delta s}$ satisfies the induction hypothesis. Let now $F,G:(\mathbb R^+)^{\overline{\Lambda} \setminus (A\setminus \{x\})}\rightarrow \mathbb R$ be increasing square integrable functions. Since the restrictions of $F,G$ to any subset are also increasing, we have that
\begin{align}\label{eq: cond abs fkg 3}
\begin{split}
\text{LHS of }\eqref{eq: cond abs fkg 2}
&=	
\mu^{A,\eta\Delta s}[FG] + \mu^{A,\eta\Delta s'}[FG]- \mu^{A,\eta\Delta s}[F]\mu^{A,\eta\Delta s'}[G]
- \mu^{A,\eta\Delta s'}[F]\mu^{A,\eta\Delta s}[G]
\\
&=
\mu^{A,\eta
\Delta s}[FG] - \mu^{A,\eta\Delta s}[F]\mu^{A,\eta\Delta s}[G]+\mu^{A,\eta \Delta s'}[FG] - \mu^{A,\eta\Delta s'}[F]\mu^{A,\eta\Delta s'}[G]
\\
&\quad\quad
+
\Big(\mu^{A,\eta\Delta s}[F]-\mu^{A,\eta\Delta s'}[F]\Big)\Big(\mu^{A,\eta\Delta s}[G]-\mu^{A,\eta\Delta s'}[G]\Big)
\\
&\geq
\Big(\mu^{A,\eta\Delta s}[F]-\mu^{A,\eta\Delta s'}[F]\Big)\Big(\mu^{A,\eta\Delta s}[G]-\mu^{A,\eta\Delta s'}[G]\Big),
\end{split}
\end{align}
where we used the induction hypothesis in the last inequality.

Assume, without loss of generality, that $s\geq s'$. We will now show that $\mu^{A,\eta\Delta s}[F] \geq \mu^{A,\eta\Delta s'}[F]$ for every increasing function $F$ as above (i.e.\ including if $F$ is swapped for $G$). This establishes that the last line of \eqref{eq: cond abs fkg 3} is positive, thereby completing the proof of the claim. On the event $\{ \mathsf a|_A = \tilde\eta\}$, where $\tilde \eta|_{A \setminus \{x\}}=\eta$ and $\tilde \eta_x= s'$, define the random variable 
\begin{equation}
Z(\mathsf a):= \frac{Z^{\rm Ising,\xi}_{\Lambda,\beta, \mathsf{h},\tilde {\mathsf{a}}}}{Z^{\rm Ising,\xi}_{\Lambda,\beta, \mathsf{h},\mathsf{a}}}, 
\end{equation}
where, $\tilde{\mathsf{a}}|_{\Lambda\setminus A}=\mathsf{a}|_{\Lambda\setminus A}$, $\tilde{\mathsf{a}}|_{A \setminus \{x\}}=\mathsf{a}|_{A\setminus\{x\}} (=\eta|_{A\setminus\{x\}})$, and $\tilde{\mathsf{a}}_x := s \geq s'=\mathsf{a}_x$. The logarithmic derivative of $Z$ with respect to $\mathsf a_y$ for $y \in \Lambda \setminus A$, is given by
\begin{equation}\label{eq:log derivative1}
\beta \sum_{z\sim y}\left(  \tilde {\mathsf a}_z \langle \sigma_y \sigma_z \rangle^{\rm Ising,\xi}_{\Lambda, \beta,\mathsf{h},\tilde {\mathsf{a}}}-\mathsf a_z \langle \sigma_y \sigma_z \rangle^{\rm Ising,\xi}_{\Lambda, \beta, \mathsf h,\mathsf{a}} \right)+ \beta \mathsf h_y \left( \langle \sigma_y \rangle^{\rm Ising,\xi}_{\Lambda, \beta, \mathsf{h},\tilde {\mathsf{a}}}- \langle \sigma_y \rangle^{\rm Ising,\xi}_{\Lambda, \beta, \mathsf h,\mathsf{a}} \right).
\end{equation}
By Griffiths' inequality, it is positive and hence $Z$ is an increasing function on $(\mathbb R^+)^{\overline{\Lambda} \setminus A}$. Therefore, by a direct computation and using that $F(\tilde{\mathsf{a}}) \geq F(\mathsf{a})$, followed by the induction hypothesis,
\begin{equation}
\mu^{A,\eta\Delta s}[F] \geq \frac{\mu^{A,\eta\Delta s'}[ZF]}{\mu^{A,\eta\Delta s'}[Z]} \geq \mu^{A,\eta \Delta s'}[F],
\end{equation}
as desired.
\end{proof}

We can now deduce that $\Psi_{\Lambda,\beta,\mathsf{h}}^{(\xi,\mathsf{b})}$ satisfies the full FKG property. We again state a stronger property that holds for conditional measures. 

\begin{proposition}[FKG]\label{prop:FKG random cluster}
Let $\Lambda\subset \mathbb{Z}^d$ finite, let $(\xi,\mathsf{b})$ be a boundary condition on $\Lambda$, and let $\mathsf h\in (\mathbb{R}^+)^{\Lambda}$. For every $A \subset \Lambda$, $\eta \in (\mathbb R^+)^\Lambda$, and every increasing, square-integrable functions $F,G: \{0,1\}^{\overline{E}(\Lambda)}\times (\mathbb R^+)^{\overline{\Lambda}}\to\mathbb{R}$,
we have
\begin{equation}
\Psi_{\Lambda,\beta,\mathsf{h}}^{(\xi,\mathsf{b})}[F\cdot G \mid \mathsf a|_A=\eta]
\geq
\Psi_{\Lambda,\beta,\mathsf{h}}^{(\xi,\mathsf{b})}[F \mid \mathsf a|_A=\eta] \Psi_{\Lambda,\beta,\mathsf{h}}^{(\xi,\mathsf{b})}[G \mid \mathsf a|_A=\eta].
\end{equation}
\end{proposition}

\begin{proof}
For every $\mathsf{a} \in (\mathbb{R}^+)^{\overline{\Lambda}}$, let $F^{\mathsf{a}}(\omega)= F(\omega,\mathsf{a})$ and $G^{\mathsf{a}}(\omega)=G(\omega,\mathsf{a})$. Clearly these are increasing as functions of $\omega$.
By the FKG property of the random cluster model (see \cite[Theorem 3.8]{Grimmett2006RCM}), for every $\mathsf{a} \in (\mathbb{R}^+)^{\overline{\Lambda}}$ we have
\begin{align}
\boldsymbol{\phi}^\xi_{\Lambda,\beta,\mathsf h,\mathsf{a}}[F^{\mathsf{a}} \cdot G^{\mathsf{a}}]
&\geq
\boldsymbol{\phi}^\xi_{\Lambda,\beta,\mathsf h,\mathsf{a}}[F^{\mathsf{a}}]\boldsymbol{\phi}^\xi_{\Lambda,\beta,\mathsf h,\mathsf{a}}[G^{\mathsf{a}}].
\end{align}
The mapping $\mathsf{a} \mapsto \boldsymbol{\phi}^\xi_{\Lambda,\beta,\mathsf{h},\mathsf{a}}[F^{\mathsf{a}}]$ is increasing due to the monotonicity of $\boldsymbol{\phi}^\xi_{\Lambda,\beta,\mathsf{h},\mathsf{a}}$ in $\mathsf{a}$ (see \cite[Theorem 3.21]{Grimmett2006RCM}), and the fact that for $\mathsf a_1 \leq \mathsf a_2$, we have $F^{\mathsf{a}_1}(\omega) \leq F^{\mathsf{a}_2}(\omega)$ for every $\omega$. Hence, by the conditional absolute value FKG property \eqref{eq: conditional abs fkg}, we have
\begin{equation}
\begin{aligned}
\Psi_{\Lambda,\beta,\mathsf{h}}^{(\xi,\mathsf{b})}[F\cdot G \mid \mathsf a|_A=\eta]
&\geq
\mu_{\Lambda,\beta,\mathsf{h}}^{(\xi,\mathsf{b})}[\boldsymbol{\phi}^\xi_{\Lambda,\beta,\mathsf{h},\mathsf{a}}[F^{\mathsf{a}}]\boldsymbol{\phi}^\xi_{\Lambda,\beta,\mathsf{h},\mathsf{a}}[G^{\mathsf{a}}] \mid \mathsf a|_A=\eta]
\\
&\geq
\Psi_{\Lambda,\beta,\mathsf{h}}^{(\xi,\mathsf{b})}[F \mid \mathsf a|_A=\eta]\Psi_{\Lambda,\beta,\mathsf{h}}^{(\xi,\mathsf{b})}[G \mid \mathsf a|_A=\eta].
\end{aligned}
\end{equation}
This concludes the proof.
\end{proof}

As a consequence of the FKG property, we obtain the following monotonicity in boundary conditions, inverse temperature, and external magnetic field.
Before stating it, let us introduce a useful notation. For boundary conditions $(\xi_1,\mathsf{b}_1), (\xi_2,\mathsf{b}_2)$ on a common set $\Lambda$, we write $(\xi_1,\mathsf{b}_1) \leq (\xi_2,\mathsf{b}_2)$ if every partition class of $\xi_1$ is contained in a partition class of $\xi_2$, and $\mathsf{b}_1\leq \mathsf{b}_2$. Also, recall that for two measures $\mu_1$ and $\mu_2$, we write $\mu_1 \preccurlyeq \mu_2$ if $\mu_1(F)\leq \mu_2(F)$ for every increasing function $F$.

\begin{proposition}[Monotonicity properties of the random cluster measure]\label{prop:monotonicity_FK}
Let $\Lambda\subset \mathbb{Z}^d$ finite and $A\subset \Lambda$. For every boundary conditions $(\xi_1,\mathsf{b}_1) \leq (\xi_2,\mathsf{b}_2)$ on $\Lambda$, every $\eta_1,\eta_2\in (\mathbb{R}^+)^{A}$ such that $\eta_1\leq \eta_2$, every external magnetic fields $\mathsf{h}_1,\mathsf{h}_2\in (\mathbb R^+)^\Lambda$ with $\mathsf h_1\leq \mathsf h_2$, and every $0\leq\beta_1\leq \beta_2$ we have
\begin{align}
\Psi_{\Lambda,\beta_1,\mathsf h_1}^{(\xi_1,\mathsf{b}_1)}[\:\cdot \mid \mathsf{a}|_A=\eta_1] \preccurlyeq \Psi_{\Lambda,\beta_2,\mathsf h_2}^{(\xi_2,\mathsf{b}_2)}[\:\cdot \mid \mathsf{a}|_A=\eta_2].
\end{align}
In particular, 
\begin{equation}
\Psi^0_{\Lambda,\beta,\mathsf h} \preccurlyeq \Psi_{\Lambda,\beta,\mathsf h}^{(\xi,\mathsf{b})}\preccurlyeq \Psi^{(w,\mathfrak{p})}_{\Lambda,\beta,\mathsf h}  
\end{equation}
for every boundary condition $(\xi,\mathsf{b})$ on $\Lambda$ such that $\mathsf{b}\leq \mathfrak{M}_{\Lambda}$, and every $\mathsf{h}\in(\mathbb R^+)^\Lambda$.
\end{proposition}

\begin{proof}
We first claim that 
\begin{equation}
\mu_{\Lambda,\beta_1,\mathsf h_1}^{(\xi_1,\mathsf{b}_1)}[\:\cdot \mid \mathsf{a}|_A=\eta_1] \preccurlyeq \mu_{\Lambda,\beta_2,\mathsf h_2}^{(\xi_2,\mathsf{b}_2)}[\:\cdot \mid \mathsf{a}|_A=\eta_2].    
\end{equation} 
Indeed, for any bounded increasing function $F:(\mathbb{R}^{+})^{\overline{\Lambda}\setminus A}\rightarrow \mathbb{R}$ we have
\begin{equation} 
\mu_{\Lambda,\beta_2,\mathsf h_2}^{(\xi_2,\mathsf{b}_2)}[F\mid \mathsf{a}|_A=\eta_2]
=
\frac{\mu_{\Lambda,\beta_1,\mathsf h_1}^{(\xi_1,\mathsf{b}_1)}[Z\cdot F\mid \mathsf{a}|_A=\eta_1]}{\mu_{\Lambda,\beta_1,\mathsf h_1}^{(\xi_1,\mathsf{b}_1)}[Z\mid \mathsf{a}|_A=\eta_1]}
\end{equation}
where
\begin{equation}
Z(\mathsf{a})
=
\frac{Z^{\textup{Ising},\xi_2}_{\Lambda,\beta_2,\mathsf{h}_2,\overline{\mathsf{a}}}}{Z^{\textup{Ising},\xi_1}_{\Lambda,\beta_1,\mathsf{h}_1,\mathsf{a}}},
\end{equation}
where $\overline{\mathsf{a}}|_A=\eta_2$, and is otherwise equal to $\mathsf{a}$.
By taking the logarithmic derivative and using Griffiths' inequality (as in \eqref{eq:log derivative1}), we see that $Z$ is an increasing function of $\mathsf{a}$. Using the absolute value FKG \eqref{eq: conditional abs fkg}, we deduce that 
\begin{equation} 
\mu_{\Lambda,\beta_1,\mathsf h_1}^{(\xi_1,\mathsf{b}_1)}[F\mid \mathsf{a}|_A=\eta_1]
\leq
\mu_{\Lambda,\beta_2,\mathsf h_2}^{(\xi_2,\mathsf{b}_2)}[F\mid \mathsf{a}|_A=\eta_2],
\end{equation}
which implies the claim.

We now proceed to the proof of the full monotonicity. Let $G:\{0,1\}^{\overline{E}(\Lambda)}\times(\mathbb{R}^{+})^{\overline{\Lambda}\setminus A}\to \mathbb{R}$ be an increasing function. For any $\mathsf{a} \in (\mathbb{R}^{+})^{\overline{\Lambda}}$, we have by the monotonicity properties of the random cluster model that
\begin{equation}
\boldsymbol{\phi}^{\xi_1}_{\Lambda,\beta_1,\mathsf h_1,\mathsf{a}}[G^{\mathsf{a}}]
\leq
\boldsymbol{\phi}^{\xi_2}_{\Lambda,\beta_2,\mathsf h_2,\mathsf{a}}[G^{\mathsf{a}}],
\end{equation}
where we recall that $G^{\mathsf{a}}(\omega)=G(\omega,\mathsf{a})$,
and that the mapping $\mathsf{a} \mapsto \boldsymbol{\phi}^{\xi_2}_{\Lambda,\beta_2,\mathsf h_2,\mathsf{a}}[G^{\mathsf{a}}]$ is increasing.
By the monotonicity of the absolute value field we deduce
\begin{equation}
\begin{aligned}
\Psi_{\Lambda,\beta_1,\mathsf h_1}^{(\xi_1,\mathsf{b}_1)}[G\mid \mathsf{a}|_A=\eta_1]
&\leq 
\mu_{\Lambda,\beta_1,\mathsf h_1}^{(\xi_1,\mathsf{b}_1)}[\boldsymbol{\phi}^{\xi_2}_{\Lambda,\beta_2,\mathsf h_2,\mathsf{a}}[G^{\mathsf{a}}]\mid \mathsf{a}|_A=\eta_1]\\
&\leq
\Psi_{\Lambda,\beta_2,\mathsf h_2}^{(\xi_2,\mathsf{b}_2)}[G\mid \mathsf{a}|_A=\eta_2].
\end{aligned}
\end{equation}
This completes the proof.
\end{proof}

As a corollary, we obtain the following monotonicity property for the free measure.

\begin{corollary}\label{cor: monotonicity free measure}
Let $\Lambda_1\subset \Lambda_2\subset \mathbb{Z}^d$ be finite and $\mathsf h_1\in (\mathbb{R}^+)^{\Lambda_1},\mathsf h_2\in (\mathbb{R}^+)^{\Lambda_2}$ such that $\mathsf h_1(x)\leq \mathsf h_2(x)$ for all $x\in \Lambda_1$. For every $\beta\geq 0$, we have
\begin{align}
\Psi_{\Lambda_1,\beta,\mathsf{h}_1}^{0} \preccurlyeq \Psi_{\Lambda_2,\beta,\mathsf{h}_2}^{0},
\end{align}    
in the sense that for every increasing, integrable function $F:\{0,1\}^{\overline{E}(\Lambda_2)}\times(\mathbb R^+)^{\overline{\Lambda_2}}\to \mathbb{R}$ that is $\Lambda_1$-measurable,
\begin{equation}
\Psi^0_{\Lambda_1,\beta,\mathsf h_1}[F|_{\Lambda_1}] \leq \Psi^0_{\Lambda_2,\beta, \mathsf h_2}[F].
\end{equation}
\end{corollary}

\begin{proof}
First note that by Proposition \ref{prop:monotonicity_FK},
\begin{equation}
\Psi_{\Lambda_2,\beta,\tilde {\mathsf h}_1}^{0}\preccurlyeq \Psi_{\Lambda_2,\beta,\mathsf{h}_2}^{0},   
\end{equation}
where $\tilde {\mathsf h}_1$ coincides with $\mathsf h_1$ on $\Lambda_1$ and is equal to $0$ on $\Lambda_2\setminus \Lambda_1$.
The desired inequality then follows by applying Proposition \ref{prop: domain markov} followed by Proposition \ref{prop:monotonicity_FK} (applied to the random boundary condition $(\xi,\mathsf{b})$ on $\Lambda_1$ induced from the domain Markov property) to yield that
\begin{equation}
\Psi_{\Lambda_1,\beta,\mathsf{h}_1}^{0}\preccurlyeq \Psi_{\Lambda_2,\beta,\tilde{\mathsf h}_1}^{0}.  
\end{equation}
This completes the proof.
\end{proof}

\subsection{Uniqueness of infinite volume measures}

In this section, we show uniqueness of full-space and half-space infinite volume measures. First, we consider the infinite volume limits arising from the measures $\Psi^0_{\Lambda,\beta}$ and $\Psi^{(w,\mathfrak{p})}_{\Lambda,\beta}$ as $\Lambda \uparrow \mathbb{Z}^d$, and show that they coincide for any $\beta\geq 0$. 

\begin{proposition}
Let $\beta\geq 0$. Then, $\Psi^0_{\Lambda,\beta}$ and $\Psi^{(w,\mathfrak{p})}_{\Lambda,\beta}$ converge weakly as $\Lambda \uparrow \mathbb{Z}^d$ to infinite volume measures denoted $\Psi^0_{\beta}$ and $\Psi^1_{\beta}$, respectively.
\end{proposition}

\begin{proof}
Note that $\Psi^{(w,\mathfrak{p})}_{\Lambda,\beta}=\Psi^0_{\Lambda,\beta,\mathfrak{p}_{\Lambda}}$. The desired result then follows from the weak convergence of the corresponding spin measures $\langle \cdot \rangle_{\Lambda,\beta}$ and $\langle \cdot \rangle_{\Lambda,\beta,\mathfrak{p}_{\Lambda}}$ and the Edwards--Sokal coupling.    
\end{proof}

Recall from \eqref{eq:free=plus+minus/2} that $\langle \cdot\rangle_\beta^0=\tfrac{1}{2}(\langle \cdot \rangle^+_{\beta}+\langle \cdot \rangle^-_{\beta})$. Using this result, we will now deduce that $\Psi^0_{\beta}$ and $\Psi^1_{\beta}$ coincide.

\begin{proposition}\label{prop: unique gibbs measure}
Let $\beta\geq 0$. One has that $\Psi^0_{\beta}=\Psi^1_{\beta}$.    
\end{proposition}

\begin{proof}
By Proposition~\ref{prop:monotonicity_FK}, we have $\Psi^0_{\beta}\preccurlyeq \Psi^1_{\beta}$. Hence, by Strassen's theorem \cite{Strassen1965}, there exists a coupling $(\mathbb{Q}, (\mathsf{a}^0,\omega^0), (\mathsf{a}^1,\omega^1))$ such that $(\mathsf{a}^0,\omega^0)\sim \Psi^0_{\beta}$, $(\mathsf{a}^1,\omega^1)\sim \Psi^1_{\beta}$, and
$\mathsf{a}^0\leq \mathsf{a}^1$, $\omega^0\leq \omega^1$ almost surely under $\mathbb{Q}$. Since $\langle \cdot \rangle_{\beta}$ is a convex combination of $\langle \cdot \rangle^+_{\beta}$ and $\langle \cdot \rangle^-_{\beta}$, and $\langle \cdot \rangle^-_{\beta}$ coincides with the pushforward of $\langle \cdot \rangle^+_{\beta}$ under the mapping $(\varphi_x) \mapsto -(\varphi_x)$, we can deduce that \begin{equation}\label{eq: abs value equality}
\langle |\varphi_x| \rangle^0_{\beta}=\langle |\varphi_x|\rangle^+_{\beta},   
\end{equation} 
and 
\begin{equation}\label{eq: two point equality}
\langle \varphi_x \varphi_y \rangle^0_{\beta}=\langle \varphi_x \varphi_y \rangle^+_{\beta}.    
\end{equation} 
It follows from \eqref{eq: abs value equality} that $\mathbb{Q}[\mathsf{a}^0_x]=\mathbb{Q}[\mathsf{a}^1_x]$, which, combined with the almost sure monotonicity, implies that
\begin{equation}\label{eq: as abs value equality}
\mathsf{a}^0=\mathsf{a}^1 \quad \mathbb{Q}\text{-almost surely}. 
\end{equation}
We now use \eqref{eq: two point equality} and \eqref{eq: as abs value equality} to deduce that $\omega^0_{xy}=\omega^1_{xy}$ $\mathbb{Q}$-almost surely, from which it follows that $\Psi^0_{\beta}=\Psi^1_{\beta}$, as desired.
We first observe that by combining Corollary \ref{cor: ES correlations}, \eqref{eq: two point equality}, and \eqref{eq: as abs value equality},
\begin{equation}
\mathbb{Q}[\mathsf{a}^0_x \mathsf{a}^0_y \mathbbm{1}\{x\overset{\omega^0}{\longleftrightarrow} y\} ]=\mathbb{Q}[\mathsf{a}^1_x \mathsf{a}^1_y\mathbbm{1}\{x\overset{\omega^1}{\longleftrightarrow} y\}]=\mathbb{Q}[\mathsf{a}^0_x \mathsf{a}^0_y\mathbbm{1}\{x\overset{\omega^1}{\longleftrightarrow} y\}].     
\end{equation}
Since $\mathsf{a}^0_x \mathsf{a}^0_y>0$ almost surely, using again the almost sure monotonicity we obtain that
\begin{equation}\label{eq: as two point equality}
\mathbbm{1}\{x\overset{\omega^0}{\longleftrightarrow} y\} =\mathbbm{1}\{x\overset{\omega^1}{\longleftrightarrow} y\} 
\quad \text{almost surely}.
\end{equation}
By the Edwards--Sokal coupling we have 
\begin{equation}\label{eq:inf vol eq proof 1}
\mathbb{Q}[\omega^0_{xy}]=\big\langle p(\beta,|\varphi|)_{xy} \mathbbm{1}\{\textup{sgn}(\varphi_x)=\textup{sgn}(\varphi_y)\} \big\rangle^0_{\beta}. \end{equation}
Note that 
$\mathbbm{1}\{\textup{sgn}(\varphi_x)=\textup{sgn}(\varphi_y)\}=\frac{1}{2}(\textup{sgn}(\varphi_x)\textup{sgn}(\varphi_y)+1)$,
and by the Edwards--Sokal coupling between the Ising and random cluster models, we have
$
\langle \textup{sgn}(\varphi_x)\textup{sgn}(\varphi_y) \mid |\varphi| \rangle^0_{\beta}= \mathbb{Q}[x\overset{\omega^0}{\longleftrightarrow} y \mid \mathsf{a}^0]
$.
Combining the two previous equations with \eqref{eq:inf vol eq proof 1} yields
\begin{equation}\label{eq: omega connectivity equality}
\mathbb{Q}[\omega^0_{xy}]=\frac{1}{2}\mathbb{Q}[p(\beta,\mathsf{a}^0)_{xy}(\mathbbm{1}\{x\overset{\omega^0}{\longleftrightarrow} y\}+1)],    
\end{equation}
and similarly,
\begin{equation}
\mathbb{Q}[\omega^1_{xy}]=\frac{1}{2}\mathbb{Q}[p(\beta,\mathsf{a}^1)_{xy}(\mathbbm{1}\{x\overset{\omega^1}{\longleftrightarrow} y\}+1)].    
\end{equation}
However, \eqref{eq: as abs value equality} and \eqref{eq: as two point equality} imply that
\begin{equation}
\frac{1}{2}\mathbb{Q}[p(\beta,\mathsf{a}^0)_{xy}(\mathbbm{1}\{x\overset{\omega^0}{\longleftrightarrow} y\}+1)]=\frac{1}{2}\mathbb{Q}[p(\beta,\mathsf{a}^1)_{xy}(\mathbbm{1}\{x\overset{\omega^1}{\longleftrightarrow} y\}+1)].   
\end{equation}
Thus, $\mathbb{Q}[\omega^0_{xy}]=\mathbb{Q}[\omega^1_{xy}]$, and by the almost sure monotonicity this implies that $\omega^0_{xy}=\omega^1_{xy}$ almost surely. This completes the proof.
\end{proof}

\begin{remark}\label{rem: equality of magnetisation} Note that Corollary~\textup{\ref{cor: ES correlations}}  and Proposition~\textup{\ref{prop: unique gibbs measure}} imply that for any $\beta\geq 0$, one has
\begin{equation}
m^*(\beta)=\Psi^1_\beta[\mathsf{a}_0 \mathbbm{1}_{\{0\longleftrightarrow\infty\}}]=\Psi^0_\beta[\mathsf{a}_0 \mathbbm{1}_{\{0\longleftrightarrow\infty\}}].
\end{equation}
\end{remark}

We conclude this section by considering infinite volume limits arising in the half-space setting. Following the setup of Proposition \ref{prop: unique half-space measure}, we let $h>0$ and define $\mathsf{h}\in(\mathbb R^+)^{\mathbb H}$ as $\mathsf{h}_x:=h\mathbbm{1}_{x\in \partial \mathbb H}$.
Define two measures $\Psi^{0,h}_{\mathbb H,\beta}$ and $\Psi^{1,h}_{\mathbb H,\beta}$ on $\{0,1\}^{E(\mathbb H[\mathsf{h}])}\times (\mathbb R^+)^{\mathbb H}$ as follows,
\begin{equation}
    \Psi^{0,h}_{\mathbb H,\beta}:=\lim_{L\rightarrow \infty}\Psi^{0}_{\Lambda_L^+,\beta,\mathsf{h}}, \qquad \Psi^{1,h}_{\mathbb H,\beta}:=\lim_{L\rightarrow \infty}\Psi^{0}_{\Lambda_L^+,\beta,\mathsf{h}+\mathfrak{p}_{\Lambda_L^+\setminus \partial \mathbb H}}
\end{equation}
Note that these limits exist thanks to Lemma \ref{lem: existence of phi4 halspace measures} and the Edwards--Sokal coupling.

\begin{proposition}\label{prop: equality half space random cluster measures} Let $d\geq 2$, $\beta\geq 0$ and $h>0$. Then,
\begin{equation}
    \Psi^{0,h}_{\mathbb H,\beta}=\Psi^{1,h}_{\mathbb H,\beta}.
\end{equation}
\end{proposition}

\begin{proof} By Proposition \ref{prop: unique half-space measure}, we have that $\langle \cdot\rangle^{+,h}_{\mathbb H,\beta}=\langle \cdot\rangle^{0,h}_{\mathbb H,\beta}$. The desired result follows readily from the Edwards--Sokal coupling.
\end{proof}

\section{Surface order bound on disconnections}\label{sec:surface_tension}

The goal of this section is to prove Theorem~\ref{thm:free_surface_tension} below, which states that disconnection probabilities for the $\varphi^4$ random cluster model with \emph{free} (and therefore any!) boundary condition decay exponentially in the surface order.

\begin{theorem}\label{thm:free_surface_tension}
Let $d\geq2$ and $\beta>\beta_c$. There exists $c_1>0$ such that for every $L\geq 1$ we have
\begin{equation}\label{eq:surface_free_box}
        \Psi^0_{\Lambda_{L},\beta}[\Lambda_\ell\longleftrightarrow \partial \Lambda_{c_1 L}]\geq 1-e^{-c_1 \ell^{d-1}}, ~~~ \forall\, 1\leq \ell\leq c_1 L.
    \end{equation}
\end{theorem}

Theorem~\ref{thm:free_surface_tension} is proved in two steps. First, we prove that the surface tension of the $\varphi^4$ model is positive for every $\beta>\beta_c$ by adapting the argument of Lebowitz and Pfister \cite{lebowitz1981surface}. We then perform a delicate comparison argument to deduce a corresponding statement for the \emph{free} $\varphi^4$ random cluster model.

\subsection{Definition of the surface tension}

Here we define a notion of surface tension for the $\varphi^4$ model and relate it to disconnection probabilities for the corresponding random cluster representation. As mentioned in the introduction, we consider a magnetic field of intensity $1$ on a thick boundary. The thick boundary will be important in Section~\ref{sec:Lebowitz-Pfister} in order to prove positivity of the surface tension with the help of Proposition~\ref{prop: thick conv}. The boundedness of the magnetic field will be important in Section~\ref{prop:from thick plus to free}, where we will rely on regularity up to the boundary (recall Proposition~\ref{prop:regularity}) in order to recover the \emph{free} $\varphi^4$ random cluster measure.

First, let us introduce some necessary definitions and notation. For $L,M\geq 1$, define the rectangle $\mathcal R(L,M) = \{ -L,\dots,L \}^{d-1} \times \{-M,\ldots,M\}$, and the infinite strip $\mathcal R(L) := \{-L,\dots,L\}^{d-1}\times \mathbb Z$. Let also $\partial^{\rm thick} \mathcal R(L,M) := \mathcal R(L,M) \setminus \{ -L+\log L+1,\dots,L-\log L-1 \}^{d-1} \times \{-M,\ldots,M\}$ and $\partial^{\rm thick} \mathcal R(L) := \mathcal R(L) \setminus \{ -L+\log L+1,\dots,L-\log L-1 \}^{d-1} \times \mathbb Z$ denote their thickened boundaries. We also define the top and bottom of the thick boundary $\partial^{\rm thick}_+\mathcal R(L,M):=\partial^{\rm thick}\mathcal{R}(L,M)\cap \{-L,\ldots,L\}^{d-1}\times \{1,\ldots, M\}$ and $\partial^{\rm thick}_-\mathcal R(L,M):=\partial^{\rm thick}\mathcal{R}(L,M)\cap \{-L,\ldots,L\}^{d-1}\times \{-M,\ldots, 0\}$. Let 
\begin{equation*}
\mathsf{h}^{+,-}_{L,M}(x)=
\begin{cases}
+1, & \text{for } x\in \partial^{\rm thick}_+\mathcal{R}(L,M),\\
-1, & \text{for } x\in\partial^{\rm thick}_-\mathcal{R}(L,M),\\
0, & \text{otherwise.}    
\end{cases}
\end{equation*}  
We also let $\mathsf{h}^{+,+}_{L.M}$ be the external magnetic field taking value $+1$ on $\partial^{\rm thick} \mathcal R(L,M)$ and $0$ otherwise. See Figure \ref{fig:R(L,M)} for an illustration. For ease of notation, we write the partition functions of the associated $\varphi^4$ models as follows:
\begin{equation}
Z^{+,+}_{\mathcal{R}(L,M),\beta}=Z^{\varphi^4}_{\mathcal{R}(L,M),\beta,\mathsf{h}^{+,+}_{L,M}} \quad \text{and} \quad Z^{+,-}_{\mathcal{R}(L,M),\beta}=Z^{\varphi^4}_{\mathcal{R}(L,M),\beta,\mathsf{h}^{+,-}_{L,M}}.
\end{equation}
We also let $\langle \cdot\rangle^{+,+}_{\mathcal{R}(L,M),\beta}=\langle \cdot\rangle_{\mathcal{R}(L,M),\beta,\mathsf{h}^{+,+}_{L,M}}$ and $\langle \cdot\rangle^{+,-}_{\mathcal{R}(L,M),\beta}=\langle \cdot\rangle_{\mathcal{R}(L,M),\beta,\mathsf{h}^{+,-}_{L,M}}$.

\begin{figure}[htb]
    \centering
    \includegraphics[width=0.4\linewidth]{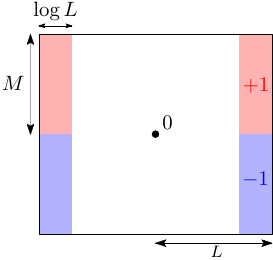}
    \caption{An illustration of the rectangle $\mathcal{R}(L,M)$. The support of the magnetic field $\mathsf{h}^{+,-}_{L,M}$ is the coloured region.}
    \label{fig:R(L,M)}
\end{figure}

\begin{definition}
The \emph{surface tension} of the $\varphi^4$ model on $\mathcal R(L,M)$ is defined as 
\begin{equation}
    \tau_{\beta}^{L,M}:=\frac{1}{L^{d-1}}\log \frac{Z^{+,+}_{\mathcal{R}(L,M),\beta}}{Z^{+,-}_{\mathcal{R}(L,M),\beta}}.
\end{equation}
We also define 
\begin{equation}
\tau_{\beta}^L:=\liminf_{M\to\infty}\tau_{\beta}^{L,M} \quad \text{and} \quad \tau_{\beta}:=\liminf_{L\to\infty}\tau_{\beta}^{L}.
\end{equation}
\end{definition}
The surface tension can be expressed naturally in terms of a disconnection event under the measure $\Psi^{0}_{\mathcal R(L,M),\beta,\mathsf{h}^{+,+}_{L,M}}$. In order to facilitate the comparison external magnetic fields $\mathsf h^{+,+}_{L,M}$ and $\mathsf h^{+,-}_{L,M}$, it is convenient to view $\Psi^{0}_{\mathcal R(L,M),\beta,\mathsf{h}^{+,+}_{L,M}}$ as a probability measure on a graph with \emph{two} ghost vertices $\{\mathfrak g^+,\mathfrak g^-\}$. We write $\mathcal{R}^{+,-}(L,M)$ for the graph with vertex set $\mathcal R(L,M)\cup \{\fg^{+},\fg^{-}\}$ and edge set $E(\mathcal{R}^{+,-}(L,M))$ given by the union of $\{xy\in E(\mathbb{Z}^d): x,y\in \mathcal R(L,M)\}$ and
\begin{equation}
\{x\fg^{-} : x\in \partial^{\rm thick}_- \mathcal R(L,M)\}\sqcup \{x\fg^{+}: x\in \partial^{\rm thick}_+ \mathcal R(L,M)\}.    
\end{equation}
We can now (re)define $\Psi^0_{\mathcal R(L,M), \beta, \mathsf h^{+,+}_{L,M}}$ on $\{0,1\}^{E(\mathcal{R}^{+,-}(L,M))}\times (\mathbb R^+)^{\mathcal{R}(L,M)}$ as in Definition~\ref{def: random cluster phi4} with straightforward modifications to the edge weights and where $k^0(\omega)$ is replaced by $\tilde k(\omega)$, where the clusters are counted in the graph obtained by identifying $\mathfrak g^-$ and $\mathfrak g^+$. Under the pushforward with respect to the bijection identifying the two ghosts, we have that $\Psi^0_{\mathcal R(L,M), \beta, \mathsf h^{+,+}_{L,M}}$ coincides with the definition of Section \ref{sec:random_cluster}. Below, we will use that the partition functions are the same.

\begin{lemma}\label{lem: ratio probablistic expression} 
For every $\beta>0$, and every $L,M\geq 1$ we have
\begin{equation}\label{eq: ratio probabilistic expression}
\frac{Z^{+,-}_{\mathcal{R}(L,M),\beta}}{Z^{+,+}_{\mathcal{R}(L,M),\beta}}= \Psi^{0}_{\mathcal R(L,M),\beta,\mathsf{h}^{+,+}_{L,M}}[\fg^{-} \centernot\longleftrightarrow \fg^{+}].
\end{equation}
In particular, $\tau_{\beta}\geq 0$ for every $\beta>0$.
\end{lemma}

\begin{proof}
Recall \eqref{eq: phi 4 partition functionsequality}, and note that 
in our case it implies that 
\begin{equation}\label{eq: pp partition equality}
Z^{+,+}_{\mathcal{R}(L,M),\beta}=\frac{1}{2}Z^0_{\mathcal R(L,M),\beta,\mathsf{h}^{+,+}_{L,M}}.    
\end{equation}
We claim that 
\begin{equation}\label{eq: pm parition equality}
Z^{+,-}_{\mathcal{R}(L,M),\beta}=\frac{1}{2}Z^0_{\mathcal R(L,M),\beta,\mathsf{h}^{+,+}_{L,M}}[\fg^{-} \centernot\longleftrightarrow \fg^{+}],   
\end{equation}
where $Z^0_{\mathcal R(L,M),\beta,\mathsf{h}^{+,+}_{L,M}}[\fg^{-} \centernot\longleftrightarrow \fg^{+}]$ denotes the partition function obtained by restricting the sum to configurations $\omega$ in which $\fg^{-}$ and $\fg^{+}$ lie in different clusters. Note that \eqref{eq: ratio probabilistic expression} follows readily from \eqref{eq: pp partition equality} and \eqref{eq: pm parition equality}.

Recall from \eqref{eq: relation phi4 Ising partition functions} that
\begin{equation}
Z^{+,-}_{\mathcal{R}(L,M),\beta,\mathsf{h}}=\int Z^{\rm Ising,0}_{\mathcal{R}(L,M),\beta, \mathsf{h}^{+,-}_{L,M},\mathsf{a}}
\prod_{x\in \mathcal{R}(L,M)}\mathrm{d} \rho_{g,a}(\mathsf{a}_x),  
\end{equation}
where $Z^{{\rm Ising},0}_{\Lambda,\beta,\mathsf h, \mathsf a}$ is defined (we only considered $\mathsf{h}\in (\mathbb R^+)^\Lambda$ but the definition extends to general $\mathsf{h}\in \mathbb R^\Lambda$) in \eqref{eq:Ising_def}. Furthermore, let $\mathbf{Z}^0_{\mathcal R(L,M),\beta,\mathsf h^{+,+}_{L,M}, \mathsf a}$ denote the partition function of the FK-Ising random cluster measure which is, as for the $\varphi^4$ random cluster measure above, viewed on the graph $\mathcal R^{+,-}(L,M)$, i.e.\ with two ghost vertices. By adapting the proof of \eqref{eq: partition function equality}, we will show: 
\begin{equation}\label{eq: partition function equality 2}
Z^{\textup{Ising},0}_{\mathcal{R}(L,M),\beta,\mathsf{h}^{+,-}_{L,M},\mathsf{a}}=\frac{1}{2}{\bf Z}^{0}_{\mathcal{R}(L,M),\beta, \mathsf{h}^{+,+}_{L,M},\mathsf{a}}[\fg^{-} \centernot\longleftrightarrow \fg^{+}]\prod_{xy\in E(\mathcal{R}(L,M))} e^{-\beta \mathsf{a}_x\mathsf{a}_y}\prod_{x\in \mathcal{R}(L,M)}e^{-\beta (\mathsf{h}^{+,+}_{L,M})_x\mathsf{a}_x}.    
\end{equation}
Upon integrating, this yields \eqref{eq: pm parition equality}.

Let us now show \eqref{eq: partition function equality 2}. Recall that we say that a spin configuration $\sigma \in \{\pm 1\}^{\mathcal R^{+,-}(L,M)}$ is compatible with $\omega\in \{0,1\}^{E(\mathcal{R}^{+,-}(L,M))}$ if $\sigma$ is constant in the clusters of $\omega$. We write $\omega\sim_{\fg}\sigma$ if in addition, $\sigma_{\mathfrak g^+}=1$ and $\sigma_{\mathfrak g^-}=-1$. Note that, given a configuration $\omega$ such that $\{\fg^{-} \centernot\longleftrightarrow \fg^{+}\}$, the number of configurations $\sigma$ such that $\omega\sim_{\fg}\sigma$ is $2^{\tilde k(\omega)-1}$, where we emphasise that $\tilde{k}(\omega)$ is the number of components in $\omega$ after identifying $\fg^+$ and $\fg^-$. Conversely, if $\omega\sim_{\mathfrak g} \sigma$, then $\omega \in \{\mathfrak g^- \centernot \longleftrightarrow \mathfrak g^+\}$.  Therefore, we can write
\begin{multline}
{\bf Z}^{0}_{\mathcal{R}(L,M),\beta, \mathsf{h}^{+,+}_{L,M},\mathsf{a}}[\fg^{-} \centernot\longleftrightarrow \fg^{+}]
\\=2 \sum_{\substack{\omega\in \{0,1\}^{E(\mathcal{R}^{+,-}(L,M))}\\ \fg^+\centernot\longleftrightarrow \fg^-}}2^{\tilde{k}(\omega)-1}\prod_{xy \in E(\mathcal{R}(L,M))} \Big(\frac{p(\beta,\mathsf{a})_{xy}}{1-p(\beta,\mathsf{a})_{xy}}\Big)^{\omega_{xy}} \prod_{x \in \mathcal{R}(L,M)} \Big(\frac{p(\beta,\mathsf{h}^{+,+}_{L,M},\mathsf{a})_{x\fg}}{1-p(\beta,\mathsf{h}^{+,+}_{L,M},\mathsf{a})_{x\fg}}\Big)^{\omega_{x\fg}}
\\= 2 \sum_{\substack{\sigma\in \{\pm 1\}^{\mathcal{R}^{+,-}(L,M)}\\\sigma_{\fg^+}=1\\\sigma_{\fg^-}=1}} \sum_{\omega\sim_\fg \sigma}\prod_{xy \in E(\mathcal{R}(L,M))} \Big(\frac{p(\beta,\mathsf{a})_{xy}}{1-p(\beta,\mathsf{a})_{xy}}\Big)^{\omega_{xy}} \prod_{x \in \mathcal{R}(L,M)} \Big(\frac{p(\beta,\mathsf{h}^{+,+}_{L,M},\mathsf{a})_{x\fg}}{1-p(\beta,\mathsf{h}^{+,+}_{L,M},\mathsf{a})_{x\fg}}\Big)^{\omega_{x\fg}}.    
\end{multline}
From there, the end of the computation is exactly the same as in the proof of \eqref{eq: partition function equality}.
\end{proof}

\subsection{Strict positivity of the surface tension}\label{sec:Lebowitz-Pfister}

Our aim now is to show that the strict inequality $\tau_{\beta}>0$ holds in the entire supercritical regime--- see Proposition~\ref{prop: surface tension positive}. This follows from an adaptation of the argument of \cite{lebowitz1981surface}, which crucially relies on the Ginibre inequality (Proposition~\ref{prop:Ginibre}). 

We will need the following lemma.

\begin{lemma}\label{lem:++}
Let $\beta> 0$ and $L\geq 1$. Then, both $\langle \cdot \rangle^{+,+}_{\cR(L,M),\beta}$ and $\langle \cdot \rangle^{+,-}_{\cR(L,M),\beta}$ converge weakly as $M\to\infty$. We write 
\begin{equation}
\langle \cdot \rangle^{+,+}_{\cR(L),\beta}:=\lim_{M\to\infty}\langle \cdot \rangle^{+,+}_{\cR(L,M),\beta}  \quad \text{and} \quad \langle \cdot \rangle^{+,-}_{\cR(L),\beta}:=\lim_{M\to\infty}\langle \cdot \rangle^{+,-}_{\cR(L,M),\beta}. 
\end{equation}
\end{lemma}

\begin{proof} Observe that by Proposition \ref{prop:stoc_monotonicity}, one has that $\langle \cdot\rangle_{\mathbb Z^d,\beta,1}\succcurlyeq \langle \cdot \rangle^{+,+}_{\cR(L,M),\beta}$ for every $L,M\geq 1$. Moreover, the measure $\langle \cdot\rangle_{\mathbb Z^d,\beta,1}$ can be obtained as the limit of the sequence $(\langle \cdot\rangle_{\Lambda_n,\beta,1})_{n\geq 1}$, which is uniformly regular in the sense of Proposition~\ref{prop:regularity}, and hence $\langle \cdot\rangle_{\mathbb Z^d,\beta,1}$ inherits this property.  In particular, this implies that the measures $ \langle \cdot \rangle^{+,+}_{\cR(L,M),\beta}$ are also uniformly regular, and therefore tight. Using again Proposition \ref{prop:stoc_monotonicity}, the sequence $( \langle \cdot \rangle^{+,+}_{\cR(L,M),\beta})_{M\geq 1}$ is stochastically increasing and hence converges as $M\rightarrow \infty$.
For $\langle \cdot \rangle^{+,-}_{\cR(L,M),\beta}$, note that by flipping the sign of all spins $\varphi_x$ for $x\in \{-L,\ldots,L\}^{d-1}\times \{-M,\ldots,0\}$, and denoting the resulting configuration $\tilde \varphi$,  we obtain
\begin{equation}
\langle F(\varphi) \rangle^{+,-}_{\cR(L,M),\beta}=  \frac{\langle F(\tilde \varphi) \prod_{x\in \{-L,\ldots,L\}^{d-1}}\exp(-2\beta \varphi_{(x,1)} \varphi_{(x,0)} )\rangle^{+,+}_{\cR(L,M),\beta}}{\langle \prod_{x\in \{-L,\ldots,L\}^{d-1}}\exp(-2\beta \varphi_{(x,1)} \varphi_{(x,0)} )\rangle^{+,+}_{\cR(L,M),\beta}}  
\end{equation}
for every bounded (local) measurable function $F$. The convergence of $\langle \cdot \rangle^{+,-}_{\cR(L,M),\beta}$ then follows from the convergence and regularity of $\langle \cdot \rangle^{+,+}_{\cR(L,M),\beta}$.
\end{proof}

\begin{proposition}\label{prop: surface tension positive}
For every $\beta > \beta_c$, 
\begin{equation}
\tau_\beta \geq \frac 1d \int_{\beta_c}^\beta \left( \langle \varphi_0 \rangle_u^+ \right)^2 \mathrm{d}u.
\end{equation}
In particular, $\tau_\beta > 0$.
\end{proposition}
\begin{proof}
Our main tool will be the Ginibre inequality stated in Proposition~\ref{prop:Ginibre}.
We first differentiate $\tau^{L,M}_\beta$ with respect to $\beta$ to get that
\begin{align}
\begin{split}
\frac{\rm d}{\rm d\beta} \tau^{L,M}_\beta =  \frac{1}{L^{d-1}}& \sum_{xy\in E(\cR(L,M))} \left(\langle \varphi_x\varphi_y \rangle^{+,+}_{\cR(L,M),\beta} - \langle \varphi_x\varphi_y \rangle^{+,-}_{\cR(L,M),\beta}\right)
\\
&\quad +\frac{1}{L^{d-1}}\sum_{x\in \cR(L,M)}\left( \langle \mathsf h^{+,+}_{L,M}(x)\varphi_x \rangle_{\mathcal R(L,M),\beta} - \langle \mathsf h^{+,-}_{L,M}(x)\varphi_x \rangle_{\mathcal R(L,M),\beta} \right).
\end{split}
\end{align}
By \eqref{eq:Ginibre}, each term in the sums is non-negative. Let us fix some $N\geq 1$. By disregarding the horizontal edges and the edges in the complement of $E(\mathcal{R}(L,N))$, and applying Proposition \ref{prop:Ginibre}, we obtain, for every $M\geq N$,
\begin{align*}
\frac{\rm d}{\rm d\beta} \tau^{L,M}_\beta \geq  \frac{1}{L^{d-1}} \sum_{\substack{ x\in \{-L,\ldots,L\}^{d-1}\\ j\in\{-N,\ldots,N-1\}}} & \Big(\langle \varphi_{(x,j)} \rangle^{+,+}_{\cR(L,M),\beta} \langle \varphi_{(x,j+1)} \rangle^{+,-}_{\cR(L,M),\beta}\\
&\qquad- \langle \varphi_{(x,j+1)}\rangle^{+,+}_{\cR(L,M),\beta} \langle \varphi_{(x,j)} \rangle^{+,-}_{\cR(L,M),\beta}\Big).
\end{align*}
Taking the limit as $M\to \infty$, and using translation invariance in the $\mathbf{e}_d$ direction we obtain
\begin{align*}
\liminf_{M\to\infty}\frac{\rm d}{\rm d\beta} \tau^{L,M}_\beta &\geq  \frac{1}{L^{d-1}} \sum_{\substack{ x\in \{-L,\ldots,L\}^{d-1}\\ j\in\{-N,\ldots,N-1\}}}  \langle \varphi_{(x,0)} \rangle^{+,+}_{\cR(L),\beta} \left( \langle \varphi_{(x,j+1)} \rangle^{+,-}_{\cR(L),\beta} - \langle \varphi_{(x,j)} \rangle^{+,-}_{\cR(L),\beta} \right) \\
&= \frac{1}{L^{d-1}} \sum_{x\in \{-L,\ldots,L\}^{d-1}}  \langle \varphi_{(x,0)} \rangle^{+,+}_{\cR(L),\beta}\left(\langle \varphi_{(x,N)} \rangle^{+,-}_{\cR(L),\beta} - \langle \varphi_{(x,-N)} \rangle^{+,-}_{\cR(L),\beta}\right).
\end{align*}
Now, we can use Lemma~\ref{lem: pm convergence} below to obtain that
\begin{equation}
\liminf_{M\to\infty}\frac{\rm d}{\rm d\beta} \tau^{L,M}_\beta \geq \frac{1}{L^{d-1}} \sum_{x\in \{-L,\ldots,L\}^{d-1}}  2\left(\langle \varphi_{(x,0)} \rangle^{+,+}_{\cR(L),\beta}\right)^2.
\end{equation}
Applying Fatou's lemma we can now deduce that
\begin{equation}
\tau_{\beta}^L\geq \tau_{\beta_c}^L+\frac{1}{L^{d-1}} \sum_{x\in \{-L,\ldots,L\}^{d-1}} 2\int_{\beta_c}^{\beta} \left(\langle \varphi_{(x,0)} \rangle^{+,+}_{\cR(L),u}\right)^2 \mathrm{d}u.
\end{equation}
Without loss of generality, we only handle $\langle \varphi_{(x,0)} \rangle^{+,+}_{\cR(L),u}$ for $x\in \{-L,\ldots,L\}^{d-1}$ satisfying $x_1=|x|\in \{0,\ldots,L\}$. Using Proposition \ref{prop:monotonicity}, for $u\in (\beta_c,\beta)$ and $x\in \{-L,\ldots,L\}^{d-1}$ satisfying $x_1=|x|\in \{0,\ldots,L\}$, one has
\begin{equation}
    \langle \varphi_{(x,0)}\rangle^{+,+}_{\mathcal{R}(L),u}\geq \langle \varphi_{(x,0)}\rangle^{+,+}_{\mathcal{R}(L,L-x_1),u}\geq \langle\varphi_0\rangle_{\Lambda_{L-x_1}(x),u,\tilde{\mathsf{h}}_{L-x_1}(x)},
\end{equation}
where $\tilde{\mathsf{h}}_{L-x_1}(x)$ was defined above Proposition \ref{prop: one side is enough}. The latter proposition gives that
\begin{equation}
    \liminf_{L\rightarrow \infty} \langle\varphi_0\rangle_{\Lambda_{L-x_1}(x),u,\tilde{\mathsf{h}}_{L-x_1}(x)}\geq \frac{1}{2d}\langle \varphi_0\rangle^+_u.
\end{equation}
It follows that
\begin{equation}
\tau_{\beta}\geq \lim_{L\to\infty} \frac{1}{d L^{d-1}} \sum_{x\in \{-L,\ldots,L\}^{d-1}} \int_{\beta_c}^{\beta} \left(\langle \varphi_{0} \rangle^+_{u}\right)^2 \mathrm{d}u= \frac{1}{d}\int_{\beta_c}^{\beta} \left(\langle \varphi_{0} \rangle^+_{u}\right)^2 \mathrm{d} u >0,
\end{equation}
as desired.
\end{proof}

We now prove the lemma mentioned in the proof of Proposition~\ref{prop: surface tension positive}.

\begin{lemma}\label{lem: pm convergence}
Let $\beta>0$. For every $L\geq 1$ we have 
\begin{equation}
\lim_{N\to\infty}\langle \varphi_{(x,N)} \rangle^{+,-}_{\cR(L),\beta}=\langle \varphi_{(x,0)} \rangle^{+,+}_{\cR(L),\beta} \quad \text{and} \quad  \lim_{N\to\infty}\langle \varphi_{(x,-N)} \rangle^{+,-}_{\cR(L),\beta} =-\langle \varphi_{(x,0)} \rangle^{+,+}_{\cR(L),\beta} 
\end{equation}
for every $x\in \{-L,\ldots,L\}^{d-1}$.
\end{lemma}

\begin{proof}
We will only obtain the limit of $\langle \varphi_{(x,N)} \rangle^{+,-}_{\cR(L),\beta}$. The limit of $\langle \varphi_{(x,-N)} \rangle^{+,-}_{\cR(L),\beta}$ can be obtained similarly, once one observes that $-\langle \varphi_{(x,0)} \rangle^{+,+}_{\cR(L),\beta}=\langle \varphi_{(x,0)} \rangle^{-,-}_{\cR(L),\beta}$.

Let us define 
\begin{equation}
\iota := \min \{i>0:~ \varphi_{(x,2i)} >0, \, \forall x \in \{-L,\ldots, L \}^{d-1}\}.
\end{equation}
We will first prove that $\iota$ is finite almost surely.

Reasoning as in the proof of Lemma \ref{lem:++}, the stochastic domination (by a product measure) of Proposition \ref{prop:regularity} extends to the measure $\langle \cdot \rangle^{+,-}_{\mathcal R(L), \beta}$. Therefore, there exists $M>0$ such that $\langle \cdot\rangle^{+,-}_{\cR(L),\beta}$--almost surely, there are infinitely many $i>0$ such that the event  
\begin{equation}
A_{M,i}:=\{|\varphi_{(x,\ell)}|\leq M ,\: \forall x\in \{-L,\ldots,L\}^{d-1}, \forall \ell \in \{2i-1,2i+1\} \}
\end{equation}
occurs. Let $\mathcal{R}^+_{\rm odd}(L):=\{(x,2j+1): x\in \{-L,\ldots,L\}^{d-1}, \: j\geq 0\}$.  Using the domain Markov property, we obtain that
\begin{equation}
\nu^{+,-}_{\mathcal{R}(L),\beta}\left[\Big\{\varphi_{(x,2i)}>0,\, \forall x\in \{-L,\ldots,L\}^{d-1}\Big\} \mid \varphi_{|\mathcal{R}^+_{\rm odd}(L)}=\eta_{|\mathcal{R}^+_{\rm odd}(L)} \right],
\end{equation}
remains bounded away from $0$, uniformly over $i$ and $\eta_{|\mathcal{R}^+_{\rm odd}(L)}\in A_{M,i}$. Furthermore, these events are independent of each other conditionally on $\varphi_{|\mathcal{R}^+_{\rm odd}(L)}$. Therefore, $\iota$ is indeed finite almost surely.

Let us now condition on the event $\{\iota = i\}$, which is measurable with respect to $\varphi_{(x,j)}$, $j\leq 2i$ and $x \in \{-L,\ldots,L\}^{d-1}$. By the Markov property and monotonicity in boundary condition, we have that for every $N\geq 2i$
\begin{equation}
\langle \varphi_{(x,N)} 
 \mid \iota = i \rangle^{+,-}_{\cR(L),\beta}\geq  \langle \varphi_{(x,0)} \rangle^{+,+}_{\cR(L,N-2i),\beta}.
\end{equation}
As a consequence,
\begin{equation}
    \langle \varphi_{(x,N)}  \rangle^{+,-}_{\cR(L),\beta}\geq \langle \varphi_{(x,0)} \rangle^{+,+}_{\cR(L,N-2i),\beta}\langle \mathds{1}_{\iota\leq i}\rangle_{\mathcal{R}(L),\beta}^{+,-}+\langle \varphi_{(x,N)}\mathds{1}_{\iota>i}  \rangle^{+,-}_{\cR(L),\beta}.
\end{equation}
By the Cauchy--Schwarz inequality and Proposition \ref{prop:regularity}, there exists a constant $C>0$ such that for every $i\geq 1$ we have
\begin{equation}
|\langle \varphi_{(x,N)} 
\mathbbm{1}_{\iota>i} \rangle^{+,-}_{\cR(L),\beta}|\leq \sqrt{\langle \varphi^2_{(x,N)}  \rangle^{+,-}_{\cR(L),\beta}\langle 
\mathbbm{1}_{\iota>i}  \rangle^{+,-}_{\cR(L),\beta}} 
\leq C\sqrt{\langle 
\mathbbm{1}_{\iota>i}  \rangle^{+,-}_{\cR(L),\beta}}=: \varepsilon_i.
\end{equation}
Thus, by Lemma \ref{lem:++}, for every $i>0$,
\begin{equation}
\liminf_{N\to\infty} \langle \varphi_{(x,N)} 
 \rangle^{+,-}_{\cR(L),\beta}\geq \langle \varphi_{(x,0)} \rangle^{+,+}_{\cR(L),\beta}\langle \mathbbm{1}_{\iota\leq i}  \rangle^{+,-}_{\cR(L),\beta}-\varepsilon_i.
 \end{equation}
 Since $\iota$ is finite almost surely, $\langle \mathbbm{1}_{\iota\leq i}  \rangle^{+,-}_{\cR(L),\beta}\rightarrow 1$ and $\varepsilon_i\rightarrow 0$ as $i\rightarrow \infty$. We can deduce that 
\begin{equation}
\liminf_{N\to\infty} \langle \varphi_{(x,N)} 
 \rangle^{+,-}_{\cR(L),\beta}\geq \langle \varphi_{(x,0)} \rangle^{+,+}_{\cR(L),\beta}.
 \end{equation}
Moreover, by Proposition \ref{prop:Ginibre} (with $A=\{(x,N)\}$ and $B=\emptyset$), for every $N$,
\begin{equation}
    \langle \varphi_{(x,N)}\rangle^{+,-}_{\mathcal{R}(L),\beta}\leq \langle \varphi_{(x,0)} \rangle^{+,+}_{\cR(L),\beta}.
\end{equation}
The proof follows readily from the two last displayed equations.


\end{proof}

\subsection{Comparison with the free $\varphi^4$ random cluster model}\label{prop:from thick plus to free}

In this section we prove Theorem~\ref{thm:free_surface_tension}. Proposition~\ref{prop: surface tension positive} and Lemma~\ref{lem: ratio probablistic expression} readily imply a surface order bound on disconnection for the $\varphi^4$ random cluster measure $\Psi^{0}_{\mathcal R(L),\beta,\mathsf{h}^{+,+}_{L}}:=\lim_{M\to\infty}\Psi^{0}_{\mathcal R(L,M),\beta,\mathsf{h}^{+,+}_{L,M}}$--- see Lemma~\ref{lem: positivity of tau}. Theorem~\ref{thm:free_surface_tension} then follows from a comparison to the free measure $\Psi^0_{R(L,\delta L),\beta}$ (for $\delta>0$ well chosen). This comparison is established in two steps in Lemmas \ref{lem: strip to rectangle} and \ref{lem: large to free bc}. The uniqueness of the $\varphi^4$ random cluster measure on the half-space with positive magnetic field (recall Propositions~\ref{prop: unique half-space measure} and \ref{prop: equality half space random cluster measures}) will be crucial in the later step.

\begin{lemma}\label{lem: positivity of tau}
Let $\beta>\beta_c$. There exists $c_2>0$ such that for every $L \geq 1$,
\begin{equation}\label{eq: proba strip}
\Psi^{0}_{\mathcal R(L),\beta,\mathsf{h}^{+,+}_{L}}[ \fg^{-} \longleftrightarrow \fg^{+}] \geq 1 - e^{-c_2L^{d-1}}.
\end{equation}
\end{lemma}

\begin{proof}
It follows from Lemma~\ref{lem: ratio probablistic expression} and Proposition~\ref{prop: surface tension positive} that 
\begin{equation}
\limsup_{L\to\infty}\frac{1}{L^{d-1}}\log\Psi^{0}_{\mathcal R(L),\beta,\mathsf{h}^{+,+}_{L}}[\fg^{-} \centernot\longleftrightarrow \fg^{+}]<0.    
\end{equation}
Since $\Psi^{0}_{\mathcal R(L),\beta,\mathsf{h}^{+,+}_{L}}[\fg^{-} \longleftrightarrow \fg^{+}]>0$ for every $L\geq 1$, there exists a constant $c_2>0$ such that 
\begin{equation}
\frac{1}{L^{d-1}}\log\Psi^{0}_{\mathcal R(L),\beta,\mathsf{h}^{+,+}_{L}}[\fg^{-} \centernot\longleftrightarrow \fg^{+}]\leq -c_2  \end{equation}
for every $L\geq 1$. This concludes the proof.
\end{proof}

We now obtain a finite volume analogue of Lemma~\ref{lem: positivity of tau} for a $\varphi^4$ random cluster measure defined on a rectangle with sufficiently large external magnetic field on the top and bottom faces. More precisely, for $h>0$, let $\Psi^{0,h}_{\mathcal{R}(L,M),\beta}$ denote the $\varphi^4$ random cluster measure on $\mathcal{R}(L,M)$ with external magnetic field equal to $h$ on $\partial^{\rm bot} \mathcal{R}(L,M)\cup \partial^{\rm top} \mathcal{R}(L,M)$ and $0$ otherwise, where
\begin{equation}
    \partial^{\rm bot} \mathcal{R}(L,M):=\{x \in \mathcal{R}(L,M): x_d=-M\}, \; \partial^{\rm top} \mathcal{R}(L,M):=\{x \in \mathcal{R}(L,M): x_d=M\}.
\end{equation}

\begin{lemma}\label{lem: strip to rectangle}
Let $\beta>\beta_c$. For every $h>0$ sufficiently large, there exist $\delta>0$ and $c_3=c_3(h)>0$ such that for every $L \geq 1$,
\begin{equation}
\Psi^{0,h}_{R(L,\delta L),\beta}[\partial^{{\rm bot}} \mathcal R(L, \delta L) \longleftrightarrow\partial^{{\rm top}} \mathcal R(L, \delta L)] \geq 1 - e^{-c_3 L^{d-1}}. 
\end{equation}
\end{lemma}

\begin{proof}
We fix $\beta>\beta_c$ and drop it from notation. Let $\delta>0$ be a constant to be determined. Recall Lemma~\ref{lem: positivity of tau}. Our aim is to first compare $\Psi^{0}_{\mathcal R(L),\mathsf{h}^{+,+}_{L}}$ to $\Psi^{0,h}_{\mathcal{R}(L,\delta L)}$ up to a cost that we can make smaller than $e^{(c_2/2) L^{d-1}}$ by tuning the value of $h$, and then forcing a particular set of edges to be closed, which in turn forces any path from $\fg^{-}$ to $\fg^{+}$ to connect $\partial^{\rm bot} \mathcal R(L, \delta L) $ to $ \partial^{\rm top} \mathcal R(L, \delta L)$. To achieve the latter at a small cost we need to choose $\delta$ sufficiently small. This naturally splits the proof into two steps.

\paragraph{Step 1.}
Let $\varepsilon>0$ be a constant to be determined. Introduce  $A:= A^{\rm top}\cup A^{\rm bot}$, where $A^{\rm top}=\partial^{\rm top}\mathcal{R}(L,\delta L+1)$ and $A^{\rm bot}=\partial^{\rm bot}\mathcal{R}(L,\delta L+1)$. By the FKG inequality and regularity, we can now choose $H>0$ to be large enough so that 
\begin{equation}
\Psi^0_{\mathcal{R}(L),\mathsf{h}^{+,+}_L} [\mathsf{a}|_A\leq H ]\geq e^{-c_2/2 L^{d-1}}, 
\end{equation}
where $c_2$ is the constant of Lemma~\ref{lem: positivity of tau}. Thus,
\begin{equation}\label{eq: proof lemma T1}
\Psi^0_{\mathcal{R}(L),\mathsf{h}^{+,+}_L} [\fg^{-}\centernot \longleftrightarrow \fg^{+}]\geq \Psi^0_{\mathcal{R}(L),\mathsf{h}^{+,+}_L} [\fg^{-}\centernot \longleftrightarrow \fg^{+} \mid \mathsf{a}|_A\leq H ] e^{-c_2/2 L^{d-1}}.
\end{equation}
By Proposition \ref{prop:monotonicity_FK}, and the FKG inequality of Proposition \ref{prop:FKG random cluster} (for the conditional measure $\Psi^0_{\mathcal R(L), \mathsf h^{+,+}_L}[\: \cdot \mid \mathsf a|_A = H]$),
\begin{equation}
\begin{aligned}
\Psi^0_{\mathcal{R}(L),\mathsf{h}^{+,+}_L} [\fg^{-}\centernot \longleftrightarrow \fg^{+} \mid \mathsf{a}|_A\leq H ]& \geq \Psi^0_{\mathcal{R}(L),\mathsf{h}^{+,+}_L} [\fg^{-}\centernot \longleftrightarrow \fg^{+} \mid \mathsf{a}|_{A}=H] \\ &\geq \Psi^0_{\mathcal{R}(L),\mathsf{h}^{+,+}_L} [\fg^{-}\centernot \longleftrightarrow \fg^{+} \mid \mathsf a|_A=H,  \mathcal{A}],
\end{aligned}
\end{equation}
where \begin{equation}
    \mathcal A:= \{\omega_e = 1, \: \forall e \in E(A)\}\cap \{  \exists \, (x,y) \in A^{\rm top}\times  A^{\rm bot} : \: \omega_{x\mathfrak g^+}=1, \: \omega_{y\mathfrak g^-}=1 \}.
\end{equation}
Let us denote $\mathsf{H}$ the external magnetic field which is equal to $H$ on $\partial^{\rm top}\mathcal{R}(L,\delta L)\cup \partial^{\rm bot}\mathcal{R}(L,\delta L)$. Since on the event $\mathcal A$, $ A^{\rm top}$ is connected to $\fg^+$, and $A^{\rm bot}$ is connected to $\fg^-$, the domain Markov property of Proposition \ref{prop: domain markov} gives
\begin{equation}\label{eq: proof lemma T2}
\Psi^{0}_{\mathcal R(L),\mathsf{h}^{+,+}_{L}}[\:\cdot \mid \mathsf a|_A=H,  \mathcal A]=  \Psi^{0}_{\mathcal R(L,\delta L),\mathsf{h}^{+,+}_{L,\delta L}+\mathsf{H}}.
\end{equation} 
Combining \eqref{eq: proof lemma T1} and \eqref{eq: proof lemma T2}, and applying Lemma \ref{lem: positivity of tau}, we obtain
\begin{equation}\label{eq: comparison 3}
\Psi^{0}_{\mathcal R(L,\delta L),\mathsf{h}^{+,+}_{L,\delta L}+\mathsf{H}}[\fg^{-}\centernot \longleftrightarrow \fg^{+}]\leq e^{-c_2/2 L^{d-1}}.
\end{equation}

\begin{figure}[htb]
    \centering
    \includegraphics[width=0.7\linewidth]{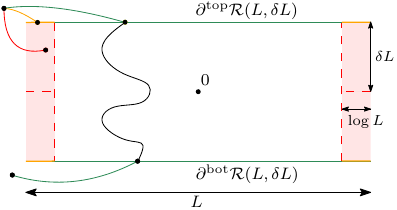}
    \put(-307,155){$\fg^+$}
    \put(-307,32){ $\fg^-$}
    \caption{An illustration of \eqref{eq:Tproofinclusion}. The magnetic field $\mathsf{h}^{+,+}_{L,\delta L}+\mathsf{H}$ is supported on the coloured regions: in the red region it is equal to $1$, in the green region it is equal to $H$, and in the overlap of the two regions--- in orange--- it is equal to $H+1$. This provides three different ways to connect to either of the two ghost vertices. If there exists an open path in $\omega$ connecting $\fg^+$ to $\fg^-$ and additionally the graph $T$--- depicted by the dotted red line--- is fully closed in $\omega$, then there must exist an open path from $\partial^{\rm bot} \mathcal{R}(L,\delta L)$ to $\partial^{\rm top} \mathcal{R}(L,\delta L)$.}
    \label{fig:t graph}
\end{figure}

\paragraph{Step 2.} We now move to the second part of the proof.
We begin by defining the graph $T=(V(T),E(T))$ defined as follows: 
\begin{multline}
    V(T):= \Big\{x \in \{-L,\ldots ,L\}^{d-1}\times \{0,1\}: L-\log L\leq |x|\leq L\Big\}\\\cup \Big\{x \in \partial \Lambda_{L-\log L }\cup \partial \Lambda_{L-\log L +1}: |x_d|\leq \delta L\Big\},
\end{multline}
and $E(T):=\{xy\in E(\mathbb Z^d): x,y\in V(T)\}$. Let us define the event
\begin{equation}
\mathcal{E}:= \{\partial^{\rm bot} \mathcal R(L, \delta L)\centernot\longleftrightarrow \partial^{\rm top} \mathcal R(L, \delta L)\}.
\end{equation}
Observe that
\begin{equation}\label{eq:Tproofinclusion}
\mathcal{E}\cap \{\omega|_{E(T)}=0\} \subset \{\fg^{-}\centernot \longleftrightarrow \fg^{+}\}\cap \{\omega|_{E(T)}=0\}.
\end{equation}
See Figure \ref{fig:t graph} for an illustration. To estimate the cost of $\{\omega|_{E(T)}=0\}$, we first bound the absolute value field on $V(T)$.
By Proposition \ref{prop:regularity}, there exists a constant $M>0$ such that
$\Psi^{0}_{R(L,\delta L),\mathsf{h}^{+,+}_{L,\delta L}+\mathsf{H}}[\mathcal{B}]\geq \frac{1}{2}$,
where
\begin{equation}
\mathcal{B}:=\left\{\sum_{xy\in E(T)} \mathsf{a}_x \mathsf{a}_y\leq M |E(T)|\right\}. \end{equation}
Since $\boldsymbol{\phi}^{0,H}_{\mathcal{R}(L,\delta),\mathsf{h}^{+,+}_{L,\delta L}+\mathsf{H},\mathsf{a}}$ is stochastically dominated by Bernoulli percolation of parameter $p(\beta,\mathsf{a})$,
\begin{equation}\label{eq: close edges}
\begin{aligned}
\Psi^{0}_{R(L,\delta L),\mathsf{h}^{+,+}_{L,\delta L}+\mathsf{H}}[\omega|_{E(T)}=0 \mid \mathcal{B}]&\geq \Psi^{0}_{R(L,\delta L),\mathsf{h}^{+,+}_{L,\delta L}+\mathsf{H}}\left[\exp\Big(-2\beta\sum _{xy\in E(T)} \mathsf{a}_x \mathsf{a}_y\Big)\mid \mathcal{B}\right]
\\&\geq e^{-2\beta M |E(T)|}.
\end{aligned}
\end{equation}
Thus,
\begin{equation}\label{eq: comparison 4}
\begin{aligned}
\Psi^{0}_{R(L,\delta L),\mathsf{h}^{+,+}_{L,\delta L}+\mathsf{H}}[\fg^{-} \centernot \longleftrightarrow \fg^{+}]&\geq \frac{1}{2} e^{-2\beta M |E(T)|}\Psi^{0}_{R(L,\delta L),\mathsf{h}^{+,+}_{L,\delta L}+\mathsf{H}}[\fg^{-} \centernot \longleftrightarrow \fg^{+}\mid \mathcal{B}\cap \{ \omega|_{E(T)}=0 \}]\\
&\geq \frac{1}{2} e^{-2\beta M |E(T)|}\Psi^{0}_{R(L,\delta L),\mathsf{h}^{+,+}_{L,\delta L}+\mathsf{H}}[\mathcal{E}\mid \mathcal{B}\cap \{ \omega|_{E(T)}=0 \}].
\end{aligned}
\end{equation}
We now compare the latter probability with $\Psi^{0,H}_{R(L-\log L,\delta L)}[\mathcal{E}]$.
Using that $e^{\beta \mathsf{a}_x \mathsf{a}_y}\geq 1$ and the definition of the event $\mathcal{B}$ we see that
\begin{equation}\label{eq: comparison 5}
\begin{aligned}
\Psi^{0}_{R(L,\delta L),\mathsf{h}^{+,+}_{L,\delta L}+\mathsf{H}}&[\mathcal{E}\mid \mathcal{B}\cap \{ \omega|_{E(T)}=0 \}] \\
 &\geq e^{-\beta M |E(T)|}\frac{\Psi^{0}_{R(L,\delta L),\mathsf{h}^{+,+}_{L,\delta L}+\mathsf{H}}[\exp(\sum_{xy \in E(T)}\beta \mathsf{a}_x \mathsf{a}_y)\mathbbm{1}_{\mathcal{E}} \mid \mathcal{B}\cap  \{ \omega|_{E(T)}=0 \}]}{\Psi^{0}_{R(L,\delta L),\mathsf{h}^{+,+}_{L,\delta L}+\mathsf{H}}[\exp(\sum_{xy \in E(T)}\beta \mathsf{a}_x \mathsf{a}_y) \mid   \{ \omega|_{E(T)}=0\}]}\\
  &\geq e^{-\beta M |E(T)|}\frac{\Psi^{0}_{R(L,\delta L),\mathsf{h}^{+,+}_{L,\delta L}+\mathsf{H}}[\exp(\sum_{xy \in E(T)}\beta \mathsf{a}_x \mathsf{a}_y)\mathbbm{1}_{\mathcal{E}} \mathbbm{1}_{\mathcal{B}} \mid   \{ \omega|_{E(T)}=0 \}]}{\Psi^{0}_{R(L,\delta L),\mathsf{h}^{+,+}_{L,\delta L}+\mathsf{H}}[\exp(\sum_{xy \in E(T)}\beta \mathsf{a}_x \mathsf{a}_y) \mid   \{ \omega|_{E(T)}=0 \}]}.
\end{aligned}
\end{equation}
By the expression \eqref{eq: random cluster explicit density} for the density, the ratio in the last line of \eqref{eq: comparison 5} is equal to $\Psi^{0}_{R(L,\delta L),\mathsf{h}^{+,+}_{L,\delta L}+\mathsf{H},J}[\mathcal{E} \cap \mathcal{B}]$, where $J_e=0$ for every $e\in E(T)$, and $J_e=1$ otherwise. Applying Proposition \ref{prop:regularity} for $\Psi^{0}_{R(L,\delta L),\mathsf{h}^{+,+}_{L,\delta L}+\mathsf{H},J}$, we can increase the value of $M$, if necessary, so that $\Psi^{0}_{R(L,\delta L),\mathsf{h}^{+,+}_{L,\delta L}+\mathsf{H},J}[\mathcal{B}]\geq 1-e^{-L^{d-1}}$. At this point, we take cases. Either
\begin{equation}
\Psi^{0}_{R(L,\delta L),\mathsf{h}^{+,+}_{L,\delta L}+\mathsf{H},J}[\mathcal{E}]\leq 2 \Psi^{0,H}_{R(L,\delta L),J}[\mathcal{B}^c]    
\quad \text{or} \quad 
\Psi^{0}_{R(L,\delta L),\mathsf{h}^{+,+}_{L,\delta L}+\mathsf{H},J}[\mathcal{E}]> 2 \Psi^{0,H}_{R(L,\delta L),J}[\mathcal{B}^c].    
\end{equation}
In the first case, by definition of $J$, the desired result follows directly since we have 
\begin{equation}
\Psi^{0}_{R(L,\delta L),\mathsf{h}^{+,+}_{L,\delta L}+\mathsf{H},J}[\mathcal{E}]=\Psi^{0,H}_{R(L-\log L,\delta L)}[\mathcal{E}].
\end{equation}
In the second case, we have
\begin{equation}\label{eq: comparison 6}
\begin{aligned}
\Psi^{0}_{R(L,\delta L),\mathsf{h}^{+,+}_{L,\delta L}+\mathsf{H},J}[\mathcal{E} \cap \mathcal{B}]&\geq \Psi^{0,H}_{R(L,\delta L),J}[\mathcal{E}]-\Psi^{0}_{R(L,\delta L),\mathsf{h}^{+,+}_{L,\delta L}+\mathsf{H},J}[\mathcal{B}^c]\\&\geq \frac{1}{2}\Psi^{0}_{R(L,\delta L),\mathsf{h}^{+,+}_{L,\delta L}+\mathsf{H},J}[\mathcal{E}]=\frac{1}{2}\Psi^{0,H}_{R(L-\log L,\delta L)}[\mathcal{E}].
\end{aligned}
\end{equation}
Combining inequalities \eqref{eq: comparison 3}, \eqref{eq: comparison 4}, \eqref{eq: comparison 5} and \eqref{eq: comparison 6} we obtain 
\begin{align}
    \Psi^{0,H}_{R(L-\log L,\delta L)}[\mathcal{E}]&\leq 4e^{3\beta M |E(T)|} \Psi^{0}_{\mathcal{R}(L),\mathsf{h}^{+,+}_{L,\delta L}+\mathsf{H}}[\fg^{-}\centernot \longleftrightarrow \fg^{+}]
    \\&\leq 4e^{3\beta M |E(T)|} e^{-c_2/2L^{d-1}}.\label{eq:Tproof4}
\end{align} 
Recall that $|E(T)|\leq C\delta^{d-1}L^{d-1}$, and choose $\delta$ small enough such that, for $L$ large enough
\begin{equation}\label{eq:Tproof5}
    4e^{3\beta M |E(T)|} \leq e^{(c_2/4)L^{d-1}}.
\end{equation}
Plugging \eqref{eq:Tproof5} in \eqref{eq:Tproof4}, and choosing $c_3$ small enough to accommodate the small values of $L$ concludes the proof.
\end{proof}

Finally, we show that we can reduce the probability to free boundary conditions.

\begin{lemma}\label{lem: large to free bc} Let $\beta>\beta_c$. There exist $\delta,c_4>0$ such that for every $L \geq 1$,
\begin{equation}\label{eq: large to free bc}
\Psi^0_{\mathcal R(L,\delta L),\beta}[ \partial^{\rm bot} \mathcal R(L, \delta L/2) \longleftrightarrow \partial^{\rm top} \mathcal R(L, \delta L/2)] \geq 1 - e^{-c_4 L^{d-1}}. 
\end{equation}
\end{lemma}

\begin{proof}
Let $h,\delta,c_3>0$ be the constants of Lemma~\ref{lem: strip to rectangle}.   
Write $\mathcal{D}=\{\partial^{\rm bot} \mathcal R(L, \delta L/2) \centernot\longleftrightarrow \partial^{\rm top} \mathcal R(L, \delta L/2)\}$.
By inclusion of events, we have 
   \begin{equation}\label{eq:mixed_b.c.}
   \Psi^{0,h}_{\mathcal R(L,\delta L),\beta}[\mathcal{D}] \leq e^{-c_3L^{d-1}}.
   \end{equation}
We now gradually decrease the value of the external magnetic field from $s=h$ to $s=0$ in order to interpolate between $\Psi^{0,h}_{\mathcal R(L,\delta L),\beta}[ \mathcal{D} ]$ and $\Psi^{0}_{\mathcal R(L,\delta L),\beta}[ \mathcal{D} ]$.
To this end, consider the derivative with respect to $s$ and note that by Russo's formula (c.f.\ the calculation of \cite[Theorem~(2.43)]{Grimmett2006RCM} in the case of the usual random cluster measure)
\begin{equation} \label{eq: log derivative}
\frac{\partial \log \Psi^{0,s}_{\mathcal R(L,\delta L),\beta}[ \mathcal{D} ]}{\partial s}= \sum_{x\in \partial^{\pm}\mathcal{R}(L,\delta L)} \Psi^{0,s}_{\mathcal R(L,\delta L),\beta}\left[F_x(\mathsf{a},\omega) \mid  \mathcal{D}\right] -\Psi^{0,s}_{\mathcal R(L,\delta L),\beta}\left[F_x(\mathsf{a},\omega)\right],
\end{equation}
where $\partial^{\pm}\mathcal{R}(L,\delta L)=\partial^{\rm bot}\mathcal{R}(L,\delta L)\cup \partial^{\rm top}\mathcal{R}(L,\delta L)$, and
\begin{equation}
F_x(\mathsf{a},\omega)=-\beta \mathsf{a}_x+\frac{\omega_{x\fg}}{p(\beta,s,\mathsf{a})_{x\fg}(1-p(\beta,s,\mathsf{a})_{x\fg})} \frac{\partial p(\beta,s,\mathsf{a})_{x\fg}}{\partial s}=-\beta\mathsf{a}_x\left(1+2\frac{\omega_{x\fg}}{p(\beta,s,\mathsf{a})_{x\fg}}\right).
\end{equation}
Here the term $-\beta \mathsf{a}_x$ arises from differentiating $e^{-\beta \mathsf{a}_x s}$, while the other term arises from differentiating $\left(\frac{p(\beta,s,\mathsf{a})}{1-p(\beta,s,\mathsf{a})}\right)^{\omega_{xg}}$.

We are going to split the sum in \eqref{eq: log derivative} according to boundary and bulk contributions. Let $\epsilon \in (0,\delta/2)$ to be fixed below. Define
\begin{equation}
{\rm Bulk}(\varepsilon, L):= \{ x \in \partial^{\pm}\mathcal{R}(L,\delta L): {\rm dist}^\infty( x, \partial R(L,\delta L) \setminus \partial^\pm \mathcal R(L,\delta L)) \geq \varepsilon L \},
\end{equation}
and write
\begin{equation}
\eqref{eq: log derivative} =: I({\rm bulk}) + I({\rm boundary})
\end{equation}
in the natural way. 

Let us treat the boundary contribution first. Arguing as in the proof of \eqref{eq: omega connectivity equality} we see that we can replace $2\omega_{x\fg}$ by $p(\beta,s,\mathsf{a})_{x\fg} \left(\mathbbm{1}\{x\overset{\omega}{\longleftrightarrow}\fg\}+1\right)$ to obtain
\begin{equation}\label{eq: derivative formula bulk}
I({\rm boundary}) = \sum_{x \in {\rm Bulk}(\varepsilon,L)^c}
\Psi^{0,s}_{\mathcal R(L,\delta L),\beta}\left[G_x(\mathsf{a},\omega) \mid  \mathcal{D}\right] -\Psi^{0,s}_{\mathcal R(L,\delta L),\beta}\left[G_x(\mathsf{a},\omega)\right],
\end{equation}
where
\begin{equation}
G_x(\mathsf{a},\omega)=-\beta \mathsf{a}_x\left(2+\mathbbm{1}\{x\overset{\omega}{\longleftrightarrow}\fg\}\right).    
\end{equation}
Notice that $G \leq 0$ and $G$ is decreasing. Additionally, note that by Proposition \ref{prop:regularity}, there exists $C>0$ such that, for every $x\in \partial^{\rm bot}\mathcal{R}(L,\delta L)$,
\begin{equation}
    \max \left( \left|\Psi^{0,s}_{\mathcal R(L,\delta L),\beta}\left[G_x(\mathsf{a},\omega)\right]\right|,  
    \left|\Psi^{0,s}_{\mathcal R(L,\delta L),\beta}\left[G_x(\mathsf{a},\omega) \mid  \mathcal{D}\right]\right| \right) \leq C,\label{eq: general control summand}
\end{equation}
where we used the FKG inequality in the second inequality. Hence, there exists $C_1 > 0$ such that, for every $s \in (0,h]$,
\begin{equation}
|I({\rm boundary})| \leq C_1 \varepsilon L^{d-1}.
\end{equation}

We now analyse the bulk contribution, which requires some more care. Without loss of generality, we only treat the case $x\in \partial^{\rm bot}\mathcal{R}(L,\delta L)$. 
Define $\mathcal{B}(L,\varepsilon,x):=\Lambda^+_{\varepsilon L}(x)=\Lambda_L(L\mathbf{e}_d+x)$, and for $\mathsf{b}\in (\mathbb{R}^+)^{\partial \mathcal{B}(L,\varepsilon,x)}$,
\begin{equation}
\mathsf{h}(x;s,\mathsf{b})_y:=
\begin{cases}
s & \text{for } y\in \partial \mathcal{B}(L,\varepsilon,x)\cap \{y\in \mathbb Z^d: y_d=-\delta L\},\\
\mathsf{b}_y & \text{for } y\in \partial \mathcal{B}(L,\varepsilon,x)\cap \{y\in \mathbb Z^d: y_d>-\delta L\},\\
0 & \text{otherwise.}
\end{cases}    
\end{equation} 
See Figure \ref{fig:removinglargemagfield} for an illustration. For $\Psi^{0,s}_{\mathcal R(L,\delta L), \beta}$-a.s. $(\omega', \mathsf a')$, 
\begin{equation}
\Psi^{(\xi^{\omega'},\mathsf a')}_{\mathcal B(L, \varepsilon, x), \beta}[F_x(\mathsf a, \omega)] = \Psi^{(\xi^{\omega'},\mathsf a')}_{\mathcal B(L, \varepsilon, x), \beta}[G_x(\mathsf a, \omega)],
\end{equation}
where we recall from the Proposition \ref{prop: domain markov} that $(\xi^{\omega'},\mathsf a')$ is the boundary condition induced by the pair $(\omega',\mathsf a')$ on $\mathcal B(L,\varepsilon, x)$. Furthermore, note that the event $\mathcal{D}$ is supported on $\mathcal{R}(L,\delta L/2)$.
Hence, by using the display above together with Propositions \ref{prop:regularity} and \ref{prop:monotonicity_FK}, we get that (recall that $G_x$ is a decreasing function) if $x\in \partial^{\rm bot}\mathcal{R}(L,\delta L)$ with $|x|\leq L-\varepsilon L$,
\begin{align}\label{eq: psi lower bound}
\Psi^{0,s}_{\mathcal R(L,\delta L),\beta}\left[F_x(\mathsf{a},\omega) \right]&\geq \Psi^{0}_{\mathcal{B}(L,\varepsilon,x),\beta,\mathsf{h}(x;s,\mathfrak{p})}\left[G_x(\mathsf{a},\omega)\right]-o(1),
\\ \label{eq: psi upper bound}
\Psi^{0,s}_{\mathcal R(L,\delta L),\beta}\left[F_x(\mathsf{a},\omega)\mid  \mathcal{D}\right]
&\leq  \Psi^{0}_{\mathcal{B}(L,\varepsilon,x),\beta,\mathsf{h}(x;s,0)}\left[G_x(\mathsf{a},\omega)\right].
\end{align}
where the $o(1)$ term tends to $0$ as $L\to\infty$, uniformly over $s\in [0,h]$ and $x$ as above. Using Proposition \ref{prop: equality half space random cluster measures}, one has that $\Psi^{0,s}_{\mathbb H,\beta}=\Psi^{1,s}_{\mathbb H,\beta}$ for every $s\in (0,h]$. As a result, 
\begin{equation}\label{eq: summand goes to 0 far away from the boundary}
   \Psi^{0}_{\mathcal{B}(L,\varepsilon,x),\beta,\mathsf{h}(x;s,0)}\left[G_x(\mathsf{a},\omega)\right]-\Psi^{0}_{\mathcal{B}(L,\varepsilon,x),\beta,\mathsf{h}(x;s,\mathfrak{p})}\left[G_x(\mathsf{a},\omega)\right]=o_s(1),
\end{equation}
where, for any $s \in (0,h]$, $o_s(1)$ tends to $0$ as $L$ tends to infinity uniformly over $x\in \partial^{\rm bot}\mathcal{R}(L,\delta L)$ such that $|x|\leq L-\varepsilon L$, and by \eqref{eq: diff g} $|o_s(1)|\leq C$.

\begin{figure}[htb]
    \centering
    \includegraphics[width=0.6\linewidth]{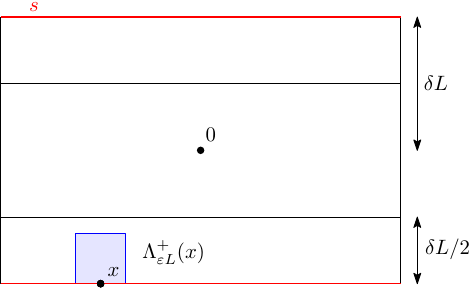}
    \caption{An illustration of the proof of Lemma \ref{lem: large to free bc}. The blue region is the box $\mathcal{B}(L,\varepsilon,x)=\Lambda^+_{\varepsilon L}(x)$. We put an external magnetic field on the bold blue boundary of $\mathcal{B}(L,\delta,x)$ which is either $0$ or $\mathfrak{p}$. The top and bottom boundaries of $\mathcal{R}(L,\delta L)$ (in red) carry a magnetic field $s$. The law of $G_x(\mathsf{a},\omega)$ within $\mathcal{B}(L,\varepsilon,x)$ approximates that of $G_x(\mathsf{a},\omega)$ under the half-space measures $\Psi^{0,s}_{\mathbb H,\beta}$ or $\Psi^{1,s}_{\mathbb H,\beta}$. Thanks to Proposition \ref{prop: equality half space random cluster measures}, these measures coincide.}
    \label{fig:removinglargemagfield}
\end{figure}

Combining the above, we obtain
\begin{equation}
\Bigg|\frac{\partial \log \Psi^{0,s}_{\mathcal R(L,\delta L),\beta}[ \mathcal{D} ]}{\partial s}\Bigg|\leq C_1 \varepsilon L^{d-1}+ \sum_{{\rm Bulk}(\varepsilon,L)} o_s(1).
\end{equation}
Integrating in $s$ and using the dominated convergence theorem, we obtain that there exists $C_2>0$ such that for $L$ large enough,
\begin{equation}
\Big|\log \Psi^{0,h}_{\mathcal R(L,\delta L),\beta}[ \mathcal{D} ]- \log \Psi^{0}_{\mathcal R(L,\delta L),\beta}[ \mathcal{D} ]\Big|\leq C_2\varepsilon L^{d-1}.  
\end{equation}
The proof follows from choosing $\varepsilon$ small enough in such a way that $C_2\varepsilon \leq c_3/2$, and from fixing $c_4$ small enough to accommodate the small values of $L$.
\end{proof}

We are now ready to prove Theorem~\ref{thm:free_surface_tension}.

\begin{proof}[Proof of Theorem~\textup{\ref{thm:free_surface_tension}}]
Lemma \ref{lem: large to free bc} implies that there exist some $\delta,c_4>0$ such that for every $L\geq 1$,
\begin{equation}\label{eq:surface_free}
\Psi^0_{\mathcal R(L,\delta L)}[ \partial \mathcal R^{\mathrm{bot}}(L, \delta L/2) \centernot\longleftrightarrow \partial \mathcal R^{\mathrm {top}}(L, \delta L/2)] \leq e^{-c_4 L^{d-1}}.
\end{equation} Let $\ell\leq \delta L/2$. Note that we can cover the hyperplane $\{-L,\ldots,L\}^{d-1}\times \{0\}$ by $m\leq C({L}/{\ell})^{d-1}$ boxes $\Lambda_\ell(x_1),\ldots,\Lambda_\ell(x_m)$ with $x_1,\ldots,x_m\in\{-L,\ldots,L\}^{d-1}\times \{0\}$. With such a choice in hands, we have 
\begin{equation}
\bigcap_{i=1}^{m} \{\Lambda_\ell(x_i)\centernot\longleftrightarrow\partial \Lambda_{\delta L/2}(x_i)\} \subset \{\partial^{\rm bot} \cR(L, \delta L/2)\centernot\longleftrightarrow\partial^{\rm top} \cR(L, \delta L/2)\}.
\end{equation}
Therefore, by the FKG inequality together with \eqref{eq:surface_free}, we can find an $i_0\in\{1,\ldots,m\}$ such that 
\begin{equation} \label{eq: diff g}
\Psi^0_{\Lambda_{L},\beta}[\Lambda_\ell(x_{i_0})\centernot\longleftrightarrow\partial \Lambda_{\delta L/2}(x_{i_0})]\leq e^{-c_4 \ell^{d-1}/C}.
\end{equation}
By the monotonicity in the volume we obtain 
\begin{equation}
\Psi^0_{\Lambda_{2L}(x_{i_0}),\beta}[\Lambda_\ell(x_{i_0})\centernot\longleftrightarrow\partial \Lambda_{\delta L/2}(x_{i_0})]\leq \Psi^0_{\Lambda_{L},\beta}[\Lambda_\ell(x_{i_0})\centernot\longleftrightarrow\partial \Lambda_{\delta L/2}(x_{i_0})]\leq e^{-c_4 \ell^{d-1}/C}.
\end{equation}
This implies the desired result with $c_1= \min\{\delta/2, c_4/C\}$.
\end{proof}

\section{Local uniqueness and renormalisation}\label{sec:uniqueness}

In this section, we prove Theorem~\ref{thm: local uniqueness}. We handle the dimensions $d=2$ and $d\geq 3$ separately in Sections~\ref{sec:uniq_d=2} and \ref{sec:uniq_d>2} respectively. 

\subsection{Proof of Theorem~\ref{thm: local uniqueness} for $d=2$}\label{sec:uniq_d=2}

In this section, we prove Theorem~\ref{thm: local uniqueness} in dimension $d=2$. The proof combines the general Russo--Seymour--Welsh theory developed by Köhler-Schindler and Tassion in \cite{KohlerTassionGeneralRSW} with Theorem~\ref{thm:free_surface_tension} to deduce that open circuits lying in dyadic annuli exist with high probability. For topological reasons, this is enough to obtain the desired local uniqueness. 

In order to state the result of \cite{KohlerTassionGeneralRSW} in our context, we first introduce some definitions and notation. For $m,n\geq 1$, define the crossing event 
\begin{equation}
      \mathcal{C}(m,n)=\left\{\text{
        \begin{minipage}{.55\linewidth}\centering
          There exists a path crossing from left to right in $[-m,m]\times[-n,n]$ made of open edges 
        \end{minipage}}\right\}.
\end{equation}

\begin{theorem}\label{thm: RSW}
Let $d=2$. There exists a continuous and increasing bijection $f:[0,1]\rightarrow [0,1]$ such that for every $\beta>0$ and every $k\geq 2$ we have 
\begin{equation}
\Psi^0_{\Lambda_{(k+4)n},\beta}[\mathcal{C}(2 n,n)] \geq f \big(\Psi^0_{\Lambda_{kn},\beta}[\mathcal{C}(n,2n)] \big).
\end{equation}
for every $n\geq 1$.
\end{theorem}
\begin{proof}
The result follows from applying \cite[Theorem~2]{KohlerTassionGeneralRSW} to the measures $\Psi^0_{\Lambda_{kn},\beta}$, $\Psi^0_{\Lambda_{(k+2)n},\beta}$, and $\Psi^0_{\Lambda_{(k+4)n},\beta}$, which satisfy the assumptions of the theorem by the monotonicity of the free measure in the volume.   
\end{proof}

We are now ready to prove Theorem~\ref{thm: local uniqueness} in dimension $d=2$.

\begin{proof}[Proof of Theorem~\textup{\ref{thm: local uniqueness}} when $d=2$] 

Let $d=2$ and $\beta>\beta_c$. By Lemma~\ref{lem: large to free bc}, 
and monotonicity in the volume (see Corollary \ref{cor: monotonicity free measure}), there exist $\delta>0$ such that
\begin{equation}\label{eq:2Dcrossing_delta}
    \lim_{L\to\infty} \Psi^0_{\Lambda_{L},\beta}[\mathcal{C}(\delta L,L)] = 1.
\end{equation}
It is classical to deduce the same statement for $\mathcal{C}(\delta L,2\delta L)$. Theorem~\ref{thm: RSW} does the much harder task of going from $\mathcal{C}(\delta L,2\delta L)$ to $\mathcal{C}(2\delta L,\delta L)$. Then, it is again classical to construct the local uniqueness event out of $\mathcal{C}(2\delta L,\delta L)$, thus concluding the proof.
The details are given below.

First, notice that if the event $\mathcal{C}(\delta L,L)$ holds, then either one of the rectangles of the form $[-\delta L, \delta L]\times[-L+k\delta L,-L+(k+2)\delta L]$ for integers $0\leq k\leq 2/\delta - 1$ is crossed horizontally or one of the squares of the form $[-\delta L ,\delta L]\times[-L+k\delta L,-L+(k+1)\delta L]$ for $0\leq k\leq 2/\delta-1$ is crossed vertically. By the FKG inequality for the complementary events, monotonicity in volume, inclusion of events and the $\pi/2$ rotational symmetry
\begin{equation}
\Psi^0_{\Lambda_{2L},\beta}[\mathcal{C}(\delta L,2\delta L)]\geq 1-\Psi^0_{\Lambda_{2L},\beta}[\mathcal{C}(\delta L,L)^c]^{\delta/4} \underset{L\rightarrow \infty}{\longrightarrow} 1.    
\end{equation}
Then, Theorem~\ref{thm: RSW} gives $\Psi^0_{\Lambda_{2L +4\delta L},\beta}[\mathcal{C}(2\delta L,\delta L)] \to 1$ as $L\to\infty$, which by monotonicity in boundary conditions implies
\begin{equation}
  \lim_{L\rightarrow \infty}\Psi^0_{\Lambda_{2L},\beta}[\mathcal{C}(2\eta L,\eta L)] =1,
\end{equation}
for some constant $\eta>0$. 

Notice that the $U(L)$ holds as soon as there exist a circuit surrounding the annulus $\Lambda_{4L}\setminus \Lambda_{2L}$ and a crossing from $\Lambda_L$ to $\partial \Lambda_{8L}$. In particular, one can easily realise $U(L)$ as the intersection of a bounded number (depending on $\eta$) of translates and $\pi/2$ rotations of the event $\mathcal{C}(2\eta L,\eta L)$ within the box $\Lambda_{8L}$. Therefore, by monotonicity and a union bound, the previous display readily implies that $\Psi^0_{\Lambda_{10L},\beta}[U(L)] \to 1$ as $L\to\infty$, as we wanted to prove.
\end{proof}

\subsection{Proof of Theorem~\textup{\ref{thm: local uniqueness}} for $d\geq 3$}\label{sec:uniq_d>2}

In this section, we prove Theorem~\textup{\ref{thm: local uniqueness}} for $d\geq 3$ following the approach of \cite{Severo24}. The proof is divided into three steps, done in separate subsections. First, we use Theorem~\ref{thm:free_surface_tension} to show that all macroscopic clusters merge into a unique cluster after sprinkling $\Psi^0_{\Lambda,\beta}$ by an independent Bernoulli bond percolation. We then show that increasing the value of $\beta$ has a stronger effect than this independent Bernoulli sprinkling of $\Psi^0_{\Lambda,\beta}$. This allows us to deduce that the model percolates in sufficiently thick 2D slabs with minimal (free) boundary conditions. Finally, we use the slab technology to obtain local uniqueness without sprinkling via an onion peeling argument.

\subsubsection{Local uniqueness with sprinkling}

We consider a measure obtained from $\Psi^0_{\Lambda,\beta}$ by sprinkling an independent Bernoulli bond percolation. We will show that in $d\geq 3$, all crossing clusters of $\Psi^0_{\Lambda,\beta}$ merge into one cluster after sprinkling. 

Let us first introduce some notation. Given $\beta>0$, $\varepsilon>0$, a finite set $\Lambda\subset \Z^d$ and boundary conditions $(\xi,\mathsf{b})$ on $\Lambda$, 
let $(\omega,\mathsf{a}) \sim_d \Psi^{(\xi,\mathsf{b})}_{\Lambda,\beta}$ and $\gamma \sim_d \mathbb P^{\rm Ber}_{\Lambda,\varepsilon}$ be independent random variables, where 
$\mathbb P^{\rm Ber}_{\Lambda,\varepsilon}=\otimes_{e\in \overline{E}(\Lambda)} \textrm{Ber}(\varepsilon)$ denotes the law of independent identically distributed Bernoulli random variables of parameter $\varepsilon$. Denote $\omega\cup \gamma$ the configuration on $\overline{E}(\Lambda)$ where $(\omega\cup \gamma)_e=1$ if and only if $\omega_e=1$ or $\gamma_e=1$. Define
${\boldsymbol \Gamma}^{(\xi,\mathsf{b})}_{\Lambda,\beta,\varepsilon}:=\Psi^{(\xi,\mathsf{b})}_{\Lambda,\beta}\otimes \mathbb P^{\rm Ber}_{\Lambda,\varepsilon}$ and let $\Gamma^{(\xi,\mathsf{b})}_{\Lambda,\beta,\varepsilon}$ be the pushforward of ${\boldsymbol \Gamma}^{(\xi,\mathsf{b})}_{\Lambda,\beta,\varepsilon}$ under the mapping $(\omega,\eta)\mapsto \omega\cup \eta$.

Let $\mathrm{Unique}(L)$ be the event that there is a cluster in $\omega\cap \Lambda_{L}$ crossing $\Lambda_L\setminus \Lambda_{L/8}$ and that every cluster of $\omega\cap \Lambda_L$ crossing $\Lambda_{L/2}\setminus \Lambda_{L/4}$ are connected to each other in $(\omega\cup\gamma)\cap \Lambda_{L/2}$. 
Our aim is to show that $\mathrm{Unique}(L)$ happens with probability tending to $1$ as $L$ tends to infinity, uniformly over all possible boundary conditions $(\xi,\mathsf{b})$ at a macroscopic distance from $\Lambda_L$.
The proof adapts the work of Benjamini \& Tassion \cite{BenjaminiTassion}, which concerns the connectedness of ``everywhere percolating subsets'' of $\mathbb{Z}^d$ after an $\varepsilon$-Bernoulli sprinkling, to the setting of very dense subgraphs. This extension appears in \cite{Severo24}.

To this end, we first introduce the following notation.
Let $c_1>0$ be the constant of Theorem~\ref{thm:free_surface_tension}, and set $\delta:=c_1/4$, $C_0:=(dc_1)^{-\frac{1}{d-1}}$, and 
\begin{equation}
\ell=\ell(L,\beta):=C_0 (\log{L})^{\frac{1}{d-1}}.    
\end{equation}
Define the event
\begin{equation}
\mathcal{A}_L\coloneqq \bigcap_{x\in \ell\Z^d\cap\Lambda_{\delta L}} \{\Lambda_\ell(x) \overset{\omega}{\longleftrightarrow} \partial \Lambda_{\delta L}\}.
\end{equation}

\begin{proposition}\label{prop: cond uniqueness}
Let $d\geq3$. For every $\beta>\beta_c$ and every $\varepsilon>0$, we have 
\begin{equation}
\lim_{L\to\infty} \inf_{(\xi,\mathsf{b})} \inf_{\theta \in \mathcal{A}_L} {\boldsymbol \Gamma}^{(\xi,\mathsf{b})}_{\Lambda_L,\beta,\varepsilon}[\textup{Unique}(\delta L) \mid \omega=\theta]=1.    
\end{equation}
\end{proposition}
\begin{proof}
The proof of this result is given in \cite[Proposition 2.2]{Severo24}. There the result is shown for the standard random cluster measure sprinkled by Bernoulli percolation, but as the desired result is a statement purely about $\gamma$, the same proof applies mutatis mutandis. We note that the proof in \cite[Proposition 2.2]{Severo24} crucially relies on the fact that $\ell=o(\log L)$, this being the only place where $d\geq3$ is used.
\end{proof}

With this result in hands we can now deduce that $\mathrm{Unique}(\delta L)$ happens with probability tending to $1$ under $\boldsymbol{\Gamma}^{(\xi,\mathsf{b})}_{\Lambda_L,\beta,\varepsilon}$.

\begin{proposition}\label{prop: uniqueness with sprinkling}
Let $d\geq 3$. For every $\beta>\beta_c$ and every $\varepsilon>0$ we have 
\begin{equation}
\lim_{L\to\infty} \inf_{(\xi,\mathsf{b})} \boldsymbol{\Gamma}^{(\xi,\mathsf{b})}_{\Lambda_L,\beta,\varepsilon}[\textup{Unique}(\delta L) ]=1,   
\end{equation}
where $\delta=c_1/4$ and $c_1$ is given by Theorem \textup{\ref{thm:free_surface_tension}}.
\end{proposition}

\begin{proof}
Let $L'\geq 1$ and set $L=2L'$. By the monotonicity on boundary conditions and Theorem~\ref{thm:free_surface_tension}, we obtain 
\begin{equation}\boldsymbol{\Gamma}^{(\xi,\mathsf{b})}_{\Lambda_L,\beta,\varepsilon}[\Lambda_\ell(x) \overset{\omega}{\longleftrightarrow} \partial \Lambda_{\delta L}] \geq \Psi^0_{\Lambda_{L'}}[\Lambda_\ell \longleftrightarrow \partial \Lambda_{c_1 L'}]\geq 1-L^{-d}
\end{equation}
for all $x\in\Lambda_{\delta L}$ and all $(\xi,\mathsf{b})$. 
By the union bound, we conclude that 
\begin{equation}
\lim_{L\to\infty} \inf_{(\xi,\mathsf{b})} \boldsymbol{\Gamma}^{(\xi,\mathsf{b})}_{\Lambda_L,\beta,\varepsilon}[\mathcal{A}_L]=1.
\end{equation}
The desired result follows now from Proposition~\ref{prop: cond uniqueness}.
\end{proof}

\subsubsection{Slab percolation}

Recall that $d\geq 3$. Consider the ``slab box'' of size $N$ and thickness $L$ given by
\begin{equation}
    \mathcal{S}(L,N)\coloneqq \{-L,\ldots,L\}^{d-2}\times \{-N,\ldots,N\}^2,
\end{equation}
and the corresponding ``slab'' of thickness $L$ given by 
\begin{equation}
    \mathcal{S}(L)\coloneqq \{-L,\ldots,L\}^{d-2}\times \mathbb{Z}^2.
\end{equation}
We will show that for every $\beta>\beta_c$, there exists a sufficiently large $L$ such that $\Psi^0_{\mathcal{S}(L),\beta}$ percolates. We will first show that this is indeed the case once we sprinkle $\Psi^0_{\mathcal{S}(L),\beta}$ with a Bernoulli percolation of an arbitrarily small parameter $\varepsilon>0$. Then, we will show that the edge marginal of the latter sprinkled measure is stochastically dominated by the edge marginal $\Phi^0_{\mathcal{S}(L),\beta'}$ of the measure $\Psi^0_{\mathcal{S}(L),\beta'}$ with a slightly higher inverse temperature $\beta'$, and deduce that $\Psi^0_{\mathcal{S}(L),\beta'}$ percolates. It is important for the next section to state our results in terms of finite-volume measures defined on the slab box $\mathcal{S}(L,N)$. In what follows, given a vertex $x$, we write $\mathrm{Unique}_x(L)$ for the translation of the event $\mathrm{Unique}(L)$ by $x$. We now implement the aforementioned strategy. We first prove the following result. 

\begin{proposition}\label{prop: slab sprinkling}
Let $d\geq 3$. For every $\beta>\beta_c$ and $\varepsilon>0$, there exists $L\geq 1$ such that
\begin{equation*}
        \inf_{N\geq 0}\inf_{x\in \mathcal{S}(L,N)} \boldsymbol{\Gamma}^0_{\mathcal{S}(L,N),\beta,\varepsilon} [0\overset{\omega\cup \gamma\:}{\longleftrightarrow} x]>0.
\end{equation*}
\end{proposition}

\begin{proof}
Since we always have $\boldsymbol{\Gamma}^0_{\mathcal{S}(L,N),\beta,\varepsilon} [0\overset{\omega\cup \gamma\:}{\longleftrightarrow} x]>0$, it suffices to show that the infimum over $N\geq L$ is strictly positive.
Let $c_1>0$ be the constant of Theorem~\ref{thm:free_surface_tension}, and recall that Proposition~\ref{prop: uniqueness with sprinkling} holds with $\delta=c_1/4$. Write $n\coloneqq \lfloor \frac{8(N-L)}{\delta L}\rfloor$, and consider the vertices of the form $x(u)=\left(0_\bot,\frac{\delta}{8}Lu\right)$ for $u\in B_n\coloneqq \{-n,\ldots,n\}^2$ (where $0_\bot$ is the null element of $\mathbb Z^{d-2}$). Note that with this choice, we have $\Lambda_L(x(u))\subset \mathcal{S}(L,N)$ for all $u\in B_n$, and if $|u-v|> 16/\delta$, then $\Lambda_L(x(u))\cap \Lambda_L(x(v))=\emptyset$. We now define a site percolation configuration $\eta$ on $B_n$ by setting $\eta_u=1$ if and only if the event $\mathrm{Unique}_{x(u)}(\delta L)$ happens. First, notice that, by the definition of the event $\mathrm{Unique}(\delta L)$, any path in $\eta$ starting at some vertex $u$ and ending at some vertex $v$ induces a path in $\omega\cup\gamma$ from $\Lambda_{\delta L/8}(x(u))$ to $\Lambda_{\delta L/8}(x(v))$. Second, by the Markov property, we have
\begin{equation}
\boldsymbol{\Gamma}^0_{\mathcal{S}(L,N),\beta,\varepsilon}[\eta_u=1 \mid (\eta_v:\, |u-v|> 16/\delta)] \geq \alpha(L) ~~~\text{a.s. } \forall u\in B_n,
\end{equation}
where $\alpha(L)\coloneqq \inf_{(\xi,\mathsf{b})} \boldsymbol{\Gamma}^{(\xi,\mathsf{b})}_{\mathcal{S}(L,L),\beta,\varepsilon}[\mathrm{Unique}(\delta L)]$, and the infimum is taken over boundary conditions on $\mathcal{S}(L,L)=\Lambda_L$. Since $\lim_{L\to\infty} \alpha(L)= 1$ by Proposition~\ref{prop: uniqueness with sprinkling}, the main result of \cite{LSS97} implies that there exists an $L\geq 1$ such that $\eta$ stochastically dominates a Bernoulli site percolation $\mathbb{P}_s^{\textup{site}}$ on $B_n$ with parameter $s>p_c^{\textup{site}}(\Z^2)$. Let us fix such an $L$. Since there exists a constant $c>0$ such that $\mathbb{P}_s^{\textup{site}}[u\longleftrightarrow v]\geq c$ for all $u,v\in B_n$ and all $N$, it follows that 
\begin{equation}
\inf_{N\geq L}\inf_{u,v\in B_n}\boldsymbol{\Gamma}^0_{\mathcal{S}(L,N),\beta,\varepsilon}[u \overset{\eta}{\longleftrightarrow} v]\geq c.
\end{equation}
Now, given any point $x\in\mathcal{S}(L,N)$, there exists $u\in B_n$ such that $x\in \Lambda_{2L}(x(u))$, so that the event $\{0 \overset{\eta}{\longleftrightarrow} u\}\cap \{\Lambda_L \text{ fully open in } \gamma\} \cap \{\Lambda_{2L}(x(u))\cap\mathcal{S}(L,N) \text{ fully open in } \gamma\}$ is contained in $\{0 \overset{\omega\cup\gamma\:}{\longleftrightarrow} x\}$. As a conclusion, conditioning on $\omega$ and using the FKG inequality for $\gamma$ (note that $\eta$ is increasing in $\gamma$) we obtain the desired uniform lower bound $\boldsymbol{\Gamma}^0_{\mathcal{S}(L,N),\beta,\varepsilon}[0 \overset{\omega\cup\gamma\:}{\longleftrightarrow} x]\geq c'\coloneqq c \varepsilon^{CL^d}>0$ for every $x\in \mathcal{S}(L,N)$ and every $N\geq L$.
\end{proof}

Recall from Definition \ref{def: random cluster phi4} that, for $\Lambda \subset \mathbb Z^d$ and $\beta > 0$, the probability measure $\Phi^0_{\Lambda,\beta}$ on $E(\Lambda)$ is the edge-marginal of the free $\varphi^4$ random cluster measure $\Psi^0_{\Lambda,\beta}$.
Let $\beta>0$ and $\beta' \in (0,\beta)$. The main technical result of this section is the proof that $\Phi^0_{\Lambda,\beta}$ stochastically dominates the superposition of $\Phi^0_{\Lambda,\beta'}$ and an independent Bernoulli bond percolation $\mathbb P^{\rm Ber}_{\Lambda,\varepsilon}$ of parameter $\varepsilon=\varepsilon(\beta',\beta)>0$ uniformly in $\Lambda$. This is formalised in the upcoming proposition. 

\begin{proposition}\label{prop: coupling}
Let $0< \beta'<\beta$. There exists $\varepsilon=\varepsilon(\beta',\beta)>0$ such that for every $\Lambda\subset \mathbb{Z}^d$ finite, 
\begin{equation}
\Phi^0_{\Lambda,\beta}\succcurlyeq \Gamma^{0}_{\Lambda,\beta',\varepsilon}. 
\end{equation}
\end{proposition}

The proof of Proposition~\ref{prop: coupling} is rather technical and we defer it to Section~\ref{sec: proof of coupling}. We are now ready to prove the main result of this section.

\begin{corollary}\label{cor: slab percolation without sprinkling}
Let $d\geq 3$. For every $\beta>\beta_c$, there exists $L\geq 1$ such that
\begin{equation}\label{eq: slab percolation without sprinkling}
        \inf_{N}\inf_{x\in \mathcal{S}(L,N)} \Psi^0_{\mathcal{S}(L,N),\beta} [0\longleftrightarrow x]>0.
\end{equation}
\end{corollary}
\begin{proof}
This follows from applying Proposition~\ref{prop: slab sprinkling} for some $\beta'\in (\beta_c,\beta)$ and $\varepsilon>0$ sufficiently small so that the stochastic domination of Proposition~\ref{prop: coupling} holds.    
\end{proof}



\subsubsection{Uniqueness without sprinkling}

We conclude the proof of Theorem~\ref{thm: local uniqueness} for $d\geq3$ by showing that the slab percolation statement of Corollary~\ref{cor: slab percolation without sprinkling} implies that the uniqueness (without sprinkling) event $U(L)$ happens with high probability, uniformly in boundary conditions. Our proof is based on a classical onion peeling argument, see e.g.~\cite[Lemma 7.89]{GrimmettPercolation1999}.

\begin{proof}[Proof of Theorem~\textup{\ref{thm: local uniqueness}} in $d\geq 3$]
Let $\beta>\beta_c$. It follows from Propositions~\ref{prop: uniqueness with sprinkling} and \ref{prop: coupling}, and a gluing argument as in the proof of Theorem~\textup{\ref{thm: local uniqueness}} in $d=2$ that 
\begin{equation}
\lim_{L\to\infty}\inf_{(\xi,\mathsf{b})}\Psi^{(\xi,\mathsf{b})}_{\Lambda_{10L},\beta}[\Lambda_L \longleftrightarrow \Lambda_{8L}]= 1.
\end{equation}
It remains to show that with high probability, there is at most one cluster crossing $\Lambda_{4L}\setminus \Lambda_{2L}$. Consider an $L'\geq 1$ large enough so that Corollary \ref{cor: slab percolation without sprinkling} holds, and let $n=2L'+1$. 
First, we partition $\Lambda_{4L}\setminus \Lambda_{2L}$ into disjoint annuli of the form $A(k,n)\coloneqq\Lambda_{(k+1)n-1}\setminus\Lambda_{kn-1}$ for $\lfloor\frac{2L}{n}\rfloor\leq k\leq \lceil\frac{4L-n}{n}\rceil$. By the domain Markov property, the FKG inequality and the monotonicity in boundary conditions, for any $x,y\in \partial \Lambda_{2L}$ and  $\lfloor\frac{2L}{n}\rfloor\leq k\leq \lceil\frac{4L-n}{n}\rceil$ we have
\begin{align*}
\Psi^{(\xi,\mathsf{b})}_{\Lambda_{10L},\beta}&[x\longleftrightarrow y \text{ in } \Lambda_{(k+1)n-1} \mid x\longleftrightarrow \partial \Lambda_{kn}, y \longleftrightarrow \partial \Lambda_{kn},x\centernot\longleftrightarrow y \text{ in } \Lambda_{kn-1}]
\\ &\geq \inf_{u',v'\in \partial \Lambda_{kn-1}}\inf_{(\xi',\mathsf{b}')}\Psi^{(\xi',\mathsf{b}')}_{A(k,n),\beta}[u'\longleftrightarrow v' \text{ in } \Lambda_{(k+1)n-1} \mid u'\longleftrightarrow \partial \Lambda_{kn}, v' \longleftrightarrow \partial \Lambda_{kn}]
\\ &\geq 
\inf_{u,v\in A(k,n)} \Psi^{0}_{A(k,n),\beta}[u\longleftrightarrow v],  
\end{align*}
where in the connectivity event of the first inequality we took into account the fact that $\Psi^{(\xi',\mathsf{b}')}_{A(k,n),\beta}$ is a measure on $\overline{E}(A(k,n))\times (\mathbb R^+)^{A(k,n)}$ (see Section \ref{sec:random_cluster}).
By Corollary~\ref{cor: slab percolation without sprinkling} and the FKG inequality, there exists a constant $\delta>0$ such that $\Psi^{0}_{A(k,n),\beta}[u\longleftrightarrow v]\geq \delta$ for any $u,v\in A(k,n)$ and any $k$ as above. Iterating this bound we see that
\begin{equation}
\Psi^{(\xi,\mathsf{b})}_{\Lambda_{10L},\beta}[x\centernot\longleftrightarrow y \mid x\longleftrightarrow \partial \Lambda_{4L}, y \longleftrightarrow \partial \Lambda_{4L}]\leq (1-\delta)^{2L/n}.  
\end{equation}
The desired result follows from the union bound over all $x,y\in \partial \Lambda_{2L}$. 
\end{proof}

\subsection{Coupling with a sprinkled measure}\label{sec: proof of coupling}

In this section we prove Proposition~\ref{prop: coupling}. As a first step towards the proof, we show that $\Phi^0_{\Lambda,\beta}$ stochastically dominates the superposition of $\Phi^0_{\Lambda,\beta'}$ and a Bernoulli bond percolation with inhomogeneous \emph{random} edge parameter.
Heuristically, conditioning $\Psi^0_{\Lambda,\beta}$ on $\mathsf a$ gives the usual random cluster measure, but with inhomogeneous edge-weights $p(\mathsf a,\beta)$ defined in \eqref{def: random weight rc}.
For these measures, an analogous stochastic domination holds, allowing for the (inhomogeneous) parameter of the independent sprinkling to depend on $\mathsf a$ by direct calculation. A disintegration argument will then yield the result. Let us now define the latter measure. For every $\mathsf{a}\in (\mathbb{R}^+)^{\Lambda}$ and $xy\in E(\Lambda)$, write
\begin{equation}
r(\beta,\mathsf{a})_{xy}=\frac{1-e^{-2\beta\mathsf{a}_x\mathsf{a}_y}}{1+e^{-2\beta\mathsf{a}_x\mathsf{a}_y}}.
\end{equation}
Denote $\Psi^0_{\Lambda,\beta',\beta}$ the measure on $\{0,1\}^{E(\Lambda)} \times (\R^+)^{\Lambda}$
defined by
\begin{equation}
\mathrm{d}\Psi_{\Lambda,\beta',\beta}^0[(\omega,\mathsf{a})]=\gamma^0_{\Lambda,\beta',\beta,\mathsf{a}}[\omega] \mathrm{d}\mu_{\Lambda,\beta'}^0(\mathsf{a}),
\end{equation}
where $\gamma^0_{\Lambda,\beta',\beta,\mathsf{a}}$ is the superposition of $\boldsymbol{\phi}^0_{\Lambda,\beta', \mathsf a}$ and an independent Bernoulli bond percolation on $E(\Lambda)$ with law $\mathbb{P}_{\Lambda,\beta',\beta,\mathsf{a}}^{\rm Ber}$ defined by
\begin{equation}
\mathbb{P}_{\Lambda,\beta',\beta,\mathsf{a}}^{\rm Ber}[\omega]=\prod_{xy \in E(\Lambda)} r(\beta-\beta',\mathsf{a})_{xy}^{\omega_{xy}} (1-r(\beta-\beta',\mathsf{a}))_{xy}^{1-\omega_{xy}}.     
\end{equation}
Finally, write $\Gamma^0_{\Lambda,\beta',\beta}$ for the projection of $\Psi^0_{\Lambda,\beta',\beta}$ on to $\{0,1\}^{E(\Lambda)}$.

In order to state the intermediate stochastic domination result, we will need the following notation. Given $\Lambda\subset \mathbb{Z}^d$, let $G=(\Lambda,E^{1,2})$ be the graph with vertex set $\Lambda$, where for each edge $xy\in E(\Lambda)$, $G$ contains two parallel edges $xy(1)$ and $xy(2)$. We write $E^i=\{xy(i) : xy\in E(\Lambda)\}$. 

\begin{lemma}\label{lem: dom random intensity}
Let $0<\beta'<\beta$. For every $\Lambda\subset \mathbb{Z}^d$ finite, $\Phi^0_{\Lambda,\beta} \succcurlyeq \Gamma^0_{\Lambda,\beta',\beta}$.
\end{lemma}

\begin{proof}
Let $\mathsf{a}\sim \mu^0_{\Lambda,\beta'}$. We write
$\boldsymbol{\phi}^0_{G,\beta',\beta,\mathsf{a}}$ for the random cluster measure on $G$ defined by
\begin{equation}
\boldsymbol{\phi}^0_{G,\beta',\beta,\mathsf{a}}[\omega]
=
\frac{1}{Z^0_{\Lambda,\beta',\beta, \mathsf{a}}}\prod_{xy \in E(\Lambda)} \Big(\frac{p(\beta',\mathsf{a})_{xy}}{1-p(\beta',\mathsf{a})_{xy}}\Big)^{\omega_{xy(1)}}  \Big(\frac{p(\beta-\beta',\mathsf{a})_{xy}}{1-p(\beta-\beta',\mathsf{a})_{xy}}\Big)^{\omega_{xy(2)}}  \, 2^{k^0(\omega)}.
\end{equation}
Consider also the product measure $\boldsymbol{\phi}^0_{\Lambda,\beta',\mathsf{a}}\otimes \mathbb{P}^{\rm Ber}_{\Lambda,\beta',\beta,\mathsf{a}}$, which we view as a measure on $\{0,1\}^{E^{1,2}}$ in the natural way. 

We first show that $\boldsymbol{\phi}^0_{G,\beta',\beta,\mathsf{a}}$ stochastically dominates $\boldsymbol{\phi}_{\Lambda,\beta',\mathsf{a}}\otimes \mathbb{P}^{\rm Ber}_{\Lambda,\beta',\beta,\mathsf{a}}$. Consider configurations $\theta,\theta'\in \{0,1\}^{E^{1,2}}$ such that $\theta\geq \theta'$. A direct calculation gives
\begin{equation}
\boldsymbol\phi^{0}_{G,\beta',\beta,\mathsf{a}}[\omega_{xy(1)}=1 \mid \omega|_{E^{1,2}\setminus\{xy(1)\}}=\theta|_{E^{1,2}\setminus\{xy(1)\}}]=\begin{cases} p(\beta',\mathsf{a})_{xy} &\text{ if $x\overset{\theta|_{E^{1,2}\setminus\{xy(1)\}}}{\longleftrightarrow} y$,}\\
\displaystyle r(\beta',\mathsf{a})_{xy}&\text{ otherwise,}\end{cases}
\end{equation}
and 
\begin{equation}
\boldsymbol\phi^{0}_{G,\beta',\beta,\mathsf{a}}[\omega_{xy(2)}=1 \mid \omega|_{E^{1,2}\setminus\{xy(2)\}}=\theta|_{E^{1,2}\setminus\{xy(2)\}}]=\begin{cases} p(\beta-\beta',\mathsf{a})_{xy} &\text{ if $x\overset{\theta|_{E^{1,2}\setminus\{xy(2)\}}}{\longleftrightarrow} y$,}\\
\displaystyle r(\beta-\beta',\mathsf{a})_{xy}&\text{ otherwise.}\end{cases}
\end{equation}
A similar calculation combined with the fact that $p(t,\mathsf{a})\geq r(t,\mathsf{a})$ for $t\geq 0$ shows that $\boldsymbol\phi^{0}_{G,\beta',\beta,\mathsf{a}}[\omega_{xy(1)}=1 \mid \omega|_{E^{1,2}\setminus\{xy(1)\}}=\theta|_{E^{1,2}\setminus\{xy(1)\}}]$ is at least $\boldsymbol\phi^{0}_{\Lambda,\beta',\mathsf{a}}\otimes \mathbb{P}^{\rm Ber}_{\Lambda,\beta',\beta.\mathsf{a}}[\omega_{xy(1)}=1 \mid \omega|_{E^{1,2}\setminus\{xy(1)\}}=\theta'|_{E^{1,2}\setminus\{xy(1)\}}]$. A standard Markov chain argument then implies that $\boldsymbol{\phi}^0_{G,\beta',\beta,\mathsf{a}}$ stochastically dominates $\boldsymbol{\phi}_{\Lambda,\beta',\mathsf{a}}\otimes \mathbb{P}_{\Lambda,\beta',\beta,\mathsf{a}}$. See e.g.\ \cite[Lemma 1.5]{DuminilLecturesOnIsingandPottsModels2019}.

Let $\omega^{\max} \in \{0,1\}^{E(\Lambda)}$ be the configuration defined by
$\omega^{\max}_{xy}=\max\{\omega_{xy(1)},\omega_{xy(2)}\}$. Notice that if $\omega\sim \boldsymbol{\phi}^0_{G,\beta',\beta,\mathsf{a}}$, then $\omega^{\max}$ is distributed according to $\boldsymbol{\phi}^0_{\Lambda,\beta,\mathsf{a}}$.
The desired result now follows from the fact that $\mu_{\Lambda,\beta}^0$ stochastically dominates $\mu_{\Lambda,\beta'}^0$ and the monotonicity of the random cluster and Bernoulli percolation measures in the edge parameter.
\end{proof}

Note that the parameter $r(\beta-\beta',\mathsf{a})_{xy}$ is close to $0$ when $\mathsf{a}_x\mathsf{a}_y$ is close to $0$. Thanks to Lemma \ref{lem: dom random intensity}, in order to prove Proposition~\ref{prop: coupling} it will suffice to show that $\mathsf{a}_x\mathsf{a}_y$ remains bounded away from $0$ with good $\mu^0_{\Lambda,\beta'}$-probability when we condition on the state of all edges. We will begin by showing that, conditionally on the state of all edges, $\mathsf{a}_x$ remains bounded with good probability for every vertex $x$. For technical reasons that will become clear later, we will work with a slightly more general family of measures, where we ignore the contribution coming from certain closed edges.

As above, let us fix $0<\beta'<\beta$ and $\Lambda \subset \mathbb Z^d$ finite. Let $A,B \subset E(\Lambda)$, and  $\omega^1,\omega^2 \in \{0,1\}^{E(\Lambda)}$. Denote by $\boldsymbol \mu=\boldsymbol \mu(\Lambda,\beta',\beta,A,B,\omega^1,\omega^2)$ the probability measure on $(\mathbb R^+)^\Lambda$ defined by the Radon--Nikodym derivative
\begin{equation}
d\boldsymbol\mu(\mathsf a) = \frac{1}{\boldsymbol Z}\boldsymbol \phi^0_{\Lambda,\beta',\mathsf a}[\omega^1] \mathbb P^{\rm Ber}_{\Lambda,\beta',\beta,\mathsf a}[\omega^2|_{E(\Lambda)\setminus B}] \prod_{xy \in A}e^{\beta' \mathsf a_x \mathsf a_y} \mathrm{d}\mu^0_{\Lambda,\beta'}(\mathsf a),
\end{equation}
where $\boldsymbol Z$ is a normalisation constant and we recall from \eqref{eq: expression dmu section 6} that
\begin{equation}
\mathrm{d}\mu^0_{\Lambda,\beta'}(\mathsf{a})=\frac{\mathbf{Z}^{0}_{\Lambda,\beta',\mathsf{a}} }{Z^0_{\Lambda,\beta'}} \prod_{xy\in E(\Lambda)} e^{-\beta' \mathsf{a}_x \mathsf{a}_y} \prod_{x\in \Lambda} \mathrm{d}\rho_{g,a}(\mathsf{a}_x).
\end{equation}
Note that when $A=B=\emptyset$, then $\boldsymbol \mu$ is the projection of the conditional measure $\Psi^0_{G,\beta',\beta}[\: \cdot \mid \omega|_{E^1}=\omega^1, \omega|_{E^2}=\omega^2]$ to $(\mathbb{R}^+)^{\Lambda}$, where $G=(\Lambda, E^{1,2})$ is as above and 
by abuse of notation we write $\Psi^0_{\Lambda,\beta',\beta}$ to denote its natural lift to the multigraph $G$. 

\begin{lemma}\label{lem: cond regularity}
Let $0<\beta'<\beta$. There exists a constant $C=C(\beta',\beta,d)>1$ such that for every $\Lambda\subset \mathbb{Z}^d$ finite, every $A,B\subset E(\Lambda)$, every pair of configurations $\omega_1, \omega_2 \in \{0,1\}^{E(\Lambda)}$, and every $x\in \Lambda$ we have
\begin{equation}
\boldsymbol \mu [ \mathsf a_x \geq s] \leq C \int_s^\infty e^{-\frac g2 t^4 - at^2} \mathrm{d}t, \qquad \forall s \geq C. 
\end{equation}
\end{lemma}

\begin{proof}
Fix a vertex $x\in \Lambda$ and let $C>1$ be a sufficiently large constant to be determined. For $s\geq C^3$, let $E_s=E_s(x)$ be the event that $\mathsf{a}_y\leq \max\{s^{d(x,y)+1}, C\max_{z\sim y}\{\mathsf{a}_z^{1/3}\}\}$ for every $y \in \Lambda$.
We claim that on the event $E_s$, $\mathsf{a}_x\leq s$. Indeed, consider a configuration $\mathsf a\in E_s$, and let $y\in \Lambda$ be such that $\mathsf a_y/s^{d(x,y)+1}$ is maximal. Then by definition of the event $E_s$, 
\begin{equation}
\begin{aligned}
M:=\frac{\mathsf{a}_y}{s^{d(x,y)+1}}&\leq \max\left\{1, \frac{C}{s^{2(d(x,y)+1)/3}}\max_{z\sim y}\left\{\left(\frac{\mathsf{a}_z}{s^{d(x,y)+1}}\right)^{1/3}\right\}\right\}\\ 
&\leq \max\left\{1, \frac{Cs^{1/3}}{s^{2(d(x,y)+1)/3}} \max_{z\sim y}\left\{\left(\frac{\mathsf{a}_z}{s^{d(x,z)+1}}\right)^{1/3}\right\}\right\}\\
&\leq \max\left\{1, \frac{Cs^{1/3}}{s^{2(d(x,y)+1)/3}} M^{1/3}\right\}\\
&\leq \max\{1, M^{1/3}\}.
\end{aligned}
\end{equation}
Now either $M\leq 1$ or $M\leq M^{1/3}$. In both cases, $M\leq 1$. This implies in particular that $\mathsf{a}_x\leq s$, as claimed.

Thus, it suffices to estimate the probability of $E_s$ happening for some constant $C>1$ large enough. We will prove this by comparing the value of the density $\mathrm{d}\boldsymbol{\mu}(\mathsf{a})$ of $\boldsymbol{\mu}$ for different values of $\mathsf{a}$. To this end,
for each edge $uv \in E(\Lambda)$, let
\begin{equation}\label{eq: a definition}
a_{uv}(\mathsf{a}_u,\mathsf{a}_v)=
\begin{cases}
e^{-\beta' \mathsf{a}_u \mathsf{a}_v}\Big(\frac{p(\beta',\mathsf{a})_{uv}}{1-p(\beta',\mathsf{a})_{uv}}\Big)^{\omega^1_{uv}}, & uv\not\in A\\
1, & uv\in A
\end{cases}
\end{equation}
and
\begin{equation}\label{eq: b definition}
b_{uv}(\mathsf{a}_u,\mathsf{a}_v)=
\begin{cases}
r(\beta-\beta',\mathsf{a})^{\omega^2_{uv}}_{uv}\left(1-r(\beta-\beta',\mathsf{a})_{uv}\right)^{1-\omega^2_{uv}}, & uv\not\in B\\
1, & uv\in B,
\end{cases}
\end{equation}
where $r(\beta-\beta',\mathsf{a})$ is as above.  We first compare the values of $a_{uv}(\mathsf{a}_u,\mathsf{a}_v)$ and $b_{uv}(\mathsf{a}_u,\mathsf{a}_v)$ when the first argument $\mathsf{a}_u\in [1,2]$ and when $\mathsf{a}_u$ is large, where in both instances we consider $\mathsf a_v$ fixed. Let $f(t)=\log(1-e^{-t})$ and note that $f(2\beta'\mathsf a_u \mathsf a_v)=\log p(\mathsf a, \beta')_{uv}$.

For $t>0$,
\begin{equation}
f'(t)=\frac{e^{-t}}{1-e^{-t}}=\frac{1}{e^t-1}\leq \frac{1}{t},
\end{equation}
since $e^t\geq 1+t$. By integrating we obtain that for $\mathsf{a}_u\geq 2$ and $t\in [1,2]$
\begin{equation}
f(2\beta' \mathsf{a}_u \mathsf{a}_v)-f(2\beta' t \mathsf{a}_v)\leq \log(2\beta' \mathsf{a}_u \mathsf{a}_v)-\log(2\beta' t \mathsf{a}_v)\leq \log(\mathsf{a}_u).
\end{equation}
Hence $p(\beta',\mathsf a)_{uv} \leq \mathsf a_u(1- e^{-2\beta' t \mathsf a_v})$.
By taking cases according to whether $uv \in A$ or not, and whether $\omega^1_{uv}=1$ or not, we see that in any case have
\begin{equation}\label{eq: a estimate}
a_{uv}(\mathsf{a}_u,\mathsf{a}_v)\leq e^{\beta'\mathsf{a}_u \mathsf{a}_v} \mathsf{a}_u a_{uv}(t,\mathsf{a}_v)   
\end{equation} 
for every $\mathsf{a}_u\geq 2$ and $t\in [1,2]$.

We now turn to $b_{uv}(\mathsf{a}_u,\mathsf{a}_v)$. Let $g(t)=\log(1+e^{-t})$ and note that $\log r(\beta-\beta',\mathsf{a})=f(2(\beta-\beta')\mathsf{a}_u \mathsf{a}_v)-g(2(\beta-\beta')\mathsf{a}_u \mathsf{a}_v)$.
For $t>0$ we have
\begin{equation}
g'(t)=-\frac{e^{-t}}{1+e^{-t}}=-\frac{1}{e^t+1}\geq -\frac{1}{t}.   
\end{equation}
Hence for every $\mathsf{a}_u\geq 2$ and $t\in [1,2]$ we have 
$\exp g(2(\beta-\beta')\mathsf{a}_u \mathsf{a}_v)\geq \mathsf{a}_u^{-1} \exp g(2(\beta-\beta')t \mathsf{a}_v)$, and hence $r(\beta-\beta',\mathsf{a}) \leq \mathsf a_u^2 \frac{1-e^{-2\beta' t \mathsf a_v}}{1+e^{-2\beta' t \mathsf a_v}}.$
Thus, by examining each case separately, we have
\begin{equation}\label{eq: b estimate}
b_{uv}(\mathsf{a}_u,\mathsf{a}_v)\leq \mathsf{a}_u^2 b_{uv}(t,\mathsf{a}_v).    
\end{equation}
Now if $\mathsf{a}_u\geq C\mathsf{a}_v^{1/3}$ for $C=\sqrt[3]{\frac{8\beta' d}{g}}$, then
\begin{equation}\label{eq: beta estimate}
e^{\beta' \mathsf{a}_u \mathsf{a}_v}\leq e^{\frac{g}{8d}\mathsf{a}_u^4}.    
\end{equation}
By increasing the value of $C$ if necessary, we can further ensure that 
\begin{equation}\label{eq: third power estimate}
\mathsf{a}_u^3 e^{-\frac{g}{8d}\mathsf{a}_u^4}\leq \min\{1,e^{-gt^4-at^2} \}    
\end{equation}
for every $\mathsf{a}_u\geq C$ and every $t\in [1,2]$. Since $u$ has degree at most $2d$, combining inequalities \eqref{eq: a estimate}, \eqref{eq: b estimate}, \eqref{eq: beta estimate} and \eqref{eq: third power estimate} we get
\begin{equation}
\begin{aligned}
\prod_{v\sim u} a_{uv}(\mathsf{a}_u,\mathsf{a}_v)b_{uv}(\mathsf{a}_u,\mathsf{a}_v)e^{-g\mathsf{a}_u^4-a\mathsf{a}_u^2}&\leq e^{-\frac{g}{2}\mathsf{a}_u^4-a\mathsf{a}_u^2}\min_{t\in [1,2]}\left\{\prod_{v \sim u} a_{uv}(t,\mathsf{a}_v) b_{uv}(t,\mathsf{a}_v) e^{-gt^4-at^2}\right\} \\
&\leq   e^{-\frac{g}{2}\mathsf{a}_u^4-a\mathsf{a}_u^2} \int_0^\infty \prod_{v \sim u} a_{uv}(t,\mathsf{a}_v)  b_{uv}(t,\mathsf{a}_v) e^{-gt^4-at^2}\mathrm{d} t  
\end{aligned}
\end{equation}
for every $\mathsf{a}_u\geq C\max\{1,\max_{v\sim u}\{\mathsf{a}_v^{1/3}\}\}$.
By integrating and letting $E_{u,s}=\{\mathsf{a}_u\leq \max\{s^{d(x,u)+1},C\max_{v\sim u}\{\mathsf{a}_v^{1/3}\}\}\}$ for $s\geq C$, we obtain that 
\begin{equation}\label{eq: ineq event S}
\boldsymbol\mu[E_{u,s}^c]\leq \int_{s^{d(x,u)+1}}^\infty e^{-\frac{g}{2}t^4-at^2}\mathrm{d}t.
\end{equation}
It is not hard to see that there exists a constant $c>0$ such that for $C>1$ large enough and every $s\geq C$ and $k\geq 2$,
\begin{equation}\label{eq: integral decay}
\begin{aligned}
 \int_{s^k}^\infty e^{-\frac{g}{2}t^4-at^2}\mathrm{d}t &\leq  \int_{s^k}^\infty e^{-\frac{g}{4}t^4-\frac{a}{2}t^2}e^{-\frac{g}{2}t^{4/k}-at^{2/k}}\mathrm{d}t \\ 
 &=\int_{s}^\infty k t^{\frac{k-1}{k}} e^{-\frac{g}{4}t^{4k}-\frac{a}{2}t^{2k}}e^{-\frac{g}{2}t^{4}-at^{2}}\mathrm{d}t \\
 &\leq e^{-c(k-1)}\int_{s}^\infty e^{-\frac{g}{2}t^4-at^2}\mathrm{d}t.
 \end{aligned}
\end{equation}
The desired result now follows (up to a redefinition of $C$ as necessary) from \eqref{eq: ineq event S} and \eqref{eq: integral decay} and a union bound over all vertices $u$ in $\Lambda$.
\end{proof}

We are now ready to prove Proposition~\ref{prop: coupling}.
\begin{proof}[Proof of Proposition~\textup{\ref{prop: coupling}}]
Consider two neighbouring vertices $x,y\in \Lambda$, and let $\omega_1\in \{0,1\}^{E(\Lambda)}$, $\omega_2 \in \{0,1\}^{E(\Lambda)\setminus\{xy\}}$. Our aim is to show that 
\begin{equation}\label{eq: conditional psi}
\Psi^0_{\Lambda,\beta',\beta}[\omega_{xy(2)}=1 \mid \omega|_{E^1}=\omega_1, \omega|_{E^2\setminus xy(2)}=\omega_2]\geq \varepsilon
\end{equation}
for some constant $\varepsilon>0$ that depends only on $\beta'$ and $\beta$. Assuming that \eqref{eq: conditional psi} holds, we can deduce that the projection of $\Psi^0_{\Lambda,\beta',\beta}$ on $E^2$ stochastically dominates a Bernoulli bond percolation on $E(\Lambda)$ which is independent from the projection of $\Psi^0_{\Lambda,\beta',\beta}$ on $E^1$. The desired result then follows from Lemma~\ref{lem: dom random intensity}.

It remains to prove \eqref{eq: conditional psi}.
Let $S_i$ be the set of edges $uv$ with $u\in \{x,y\}$ such that $\omega^i_{uv}=0$, and let $A=S_1$ and $B=S_2\cup\{xy\}$. 
It is not hard to see that for $\boldsymbol{\mu}=\boldsymbol{\mu}(\Lambda,\beta',\beta,A,B,\omega_1,\omega_2)$, the Radon--Nikodym derivative
\begin{equation}\label{eq:density increasing formula}
\frac{\mathrm{d}\boldsymbol{\mu}(\mathsf{a})}{\prod_{x\in \Lambda}\mathrm{d}\rho_{g,a}(\mathsf{a}_x)} =\frac{1}{Z} \prod_{uv\in E(\Lambda)} a_{uv}(\mathsf{a}_u,\mathsf{a}_v)b_{uv}(\mathsf{a}_u,\mathsf{a}_v)=:f(\mathsf{a})  
\end{equation} 
is an increasing function of $\mathsf{a}_x$ and $\mathsf{a}_y$, where $a_{uv}$ and $b_{uv}$ are defined in \eqref{eq: a definition} and \eqref{eq: b definition}, and where $Z=Z(\Lambda,\beta',\beta,A,B,\omega_1,\omega_2)$ is a normalisation constant. Indeed, this follows from the fact that the functions $a_{uv}$ and $b_{uv}$ are increasing for all edges $uv$ such that $u\in \{x,y\}$. Let $C>1$ a large enough constant to be determined. Define
\begin{equation}
F(\mathsf{a}_x,\mathsf{a}_y):=\int_{(\mathbb R)^{\Lambda\setminus \{x,y\}}}\mathbbm{1}_{G_{x,y}}f(\mathsf{a})\prod_{x\in \Lambda\setminus \{x,y\}}\mathrm{d}\rho_{g,a}(\mathsf{a}_x),
\end{equation}
where $G_{x,y}$ is the event that $\mathsf{a}_u\leq C$ for every $u$ which either lies in $\{x,y\}$ or has a neighbour in $\{x,y\}$. Note that since $f(\mathsf{a})$ is increasing in $\mathsf{a}_x$ and $\mathsf{a}_y$, the function $F$ is increasing.

Using the absolute value FKG for the product measure $\rho^{\otimes 2,\leq C}:=\rho_{g,a}[\:\cdot \mid \mathsf{a}_x\leq C]\otimes \rho_{g,a}[\:\cdot \mid \mathsf{a}_y\leq C]$, which is a special case of Proposition~\ref{prop: absolute value FKG} for $\beta=0$, we obtain that
\begin{equation}\label{eq: cond FKG}
\boldsymbol{\mu}[\mathsf{a}_x \geq 1, \mathsf{a}_y\geq 1 \mid G_{x,y}]=\frac{\rho^{\otimes 2,\leq C}[F(\mathsf{a}_x,\mathsf{a}_y)\mathbbm{1}_{\mathsf{a}_x\geq 1,\:\mathsf{a}_y\geq 1}]}{\rho^{\otimes 2,\leq C}[F(\mathsf{a}_x,\mathsf{a}_y)]}\geq \left(\rho_{g,a}[\mathsf{a} \geq 1 \mid \mathsf{a}\leq C]\right)^2,   
\end{equation}
On the event $G_{x,y}$ we have that $1-r(\beta-\beta',\mathsf{a})_{uv}$ and $e^{-\beta' \mathsf{a}_u \mathsf{a}_v}$ are bounded away from $0$ for every edge $uv$ with $u\in \{x,y\}$, while on the event $\{\mathsf{a}_x\geq 1, \mathsf{a}_y\geq 1\}$ we have $r(\beta-\beta',\mathsf{a})_{xy}$ is bounded away from $0$. Writing $S(\omega_1,\omega_2)=\{\omega|_{E^1}=\omega_1, \omega|_{E^2\setminus xy(2)}=\omega_2\}$, we can now deduce that 
\begin{equation}
\begin{aligned}
\frac{\Psi^0_{\Lambda,\beta',\beta}[\omega_{xy(2)}=1 \mid S(\omega_1,\omega_2)]}{\boldsymbol{\mu}[G_{x,y}]}&\geq
\frac{\Psi^0_{\Lambda,\beta',\beta}[\omega_{xy(2)}=1, \mathsf{a}_x \geq 1, \mathsf{a}_y\geq 1, G_{x,y}\mid S(\omega_1,\omega_2)]}{\boldsymbol{\mu}[G_{x,y}]}\\&\geq
\delta \boldsymbol{\mu}[ \mathsf{a}_x \geq 1, \mathsf{a}_y\geq 1 \mid G_{x,y}] \\
&\geq \delta \left(\rho_{g,a}[\mathsf{a} \geq 1 \mid \mathsf{a}\leq C]\right)^2
\end{aligned}
\end{equation}
for some constant $\delta>0$ depending only on $\beta'$ and $\beta$, where we used that once the contribution of the closed edges in $S_1$ and $S_2$ is removed from $\Psi^0_{\Lambda,\beta',\beta}[\:\cdot \mid S(\omega_1,\omega_2)]$, we obtain a measure whose projection on $\mathsf{a}$ coincides with $\boldsymbol{\mu}$. By Lemma~\ref{lem: cond regularity} and a union bound, there exists $C>1$ large enough so that $\boldsymbol{\mu}[G_{x,y}]\geq 1/2$.
The desired inequality \eqref{eq: conditional psi} follows readily. 
\end{proof}

\section{Surface order large deviations}\label{sec:coarse_graining}

In this section, we prove the surface order (lower) large deviations of Theorem~\ref{thm:ldp free}, which will follow rather easily from Lemmas~\ref{lem: cutting a to finite values} and \ref{lem: volume ldp small average} and the key Proposition~\ref{prop: cluster renormalisation} below. The latter is proved by combining Theorem~\ref{thm: local uniqueness} and a Pisztora's coarse graining. 
We also show the volume order (upper) large deviations mentioned in Remark~\ref{rem: vol order}, which follows from Lemma~\ref{lem: volume ldp large average}.

The following simple lemma allows us to disregard large absolute values of the field.

\begin{lemma}\label{lem: cutting a to finite values}
Let $\beta\geq 0$ and $\delta>0$. There exists $M_0=M_0(\delta)$ such that, for every $M\geq M_0$ and every $n$ large enough,
\begin{align}\label{eq:dev_large_a}
\Psi^{0}_{\Lambda_n,\beta}\left[\frac{1}{|\Lambda_n|}\sum_{x\in \Lambda_n} \mathsf{a}_x \mathbbm{1}_{\mathsf{a}_x\geq M}\geq \delta\right]&\leq e^{-n^d}.
\end{align}
\end{lemma}

\begin{proof} 
    By Proposition~\ref{prop:regularity}, the field $(\mathsf{a}_x)_{x\in \Lambda_n}$ is stochastically dominated by a field $\mathsf{a}'=(\mathsf{a}'_x)_{x\in \Lambda_n}$ with independent and identically distributed (i.i.d.)~marginals with quartic exponential tails and, in particular, finite exponential moments. Since the event in \eqref{eq:dev_large_a} is increasing, it is enough to prove the bound for the i.i.d.~field $\mathsf{a}'$. This follows from classical large deviations results, see e.g.~\cite[Theorem 2.7.7]{durrett2019probability}.
\end{proof}

The next lemma allows us to control the lower large deviations for vertices belonging to large clusters. 

\begin{lemma}\label{lem: volume ldp small average} Let $\beta\geq 0$. For every $\delta>0$ and $K\geq 1$, there exists $c=c(K,\delta)>0$ such that, for every boundary condition $(\xi,\mathsf{b})$, and for every $n$ large enough,
\begin{equation}\label{eq: volume ldp small average}
\Psi_{\Lambda_n,\beta}^{(\xi,\mathsf{b})}\left[\frac{1}{|\Lambda_n|}\sum_{x\in \Lambda_n}\mathsf{a}_x\mathbbm{1}_{|\mathcal{C}_x|\geq K}\leq m^*(\beta)-\delta\right]\leq e^{-cn^d}.
\end{equation}
\end{lemma}

\begin{proof} We can always assume that $m^*(\beta)>\delta$ (in particular $\beta>\beta_c)$. We fix $K\geq 0$. For $x\in \mathbb Z^d$, we let $Y_x^{K}:=\mathsf{a}_x\mathbbm{1}_{|\mathcal{C}_x|\geq K}$. First of all, notice that, since the event in \eqref{eq: volume ldp small average} is decreasing, it is sufficient to prove the above statement for the measure $\Psi_{\Lambda_n,\beta}^0$. Using Remark \ref{rem: equality of magnetisation}, one has that, for every $x\in \mathbb Z^d$,
\begin{equation}
    \Psi^{0}_{\beta}[\mathsf{a}_x\mathbbm{1}_{0\leftrightarrow \infty}]=\Psi^{1}_{\beta}[\mathsf{a}_x\mathbbm{1}_{0\leftrightarrow \infty}]=m^*(\beta).
\end{equation}
As a result, for every $x\in \mathbb Z^d$,
\begin{equation}
     \Psi^{0}_{\beta}[Y_x^{K}]\geq m^*(\beta).
\end{equation}
Moreover, since $Y_x^K$ is a local increasing function measurable with respect to $\Lambda_K(x)$, there exists $L=L(K,\delta)>K$ such that, for every $x\in \mathbb Z^d$,
\begin{equation}\label{eq:proof lower ldp1}
    \Psi^{0}_{\Lambda_L(x),\beta}[Y_x^K]\geq m^*(\beta)-\frac{\delta}{2}.
\end{equation}
Note that by Proposition \ref{prop:regularity}, the random variable $Y_0^K$ has exponential tails (under the measure $\Psi^0_{\Lambda_L}$).
Choose $n\geq 10L$. Let $\mathbb P$ be the law of a collection of i.i.d.\ random variables $(\tilde{Y}_x^K)_{x\in \Lambda_{n-2L}}$ of law given by the law of $Y_0^K$ under $\Psi^0_{\Lambda_L,\beta}$. Let $m:=\lfloor \tfrac{n-2L}{4L}\rfloor$.
As a consequence of monotonicity in boundary conditions, we find that, for any $y\in \Lambda_{4L}$, the collection of random variables $(Y_x^{K})_{x \in (y+4L\Lambda_m)}$ under the measure $\Psi^0_{\Lambda_n,\beta}$ stochastically dominates  $(\tilde{Y}_x^K)_{x \in (y+4L\Lambda_m)}$. For $y\in \Lambda_{4L}$, introduce the (decreasing) event
\begin{equation}
    \mathcal{E}_y(Y^K):=\left\{\frac{1}{ |y+4L\Lambda_m|}\sum_{x\in (y+4L\Lambda_m)}Y^K_x\leq m^*(\beta)-\frac{3\delta}{4}\right\}.
\end{equation}
Similarly, one may define $\mathcal{E}_w(\tilde{Y}^K).$ Classical large deviations estimates for independent random variables with exponential tails (see \cite[Theorem~2.7.7]{durrett2019probability}) imply the existence of $c_1=c_1(K,\delta)>0$ such that, for every $n$ large enough,
\begin{equation}\label{eq:proof lower ldp2}
    \sup_{y\in \Lambda_{4L}}\mathbb P [\mathcal{E}_y(\tilde{Y}^K)]\leq e^{-c_1n^d}.
\end{equation}
Now, letting $n$ be large enough so that $\tfrac{|\Lambda_n|}{|\Lambda_{n-2L}|}\leq 1+\tfrac{\delta/4}{m^*(\beta)-\delta}$, there exists $c_2=c_2(K,\delta)>0$ such that,
\begin{align}
    \Psi^0_{\Lambda_n,\beta}\left[\frac{1}{|\Lambda_n|}\sum_{x\in \Lambda_n}Y_x^{K}\leq m^*(\beta)-\delta\right]&\leq \Psi^0_{\Lambda_n,\beta}\left[\frac{1}{|\Lambda_{n-2L}|}\sum_{x\in \Lambda_{n-2L}}Y_x^{K}\leq m^*(\beta)-\frac{3\delta}{4}\right]
    \\&\leq \sum_{w\in \Lambda_{4L}} \Psi_{\Lambda_n,\beta}^0[\mathcal{E}_w(Y^K)]
    \\&\leq \sum_{w\in \Lambda_{4L}} \mathbb P[\mathcal{E}_w(\tilde{Y}^K)] \leq e^{-c_2n^d},
\end{align}
where we used the stochastic domination in the third line, and \eqref{eq:proof lower ldp2} in the last line. This concludes the proof.
\end{proof}

The next lemma is the upper deviation counterpart of Lemma~\ref{lem: volume ldp small average}. It will only be used to prove the upper large deviations stated in Remark~\ref{rem: vol order}.

\begin{lemma}\label{lem: volume ldp large average} Let $\beta\geq0$. For every $\delta>0$, there exist $K_1=K_1(\delta)\geq1$ and $c=c(\delta)>0$ such that for every $K\geq K_1$, and for every $n$ large enough,
\begin{equation}
\Psi^0_{\Lambda_n,\beta}\Bigg[ \frac{1}{|\Lambda_n|} \sum_{
x\in \Lambda_n}\mathsf{a}_x\mathbbm{1}_{|\mathcal{C}_x|\geq K}\geq m^*(\beta) + \delta\Bigg]\leq e^{-cn^d}.    
\end{equation}
\end{lemma}

\begin{proof}
We will argue as in the proof of the previous lemma, with certain additional complexities arising from the presence of large absolute values. Again, for $x\in \mathbb Z^d$, we let $Y_x^K=\mathsf{a}_x\mathbbm{1}_{|\mathcal{C}_x|\geq K}$. We first consider $\Psi^1_{\beta}$ and we note that there exists $K>0$ such that, for every $x\in \mathbb Z^d$,    
\begin{equation}
\Psi^1_{\beta}[Y_x^K]\leq m^*(\beta)+\frac{\delta}{8}.    
\end{equation}
Since $Y_x^K$ is a local function (around $x$), there exists $L> K$ large enough such that, for every $x\in \mathbb Z^d$,
\begin{equation}\label{eq: y smaller m}
\Psi^{(w,\mathfrak{p})}_{\Lambda_L(x),\beta}[Y_x^K]\leq m^*(\beta)+\frac{\delta}{4}. 
\end{equation}
We now consider $n\geq 10L$, and let $y\in \Lambda_{4L}$. We consider the collection of boxes of the form $\Lambda_L(y+4Lz)$ for $z\in \Lambda_m$, where $m=\lfloor \frac{n-2L}{4L}\rfloor$. Note that the union of all these boxes over $y\in \Lambda_{4L}$ and $z\in \Lambda_m$ covers $\Lambda_{n-5L}$. Recall the definition of $\mathfrak{p}$. For each $y\in \Lambda_{4L}$ and $z\in \Lambda_m$, let $B_{z}$ denote the event $\{\exists z\in \partial \Lambda_L(y+4Lz) \text{ such that } \mathsf{a}_z\geq C_{M} (1\vee\log |\Lambda_L|)^{1/4}\}$. With this definition in hand, we write 
\begin{align}
\Psi^0_{\Lambda_n,\beta}\Bigg[ \frac{1}{|\Lambda_n|} \sum_{
x\in \Lambda_n}Y_x^K\geq &\:m^*(\beta) + \delta\Bigg]\leq
\Psi^0_{\Lambda_n,\beta}\Bigg[ \frac{1}{|\Lambda_n|} \sum_{
x\in \Lambda_n\setminus\Lambda_{n-5L}} Y_{x}^K \geq \frac{\delta}{3}\Bigg]\label{eq:first term ldp volume lem}
\\&+
\sum_{y\in \Lambda_{4L}}\Psi^0_{\Lambda_n,\beta}\Bigg[ \frac{1}{|\Lambda_m|} \sum_{
z\in \Lambda_m}Y_{y+4Lz}^K \mathbbm{1}_{B_z}\geq \frac{\delta}{3}\Bigg]\label{eq:second term ldp volume lem}\\&+\sum_{y\in \Lambda_{4L}}\Psi^0_{\Lambda_n,\beta}\Bigg[ \frac{1}{|\Lambda_m|} \sum_{
z\in \Lambda_m}Y_{y+4Lz}^K \mathbbm{1}_{B_z^c}\geq m^*(\beta)  + \frac{\delta}{3}\Bigg],\label{eq:third term ldp volume lem}
\end{align}
and we bound each probability separately.

We begin with the term on the right-hand side of \eqref{eq:first term ldp volume lem}. Let $M_0=M_0(\delta/3)$ be given by Lemma~\ref{lem: cutting a to finite values}. If $n$ is large enough so that $\tfrac{|\Lambda_n\setminus\Lambda_{n-5L}|}{|\Lambda_n|}M_0<\tfrac{\delta}{3}$, it follows from Lemma \ref{lem: cutting a to finite values} that
\begin{equation}\label{eq:boundterm1}
\Psi^0_{\Lambda_n,\beta}\Bigg[ \frac{1}{|\Lambda_n|} \sum_{
x\in \Lambda_n\setminus\Lambda_{n-5L}} Y_{x}^K \geq \frac{\delta}{3}\Bigg]\leq \Psi^0_{\Lambda_n,\beta}\left[\frac{1}{|\Lambda_n|}\sum_{x\in \Lambda_n}\mathsf{a}_x\mathbbm{1}_{\mathsf{a}_x\geq M_0}\geq \frac{\delta}{3}\right]\leq e^{-n^d}.    
\end{equation}

To estimate the term in \eqref{eq:second term ldp volume lem},
recall that $\mathsf{a}\sim \Psi^0_{\Lambda_n,\beta}$ is stochastically dominated by a sequence of i.i.d.\ random variables $\mathsf{a}'=(\mathsf{a}'_x)_{x\in \Lambda_n}$ that have quartic tails. Writing $\mathbb{P}$ for the law of these random variables, we see that--- to the cost of potentially increasing the value of $L$--- we can assume that
$\mathbb{P}[\mathsf{a}'_{y+4Lz}\mathbbm{1}_{B_{z}}]\leq \tfrac{\delta}{4}$ for every $z\in \Lambda_m$. Since for fixed $y$, the random variables $\mathsf{a}'_{y+4Lz} \mathbbm{1}_{B_{z}}$ are independent under $\mathbb{P}$, it follows again from classical large deviations estimates that there exists $c_1>0$ such that, for every $n$ large enough and every $y\in \Lambda_{4L}$,
\begin{equation}\label{eq:boundterm2}
\Psi^0_{\Lambda_n,\beta}\Bigg[ \frac{1}{|\Lambda_m|} \sum_{
z\in \Lambda_m}Y_{y+4Lz}^K \mathbbm{1}_{B_z}\geq \frac{\delta}{3}\Bigg]\leq \mathbb{P}\Bigg[\frac{1}{|\Lambda_m|}\sum_{z\in \Lambda_m} \mathsf{a}'_{y+4Lz}\mathbbm{1}_{B_z}\geq \frac{\delta}{3}\Bigg]\leq e^{-c_1n^d}.
\end{equation}

Finally, we estimate the term in \eqref{eq:third term ldp volume lem}. We use the monotonicity in the boundary conditions to deduce that for each $y\in \Lambda_{4L}$,  the collection of random variables $(Y_{y+4Lz}^K)_{z\in \Lambda_m}$ is stochastically dominated by a sequence of independent random variables $(\tilde{Y}^K_{y+4Lz})_{z\in \Lambda_m}$ where each marginal is distributed according to the law of $Y_{y}^K$ under $\Psi^{(w,\mathfrak{p})}_{\Lambda_L(y),\beta}$, and in particular have exponential tails by Proposition \ref{prop:regularity}. By \eqref{eq: y smaller m}, the latter have mean at most $m^*(\beta)+\tfrac{\delta}{4}$. Hence, we can use again the large deviation estimates for independent random variables with exponential tails (see \cite{durrett2019probability}) to obtain that for some $c_2>0$ and for every $n$ large enough,
\begin{equation}\label{eq:boundterm3}
\Psi^0_{\Lambda_n,\beta}\Bigg[ \frac{1}{|\Lambda_m|} \sum_{
z\in \Lambda_m}Y_{y+4Lz}^K \mathbbm{1}_{B_z^c}\geq m^*(\beta)  + \delta/3\Bigg]\leq e^{-c_2n^d}.    
\end{equation}
Plugging \eqref{eq:boundterm1}--\eqref{eq:boundterm3} in \eqref{eq:first term ldp volume lem}--\eqref{eq:third term ldp volume lem} concludes the proof.
\end{proof}

Let $\mathfrak{C}=\mathfrak{C}(\Lambda_n)$ be the set of all clusters in $\Lambda_n$. Define $\mathcal{C}_{\rm max}$ to be the element of $\mathfrak{C}$ of largest cardinality (with an arbitrary rule to break ties). 
We can now state the main crucial ingredient for the proof of Theorem~\ref{thm:ldp free}. This proposition allows us to deal with the contributions of large clusters which are not the (unique) giant one.
The proof uses Theorem~\ref{thm: local uniqueness} and a coarse graining argument in order to reduce the problem to a highly supercritical Bernoulli percolation.

\begin{proposition}\label{prop: cluster renormalisation}
Let $\beta>\beta_c$. For every $\delta>0$, there exist $N_0=N_0(\delta)\geq 1$ and $c=c(\delta)>0$ such that, for every boundary condition $(\xi,\mathsf{b})$, for every $N\geq N_0$, and for every $n$ large enough,
\begin{equation}\label{eq:large_clusters_dev}
\Psi^{(\xi,\mathsf{b})}_{\Lambda_n,\beta}\Bigg[ \frac{1}{|\Lambda_n|} \sum_{
\substack{\mathcal{C}\in \mathfrak{C}\setminus\{\mathcal{C}_{\rm max}\}\\ |\mathcal{C}|\geq N}} |\mathcal{C}| \geq \delta \Bigg]\leq e^{-cn^{d-1}}.     
\end{equation}
\end{proposition}

We will use the following result about the largest cluster in highly supercritical Bernoulli percolation.
Given a set of vertices $A\subset \Lambda_n$, denote by $\mathfrak{S}(A)$ the set of connected components (not to be confused with percolation clusters!) of $A$. 
Below we denote by $\mathbb{P}^{\text{site}}_p$ the law of Bernoulli site percolation of parameter $p$ on $\mathbb{Z}^d$.

\begin{lemma}\label{lem:Bernoulli_surf_dev}
There exists $p_0<1$ such that for every $\varepsilon>0$ there exist $\ell=\ell(\varepsilon)\geq 1$ and $c=c(\varepsilon)>0$ such that    
\begin{equation}\label{eq:Bernoulli_surf_dev}
\mathbb{P}^{\textup{site}}_{p_0}\Bigg[\exists \, {\mathcal{C}}\in \mathfrak{C}(\Lambda_m):~ |{\mathcal{C}}|\geq \tfrac{3}{4}|\Lambda_m| \text{ and } \sum_{\substack{\mathcal{S}\in \mathfrak{S}(\Lambda_m\setminus \mathcal{C}) \\ |\mathcal{S}|\geq M}}|\mathcal{S}|\leq \varepsilon |\Lambda_m|\Bigg]\geq 1-e^{-cm^{d-1}}    
\end{equation}
for every $m\geq 1$.
\end{lemma}

Lemma~\ref{lem:Bernoulli_surf_dev} is a slight modification of the main result of \cite{DP96}. The proof, very similar to that of \cite{DP96}, is presented in the Appendix~\ref{sec:appendix_high_perco}.

\begin{proof}[Proof of Proposition~\textup{\ref{prop: cluster renormalisation}}]
Fix $\beta>\beta_c$ and $\delta>0$. We will prove the desired result by a renormalisation argument. Take $p_0<1$ given by Lemma~\ref{lem:Bernoulli_surf_dev}. For $x\in \mathbb{Z}^d$ and $L\geq 1$, let $U_x(L)$ denote the translation of the event $U(L)$ from Theorem~\ref{thm: local uniqueness} by $x$. Consider the renormalised percolation model $\eta\in \{0,1\}^{\Lambda_m}$, where $m=\lfloor \frac{n-10L}{L}\rfloor$, given by 
$$\eta_x \coloneqq \mathbbm{1}_{U_{Lx}(L)}.$$
Since $\inf_{\#}\Psi^{\#}_{\Lambda_{10L},\beta}[U(L)] \to 1$ as $L\to\infty$, it follows from the main result of \cite{LSS97} that there exists an $L=L(p_0)\geq 1$ sufficiently large such that, no matter the boundary condition $(\xi,\mathsf{b})$, $\eta$ stochastically dominates Bernoulli percolation with parameter $p_0$. Let $\varepsilon>0$ be a constant to be chosen later. By Lemma~\ref{lem:Bernoulli_surf_dev}, there exists $M=M(\varepsilon)\geq1$ and $c=c(\varepsilon)>0$ such that  
\begin{equation}\label{eq:large_clusters_dev_2}
\Psi^{(\xi,\mathsf{b})}_{\Lambda_n,\beta} \Bigg[ \exists \, \mathcal{S}_{\rm max} \in \mathfrak{C}(\eta):~ |\mathcal{S}_{\rm max}|\geq \tfrac{3}{4}|\Lambda_m| \text{ and } \frac{1}{|\Lambda_m|} \sum_{
\substack{\mathcal{S}\in \mathfrak{S}(\Lambda_m\setminus \mathcal{S}_{\rm max})\\ |\mathcal{S}|\geq M}} |\mathcal{S}| \leq \varepsilon \Bigg]\geq 1- e^{-cm^{d-1}},
\end{equation}
for every boundary condition $(\xi,\mathsf{b})$ and for every $n$ large enough. 

We call $\mathcal{G}$ the event in \eqref{eq:large_clusters_dev_2}.
We claim that $\mathcal{G}$ is contained in the complement of the event in \eqref{eq:large_clusters_dev} for $\varepsilon:=\frac{\min\{\delta,1\}}{2|\Lambda_{16L}|}$. This fact, combined with \eqref{eq:large_clusters_dev_2}, implies the desired result. Let us prove this claim. The proof goes in two steps.

Because of the way the event $U(L)$ is defined, the cluster $\mathcal{S}_{\rm max}$ as in \eqref{eq:large_clusters_dev_2} induces an $\omega$ cluster $\mathcal{C}_0$ which intersects all boxes $\Lambda_L(Lx)$ for $x\in \mathcal{S}_{\rm max}$. In particular if $\mathcal{G}$ occurs, $\mathcal{C}_0$ has size $|\mathcal{C}_0|\geq \frac{3}{4}|\Lambda_m|$. We first prove that $\mathcal{C}_0$ is the largest $\omega$ cluster, i.e.\ $\mathcal{C}_0=\mathcal{C}_{\rm max}$. To see this, let $N':=M|\Lambda_{16L}|$. Consider any other $\omega$-cluster $\mathcal{C}$ such that $|\mathcal{C}\cap \Lambda_{n-11L}|\geq N'$, and let $\mathcal{S}$ be the set of vertices $x\in \Lambda_m$ such that the annulus $\Lambda_{4L}(Lx)\setminus\Lambda_{2L}(Lx)$ is crossed by $\mathcal{C}$. Note that 
\begin{equation}\label{eq: simple inequality}
|\mathcal{C}\cap \Lambda_{n-11L}|\leq |\Lambda_{16L}||\mathcal{S}|,
\end{equation}
hence $|\mathcal{S}|\geq M$. Moreover, we have 
\begin{equation}\label{eq: simple inequality 2}
\mathcal{S}\subset \Lambda_m\setminus \mathcal{S}_{\rm max}.
\end{equation}
Thus, $\mathcal{S}$ lies in a connected component of $\Lambda_m\setminus \mathcal{S}_{\rm max}$, and in particular, $|\mathcal{S}|\leq \varepsilon |\Lambda_m|$. Hence, by our choice of $\varepsilon$, $|\mathcal{C}|\leq \frac{1}{2} |\Lambda_m|+|\Lambda_{n}\setminus\Lambda_{n-11L}|<|\mathcal{C}_0|$ for every $n$ large enough. This implies that $\mathcal{C}_0=\mathcal{C}_{\rm max}$. 

Now, if $\mathcal{G}$ occurs, we can use \eqref{eq: simple inequality} and \eqref{eq: simple inequality 2} to deduce that 
\begin{equation}
\frac{1}{|\Lambda_n|}\sum_{
\substack{\mathcal{C}\in \mathfrak{C}\setminus\{\mathcal{C}_{\rm max}\}\\ |\mathcal{C}\cap \Lambda_{n-11L}|\geq N'}}|\mathcal{C}\cap \Lambda_{n-11L}|\leq \frac{|\Lambda_{16L}|}{|\Lambda_n|}\sum_{
\substack{\mathcal{S}\in \mathfrak{S}(\Lambda_m\setminus \mathcal{S}_{\rm max})\\ |\mathcal{S}|\geq M}} |\mathcal{S}|\leq \frac{\varepsilon |\Lambda_{16L}| |\Lambda_m|}{|\Lambda_n|}.    
\end{equation}
Hence, for $N=2N'$,
\begin{equation}
\begin{aligned}
\frac{1}{|\Lambda_n|} \sum_{
\substack{\mathcal{C}\in \mathfrak{C}\setminus\{\mathcal{C}_{\rm max}\}\\ |\mathcal{C}|\geq N}} |\mathcal{C}| &\leq \frac{1}{|\Lambda_n|} \sum_{
\substack{\mathcal{C}\in \mathfrak{C}\setminus\{\mathcal{C}_{\rm max}\}\\ |\mathcal{C}\cap \Lambda_{n-11L}|\leq N'\\ |\mathcal{C}|\geq N}} |\mathcal{C}|+\frac{1}{|\Lambda_n|} \sum_{
\substack{\mathcal{C}\in \mathfrak{C}\setminus\{\mathcal{C}_{\rm max}\}\\ |\mathcal{C}\cap \Lambda_{n-11L}|\geq N'}} |\mathcal{C}\setminus \Lambda_{n-11L}|\\
&+\frac{1}{|\Lambda_n|} \sum_{
\substack{\mathcal{C}\in \mathfrak{C}\setminus\{\mathcal{C}_{\rm max}\}\\ |\mathcal{C}\cap \Lambda_{n-11L}|\geq N'}} |\mathcal{C}\cap \Lambda_{n-11L}|\\
&\leq \frac{|\Lambda_n\setminus \Lambda_{n-11L-N'}|}{|\Lambda_n|}+\frac{\varepsilon |\Lambda_{16L}| |\Lambda_m|}{|\Lambda_n|}
\leq \frac{\delta}{4}+\frac{\varepsilon |\Lambda_{16L}| |\Lambda_m|}{|\Lambda_n|}<\delta
\end{aligned}
\end{equation}
for every $n$ large enough. This proves the claim that the event in \eqref{eq:large_clusters_dev_2} is contained in the complement of the event in \eqref{eq:large_clusters_dev}. The desired result follows.
\end{proof}

We are now in a position to prove the main result of this section.
\begin{proof}[Proof of Theorem \textup{\ref{thm:ldp free}}] We fix $\beta>\beta_c$ and drop it from the notations. Let $\delta>0$.

\paragraph{Proof of the upper bound in \eqref{eq:ldp surface}.} Let $M:=M_0(\delta/8)$ be given by Lemma \ref{lem: cutting a to finite values}. Let $N:=N_0(\delta/(8M))$ be given by Proposition \ref{prop: cluster renormalisation}. Applying these two results, and also Lemma \ref{lem: volume ldp small average} to $K=N$ and $\tfrac{\delta}{4}$, we obtain $c_1>0$ such that, for every $n$ large enough,
\begin{equation}\label{eq:final ldp surf1}
    \Psi^0_{\Lambda_n}[B_1\cup B_2]\leq e^{-c_1n^d}, \qquad \Psi^0_{\Lambda_n}[ B_3]\leq e^{-c_1n^{d-1}},
\end{equation}
where $B_1:=\left\{\tfrac{1}{|\Lambda_n|}\sum_{x\in \Lambda_n}\mathsf{a}_x\mathbbm{1}_{\mathsf{a}_x\geq M}\geq \frac{\delta}{8}\right\}$, $B_2:=\left\{\tfrac{1}{|\Lambda_n|}\sum_{x\in \Lambda_n}\mathsf{a}_x\mathbbm{1}_{|\mathcal{C}_x|\geq N}\leq m^*-\frac{\delta}{4}\right\}$, and $B_3:=\left\{\tfrac{1}{|\Lambda_n|}\sum_{\mathcal{C}\in \mathfrak{C}\setminus \{\mathcal{C}_{\rm max}\}: |\mathcal{C}|\geq N}|\mathcal{C}|\geq \tfrac{\delta}{8M}\right\}$. First, we remark that
\begin{equation}
    \Psi^0_{\Lambda_n}[B_4]\leq \Psi^0_{\Lambda_n}[B_1\cup B_3],
\end{equation}
where $B_4:=\left\{\tfrac{1}{|\Lambda_n|}\sum_{\mathcal{C}\in \mathfrak{C}\setminus \{\mathcal{C}_{\rm max}\}: |\mathcal{C}|\geq N}\sum_{x\in \mathcal{C}}\mathsf{a}_x\geq \tfrac{\delta}{4}\right\}$. Combined with \eqref{eq:final ldp surf1}, this implies that for every $n$ large enough,
\begin{equation}\label{eq:final ldp surf1.5}
    \Psi^0_{\Lambda_n}[B_4]\leq 2e^{-c_1n^{d-1}}.
\end{equation}
Now, remark that,
\begin{equation}\label{eq:final ldp surf2}
    \Psi^0_{\Lambda_n}\left[\frac{1}{|\Lambda_n|}\sum_{x\notin \mathcal{C}_{\rm max}}\varphi_x\notin(-\tfrac{\delta}{2},\tfrac{\delta}{2})\right]\leq \Psi_{\Lambda_n}^0[B_1\cup B_4]+2\Psi^0_{\Lambda_n}[C_1],
\end{equation}
where
\begin{equation}
    C_1:=\left\{\frac{1}{|\Lambda_n|}\sum_{\substack{\mathcal{C}\in \mathfrak{C}\setminus \{\mathcal{C}_{\rm max}\}\\ |\mathcal{C}|\leq N}}\sigma_{\mathcal{C}}\sum_{x\in \mathcal{C}}\mathsf{a}_x\mathbbm{1}_{\mathsf{a}_x\leq M}\geq\frac{\delta}{8}\right\},
\end{equation}
and we recall that $\sigma_{\mathcal{C}}$ denotes the sign of $\mathcal{C}$.
Using the Edwards--Sokal coupling, we get that--- conditionally on $\mathsf{a}$--- the random variables
\begin{equation}
  \mathbbm{1}_{|\mathcal{C}|\leq N}\sigma_{\mathcal{C}}\sum_{x\in \mathcal{C}}\mathsf{a}_x\mathbbm{1}_{\mathsf{a}_x\leq M}, \qquad \mathcal{C}\in \mathfrak{C}\setminus\{\mathcal{C}_{\rm max}\}
\end{equation}
are centred, independent, and bounded (by $N\cdot M$). Classical large deviations estimates (see \cite{durrett2019probability}) then imply the existence of $c_2>0$ such that for every $n$ large enough,
\begin{equation}\label{eq:final ldp surf3}
    \Psi^0_{\Lambda_n}[C_1]\leq e^{-c_2n^d}.
\end{equation}
Combining \eqref{eq:final ldp surf1}, \eqref{eq:final ldp surf1.5}, \eqref{eq:final ldp surf2}, and \eqref{eq:final ldp surf3}, we obtain $c_3>0$ such that for all $n$ large enough,
\begin{equation}\label{eq:final ldp surf4}
    \nu_{\Lambda_n}\left[\frac{1}{|\Lambda_n|}\sum_{x\in \Lambda_n}\varphi_x\in [-m^*+\delta,m^*-\delta]\right]\leq e^{-c_3n^{d-1}}+\Psi^0_{\Lambda_n}\left[\frac{1}{|\Lambda_n|}\sum_{x\in \Lambda_n}\mathsf{a}_x\mathbbm{1}_{x\in \mathcal{C}_{\rm max}}\leq m^*-\frac{\delta}{2}\right].
\end{equation}
Now, writing $\mathbbm{1}_{x\in \mathcal{C}_{\rm max}}=\mathbbm{1}_{|\mathcal{C}_x|\geq N}-\mathbbm{1}_{|\mathcal{C}_x|\geq N, \: \mathcal{C}_x\neq \mathcal{C}_{\rm max}}$, observe that
\begin{equation}\label{eq:final ldp surf5}
    \Psi^0_{\Lambda_n}\left[\frac{1}{|\Lambda_n|}\sum_{x\in \Lambda_n}\mathsf{a}_x\mathbbm{1}_{x\in \mathcal{C}_{\rm max}}\leq m^*-\frac{\delta}{2}\right]\leq \Psi_{\Lambda_n}^0[B_2\cup B_4].
\end{equation}
Combining \eqref{eq:final ldp surf4} and \eqref{eq:final ldp surf5}, and using again \eqref{eq:final ldp surf1} and \eqref{eq:final ldp surf1.5}, we obtain that there exists $c_4>0$ such that for all $n$ large enough
\begin{equation}
    \nu_{\Lambda_n}\left[\frac{1}{|\Lambda_n|}\sum_{x\in \Lambda_n}\varphi_x\in [-m^*+\delta,m^*-\delta]\right]\leq e^{-c_4n^{d-1}}.
\end{equation}

\paragraph{Proof of the lower bound in \eqref{eq:ldp surface}.} Let $\varepsilon=\varepsilon(\delta)>0$ be a small constant to be defined. We partition $\Lambda_n$ into boxes of the form $\Lambda_{\varepsilon n}((\varepsilon n +1)x)\cap \Lambda_n$, where $x\in \Lambda_m$, $m=\lfloor\frac{n}{\varepsilon n+1}\rfloor$. Let $S$ be the set of all the edges in $\Lambda_n$ that connect neighbouring boxes $\Lambda_{\varepsilon n}((\varepsilon n +1)y)$ and $\Lambda_{\varepsilon n}((\varepsilon n +1)x)$. Since $|S|\leq m^d \cdot 2d (2\varepsilon n)^{d-1}$, by the FKG inequality,
\begin{equation}
\Psi^0_{\Lambda_n}[\omega_e=0 ,\; \forall \; e\in S]\geq e^{-Cn^{d-1}}
\end{equation}
for some constant $C=C(\varepsilon)>0$. Furthermore, by Proposition \ref{prop:regularity}, there exists a constant $R>0$, independent of $\varepsilon$, such that 
\begin{equation}
\Psi^0_{\Lambda_n}\Big[\sum_{y\in \Lambda_{\varepsilon n}((\varepsilon n+1)x)} \mathsf{a}_y\leq R|\Lambda_{\varepsilon n}| ,\; \forall x\in \Lambda_m\Big]\geq 1- |\Lambda_m| e^{-|\Lambda_{\varepsilon n}|}.
\end{equation}
A union bound gives that for every $n$ large enough
\begin{equation}
\Psi^0_{\Lambda_n}\Big[\omega_e=0 ,\, \forall \; e\in S, \sum_{y\in \Lambda_{\varepsilon n}((\varepsilon n+1)x)} \mathsf{a}_y\leq R|\Lambda_{\varepsilon n}| ,\; \forall x\in \Lambda_m\Big]\geq e^{-2Cn^{d-1}}.
\end{equation}
On the latter event, we use the Edwards--Sokal coupling to assign to each cluster independent $\pm 1$ spins. Conditionally on such a pair $(\omega,\mathsf{a})$, the expectation of $\sum_{z\in \Lambda_n} \varphi_z$ is equal to $0$. Furthermore, since on the event $\{\omega_e=0 ,\, \forall \; e\in S\}$ spins in different $\varepsilon n$-boxes are independent, the (conditional) variance of $\sum_{z\in \Lambda_n} \varphi_z$ is at most
\begin{equation}
\sum_{x\in \Lambda_m} \sum_{u,v\in \Lambda_{(\varepsilon n+1)x}} \mathsf{a}_u \mathsf{a}_v\leq m^d R^2 |\Lambda_{\epsilon n}|^2.
\end{equation}
This implies that conditionally on such a pair $(\omega,\mathsf{a})$, with probability $1/2$ we have that $|\sum_{z\in \Lambda_n} \varphi_z|\leq \sqrt{2 m^d R^2 |\Lambda_{\epsilon n}|^2}$. Now, we choose $\varepsilon>0$ to be small enough so that we have $2 m^d R^2 |\Lambda_{\epsilon n}|^2\leq (m^*(\beta)-\delta)^2 |\Lambda_n|^2$. The desired result follows by combining the above inequalities.
\end{proof}

We now prove the volume order large deviations mentioned in Remark \ref{rem: vol order}.
\begin{proof}[Proof of \textup{\eqref{eq:ldp volume}}.]
We keep the notations of the preceding proof. For the lower bound, we note that for each $x\in \Lambda_n$, the probability that $\varphi_x\leq m^*(\beta)+\delta$ stays bounded away from $0$. The FKG inequality then gives the desired lower bound.

For the upper bound, let $K_1=K_1(\delta/2)$ be given by Lemma \ref{lem: volume ldp large average}. Using this result, there exists $c_1>0$ such that for every $n$ large enough,
\begin{equation}\label{eq:final ldp vol1}
    \Psi^0_{\Lambda_n}\left[B_5\right]\leq e^{-c_1n^d},
\end{equation}
where $B_5:=\left\{\tfrac{1}{|\Lambda_n|}\sum_{x\in \Lambda_n}\mathsf{a}_x\mathbbm{1}_{|\mathcal{C}_x|\geq K_1}\geq m^*+\frac{\delta}{2}\right\}$. Notice that
\begin{equation}\label{eq:final ldp vol2}
    \nu_{\Lambda_n}\Big[m_{\Lambda_n}\notin [-m^*-\delta,m^*+\delta]\Big]\leq \Psi^0_{\Lambda_n}\left[B_1\cup B_5\right]+\Psi^0_{\Lambda_n}[C_2],
\end{equation}
where
\begin{equation}
    C_2:=\left\{\frac{1}{|\Lambda_n|}\sum_{\mathcal{C}\in \mathfrak{C}:|\mathcal{C}|\leq K_1}\text{sgn}(\mathcal{C})\sum_{x\in \mathcal{C}}\mathsf{a}_x\mathbbm{1}_{\mathsf{a}_x\leq M}\geq\frac{\delta}{4}\right\}.
\end{equation}
Using again a classical large deviations estimate, we argue the existence of $c_2>0$ such that for every $n$ large enough
\begin{equation}\label{eq:final ldp vol3}
    \Psi^0_{\Lambda_n}[C_2]\leq e^{-c_2n^d}.
\end{equation}
Plugging \eqref{eq:final ldp vol1}, \eqref{eq:final ldp vol3}, and \eqref{eq:final ldp surf1} in \eqref{eq:final ldp vol2} yields the existence of $c_3>0$ such that for every $n$ large enough,
\begin{equation}
    \nu_{\Lambda_n}\Big[m_{\Lambda_n}\notin [-m^*-\delta,m^*+\delta]\Big]\leq e^{-c_3n^d}.
\end{equation}
\end{proof}

\section{Spectral gap decay}\label{section:spectral gap}

In this section, we prove Theorem~\ref{theorem: dynamics} as a consequence of Theorem~\ref{thm:ldp free}. We begin by introducing the necessary definitions before proceeding to the proof.

\subsection{Definition of dynamics}
Let $\Lambda \subset \mathbb Z^d$ be finite and let $\beta > 0$. A dynamical $\varphi^4$ model on $\Lambda$ at inverse temperature $\beta$ is a continuous-time Markov process evolving on the state space $\mathbb R^\Lambda$ with unique invariant measure $\nu_{\Lambda,\beta}$. As stated in the introduction, we focus on Langevin and heat-bath dynamics.

In order to define both dynamics precisely, let us recall some basic facts about Feller processes. Given a Markov process $(X_t)_{t \geq 0}$, we consider the semigroup $(P_t)_{t \geq 0}$ which acts on bounded measurable functions $f:\mathbb R^\Lambda \rightarrow \mathbb R$ via
\begin{equation}
P_t f (\varphi) := \mathbf E_\varphi[ f(X_t)],
\end{equation}
where $\mathbf E_\varphi$ denotes expectation with respect to the path measure $\mathbf P_\varphi$ under which the process is conditioned to satisfy $X_0=\varphi$ almost surely. Feller processes correspond to processes whose semigroups have a further regularity property when restricting its action to $C_0(\mathbb R^\Lambda)$, the space of continuous functions vanishing at infinity. Namely, for every $t \geq 0$, $P_t:C_0(\mathbb R^\Lambda) \rightarrow C_0(\mathbb R^\Lambda)$. It is well known that Feller processes are characterised by a generator, i.e.\ $P_t = e^{t\mathcal L}$ for some negative-definite linear operator $\mathcal L$ acting on a dense domain of $C_0(\mathbb R^\Lambda)$. Furthermore, we only consider examples where $\mathcal L$ is (or rather, extends to) a self-adjoint operator on $L^2(\nu_{\Lambda,\beta})$. This means that the process is reversible.

We now turn to the precise definition of Langevin and heat-bath dynamics.

\subsubsection{Langevin dynamics}

We begin by describing the Langevin dynamics. Recall that the graph Laplacian $\Delta_\Lambda: \mathbb R^\Lambda \rightarrow \mathbb R^\Lambda$ is the linear map defined by
\begin{equation}
(\Delta_\Lambda \varphi)_x = \sum_{y \sim x} (\varphi_y-\varphi_x), \qquad \forall \varphi \in \mathbb R^\Lambda, \, \forall x \in \Lambda.
\end{equation} 
Let us write
\begin{equation}
e^{-\beta H_{\Lambda}(\varphi)} \prod_{x \in \Lambda} {\rm d}\rho_{g,a}(\varphi_x) = e^{-U_{\Lambda,\beta}(\varphi)-\frac \beta 2\sum_{xy \in E}(\varphi_x-\varphi_y)^2} \prod_{x \in \Lambda}{\rm d}\varphi_x,
\end{equation}
where the potential $U_{\Lambda,\beta}$ is given by
\begin{equation}
U_{\Lambda,\beta}(\varphi) = \sum_{x \in V} \left( g\varphi_x^4 + (a-\beta {\rm deg}_\Lambda(x)/2)\varphi_x^2 \right),
\end{equation}
and ${\rm deg}_\Lambda(x)$ is the degree of $x$ in $\Lambda$.

The Langevin dynamics $(\mathbf X_t^{\rm LA})_{t \geq 0}$ started from $\varphi \in \mathbb R^\Lambda$ is the solution of a system of stochastic differential equations:
\begin{equation}
{\rm d} \mathbf X_t^{\rm LA} = \left(\beta \Delta_\Lambda \mathbf X^{\rm LA}_t -\nabla U_{\Lambda,\beta} ( \mathbf X^{\rm LA}_t) \right) {\rm d}t + \sqrt 2  {\rm d}B_t^{\Lambda}, \qquad \mathbf X^{\rm LA}_0 = \varphi,
\end{equation}
where $\nabla$ is the gradient on $\mathbb R^\Lambda$, and $(B_t^\Lambda)_{t \geq 0}$ is a vector of i.i.d.\ Brownian motions started from $0$ at each lattice point in $\Lambda$. It is classical that $(\mathbf X_t^{\rm LA})_{t \geq 0}$ defines a continuous-time Markov process with invariant measure $\nu_{\Lambda,\beta}$ and with corresponding generator acting on smooth functions $f:\mathbb R^\Lambda \rightarrow \mathbb R$ via
\begin{equation}
\mathcal L_{\Lambda_L,\beta}^{\rm LA} f(\varphi) = -\frac 12\nabla U_{\Lambda,\beta}(\varphi) \cdot \nabla f(\varphi) + \frac 12 \Delta_\Lambda f(\varphi).
\end{equation}
We refer to \cite{RT96} for further details.

\subsubsection{Heat-bath dynamics}

We now turn our attention to defining heat-bath dynamics, which is a general state-space Markov chain (see \cite{meyn2012markov}). We consider a family of \emph{transition rates}  $\{ p_{\Lambda,\beta}(\varphi, x;s) :\varphi \in \mathbb R^\Lambda, x \in \Lambda, s \in \mathbb R \}$. The transition rate $p_{\Lambda,\beta}(\varphi,x;s)$ is the infinitesimal rate at which $\varphi$ jumps to $\varphi^{(x,s)}$, which is the configuration obtained from $\varphi$ by replacing its value at $x$ by $s$. Heat-bath dynamics corresponds to defining the transition rates according to the conditional $\varphi^4$ measure on the singleton $\{x\}$ with boundary conditions given by $\varphi|_{\Lambda\setminus\{x\}}$. Let us give a precise definition: for every $\varphi$, $x$, and $s$, 
\begin{equation}
p_{\Lambda,\beta}(\varphi, x; s) := \nu_{\Lambda, \beta} (\varphi_x = s \mid \varphi_{\Lambda\setminus\{x\}}) = \frac{e^{ - gs^4-as^2 + \beta s \sum_{y \sim x}\varphi_y}}{\int_{\mathbb R} e^{- gt^4-at^2 + \beta t \sum_{y \sim x}\varphi_y} {\rm d}t }.
\end{equation}

We build a Markov process $(\mathbf X^{\rm HB}_t)_{t \geq 0}$ associated to these jump rates by a classical construction. Let $\mathcal P$ be a homogeneous Poisson point process of rate 1 on $\mathbb R^+\times \Lambda$. Let us consider a sample of $\mathcal P$, i.e.\ a sequence of random space-time points $\mathcal T:= ( (t_i,x_i)_{i \geq 1} : 0\leq  t_1 < t_2 < \dots < t_n \leq \dots, \, x_i \in \Lambda\}$. Since $\mathcal P$ has no atoms, the strict monotonicity of the times occurs almost surely. Starting from some initial configuration $\varphi^0$, or initial law $\tilde \nu$ on $\mathbb R^\Lambda$, we then sequentially update the process $(\mathbf X_t^{\rm HB})_{t \geq 0}$ conditionally on $\mathcal P$ at the space-time points $\mathcal T$, where each time the spin configuration is resampled according to the transition rates defined above.

The above construction yields a Markov process with invariant measure $\nu_{\Lambda,\beta}$ and generator $\mathcal L_{\Lambda,\beta}^{\rm HB}$ whose action for every $f \in C_0(\mathbb R^\Lambda)$ is given by
\begin{equation}
\mathcal L_{\Lambda,\beta}^{\rm HB} f(\varphi) = \int_{\mathbb R} \sum_{x \in \Lambda} p_{\Lambda,\beta}(\varphi,x;s)(f(\varphi^{(x,s)})-f(\varphi)) {\rm d}s, \qquad \forall \varphi \in \mathbb R^\Lambda. 
\end{equation}

\begin{remark}
We restrict to heat-bath dynamics for convenience. We expect our results to hold for a larger class of Glauber dynamics jump rates, but we do not pursue this direction further.
\end{remark}

\subsection{Spectral gap and proof of Theorem \ref{theorem: dynamics}}

We first recall some basic facts about spectral gaps in our context. As above, let $\mathcal L$ be the generator of a Feller process that is reversible with respect to/self-adjoint in (a dense domain) of $L^2(\nu_{\Lambda,\beta})$. In the cases we consider, $\mathcal L$ is a negative-definite linear operator and hence the spectrum is contained in $(-\infty,0]$. Furthermore, constant functions are eigenvectors of eigenvalue $0$. The spectral gap $\lambda(\mathcal L)$ is the largest $\lambda \geq 0$ such that the spectrum of $\mathcal L$ on the subspace orthogonal to constant functions is contained in $(-\infty,-\lambda]$. We can give an equivalent formulation in terms of the Dirichlet form $\mathcal{E}$ defined by:
\begin{equation}
\mathcal E(f):= \langle -\mathcal L f \cdot f \rangle_{\Lambda,\beta},  
\end{equation}
where $f$ is taken in the domain $\mathcal{D}(\mathcal{E})$ of $\mathcal E$ (appropriately defined). The spectral gap is characterised by the best constant in the Poincaré inequality, that is 
\begin{equation}
    \lambda(\mathcal{L})=\inf_{f \in \mathcal{D}(\mathcal{E}): {\rm Var}_{\Lambda,\beta}(f)\neq 0}\frac{\mathcal{E}(f)}{{\rm Var}_{\Lambda,\beta}(f)}.
\end{equation}

\begin{remark}
Recall the semigroup $(P_t)_{t \geq 0}$. For every $t \geq 0$, $P_t = e^{t\mathcal L}$. By this observation and standard results from spectral theory, $\lambda(\mathcal{L})>0$ implies an exponential relaxation of variances. For every $f$ in the domain of $\mathcal L$ and every $t \geq 0$,
\begin{equation}
{\rm Var}_{\Lambda,\beta} \left( \mathbf E_\varphi[f(X_t)] \right) \leq e^{-2\lambda(\mathcal{L}) t} {\rm Var}_{\Lambda,\beta}(f).
\end{equation}
\end{remark}

It turns out that the two dynamics we consider are known to have strictly positive spectral gaps. For Langevin dynamics, this follows from functional inequalities developed in \cite{ledoux2001logarithmic} (which apply since the potential $U_{\Lambda,\beta}$ is strictly convex at infinity). 

In the case of heat-bath dynamics, it is known that strict positivity of spectral gaps is equivalent to a condition known as \emph{geometric ergodicity} \cite{roberts2001geometric}. In order to state this, let us observe that, for every $\varphi$ and every $t\geq 0$, the map $A \mapsto \mathbf E_\varphi[\mathbbm 1_A(X_t)]$ defines a probability measure. We call this measure $P_t^\varphi$. Geometric ergodicity asserts that there exists $c>0$ such that, almost surely for every $\varphi \sim \nu_{\Lambda,\beta}$, there exists $C(\varphi)>0$, such that
\begin{equation}
\| P_t^\varphi - \nu_{\Lambda,\beta} \|_{{\rm TV}} \leq C(\varphi) e^{-c t},
\end{equation}
where $\|\cdot\|_{{\rm TV}}$ is the total variation distance of these measures. Let us stress that, unlike the case of finite state-space chains, we cannot take $C(\varphi)$ uniform in $\varphi$. The geometric ergodicity of heat-bath dynamics was established, as a special case of geometric ergodicity of Metropolis--Hastings chains (see the definition of Gibbs sampler in \cite{RR04}), in \cite{RT96b}.

\begin{remark}
For full disclosure, the aforementioned geometric ergodicity results are proven for the discrete-time analogues of heat-bath chains. As with finite state-space Markov chains, one can deduce ergodicity results for continuous-time chains from the results for discrete-time chains.
\end{remark}

We now prove the surface order exponential decay of spectral gaps for $\varphi^4$ dynamics in the supercritical regime. We only do the case of heat-bath dynamics since the proof for Langevin dynamics is similar and simpler (see for example the spectral gap estimates in \cite{CGW22}). Additionally, we drop the superscript ``HB'' in all the notations. 

\begin{proof}[Proof of Theorem~\textup{\ref{theorem: dynamics}}]

We let $d\geq 2$ and $\beta>\beta_c$. We also take $\Lambda=\Lambda_n$ for some $n$. Let $m \in (0,m^*(\beta)-3\delta)$ and let $\chi_m:\mathbb R \rightarrow [-1,1]$ be a smooth, increasing, and odd function satisfying: $\chi_m(a) \equiv -1$ for $a \leq -m$ and $\chi_m(a) \equiv 1$ for $a \geq m$. Let us write
\begin{equation}
    f(\varphi) := \chi_m(m_{\Lambda_n}(\varphi)).
\end{equation}
By symmetry, one has that $\langle f \rangle_{\Lambda_n,\beta} = 0$. Moreover,
\begin{equation}\label{eq: upper bound sg 0}
    {\rm Var}_{\Lambda_n,\beta}(f)=\langle f^2\rangle_{\Lambda_n,\beta}\geq 2\nu_{\Lambda_n,\beta}[m_{\Lambda_n}\geq m]\geq 2\nu_{\Lambda_n,\beta}[m_{\Lambda_n}\geq m^*(\beta)-3\delta].
\end{equation}
Recall the definition of the Dirichlet form. In the case of heat-bath dynamics, it is explicitly given by:
\begin{equation}\label{eq: upper bound sg 1}
    \mathcal E_{\Lambda_n,\beta}(f) = \frac 12 \sum_{x \in \Lambda_n} \Big\langle \int_{\mathbb R} p_{\Lambda_n,\beta}(\varphi,x;s) |f(\varphi)-f(\varphi^{(x,s)})|^2 {\rm d}s \Big\rangle_{\Lambda_n,\beta}.
\end{equation}
Now, observe that
\begin{equation}\label{eq: upper bound sg 1.5}
|f(\varphi)-f(\varphi^{(x,s)})| \leq |\chi_m'|_\infty \frac{|\varphi_x| + |s|}{n^d} \mathbbm 1 \{ m_{\Lambda_n} \in [-m-(|\varphi_x|+|s|)n^{-d},m+(|\varphi_x|+|s|)n^{-d}] \}.
\end{equation}
We need to control the contribution of the term $|\varphi_x|+|s|$. In order to do so, let us define the events
\begin{equation}
    \mathcal G_1(x):= \{ |\varphi_x|\leq \delta n^d \}, \qquad \mathcal G_2:= \{ |s|\leq \delta n^d \}.  
\end{equation}
Note that by the Markov property and regularity, there exists $c_1>0$ such that, for every $n$ large enough,
\begin{equation}\label{eq: upper bound sg 2}
\Big\langle\int_{\mathbb R} p_{\Lambda_n,\beta}(\varphi,x;s) \mathbbm 1\{ \mathcal{G}_2^c \} {\rm d}s \Big\rangle_{\Lambda_n,\beta } = \nu_{\Lambda_n,\beta}[\mathcal{G}_1(x)^c] \leq \exp(-c_1 n^{4d}). 
\end{equation}
Thus, since $|f|_\infty \leq 1$, there exists $c_2>0$ such that, for every $n$ large enough,
\begin{align}
\sum_{x \in \Lambda_n}\Big\langle \int_{\mathbb R} p_{\Lambda_n,\beta}(\varphi,x;s) |f(\varphi)-f(\varphi^{(x,s)})|^2 &(\mathbbm 1\{ \mathcal{G}_1(x)^c \} + \mathbbm 1\{ \mathcal{G}_2^c\}){\rm d}s  \Big\rangle_{\Lambda_n,\beta} 
\\&\leq 8 \sum_{x\in \Lambda_n}\nu_{\Lambda_n,\beta}[\mathcal{G}_1(x)^c]\\&\leq\exp(-c_2 n^{4d}),\label{eq: upper bound sg 2.5}
\end{align}
where we used the fact that $\int_{\mathbb R}p_{\Lambda_n,\beta}(\varphi,x;s){\rm d}s=1$ and $\Vert f\Vert_\infty=1$ in the second line, and \eqref{eq: upper bound sg 2} in the third line.
Let us define the event
\begin{equation}
\mathcal{E} = \{ m_{\Lambda_n} \in (-m^*(\beta)+\delta, m^*(\beta)-\delta) \}.
\end{equation}
Using Theorem \ref{thm:ldp free}, there exists $c_3>0$ such that for every $n$ large enough,
\begin{equation}\label{eq: upper bound sg 3}
\nu_{\Lambda_n,\beta}[\mathcal{E}] \leq \exp(-c_3 n^{d-1}).
\end{equation}
Then, there exists $c_4>0$ such that for every $n$ large enough,
\begin{align}
\sum_{x \in \Lambda_n}\Big\langle \int_{\mathbb R} p_{\Lambda_n,\beta}(\varphi,x;s) |f(\varphi)&-f(\varphi^{(x,s)})|^2 (\mathbbm 1\{ \mathcal{G}_1(x) \cap \mathcal{G}_2 \}){\rm d}s  \Big\rangle_{\Lambda_n,\beta} 
\\&\leq \frac{4|\chi_m'|_\infty^2}{n^{2d}}\sum_{x \in \Lambda_n}\Big\langle \int_{\mathbb R} p_{\Lambda_n,\beta}(\varphi,x;s)(|\varphi_x|^2+|s|^2) \mathbbm 1\{\mathcal{E}\}{\rm d}s  \Big\rangle_{\Lambda_n,\beta}
\\&\leq \frac{8|\chi_m'|_\infty^2}{n^{2d}}\sum_{x\in \Lambda_n}\nu_{\Lambda_n,\beta}[|\varphi_x|^4]^{1/2}\nu_{\Lambda_n,\beta}[\mathcal{E}]^{1/2}
\\&\leq \exp(-c_4 n^{d-1}),\label{eq: upper bound sg 4}
\end{align}
where in the second line we used \eqref{eq: upper bound sg 1.5} and the fact that $\{ m_{\Lambda_n} \in [-m-(|\varphi_x|+|s|)n^{-d},m+(|\varphi_x|+|s|)n^{-d}] \}\cap \mathcal{G}_1\cap \mathcal{G}_2\subset \mathcal{E}$, in the third line we used the Cauchy--Schwarz inequality (twice)--- applied to the measure $\nu_{\Lambda_n,\beta}({\rm d}\varphi)p_{\Lambda_n,\beta}(\varphi,x;s){\rm d}s$--- to argue that 
\begin{align}
\Big\langle \int_{\mathbb R} p_{\Lambda_n,\beta}(\varphi,x;s)|s|^2\mathbbm 1\{\mathcal{E}\}{\rm d}s  \Big\rangle_{\Lambda_n,\beta}&\leq \Big\langle \int_{\mathbb R}p_{\Lambda_n,\beta}(\varphi,x;s)|s|^4\mathrm{d}s\Big\rangle_{\Lambda_n,\beta}^{1/2}\nu_{\Lambda_n,\beta}[\mathcal{E}]^{1/2}\\&=\langle \varphi_x^4\rangle_{\Lambda_n,\beta}^{1/2}\nu_{\Lambda_n,\beta}[\mathcal{E}]^{1/2},
\end{align}
and that
\begin{equation}
\Big\langle \int_{\mathbb R} p_{\Lambda_n,\beta}(\varphi,x;s)|\varphi_x|^2\mathbbm 1\{\mathcal{E}\}{\rm d}s  \Big\rangle_{\Lambda_n,\beta}\leq \langle \varphi_x^4\rangle_{\Lambda_n,\beta}^{1/2}\nu_{\Lambda_n,\beta}[\mathcal{E}]^{1/2},
\end{equation}
and in the fourth line we used \eqref{eq: upper bound sg 3}.

Combining \eqref{eq: upper bound sg 1}, \eqref{eq: upper bound sg 1.5}, \eqref{eq: upper bound sg 2.5}, and \eqref{eq: upper bound sg 4} yields, for every $n$ large enough
\begin{equation}
    \mathcal{E}_{\Lambda_n,\beta}(f)\leq \exp(-c_5n^{d-1}).
\end{equation}
Moreover, \eqref{eq: upper bound sg 0} and \eqref{eq: upper bound sg 3} imply that, for every $n$ large enough,
\begin{equation}
    \textup{Var}_{\Lambda_n,\beta}(f)\geq \frac{1}{2}.
\end{equation}
Putting the two last displayed equations together yields, for every $n$ large enough,
\begin{equation}
    \lambda(\Lambda_n)\leq \frac{\mathcal{E}_{\Lambda_n,\beta}(f)}{\textup{Var}_{\Lambda_n,\beta}(f)}\leq 2\exp(-c_5n^{d-1}),
\end{equation}
and concludes the proof.
\end{proof}

\appendix

\section{Toolbox}

\subsection{Large degree deviations}

In this section, we obtain a large deviations estimate for the degrees $\Delta{\n}(x)$ of a random current $\n$. Recall that for $e=uv$, $\varphi_e$ denotes the product $\varphi_u\varphi_v$, and by convention $\varphi_{\fg}\equiv 1$.

\begin{lemma}\label{lem:degree_exp.moments} 
For every $\Lambda\subset \Z^d$, $o\in\Lambda$, $\mathsf{h}\in(\mathbb{R}^+)^\Lambda$ and every set of edges $\mathcal{E}$ in $\Lambda[\mathsf{h}]$, we have
\begin{equation}\label{eq:degree_exp.moments_sources}
\langle\varphi_o\rangle_{\Lambda,\beta,\mathsf{h}}\mathbf{E}_{\Lambda[\mathsf{h}],\beta}^{o\fg} \Big[2^{\sum_{e\in\mathcal{E}} \n_e}\Big] = \langle \varphi_o\prod_{e\in \mathcal{E}}\exp(J_e\varphi_e)\rangle_{\Lambda,\beta,\mathsf{h}}
\end{equation}
and
\begin{equation}\label{eq:degree_exp.moments}
\mathbf{E}_{\Lambda[\mathsf{h}],\beta}^{\emptyset}\Big[2^{\sum_{e\in\mathcal{E}} \n_e}\Big] = \langle \prod_{e\in \mathcal{E}}\exp(J_e \varphi_e)\rangle_{\Lambda,\beta,\mathsf{h}},
\end{equation}
where $J_e:=\beta \mathbbm{1}_{e\in E(\Lambda)}+\sum_{x\in \Lambda}\beta \mathsf{h}_x\mathbbm{1}_{e=x\fg}$.
In particular, if the endpoints of $\mathcal{E}$ are all contained in $\Delta\cup\{\fg\}$, for a set $\Delta\subset \Lambda$ satisfying the stochastic domination \eqref{eq:regularity} from Proposition~\textup{\ref{prop:regularity}} with some $C\in(0,\infty)$, then the right hand sides of both \eqref{eq:degree_exp.moments_sources} and \eqref{eq:degree_exp.moments} are smaller than $e^{C'|\mathcal{E}|}$ for some constant $C'\in(0,\infty)$ depending only on $C$. 
\end{lemma}

\begin{proof} 
We only prove \eqref{eq:degree_exp.moments_sources}. The proof of \eqref{eq:degree_exp.moments} is similar and simpler. For ease of notation, we write $\mathbf{P}=\mathbf{P}_{\Lambda[\mathsf{h}],\beta}^{o\fg}$. 
We use the binomial theorem to write 
\begin{equation}\label{eq: binomial}
    \mathbf{E}\left[2^{\sum_{e\in \mathcal{E}}\n_e}\right]=\sum_{e\in \mathcal{E}}\sum_{k_e\geq 0}\mathbf{E}\left[\prod_{e\in \mathcal{E}}\binom{\n_e}{k_e}\mathbbm{1}_{\n_e\geq k_e}\right].
\end{equation}
To express the latter, given a current $\n$ that satisfies $\n_e\geq k_e$ for all $e\in \mathcal{E}$, let $\tilde \n$ be the current defined by $\tilde\n_f:=\n_f-k_f\mathbbm 1\{f\in \mathcal{E}\}$. With this definition in hands, we can write
\begin{equation}
    w^{o\fg}_{\beta,\mathsf{h}}(\n)\prod_{e\in \mathcal{E}}\binom{\n_e}{k_e}=w^{A}_{\beta,\mathsf{h}}(\tilde \n)\prod_{e\in \mathcal{ E}\cap E(\Lambda)}\frac{\beta^{k_e}}{k_e!}\prod_{\substack{x\in \Lambda\\e=x\fg\in \mathcal{E}}}\frac{(\beta\mathsf{h}_x)^{k_e}}{k_e!},
\end{equation}
where $A\in \mathcal{M}(\Lambda^\fg)$ is defined by 
\begin{align}
    A_x&:=\sum_{\substack{e\in \mathcal{E}\\ \: e\ni x}}k_e+\mathbbm{1}_{x=o}, \quad x\in \Lambda, 
    \\A_\fg&:=\sum_{\substack{e\in \mathcal{E}\\e\ni \fg}}k_e+1\text{ mod }2.
\end{align}
Moreover, $\partial\tilde \n=\partial A$. 
Using Lemma~\ref{lem: current expansion} we get
\begin{equation}
\begin{aligned}
    \mathbf{E}\left[\prod_{e\in \mathcal{E}}\binom{\n_e}{k_e}\mathbbm{1}_{\n_e\geq k_e}\right]&=\frac{1}{\langle \varphi_o \rangle_{\Lambda,\beta,\mathsf{h}}}\frac{\sum_{\partial \tilde \n=\partial A} w^A_{\beta,\mathsf{h}}(\tilde\n)}{\sum_{\partial \tilde \n=\emptyset} w_{\beta,\mathsf{h}}(\tilde\n)}\prod_{e\in \mathcal{ E}\cap E(\Lambda)}\frac{\beta^{k_e}}{k_e!}\prod_{\substack{x\in \Lambda\\e=x\fg\in \mathcal{E}}}\frac{(\beta\mathsf{h}_x)^{k_e}}{k_e!}\\
    &=\frac{\langle \prod_{x\in \Lambda}\varphi_x^{A_x}\rangle_{\Lambda,\beta,\mathsf{h}}}{\langle \varphi_o\rangle_{\Lambda,\beta,\mathsf{h}}}\prod_{e\in \mathcal{ E}\cap E(\Lambda)}\frac{\beta^{k_e}}{k_e!}\prod_{\substack{x\in \Lambda\\e=x\fg\in \mathcal{E}}}\frac{(\beta\mathsf{h}_x)^{k_e}}{k_e!}
    \\&=
    \frac{\langle \varphi_o\prod_{e\in \mathcal{E}\cap E(\Lambda)}(\beta\varphi_e)^{k_e}\prod_{\substack{x\in \Lambda\\e=x\fg\in \mathcal{E}}}(\beta \mathsf{h}_x\varphi_x)^{k_e}\rangle_{\Lambda,\beta,\mathsf{h}}}{\langle \varphi_o\rangle_{\Lambda,\beta,\mathsf{h}}}\prod_{e\in \mathcal{E}}\frac{1}{k_e!}.
\end{aligned}    
\end{equation}
We can now use \eqref{eq: binomial} to deduce that 
\begin{equation}
    \mathbf{E}\left[2^{\sum_{e\in \mathcal{E}}\n_e}\right]=
    \frac{\langle \varphi_o\prod_{e\in \mathcal{E}}\exp(J_e\varphi_e)\rangle_{\Lambda,\beta,\mathsf{h}}}{\langle \varphi_o\rangle_{\Lambda,\beta,\mathsf{h}}}.
\end{equation}
\end{proof}

\subsection{Tanglings estimates}\label{appendix:strict stochastic dom tanglings}

In this section we prove Proposition \ref{prop: tanglings ECT}.
We recall from \cite{GunaratnamPanagiotisPanisSeveroPhi42022} that the measure $\rho^{2k}$ (resp.\ $\rho^{2k_1,2k_2}$) is constructed from taking a weak limit of the single (resp.\ double) random current measure associated with a near-critical Ising model on the complete graph $K_n$. We refer to \cite{GunaratnamPanagiotisPanisSeveroPhi42022,KPP24} for more details. In order to prove the desired result, we will use some properties of this random current expansion established in \cite{KPP24}. We denote the random current measure on $K_n$ with source set $S$ by $P^S_{K_n}$.

\begin{proof}[Proof of Proposition \textup{\ref{prop: tanglings ECT}}]
Let $S_1=\{1,2,\ldots,2k_1\}$ and $S_2=\{2k_1+1,2k_1+2,\ldots,2k_1+2k_2\}$. Let $\n_1$ and $\n_2$ be random currents on $K_n$ with source sets $S_1$ and $S_2$, respectively. Let $\Pi_i(\n_i)$ be the random partition of $S_i$ induced by the current $\n_i$, where $x,y$ are in the same partition class if and only if $x\longleftrightarrow y$ in $\n_i$. Define similarly $\Pi(\n_1,\n_2)$ to be the random partition of $S_1\cup S_2$ induced by the current $\n_1+\n_2$. We will show that for every even partition $P_1$ of $S_1$ and every even partition $P_2$ of $S_2$, there exists $\varepsilon>0$ such that
\begin{equation}
P^{S_1,S_2}_{K_n}[\Pi(\n_1,\n_2)=\{S_1\cup S_2\} \mid \Pi_1(\n_1)=P_1, \Pi_2(\n_2)=P_2]\geq \varepsilon    
\end{equation}
for every $n$ large enough. The desired result will then follow by taking the limit as $n$ tends to infinity.

Indeed, by the main result of \cite{KPP24}, there exist $\delta>0$ and $c>0$ such that 
\begin{equation}
P^{S_1}_{K_n}[\Pi_1(\n_1)=P_1, |C_{\n_1}(i)|\geq c\sqrt{n} \; \forall i\in S_1]\geq \delta     
\end{equation}
and 
\begin{equation}
P^{S_2}_{K_n}[\Pi_2(\n_2)=P_2, \sqrt{n}/c\geq |C_{\n_2}(i)|\geq c\sqrt{n} \; \forall i\in S_2]\geq \delta     
\end{equation}
for every $n$ large enough, where $C_{\n_i}(x)$ denotes the connected component of $x$ induced by the current $\n_i$, and $|C_{\n_i}(x)|$ denotes its cardinality (number of vertices). Our aim is to show that conditionally on the above events, there is positive probability that each cluster $C_{\n_2}(i)$ with $i\in S_2$ intersects all clusters $C_{\n_1}(j)$ with $j\in S_1$. The desired result then follows readily.

To this end, it suffices to show that for every $i\in S_2$, conditionally on the clusters $C_{\n_1}(1),\ldots,C_{\n_1}(2k_1)$ and $C_{\n_2}(2k_1+1),\ldots,C_{\n_2}(i-1)$, the cluster $C_{\n_2}(i)$ has positive probability to intersect all clusters $C_{\n_1}(1),\ldots,C_{\n_1}(2k_1)$. Note that under this conditioning, $C_{\n_2}(i)\setminus S_2$ is sampled uniformly at random from $K_n$ without $C_{\n_2}(2k_1+1), \ldots, C_{\n_2}(i-1)$ and $S_2$. For ease of notation we prove the above only for $i=2k_1+1$, and the general case follows similarly, up to changing the value of $n$. 

Let $P_1=\{Q_1,\ldots, Q_j\}$ be an even partition of $S_1$, and on the event $\{\Pi(\n_1)=P_1\}$, write $C_{\n_1}(Q_\ell)$ for the cluster of $Q_\ell$ in $\n_1$. Denote $A=A(B_1,\ldots,B_j,r)$ the event $\{C_{\n_1}(Q_\ell)=B_\ell$ for every $\ell\in \{1,2,\ldots,j\}, |C_{\n_2}(2k_1+1)|=r\}$, where the sets $B_\ell$ are possible realisations satisfying $|B_\ell|\geq c\sqrt{n}$, and where $\sqrt{n}/c\geq r\geq c\sqrt{n}$. 
Note that under $P^{S_2}_{K_n}[\:\cdot \mid |C_{\n_2}(2k_1+1)|=r]$, the vertices lying in $C_{\n_2}(2k_1+1)\setminus \{2k_1+1,\ldots,2k_1+2k_2\}$ are distributed uniformly at random among subsets of cardinality $r-2k_2$ chosen from a set of cardinality $n-2k_2$. Furthermore, the probability that the set of these $r-2k_2$ points intersects each $B_\ell$ is increasing as a function of $r$ and the size of each $B_\ell$. Thus, we may assume that each $B_\ell$ has size $r_0:=\lfloor c\sqrt{n} \rfloor$, and $|C_{\n_2}(2k_1+1)|=r_0$. By asking $C_{\n_2}(2k_1+1)$ to contain exactly one point from each $B_\ell$ we obtain the lower bound
\begin{equation}
P^{S_1,S_2}_{K_n}[2k_1+1 \connect{\n_1+\n_2\:} Q_\ell, \: \forall \ell=1,2,\ldots,j \mid A]\geq \frac{(r_0-2k_1-2k_2)^j\binom{n-m-2k_2}{r_0-j-2k_2}}{\binom{n-2k_2}{r_0-2k_2}},    
\end{equation}
where $m:=j(r_0-2k_1-2k_2)$. Here the term $(r_0-2k_1-2k_2)^j$ is a lower bound for the number of ways to choose exactly one vertex from each $B_\ell$, and $\binom{n-m-2k_2}{r_0-j-2k_2}$ is a lower bound for the number of ways to choose the remaining $r_0-j-2k_2$ vertices. 

For every $n$ large enough we have
\begin{equation}
\begin{aligned}
\frac{\binom{n-m-2k_2}{r_0-j-2k_2}}{\binom{n-2k_2}{r_0-j-2k_2}}&=\frac{(n-m-2k_2)!(n-r_0+j)!}{(n-2k_2)!(n-m-r_0+j)!}=\prod_{p=0}^{m-1} \frac{n-r_0+j-p}{n-2k_2-p}\\&=\prod_{p=0}^{m-1} \left(1-\frac{r_0-j-2k_2}{n-2k_2-p}\right)
\geq \left(1-\frac{2r_0}{n}\right)^m=(1+o(1))e^{-2jc^2}.
\end{aligned}
\end{equation} 
Moreover, for some constant $c'>0$ we have
\begin{equation}
\begin{aligned}
\frac{\binom{n-2k_2}{r_0-j-2k_2}}{\binom{n-2k_2}{r_0-2k_2}}=
\frac{(r_0-2k_2)!(n-r_0)!}{(r_0-j-2k_2)!(n-r_0+j)!}
=(1+o(1))\frac{r_0^j}{n^j}\geq \frac{c'}{(r_0-2k_1-2k_2)^j}.
\end{aligned}
\end{equation}
This implies that 
\begin{equation}
P^{S_1,S_2}_{K_n}[2k_1+1 \connect{\n_1+\n_2\:} Q_\ell ,\; \forall \ell=1,2,\ldots,j \mid A]\geq (1+o(1))c' e^{-2jc^2}.   
\end{equation}
The desired result follows.
\end{proof}

\subsection{Large deviation for highly supercritical percolation}\label{sec:appendix_high_perco}

Here we prove Lemma~\ref{lem:Bernoulli_surf_dev} by adapting \cite{DP96}. We start by introducing some notations and recalling basic facts. Given a graph $G=(V,E)$ and a subset $S\subset V$, we define the inner and outer boundaries  $\partial^{\rm in}_G S\coloneqq\{x\in S:~\exists\, y\in V\setminus S,\, \{x,y\}\in E\}$ and $\partial^{\rm out}_G S\coloneqq\{y\in V\setminus S:~\exists\, x\in S,\, \{x,y\}\in E\}$. We define the notion of $*$-connectivity in $\Z^d$, where every pair of vertices $x,y$ satisfying $|x-y|_\infty=1$ are connected by an edge. 
We will use the fact, proved in \cite[Lemma 2.1 (ii)]{DP96}, that for every connected set $C\subset \Lambda_m$ and every $S\in\mathfrak{S}(\Lambda_m\setminus C)$, both $\partial^{\rm in}_{\Lambda_m} S$ and $\partial^{\rm out}_{\Lambda_m} S$ are $*$-connected. 

Our main tool will be a local isoperimetric inequality for boxes of $\Z^d$, $d\geq2$, proved in \cite[Proposition 2.2]{DP96}: for every $\varepsilon>0$, there exists a constant $c=c(\varepsilon)>0$ such that for every $m\geq 1$ and every connected set $S\subset \Lambda_m$ satisfying $|S|\leq (1-\varepsilon)|\Lambda_m|$, we have
\begin{equation}\label{eq:local_isoperimetric_ineq}
    \sum_{i=1}^{\ell} |\Delta_i S|^{\frac{d}{d-1}} \geq c |S|,
\end{equation}
where $\Delta S$ can be either $\partial^{\rm in}_{\Lambda_m} S$ or $\partial^{\rm out}_{\Lambda_m} S$ and $(\Delta_i S)_{i=1}^{\ell}$ are the $*$-connected components of $\Delta S$.

\begin{proof}[Proof of Lemma~\textup{\ref{lem:Bernoulli_surf_dev}}]
Fix $m\geq1$. Recall that $\mathfrak{C}=\mathfrak{C}(\Lambda_m)$ denotes the set of open clusters.  We also denote by $\mathfrak{C}^*=\mathfrak{C}^*(\Lambda_m)$ to be the set of closed $*$-connected clusters. 

Assume that $|\mathcal{C}|< \tfrac{3}{4}|\Lambda_m|$ for every $\mathcal{C}\in\mathfrak{C}$. For convenience, we extend $\mathfrak{C}$ to $\overline{\mathfrak{C}}$ by including each closed vertex as a singleton. For every $\mathcal{C}\in \overline{\mathfrak{C}}$, let $\Delta \mathcal{C}=\partial^{\rm out}_{\Lambda_m} \mathcal{C}$ if $\mathcal{C}\in\mathfrak{C}$, and $\Delta \mathcal{C}=\{x\}$ if $\mathcal{C}=\{x\}$ is a closed singleton. In particular, the $*$-connected components $(\Delta_i \mathcal{C})_{i=1}^{\ell(\mathcal{C})}$ of $\Delta \mathcal{C}$, $\mathcal{C}\in \overline{\mathfrak{C}}$, are all fully closed and, by our assumption, satisfy the local isoperimetric inequality \eqref{eq:local_isoperimetric_ineq} with $c=c(1/4)$. Since each vertex can only appear in at most $2d+1$ many components $\Delta_i \mathcal{C}$, one has $(2d+1)|\mathcal{C}^*|\geq \sum_{\Delta_i \mathcal{C} \subset \mathcal{C}^* } |\Delta_i \mathcal{C}|$. Taking both sides to the power $d/(d-1)$ and using \eqref{eq:local_isoperimetric_ineq}, we conclude that
\begin{equation}
|\Lambda_m|=\sum_{\mathcal{C}\in\overline{\mathfrak{C}}} |\mathcal{C}| \leq \tfrac{1}{c} \sum_{\mathcal{C}\in\overline{\mathfrak{C}}} \sum_{i=1}^{\ell(\mathcal{C})} |\Delta_i \mathcal{C}|^{\frac{d}{d-1}} \leq \tfrac{1}{c'} \sum_{\mathcal{C}^*\in \mathfrak{C}^*} |\mathcal{C}^*|^{\frac{d}{d-1}},
\end{equation}
where $c'=c/(2d+1)^{\frac{d}{d-1}}$. 

Let us now assume that there exists a component $\mathcal{C}\in\mathfrak{C}$ such that $|\mathcal{C}|\geq \tfrac{3}{4}|\Lambda_m|$, but that $\sum_{\substack{\mathcal{S}\in\mathfrak{S}(\Lambda_m\setminus \mathcal{C})\\ |\mathcal{S}|\geq M}} |\mathcal{S}| > \varepsilon|\Lambda_m|$. 
For each $\mathcal{S}\in\mathfrak{S}(\Lambda_m\setminus \mathcal{C})$, let $\Delta \mathcal{S} \coloneqq \partial_{\Lambda_m}^{\rm in} \mathcal{S}$ and notice that $\Delta \mathcal{S}$ is closed, $*$-connected and satisfies the local isoperimetric inequality \eqref{eq:local_isoperimetric_ineq}. Using again that each vertex appears in at most $2d+1$ many $\Delta \mathcal{S}$, we conclude that
\begin{equation}
\varepsilon|\Lambda_m|< \sum_{ \substack{\mathcal{S}\in\mathfrak{S}(\Lambda_m\setminus \mathcal{C})\\ |\mathcal{S}|\geq M}} |\mathcal{S}| \leq \tfrac{1}{c} \sum_{ \substack{\mathcal{S}\in\mathfrak{S}(\Lambda_m\setminus \mathcal{C})\\ |\Delta \mathcal{S}|\geq (cM)^{(d-1)/d} }} |\Delta \mathcal{S}|^{\frac{d}{d-1}} \leq \tfrac{1}{c'} \sum_{\substack{ \mathcal{C}^*\in \mathfrak{C}^* \\ |\mathcal{C}^*|\geq (cM)^{(d-1)/d} }} |\mathcal{C}^*|^{\frac{d}{d-1}}.
\end{equation}

Combining the two last paragraphs, we have
\begin{equation}
    1-\mathbb{P}^{\textup{site}}_{p_0}[\mathcal{E}]\leq  \mathbb{P}^{\textup{site}}_{p_0} \Bigg[ \sum_{\mathcal{C}^*\in \mathfrak{C}^*} |\mathcal{C}^*|^{\frac{d}{d-1}} \geq c' |\Lambda_m| \Bigg]  + \mathbb{P}^{\textup{site}}_{p_0} \Bigg[ \sum_{\substack{ \mathcal{C}^*\in \mathfrak{C}^*\\ |\mathcal{C}^*|\geq (cM)^{{(d-1)}/{d}} }} |\mathcal{C}^*|^{\frac{d}{d-1}} \geq c'\varepsilon |\Lambda_m| \Bigg],
\end{equation}
where $\mathcal{E}$ is the event in \eqref{eq:Bernoulli_surf_dev}.
Under a certain probability measure $\mathbb{P}$, let ($\tilde{\mathcal{C}}^*_x)_{x\in \Lambda_m}$ be independent random variables, each distributed as the closed $*$-connected cluster of $x$ under $\mathbb{P}^{\textup{site}}_{p_0}$.
As proved in \cite[Lemma 2.3]{DP96}, for every increasing function $\rho:\N\to \R_+$, we have that $\sum_{\mathcal{C}^* \in\mathfrak{C}(\Lambda_m)} \rho(|\mathcal{C}^*|)$ is stochastically dominated by $\sum_{x\in\Lambda_m} \rho(|\tilde{\mathcal{C}}^*_x|)$.
Therefore,
\begin{equation}\label{eq:Bernoulli_surf_dom}
    1-\mathbb{P}^{\textup{site}}_{p_0}[\mathcal{E}]\leq  
    \mathbb{P} \Bigg[ \sum_{x\in\Lambda_m} |\tilde{\mathcal{C}}^*_x|^{\frac{d}{d-1}} \geq c' |\Lambda_m| \Bigg]  + \mathbb{P} \Bigg[ \sum_{x\in\Lambda_m } |\tilde{\mathcal{C}}^*_x|^{\frac{d}{d-1}} \mathbbm{1}\{|\tilde{\mathcal{C}}^*_x|\geq (c M)^{\frac{d-1}{d}}\} \geq c'\varepsilon |\Lambda_m| \Bigg].
\end{equation}
It is standard (see e.g.~\cite[Equation (3.8)]{DP96}) to prove that for $p_0<1$ sufficiently close to $1$, the tail $\mathbb{P}[|\tilde{\mathcal{C}}^*_x|\geq k]$ of the random variable $|\tilde{\mathcal{C}}^*_x|$ decays exponentially fast to $0$ in $k$. We can further choose $p_0<1$ such that $\mathbb{E}[|\tilde{\mathcal{C}}^*_x|]\leq c'/2$, and then choose $M=M(\varepsilon)\geq 1$ large enough such that $\mathbb{E}[|\tilde{\mathcal{C}}^*_x| \mathbbm{1}\{|\tilde{\mathcal{C}}^*_x|\geq (c M)^{\frac{d-1}{d}}\}]\leq c'\varepsilon/2$. It then follows from standard large deviation estimates for i.i.d.~random variables (see e.g.~\cite{durrett2019probability}) that both probabilities in the right hand side of \eqref{eq:Bernoulli_surf_dom} are decay exponentially in $|\Lambda_m|^{\frac{d-1}{d}} \asymp m^{d-1}$, as we wanted to prove.
\end{proof}

\section{The weak plus measure}

In this section, we consider the question of how small can a boundary field $\mathsf{h}_L$ be so that $\langle \cdot \rangle_{\Lambda_L,\beta,\mathsf{h}_L}\underset{L\to\infty}{\longrightarrow} \langle \cdot \rangle^+_\beta$. We show that the convergence holds for boundary fields satisfying $\mathsf{h}_L\leq \mathfrak{p}_{\Lambda_L}$ and $L^{d-1}\mathsf{h}_L\to\infty$, which was claimed in Remark~\ref{rem:def_p}. We first need the following lemma.

\begin{lemma}\label{lem: coupling lemma}
Let $\Lambda\subset \mathbb{Z}^d$ be a finite set and let $\eta\in \mathbb{R}^{\partial \Lambda}$ such that $|\eta|\leq \mathfrak{p}=\mathfrak{p}_{\Lambda}$. There exists a coupling $(\mathbb{P}_{\Lambda,\beta},\varphi^\mathfrak{p},\varphi^{\eta})$, $\varphi^\mathfrak{p}\sim \langle \cdot \rangle_{\Lambda,\beta,\mathfrak{p}}$, $\varphi^{\eta}\sim \langle \cdot \rangle_{\Lambda,\beta,\eta}$, such that $\mathbb{P}_{\Lambda,\beta}$-almost surely, $|\varphi^\mathfrak{p}_x|\geq |\varphi^{\eta}_x|$ and
\begin{equation}
 \mathbb{E}_{\Lambda,\beta}\left[\textup{sgn}(\varphi^\mathfrak{p}_x)-\textup{sgn}(\varphi^{\eta}_x) \mid |\varphi^\mathfrak{p}|,|\varphi^{\eta}|\right]
\geq 0
\end{equation}
for every $x\in \Lambda$.
\end{lemma}
\begin{proof}
Using Lemma 2.13 in \cite{GunaratnamPanagiotisPanisSeveroPhi42022} and Strassen's theorem \cite{Strassen1965}, we obtain a coupling $(\mathbb{P}_{\Lambda,\beta},\psi^\mathfrak{p},\psi^{\eta})$, where $\psi^\mathfrak{p}\sim \langle |\cdot |\rangle_{\Lambda,\beta,\mathfrak{p}}$ and $\psi^{\eta}\sim \langle |\cdot |\rangle_{\Lambda,\beta,\eta}$, such that $\mathbb{P}_{\Lambda,\beta}$-almost surely, 
\begin{equation}
    \psi^\mathfrak{p}_x\geq \psi^{\eta}_x \quad \text{for every } x\in \Lambda.
\end{equation}
Enlarging our probability space, we can assume that in the same probability space, there is a family $I=\{\sigma^\mathfrak{p}(\mathsf{a}) \sim \langle \cdot \rangle^{\textup{Ising},0}_{\Lambda,\beta,\mathfrak{p},\mathsf{a}}, \sigma^{\eta}(\mathsf{a}) \sim \langle \cdot \rangle^{\textup{Ising},0}_{\Lambda,\beta,\eta,\mathsf{a}}  \mid \mathsf{a}\in (\mathbb{R}^+)^{\Lambda}\}$ of independent Ising models that are also independent from $\psi^\mathfrak{p}$ and $\psi^{\eta}$. Now let $\varphi^\mathfrak{p}:=\psi^\mathfrak{p} \cdot \sigma^\mathfrak{p}(\psi^\mathfrak{p})$ and $\varphi^\eta:=\psi^\eta \cdot\sigma^\eta(\psi^\eta)$. It follows that $\varphi^\mathfrak{p}\sim \langle \cdot \rangle_{\Lambda,\beta,\mathfrak{p}}$, $\varphi^{\eta}\sim \langle \cdot \rangle_{\Lambda,\beta,\eta}$, and $|\varphi^\mathfrak{p}_x|\geq |\varphi^{\eta}_x|$ for every $x\in \Lambda$.
Moreover, by the Ginibre
inequality for the Ising model,
\begin{equation}
\mathbb{E}_{\Lambda,\beta}\left[\textup{sgn}(\varphi^\mathfrak{p}_x) -\textup{sgn}(\varphi^\eta_x)  \mid |\varphi^\mathfrak{p}|,|\varphi^\eta|\right]
=\langle \sigma_x  \rangle^{\textup{Ising},0}_{\Lambda,\beta,\mathfrak{p},\psi^\mathfrak{p}}-\langle \sigma_x \rangle^{\textup{Ising},0}_{\Lambda,\beta,\eta, \psi^\eta} \geq 0.
\end{equation}
\end{proof}

We now prove the main result of this section.

\begin{proposition}\label{prop: weak plus measure}
Let $\beta>\beta_c$. Let $\eta_L \in (\mathbb{R}^+)^{\partial \Lambda_L}$ be an external magnetic field such that $\eta_L\leq \mathfrak{p}_{\Lambda_L}$ and $L^{d-1}\inf_{x\in \partial \Lambda}\eta_L(x)\to\infty$ as $L\to \infty$. Then for every $x\in \mathbb Z^d$, 
\begin{equation}\label{eq:weak plus equal plus}
\lim_{L\to\infty}\langle \varphi_x \rangle_{\Lambda_L,\beta,\eta_L}= \langle \varphi_x \rangle^+_{\beta}.    
\end{equation}
As a consequence, the sequence of measures $(\langle \cdot \rangle_{\Lambda_L,\beta,\eta_L})_{L\geq 1}$ converges weakly towards $\langle\cdot\rangle_\beta^+$.
\end{proposition}
\begin{proof} Assuming the first part, the second part of the statement follows from the strategy described in Remark \ref{rem: equality mag implies equality meas}. We therefore focus on proving \eqref{eq:weak plus equal plus}. We do the computation for $x=0$ as the general case follows by the same argument. 

We claim that it suffices to show that  
\begin{equation}\label{eq: sign convergence}
\lim_{L\to\infty}\langle \textup{sgn}(\varphi_0) \rangle_{\Lambda_L,\beta,\eta_L}= \langle \textup{sgn}(\varphi_0) \rangle^+_{\beta}.    
\end{equation}
Indeed, assume that \eqref{eq: sign convergence} holds and recall the monotone coupling $(\mathbb{P}_{\Lambda,\beta},\varphi^{\mathfrak{p}},\varphi^{\eta_L})$ of Lemma~\ref{lem: coupling lemma}. Let $Y=\textup{sgn}(\varphi^{\mathfrak{p}}_0)-\textup{sgn}(\varphi^{\eta}_0)$. Using the defining properties of the monotone coupling, the Cauchy-Schwarz inequality and the fact that $0\leq Y\leq 2$ we obtain that 
\begin{equation}
\begin{aligned}
\langle \varphi_0\rangle_{\Lambda_L,\beta,\mathfrak{p}_L}-\langle \varphi_0\rangle_{\Lambda_L,\beta,\eta_L}&=\mathbb{E}_{\Lambda,\beta}[(|\varphi^{\mathfrak{p}}_0|-|\varphi^{\eta_L}_0|)\textup{sgn}(\varphi^{\mathfrak{p}}_0)]+\mathbb{E}_{\Lambda,\beta}[|\varphi^{\eta_L}_0|Y] \\  
&\leq \mathbb{E}_{\Lambda,\beta}[|\varphi^{\mathfrak{p}}_0|-|\varphi^{\eta_L}_0|]+\sqrt{\mathbb{E}_{\Lambda,\beta}[(\varphi^{\eta_L}_0)^2]\mathbb{E}_{\Lambda,\beta}[Y^2]} \\
&\leq \mathbb{E}_{\Lambda,\beta}[|\varphi^{\mathfrak{p}}_0|-|\varphi^{\eta_L}_0|]+\sqrt{2\mathbb{E}_{\Lambda,\beta}[(\varphi^{\eta_L}_0)^2]\mathbb{E}_{\Lambda,\beta}[Y]}.
\end{aligned}  
\end{equation}
Since $\langle |\varphi_0|\rangle_{\Lambda,\beta,\eta_L}\geq \langle |\varphi_0|\rangle^0_{\Lambda,\beta}$, and $\Psi^0=\Psi^1$ by Proposition~\ref{prop: unique gibbs measure}, it follows that $\mathbb{E}_{\Lambda,\beta}[|\varphi^{\mathfrak{p}}_0|-|\varphi^{\eta_L}_0|]$ tends to $0$. Furthermore, $\mathbb{E}_{\Lambda,\beta}[Y]$ tends to $0$ by our assumption. We can thus conclude that 
\begin{equation}
\lim_{L\to\infty}\langle \varphi_0 \rangle_{\Lambda_L,\beta,\eta_L}= \langle \varphi_0 \rangle^+_{\beta}, 
\end{equation}
as desired.

We now proceed with the proof of \eqref{eq: sign convergence}.
By the Edwards--Sokal coupling and Proposition~\ref{prop: unique gibbs measure}, it suffices to prove that 
\begin{equation}\label{eq: theta convergence}
\lim_{L\to\infty} \Psi_{\Lambda_L,\beta,\eta_L}^0[0\longleftrightarrow \fg]= \Psi^0_{\beta}[0\longleftrightarrow \infty].    
\end{equation}
We will prove this by using a renormalisation argument. 
To this end, recall Theorem~\ref{thm: local uniqueness}. Let $\varepsilon>0$ and consider some $s\in (0,1)$ and $k\geq 1$ to be chosen in terms of $\varepsilon$. Let $L\geq 1$ be much larger than $k$, and consider boxes of the form $\Lambda_k(kx)$ for $x\in \Lambda_m$, where $m=\lfloor \frac{L-10k}{k} \rfloor$. We define a site percolation configuration $\gamma$ on $\Lambda_m$ by letting $\gamma_x=1$ if $U(k,kx)$ happens and $\gamma_x=0$ otherwise. As in the proof of Proposition~\ref{prop: slab sprinkling}, there exists $k=k(s)$ large enough, such that $\gamma$ dominates a Bernoulli site percolation $\mathbb{P}_{\Lambda_m,s}$ of parameter $s$. We will choose $s$ sufficiently close to $1$ (which corresponds to choosing $k$ large enough) below.

Let $P_m$ denote the projection of $\Lambda_{m/2}$ on a fixed side of $\partial \Lambda_m$. Let also $\mathcal{E}$ denote the event that the number of vertices $x\in P_m$ that are connected to $\Lambda_{m/2}$ by a path of open vertices is at least $m^{d-1}/2^{d}$. We will use a Peierls-type argument to show that $\mathcal{E}$ occurs with probability close to $1$ under $\mathbb{P}_{\Lambda_m,s}$. To this end, for $x\in \partial \Lambda_m$, let $\mathcal{C}_x$ denote the connected component of $x$ in $\Lambda_m$ consisting of open vertices, and define $\mathcal{C}=\bigcup_{x\in \partial \Lambda_m}\mathcal{C}_x$. For $x\in P_m$, let $K_x$ denote the column that contains $x$ and intersects $\Lambda_{m/2}$. Note that when there is no path of open vertices connecting a vertex $x\in P_m$ to  $\Lambda_{m/2}$, there exists a vertex $y\in \partial^{\mathrm{ext}} \mathcal{C}$ that lies in $K_x$. Indeed, let $z$ be the last vertex of $K_{x}$ that lies in $\mathcal{C}_x$, and let $y$ be the vertex in $K_x$ after $z$. Then $y$ lies in $\partial^{\mathrm{ext}} \mathcal{C}$, as claimed. We can now deduce that
\begin{equation}
\mathbb{P}_{A_m,s}[\mathcal{E}^c]\leq \mathbb{P}_{A_m,s}\left[|\partial^{\mathrm{ext}} \mathcal{C}|\geq m^{d-1}/2^d\right],    
\end{equation}
where we used that $|P_m|-m^{d-1}/2^{d}\geq m^{d-1}/2^{d}$.

To bound the latter probability, let $\Delta \mathcal{C}=\partial^{\mathrm{ext}} \mathcal{C} \cup \partial^{\rm ext} \Lambda_m$. Note that $\Delta \mathcal{C}$ is $*$-connected, since each $\partial^{\mathrm{ext}} \mathcal{C}_x$ is $*$-connected. Since the number of $*$-connected subgraphs of $\mathbb{Z}^d$ with $k$ vertices that contain $0$ is at most $e^{Ck}$ for some constant $C>0$, we can use a union bound to obtain that
\begin{equation}
\mathbb{P}_{\Lambda_m,s}\left[|\partial^{\mathrm{ext}} \mathcal{C}|\geq m^{d-1}/2^d\right]\leq \sum_{i\geq m^{d-1}/2^d}\exp\left(C|\partial^{\rm ext} \Lambda_m|+Ci\right) (1-s)^{i}.
\end{equation}
By choosing $s$ to be close enough to $1$, we can ensure that the latter is at most $\varepsilon$, which implies that
\begin{equation}
\limsup_{L\to\infty} \Psi_{\Lambda_L,\beta,\eta_L}^0[\gamma \in \mathcal{E}^c]\leq \varepsilon.
\end{equation}

Now, note that when the events $\mathcal{E}$ and $\{0 \longleftrightarrow \partial \Lambda_{m}\}$ happen under $\gamma$, and additionally, $\{0\longleftrightarrow \partial \Lambda_k\}$ and $U(L)$ happen under $\omega$, then the event $\mathcal{B}=\{0\longleftrightarrow \partial \Lambda_{L-10k}, |\mathcal{C}_0 \cap \partial \Lambda_{L-10k}|\geq  cL^{d-1}\}$ happens under $\omega$, for some $c>0$. By further increasing the value of $s$, we can assume that 
$\mathbb{P}_{\Lambda_m,s}[0 \longleftrightarrow \partial \Lambda_m]\geq 1-\varepsilon$. Choosing $L$ to be large enough so that $\Psi^0_{\Lambda_L,\beta,\eta_L}[U(L)]\geq 1-\varepsilon$ and $\Psi^0_{\Lambda_L,\beta,\eta_L}[0\longleftrightarrow \partial \Lambda_k]\geq \Psi^0_{\beta}[0\longleftrightarrow \partial \Lambda_k]-\varepsilon$, we obtain that
\begin{equation}
\Psi^0_{\Lambda_L,\beta,\eta_L}[\mathcal{B}]\geq \Psi^0_{\beta}[0\longleftrightarrow \partial \Lambda_k]-4\varepsilon\geq \Psi^0_{\beta}[0\longleftrightarrow \infty]-4\varepsilon.     
\end{equation}

It remains to show that conditionally on $\mathcal{B}$, $0$ is connected to $\fg$ with high probability.
Indeed, if some $x\in \Lambda_{L-10k}$ is connected to $0$, when we fully open $E(\Lambda_{10k}(x))$ and we also open an edge of the form $u\fg$ for some $u \in \Lambda_{10k}(x)$, $0$ is connected to $\fg$. Now note that 
\begin{equation}
\Psi^0_{\Lambda_{10k}(x),\beta,\eta_L}[\omega|_{E(\Lambda_{10k}(x))\cup\{u\fg\}}=1]\geq \Psi^0_{\Lambda_{10k}(x),\beta}[\omega|_{E(\Lambda_{10k}(x))}=1]
\Psi^0_{\{u\},\beta,\eta_L}[\omega_{u\fg}=1]\geq r\eta_L,
\end{equation}
for some $r>0$ by the FKG inequality and monotonicity in the volume.
It follows that the family of events $\{\omega|_{E(\Lambda_{10k}(x))\cup\{u\fg\}}=1\}$ for $x\in \mathcal{C}_0\cap \partial \Lambda_{L-10k}$ at distance at least $20k$ apart from each other stochastically dominates a sequence of Bernoulli random variables of parameter $r\eta_L>0$. Hence, conditionally on $\mathcal{B}$, $\bigcup_{x\in \mathcal{C}_0\cap \partial \Lambda_{L-10k}}\{\omega|_{E(\Lambda_{10k}(x))\cup\{u\fg\}}=1\}$ happens with high probability, where here we use that $L^{d-1}\eta_L\to\infty$. 
Thus, for every $L$ large enough we have 
\begin{equation}
\Psi^0_{\Lambda_L,\beta,\eta_L}[0\longleftrightarrow\fg \mid \mathcal{B}]\geq 1-\varepsilon.    
\end{equation}
It follows that
\begin{equation}
\liminf_{L\to\infty}\Psi^0_{\Lambda_L,\beta,\eta_L}[0\longleftrightarrow\fg]\geq (1-\varepsilon)(\Psi^0_{\beta}[0\longleftrightarrow \infty]-4\varepsilon).
\end{equation}
Since $\varepsilon$ is arbitrary, we can conclude that
\eqref{eq: theta convergence} holds, as desired. 
\end{proof}


\begin{thebibliography}{DGRST23b}
\expandafter\ifx\csname fonteauteurs\endcsname\relax
\def\fonteauteurs{\scshape}\fi

\bibitem[ABF87]{AizenmanBarskyFernandezSharpnessIsing1987}
M.~\bgroup\fonteauteurs\bgroup Aizenman\egroup\egroup{}, D.J. \bgroup\fonteauteurs\bgroup Barsky\egroup\egroup{} and R.~\bgroup\fonteauteurs\bgroup Fern{\'a}ndez\egroup\egroup{} :
\newblock The phase transition in a general class of {I}sing-type models is sharp.
\newblock {\em Journal of Statistical Physics}, \textbf{47}\string:\penalty500\relax 343--374, 1987.

\bibitem[AD21]{AizenmanDuminilTriviality2021}
M.~\bgroup\fonteauteurs\bgroup Aizenman\egroup\egroup{} and H.~\bgroup\fonteauteurs\bgroup Duminil-Copin\egroup\egroup{} :
\newblock Marginal triviality of the scaling limits of critical 4{D} {I}sing and $\varphi^4_4$ models.
\newblock {\em Annals of Mathematics}, \textbf{194}(1)\string:\penalty500\relax 163--235, 2021.

\bibitem[Aiz82]{AizenmanGeometricAnalysis1982}
M.~\bgroup\fonteauteurs\bgroup Aizenman\egroup\egroup{} :
\newblock Geometric analysis of $\varphi^4$ fields and {I}sing models. {P}arts {I} and {II}.
\newblock {\em Communications in Mathematical Physics}, \textbf{86}(1)\string:\penalty500\relax 1--48, 1982.

\bibitem[Aiz21]{A21}
M.~\bgroup\fonteauteurs\bgroup Aizenman\egroup\egroup{} :
\newblock {A geometric perspective on the scaling limits of critical {I}sing and $\phi^4_d$ models}.
\newblock {\em Preprint}, available at arXiv:2112.04248, 2021.

\bibitem[ATT18]{ATT18}
D.~\bgroup\fonteauteurs\bgroup Ahlberg\egroup\egroup{}, V.~\bgroup\fonteauteurs\bgroup Tassion\egroup\egroup{} and A.~\bgroup\fonteauteurs\bgroup Teixeira\egroup\egroup{} :
\newblock Sharpness of the phase transition for continuum percolation in $\mathbb{R}^2$.
\newblock {\em Probability Theory and Related Fields}, 172\string:\penalty500\relax 525--581, 2018.

\bibitem[BBS14]{BauerschmidtBrydgesSlade2014Phi4fourdim}
R.~\bgroup\fonteauteurs\bgroup Bauerschmidt\egroup\egroup{}, D.C. \bgroup\fonteauteurs\bgroup Brydges\egroup\egroup{} and G.~\bgroup\fonteauteurs\bgroup Slade\egroup\egroup{} :
\newblock Scaling limits and critical behaviour of the $4$-dimensional $n$-component $|\varphi|^4$ spin model.
\newblock {\em Journal of Statistical Physics}, \textbf{157}\string:\penalty500\relax 692--742, 2014.

\bibitem[BBS19]{BauerschmidtBrydgesSladeBOOKRG2019}
R.~\bgroup\fonteauteurs\bgroup Bauerschmidt\egroup\egroup{}, D.C. \bgroup\fonteauteurs\bgroup Brydges\egroup\egroup{} and G.~\bgroup\fonteauteurs\bgroup Slade\egroup\egroup{} :
\newblock {\em Introduction to a {R}enormalisation {G}roup {M}ethod}, volume \textbf{2242}.
\newblock Springer Nature, 2019.

\bibitem[BD24]{BD24}
R.~\bgroup\fonteauteurs\bgroup Bauerschmidt\egroup\egroup{}, B. \bgroup\fonteauteurs\bgroup Dagallier\egroup\egroup{} :
\newblock {\em Communications on Pure and Applied Mathematics}, 
\textbf{77}(3-4)\string:\penalty500\relax 2579--2612, 2024.


\bibitem[BD12]{BeffaraDuminilSelfDualPoint}
V.~\bgroup\fonteauteurs\bgroup Beffara\egroup\egroup{} and H.~\bgroup\fonteauteurs\bgroup Duminil-Copin\egroup\egroup{} :
\newblock The self-dual point of the two-dimensional random-cluster model is critical for $q\geq 1$.
\newblock {\em Probability Theory and Related Fields}, \textbf{153}(3-4)\string:\penalty500\relax 511--542, 2012.

\bibitem[BIV00]{bodineau2000rigorous}
T.~\bgroup\fonteauteurs\bgroup Bodineau\egroup\egroup{}, D.~\bgroup\fonteauteurs\bgroup Ioffe\egroup\egroup{} and Y.~\bgroup\fonteauteurs\bgroup Velenik\egroup\egroup{} :
\newblock Rigorous probabilistic analysis of equilibrium crystal shapes.
\newblock {\em Journal of Mathematical Physics}, \textbf{41}(3)\string:\penalty500\relax 1033--1098, 2000.

\bibitem[BK89]{BurtonKeane1989density}
R.M. \bgroup\fonteauteurs\bgroup Burton\egroup\egroup{} and M.~\bgroup\fonteauteurs\bgroup Keane\egroup\egroup{} :
\newblock Density and uniqueness in percolation.
\newblock {\em Communications in Mathematical Physics}, \textbf{121}\string:\penalty500\relax 501--505, 1989.

\bibitem[Bod05]{Bod05}
T.~\bgroup\fonteauteurs\bgroup Bodineau\egroup\egroup{} :
\newblock {Slab percolation for the Ising model}.
\newblock {\em Probability Theory and Related Fields}, \textbf{132}\string:\penalty500\relax 83--118, 2005.

\bibitem[BR06]{BR06}
B.~\bgroup\fonteauteurs\bgroup Bollob{\'a}s\egroup\egroup{} and O.~\bgroup\fonteauteurs\bgroup Riordan\egroup\egroup{} :
\newblock {The critical probability for random Voronoi percolation in the plane is 1/2}.
\newblock {\em Probability Theory and Related Fields}, 136(3)\string:\penalty500\relax 417--468, 2006.


\bibitem[BT17]{BenjaminiTassion}
I.~\bgroup\fonteauteurs\bgroup Benjamini\egroup\egroup{} and V.~\bgroup\fonteauteurs\bgroup Tassion\egroup\egroup{} :
\newblock Homogenization via sprinkling.
\newblock {\em Annales de l'Institut Henri Poincar{\'e}-Probabilit{\'e}s and Statistiques}, \textbf{53}(2)\string:\penalty500\relax 997--1005, 2017.

\bibitem[BT89]{brower1989embedded}
R.C. \bgroup\fonteauteurs\bgroup Brower\egroup\egroup{} and P.~\bgroup\fonteauteurs\bgroup Tamayo\egroup\egroup{} :
\newblock Embedded dynamics for $\varphi^4$ theory.
\newblock {\em Physical Review Letters}, \textbf{62}(10)\string:\penalty500\relax 1087, 1989.

\bibitem[CP00]{cerf2000wulff}
R.~\bgroup\fonteauteurs\bgroup Cerf\egroup\egroup{} and {\'A}.~\bgroup\fonteauteurs\bgroup Pisztora\egroup\egroup{} :
\newblock On the wulff crystal in the {I}sing model.
\newblock {\em The Annals of Probability}, \textbf{23}(3)\string:\penalty500\relax 947--1017, 2000.



\bibitem[CGMS96]{cesi1996two}
F.~\bgroup\fonteauteurs\bgroup Cesi\egroup\egroup{}, G.~\bgroup\fonteauteurs\bgroup Guadagni\egroup\egroup{}, F.~\bgroup\fonteauteurs\bgroup Martinelli\egroup\egroup{} and R.H. \bgroup\fonteauteurs\bgroup Schonmann\egroup\egroup{} :
\newblock On the two-dimensional stochastic {I}sing model in the phase coexistence region near the critical point.
\newblock {\em Journal of Statistical Physics}, \textbf{85}\string:\penalty500\relax 55--102, 1996.

\bibitem[CGW22]{CGW22}
A.~\bgroup\fonteauteurs\bgroup Chandra\egroup\egroup{} and T.S.~\bgroup\fonteauteurs\bgroup Gunaratnam\egroup\egroup{} and H.~\bgroup\fonteauteurs\bgroup Weber\egroup\egroup{}:
\newblock Phase Transitions for $\varphi^4_3$.
\newblock {\em Communications in Mathematical Physics}, \textbf{392}(3)\string:\penalty500\relax 691--782, 2022.

\bibitem[CCS87]{CCS87}
J.T. \bgroup\fonteauteurs\bgroup Chayes\egroup\egroup{}, L.~\bgroup\fonteauteurs\bgroup Chayes\egroup\egroup{} and R.H. \bgroup\fonteauteurs\bgroup Schonmann\egroup\egroup{} :
\newblock Exponential decay of connectivities in the two-dimensional {I}sing model.
\newblock {\em Journal of Statistical Physics}, 49\string:\penalty500\relax 433--445, 1987.

\bibitem[CF86]{comets1986grandes}
\bgroup\fonteauteurs\bgroup F.~Comets and R.~Fortet\egroup\egroup. 
\emph{Grandes d\'eviations pour des champs de Gibbs sur $\mathbb{Z}^d$}. 
\emph{Comptes Rendus de l'Académie des Sciences. Série 1, Mathématique}, \textbf{303}:\penalty500\relax 511--513, 1986.

\bibitem[CMT24]{CMT24}
D.~\bgroup\fonteauteurs\bgroup Contreras\egroup\egroup{}, S.~\bgroup\fonteauteurs\bgroup Martineau\egroup\egroup{} and V.~\bgroup\fonteauteurs\bgroup Tassion\egroup\egroup{} :
\newblock Supercritical percolation on graphs of polynomial growth.
\newblock {\em Duke Mathematical Journal}, 173(4)\string:\penalty500\relax 745--806, 2024.

\bibitem[DS23]{DS23}
B.~\bgroup\fonteauteurs\bgroup Dembin\egroup\egroup{} and F.~\bgroup\fonteauteurs\bgroup Severo\egroup\egroup{} :
\newblock Supercitical sharpness for {V}oronoi percolation.
\newblock {\em To appear in Probability Theory and Related Fields}, available at arXiv:2311.00555, 2023.

\bibitem[DT22]{DT22}
B.~\bgroup\fonteauteurs\bgroup Dembin\egroup\egroup{} and V.~\bgroup\fonteauteurs\bgroup Tassion\egroup\egroup{} :
\newblock Almost sharp sharpness for {P}oisson {B}oolean percolation.
\newblock {\em Preprint}, available at arXiv:2209.00999, 2022.

\bibitem[DP96]{DP96}
J.-D. \bgroup\fonteauteurs\bgroup Deuschel\egroup\egroup{} and A.~\bgroup\fonteauteurs\bgroup Pisztora\egroup\egroup{} :
\newblock Surface order large deviations for high-density percolation.
\newblock {\em Probability Theory and Related Fields}, \textbf{104}\string:\penalty500\relax 467--482, 1996.




\bibitem[DKS92]{dobrushin1992wulff}
R.L. \bgroup\fonteauteurs\bgroup Dobrushin\egroup\egroup{}, R.~\bgroup\fonteauteurs\bgroup Koteck{\`y}\egroup\egroup{} and S.~\bgroup\fonteauteurs\bgroup Shlosman\egroup\egroup{} :
\newblock {\em Wulff construction: a global shape from local interaction}, volume \textbf{104}.
\newblock American Mathematical Society Providence, 1992.



\bibitem[Dum19]{DuminilLecturesOnIsingandPottsModels2019}
H.~\bgroup\fonteauteurs\bgroup Duminil-Copin\egroup\egroup{} :
\newblock Lectures on the {I}sing and {P}otts models on the hypercubic lattice.
\newblock \emph{In} {\em Random Graphs, Phase Transitions, and the Gaussian Free Field: PIMS-CRM Summer School in Probability, Vancouver, Canada, June 5--30, 2017}, pages 35--161. Springer, 2019.

\bibitem[DGR20]{duminil2020exponential}
H.~\bgroup\fonteauteurs\bgroup Duminil-Copin\egroup\egroup{}, S.~\bgroup\fonteauteurs\bgroup Goswami\egroup\egroup{} and A.~\bgroup\fonteauteurs\bgroup Raoufi\egroup\egroup{} :
\newblock Exponential decay of truncated correlations for the {I}sing model in any dimension for all but the critical temperature.
\newblock {\em Communications in Mathematical Physics}, \textbf{374}(2)\string:\penalty500\relax 891--921, 2020.

\bibitem[DGRST23a]{DGRST23b}
H.~\bgroup\fonteauteurs\bgroup Duminil-Copin\egroup\egroup{}, S.~\bgroup\fonteauteurs\bgroup Goswami\egroup\egroup{}, P.-F. \bgroup\fonteauteurs\bgroup Rodriguez\egroup\egroup{}, F.~\bgroup\fonteauteurs\bgroup Severo\egroup\egroup{} and A.~\bgroup\fonteauteurs\bgroup Teixeira\egroup\egroup{} :
\newblock Finite-range interlacements and couplings.
\newblock {\em To appear in The Annals of Probability}, available at arXiv:2308.07303, 2023.

\bibitem[DGRST23b]{DGRST23a}
H.~\bgroup\fonteauteurs\bgroup Duminil-Copin\egroup\egroup{}, S.~\bgroup\fonteauteurs\bgroup Goswami\egroup\egroup{}, P.-F. \bgroup\fonteauteurs\bgroup Rodriguez\egroup\egroup{}, F.~\bgroup\fonteauteurs\bgroup Severo\egroup\egroup{} and A.~\bgroup\fonteauteurs\bgroup Teixeira\egroup\egroup{} :
\newblock Phase transition for the vacant set of random walks and random interlacements.
\newblock {\em Preprint}, available at arXiv:2308.07919, 2023.

\bibitem[DGRST24]{DGRST23c}
H.~\bgroup\fonteauteurs\bgroup Duminil-Copin\egroup\egroup{}, S.~\bgroup\fonteauteurs\bgroup Goswami\egroup\egroup{}, P.-F. \bgroup\fonteauteurs\bgroup Rodriguez\egroup\egroup{}, F.~\bgroup\fonteauteurs\bgroup Severo\egroup\egroup{} and A.~\bgroup\fonteauteurs\bgroup Teixeira\egroup\egroup{} :
\newblock A characterization of strong percolation via disconnection.
\newblock {\em Proceedings of the London Mathematical Society}, 129(2)\string:\penalty500\relax e12622, 2024.

\bibitem[DGRS23]{DGRS19}
H.~\bgroup\fonteauteurs\bgroup Duminil-Copin\egroup\egroup{}, S.~\bgroup\fonteauteurs\bgroup Goswami\egroup\egroup{}, P.-F. \bgroup\fonteauteurs\bgroup Rodriguez\egroup\egroup{} and F.~\bgroup\fonteauteurs\bgroup Severo\egroup\egroup{} :
\newblock {Equality of critical parameters for percolation of Gaussian free field level sets}.
\newblock {\em Duke Mathematical Journal}, 172(5)\string:\penalty500\relax 839 -- 913, 2023.

\bibitem[DGT23]{duminil2023long}
H.~\bgroup\fonteauteurs\bgroup Duminil-Copin\egroup\egroup{}, C.~\bgroup\fonteauteurs\bgroup Garban\egroup\egroup{} and V.~\bgroup\fonteauteurs\bgroup Tassion\egroup\egroup{} :
\newblock Long-range order for critical {B}ook-{I}sing and {B}ook-percolation.
\newblock {\em Communications in Mathematical Physics}, \textbf{404}(3)\string:\penalty500\relax 1309--1339, 2023.

\bibitem[DP24]{DuminilPanis2024newLB}
H.~\bgroup\fonteauteurs\bgroup Duminil-Copin\egroup\egroup{} and R.~\bgroup\fonteauteurs\bgroup Panis\egroup\egroup{} :
\newblock New lower bounds for the (near) critical {I}sing and $\varphi^4$ models' two-point functions.
\newblock {\em Communications in Mathematical Physics}, \textbf{406}(3)\string:\penalty500\relax, 2025.




\bibitem[Dur19]{durrett2019probability}
R.~\bgroup\fonteauteurs\bgroup Durrett\egroup\egroup{} :
\newblock {\em Probability: theory and examples}, Volume \textbf{49}.
\newblock Cambridge university press, 2019.


\bibitem[FMRS87]{FeldmanMagnenRivasseau1987construction}
J.~\bgroup\fonteauteurs\bgroup Feldman\egroup\egroup{}, J.~\bgroup\fonteauteurs\bgroup Magnen\egroup\egroup{}, V.~\bgroup\fonteauteurs\bgroup Rivasseau\egroup\egroup{} and R.~\bgroup\fonteauteurs\bgroup Sénéor\egroup\egroup{} :
\newblock Construction and {B}orel summability of infrared $\varphi^4_4$ by a phase space expansion.
\newblock {\em Communications in Mathematical Physics}, \textbf{109}\string:\penalty500\relax 437--480, 1987.

\bibitem[FP87]{FP87}
J.~\bgroup\fonteauteurs\bgroup Fr{\"o}hlich\egroup\egroup{} and C.-E. \bgroup\fonteauteurs\bgroup Pfister\egroup\egroup{} :
\newblock {Semi-infinite Ising model: {II}. The wetting and layering transitions}.
\newblock {\em Communications in Mathematical Physics}, \textbf{112}\string:\penalty500\relax 51--74, 1987.

\bibitem[Fr{\"o}82]{FrohlichTriviality1982}
J.~\bgroup\fonteauteurs\bgroup Fr{\"o}hlich\egroup\egroup{} :
\newblock On the triviality of $\lambda\phi_d^4$ theories and the approach to the critical point in $d>4$ dimensions.
\newblock {\em Nuclear Physics B}, \textbf{200}(2)\string:\penalty500\relax 281--296, 1982.


\bibitem[FILS78]{FrohlichIsraelLiebSimon1978}
J.~\bgroup\fonteauteurs\bgroup Fr{\"o}hlich\egroup\egroup{}, R.~\bgroup\fonteauteurs\bgroup Israel\egroup\egroup{}, E.H. \bgroup\fonteauteurs\bgroup Lieb\egroup\egroup{} and B.~\bgroup\fonteauteurs\bgroup Simon\egroup\egroup{} :
\newblock Phase transitions and reflection positivity. {I}. {G}eneral theory and long-range lattice models.
\newblock {\em Communications in Mathematical Physics}, \textbf{62}(1)\string:\penalty500\relax 1--34, 1978.

\bibitem[FSS76]{FrohlichSimonSpencerIRBounds1976}
J.~\bgroup\fonteauteurs\bgroup Fr{\"o}hlich\egroup\egroup{}, B.~\bgroup\fonteauteurs\bgroup Simon\egroup\egroup{} and T.~\bgroup\fonteauteurs\bgroup Spencer\egroup\egroup{} :
\newblock Infrared bounds, phase transitions and continuous symmetry breaking.
\newblock {\em Communications in Mathematical Physics}, \textup{\textbf{50}}(1)\string:\penalty500\relax 79--95, 1976.



\bibitem[GK85]{GawedzkiKupiainen1985massless}
K.~\bgroup\fonteauteurs\bgroup Gawedzki\egroup\egroup{} and A.~\bgroup\fonteauteurs\bgroup Kupiainen\egroup\egroup{} :
\newblock Massless lattice $\varphi_4^4$ theory: rigorous control of a renormalizable asymptotically free model.
\newblock {\em Communications in Mathematical Physics}, \textup{\textbf{99}}\string:\penalty500\relax 197--252, 1985.

\bibitem[GJ73]{GlimmJaffe1973PHI43d}
J.~\bgroup\fonteauteurs\bgroup Glimm\egroup\egroup{} and A.~\bgroup\fonteauteurs\bgroup Jaffe\egroup\egroup{} :
\newblock Positivity of the $\varphi^4_3$ hamiltonian.
\newblock {\em Fortschritte der Physik}, \textbf{21}(7)\string:\penalty500\relax 327--376, 1973.

\bibitem[GJS75]{GlimmJaffeSpencer1975PHI4}
J.~\bgroup\fonteauteurs\bgroup Glimm\egroup\egroup{}, A.~\bgroup\fonteauteurs\bgroup Jaffe\egroup\egroup{} and T.~\bgroup\fonteauteurs\bgroup Spencer\egroup\egroup{} :
\newblock Phase transitions for $\varphi^4_2$ quantum fields.
\newblock {\em Communications in Mathematical Physics}, 45\string:\penalty500\relax 203--216, 1975.

\bibitem[GJ12]{GlimmJaffeQuantumBOOK}
J.~\bgroup\fonteauteurs\bgroup Glimm\egroup\egroup{} and A.~\bgroup\fonteauteurs\bgroup Jaffe\egroup\egroup{} :
\newblock {\em Quantum Physics: A Functional Integral Point of View}.
\newblock Springer New York, NY, 2012.

\bibitem[GG06]{Graham2006random}
B.~\bgroup\fonteauteurs\bgroup Graham\egroup\egroup{} and G.~\bgroup\fonteauteurs\bgroup Grimmett\egroup\egroup{} :
\newblock Random-cluster representation of the Blume--Capel model.
\newblock {\em Journal of Statistical Physics}, \textbf{125}\string:\penalty500\relax 283--316, 2006.

\bibitem[GM90]{GM90}
G.~\bgroup\fonteauteurs\bgroup Grimmett\egroup\egroup{} and J.M. \bgroup\fonteauteurs\bgroup Marstrand\egroup\egroup{} :
\newblock The supercritical phase of percolation is well behaved.
\newblock {\em Proceedings of the Royal Society of London. Series A: Mathematical and Physical Sciences}, \textbf{430}(1879)\string:\penalty500\relax 439--457, 1990.


\bibitem[Gri67]{GriffithsCorrelationsIsing1-1967}
R.B. \bgroup\fonteauteurs\bgroup Griffiths\egroup\egroup{} :
\newblock Correlations in {I}sing ferromagnets. {I}.
\newblock {\em Journal of Mathematical Physics}, \textbf{8}(3)\string:\penalty500\relax 478--483, 1967.

\bibitem[Gri99]{GrimmettPercolation1999}
G.~\bgroup\fonteauteurs\bgroup Grimmett\egroup\egroup{} :
\newblock {\em Percolation}, Volume \textbf{321}.
\newblock Springer, 1999.

\bibitem[Gri06]{Grimmett2006RCM}
G.~\bgroup\fonteauteurs\bgroup Grimmett\egroup\egroup{} :
\newblock {\em The Random-Cluster Model}, Volume \textbf{333}.
\newblock Springer, 2006.

\bibitem[GPPS22]{GunaratnamPanagiotisPanisSeveroPhi42022}
T.S. \bgroup\fonteauteurs\bgroup Gunaratnam\egroup\egroup{}, C.~\bgroup\fonteauteurs\bgroup Panagiotis\egroup\egroup{}, R.~\bgroup\fonteauteurs\bgroup Panis\egroup\egroup{} and F.~\bgroup\fonteauteurs\bgroup Severo\egroup\egroup{} :
\newblock Random tangled currents for $\varphi^4$: translation invariant {G}ibbs measures and continuity of the phase transition.
\newblock {\em To appear in Journal of the European Mathematical Society}, available at arXiv:2211.00319, 2022.

\bibitem[GKP24]{gunaratnam2024blume}
T.S. \bgroup\fonteauteurs\bgroup Gunaratnam\egroup\egroup{}, D.~\bgroup\fonteauteurs\bgroup Krachun\egroup\egroup{} and C.~\bgroup\fonteauteurs\bgroup Panagiotis\egroup\egroup{} :
\newblock Existence of a tricritical point for the {B}lume--{C}apel model on $\mathbb{Z}^d$.
\newblock {\em Probability and Mathematical Physics}, \textbf{5}(3)\string:\penalty500\relax 785--845, 2024.


\bibitem[HSV14]{hairer2014spectral}
\bgroup\fonteauteurs\bgroup M.~Hairer, A.~M.~Stuart, and S.~J.~Vollmer\egroup\egroup. 
\emph{Spectral gaps for a Metropolis--Hastings algorithm in infinite dimensions}. 
2014.

\bibitem[Har87]{Hara1987rigorous}
T.~\bgroup\fonteauteurs\bgroup Hara\egroup\egroup{} :
\newblock A rigorous control of logarithmic corrections in four-dimensional $\phi^4$ spin systems: {I}. {T}rajectory of effective {H}amiltonians.
\newblock {\em Journal of Statistical Physics}, \textbf{47}\string:\penalty500\relax 57--98, 1987.

\bibitem[Hu25]{Hu2025}
H.~\bgroup\fonteauteurs\bgroup Hu\egroup\egroup{} :
\newblock Polynomial rate of relaxation for the Glauber dynamics of infinite-volume critical Ising model.
\newblock {\em Electronic Communications in Probability}, \textbf{30}\string:\penalty500\relax1--12, 2025.

\bibitem[Iof95]{ioffe1995exact}
D.~\bgroup\fonteauteurs\bgroup Ioffe\egroup\egroup{} :
\newblock Exact large deviation bounds up to $T_c$ for the {I}sing model in two dimensions.
\newblock {\em Probability Theory and Related Fields}, \textbf{102}\string:\penalty500\relax 313--330, 1995.

\bibitem[Kes80]{Kesten1980criticalproba}
H.~\bgroup\fonteauteurs\bgroup Kesten\egroup\egroup{} :
\newblock The critical probability of bond percolation on the square lattice equals $1/2$.
\newblock {\em Communications in Mathematical Physics}, \textbf{74}(1)\string:\penalty500\relax 41--59, 1980.


\bibitem[KST23]{KohlerTassionGeneralRSW}
L.~\bgroup\fonteauteurs\bgroup K{\"o}hler-Schindler\egroup\egroup{} and V.~\bgroup\fonteauteurs\bgroup Tassion\egroup\egroup{} :
\newblock Crossing probabilities for planar percolation.
\newblock {\em Duke Mathematical Journal}, \textbf{172}(4)\string:\penalty500\relax 809--838, 2023.

\bibitem[KPP24]{KPP24}
D.~\bgroup\fonteauteurs\bgroup Krachun\egroup\egroup{}, C.~\bgroup\fonteauteurs\bgroup Panagiotis\egroup\egroup{} and R.~\bgroup\fonteauteurs\bgroup Panis\egroup\egroup{} :
\newblock Scaling limit of the cluster size distribution for the random current measure on the complete graph.
\newblock {\em Electronic Journal of Probability}, \textbf{29}\string:\penalty500\relax 1--24, 2024.

\bibitem[LO24]{LammersOtt2021}
P.~\bgroup\fonteauteurs\bgroup Lammers\egroup\egroup{} and S.~\bgroup\fonteauteurs\bgroup Ott\egroup\egroup{} :
\newblock {Delocalisation and absolute-value-FKG in the solid-on-solid model}.
\newblock {\em Probability Theory and Related Fields}, \textbf{188}\string:\penalty500\relax 63--87, 2024.

\bibitem[LP76]{LebowitzPresutti1976}
J.L. \bgroup\fonteauteurs\bgroup Lebowitz\egroup\egroup{} and E.~\bgroup\fonteauteurs\bgroup Presutti\egroup\egroup{} :
\newblock Statistical mechanics of systems of unbounded spins.
\newblock {\em Communications in Mathematical Physics}, \textbf{50}(3)\string:\penalty500\relax 195--218, 1976.


\bibitem[Leb77]{Lebowitz1977CoexistencePhasesIsing}
J.L. \bgroup\fonteauteurs\bgroup Lebowitz\egroup\egroup{} :
\newblock Coexistence of phases in {I}sing ferromagnets.
\newblock {\em Journal of Statistical Physics}, \textbf{16}(6)\string:\penalty500\relax 463--476, 1977.

\bibitem[LP81]{lebowitz1981surface}
J.L. \bgroup\fonteauteurs\bgroup Lebowitz\egroup\egroup{} and C.-E. \bgroup\fonteauteurs\bgroup Pfister\egroup\egroup{} :
\newblock Surface tension and phase coexistence.
\newblock {\em Physical Review Letters}, \textbf{46}(15)\string:\penalty500\relax 1031, 1981.

\bibitem[Led01]{ledoux2001logarithmic}
\bgroup\fonteauteurs\bgroup M.~Ledoux\egroup\egroup. 
\emph{Logarithmic Sobolev inequalities for unbounded spin systems revisited}. 
\emph{S{\'e}minaire de Probabilit{\'e}s XXXV}, \textbf{}:\penalty500\relax 167--194, 2001.

\bibitem[LSS97]{LSS97}
T.M. \bgroup\fonteauteurs\bgroup Liggett\egroup\egroup{}, R.H. \bgroup\fonteauteurs\bgroup Schonmann\egroup\egroup{} and A.M. \bgroup\fonteauteurs\bgroup Stacey\egroup\egroup{} :
\newblock Domination by product measures.
\newblock {\em The Annals of Probability}, \textbf{25}(1)\string:\penalty500\relax 71--95, 1997.

\bibitem[LS13]{LubSly13}
E.~\bgroup\fonteauteurs\bgroup Lubetzky\egroup\egroup{} and A.~\bgroup\fonteauteurs\bgroup Sly\egroup\egroup{} :
\newblock Cutoff for the Ising model on the lattice.
\newblock {\em Inventiones Mathematicae}, \textbf{191}\string:\penalty500\relax 719--755, 2013.

\bibitem[LS16]{LubSly16}
E.~\bgroup\fonteauteurs\bgroup Lubetzky\egroup\egroup{} and A.~\bgroup\fonteauteurs\bgroup Sly\egroup\egroup{} :
\newblock Information percolation and cutoff for the stochastic Ising model.
\newblock {\em Journal of the American Mathematical Society}, \textbf{29}\string:\penalty500\relax 729--774, 2016.

\bibitem[LS12]{LubSly12}
E.~\bgroup\fonteauteurs\bgroup Lubetzky\egroup\egroup{} and A.~\bgroup\fonteauteurs\bgroup Sly\egroup\egroup{} :  
\newblock Critical Ising on the square lattice mixes in polynomial time.  
\newblock {\em Communications in Mathematical Physics}, vol. 313, no. 3, pp. 815--836, 2012.


\bibitem[Mar99]{martinelli99}
F.~\bgroup\fonteauteurs\bgroup Martinelli\egroup\egroup{}:
\newblock Lectures on Glauber dynamics for discrete spin models.
\newblock {\em Lectures on probability theory and statistics (Saint-Flour, 1997)}, \textbf{1717}\string:\penalty500\relax 93--191, 1999.

\bibitem[MMSP84]{MMP84}
A.~\bgroup\fonteauteurs\bgroup Messager\egroup\egroup{}, S.~\bgroup\fonteauteurs\bgroup Miracle-Sol{\'e}\egroup\egroup{} and C.-E. \bgroup\fonteauteurs\bgroup Pfister\egroup\egroup{} :
\newblock On classical ferromagnets with a complex external field.
\newblock {\em Journal of Statistical Physics}, \textbf{34}\string:\penalty500\relax 279--286, 1984.

\bibitem[MT12]{meyn2012markov}
\bgroup\fonteauteurs\bgroup S.~P.~Meyn and R.~L.~Tweedie\egroup\egroup. 
\emph{Markov chains and stochastic stability}. 
Springer Science \& Business Media, 2012.

\bibitem[Nel66]{Nelson1966PHI42d}
E.~\bgroup\fonteauteurs\bgroup Nelson\egroup\egroup{} :
\newblock A quartic interaction in two dimensions.
\newblock \emph{In} {\em Mathematical Theory of Elementary Particles, (Dedham, Mass., 1965)}, pages 69--73, 1966.

\bibitem[Oll88]{olla1988large}
\bgroup\fonteauteurs\bgroup S.~Olla\egroup\egroup. 
\emph{Large deviations for Gibbs random fields}. 
\emph{Probability Theory and Related Fields}, \textbf{77}:\penalty500\relax 343--357, 1988.

\bibitem[Pan23]{PanisTriviality2023}
R.~\bgroup\fonteauteurs\bgroup Panis\egroup\egroup{} :
\newblock Triviality of the scaling limits of critical {I}sing and $\varphi^4$ models with effective dimension at least four.
\newblock {\em Preprint}, available at arXiv:2309.05797, 2023.

\bibitem[Pan24]{PanisThesis}
R.~\bgroup\fonteauteurs\bgroup Panis\egroup\egroup{} :
\newblock {\em Applications of {P}aths {E}xpansions to {S}tatistical {M}echanics}.
\newblock PhD thesis, Université de Genève, 2024.

\bibitem[Pis96]{Pis96}
A.~\bgroup\fonteauteurs\bgroup Pisztora\egroup\egroup{} :
\newblock Surface order large deviations for {I}sing, {P}otts and percolation models.
\newblock {\em Probability Theory and Related Fields}, 104\string:\penalty500\relax 427--466, 1996.

\bibitem[RT96]{RT96}
G.O.~\bgroup\fonteauteurs\bgroup Roberts\egroup\egroup{} and R.L.~\bgroup\fonteauteurs\bgroup Tweedie\egroup\egroup{} :
\newblock Exponential Convergence of Langevin Distributions and Their Discrete Approximations.
\newblock {\em Bernoulli}, \textbf{2}\string:\penalty500\relax 341--363, 1996.

\bibitem[RT01]{roberts2001geometric}
\bgroup\fonteauteurs\bgroup G.~O.~Roberts and R.~L.~Tweedie\egroup\egroup. 
\emph{Geometric $L^2$ and $L^1$ convergence are equivalent for reversible Markov chains}. 
\emph{Journal of Applied Probability}, \textbf{38A}:\penalty500\relax 37--41, 2001.

\bibitem[RT96b]{RT96b}
\bgroup\fonteauteurs\bgroup G.~O.~Roberts and R.~L.~Tweedie\egroup\egroup. 
\emph{Geometric convergence and central limit theorems for multidimensional Hastings and Metropolis algorithms}. 
\emph{Biometrika}, \textbf{83}:\penalty500\relax 95--110, 1996.

\bibitem[RR04]{RR04}
\bgroup\fonteauteurs\bgroup G.~O.~Roberts and J.~S.~Rosenthal\egroup\egroup. 
\emph{General state space Markov chains and MCMC algorithms}. 
2004.


\bibitem[Rue70]{Ruelle1970}
D.~\bgroup\fonteauteurs\bgroup Ruelle\egroup\egroup{} :
\newblock Superstable interactions in classical statistical mechanics.
\newblock {\em Communications in Mathematical Physics}, \textbf{18}(2)\string:\penalty500\relax 127--159, 1970.

\bibitem[Sak15]{Sakai2015Phi4}
A.~\bgroup\fonteauteurs\bgroup Sakai\egroup\egroup{} :
\newblock Application of the lace expansion to the $\varphi^4$ model.
\newblock {\em Communications in Mathematical Physics}, \textbf{336}\string:\penalty500\relax 619--648, 2015.

\bibitem[Sch87]{schonmann1987second}
\bgroup\fonteauteurs\bgroup R.~H.~Schonmann\egroup\egroup. 
\emph{Second order large deviation estimates for ferromagnetic systems in the phase coexistence region}. 
\emph{Communications in Mathematical Physics}, \textbf{112}:\penalty500\relax 409--422, 1987.



\bibitem[Sev22]{Sev21}
F.~\bgroup\fonteauteurs\bgroup Severo\egroup\egroup{} :
\newblock Sharp phase transition for {Gaussian} percolation in all dimensions.
\newblock {\em Annales Henri Lebesgue}, 5\string:\penalty500\relax 987--1008, 2022.

\bibitem[Sev24]{Severo24}
F.~\bgroup\fonteauteurs\bgroup Severo\egroup\egroup{} :
\newblock Slab percolation for the {I}sing model revisited.
\newblock {\em Electronic Communications in Probability}, 29\string:\penalty500\relax 1--11, 2024.

\bibitem[Sok82]{Sokal1982Destructive}
A.D. \bgroup\fonteauteurs\bgroup Sokal\egroup\egroup{} :
\newblock An alternate constructive approach to the $\varphi^4_3$ quantum field theory, and a possible destructive approach to $\varphi^4_4$.
\newblock {\em Annales de l'IHP Physique Théorique}, \textup{\textbf{37}}(4)\string:\penalty500\relax 317--398, 1982.

\bibitem[ST16]{SladeTombergRGWeakPhi4in2016}
G.~\bgroup\fonteauteurs\bgroup Slade\egroup\egroup{} and A.~\bgroup\fonteauteurs\bgroup Tomberg\egroup\egroup{} :
\newblock Critical correlation functions for the $4$-dimensional weakly self-avoiding walk and $n$-component $|\varphi|^4$ model.
\newblock {\em Communications in Mathematical Physics}, \textbf{342}\string:\penalty500\relax 675--737, 2016.


\bibitem[Str65]{Strassen1965}
V.~\bgroup\fonteauteurs\bgroup Strassen\egroup\egroup{} :
\newblock The existence of probability measures with given marginals.
\newblock {\em The Annals of Mathematical Statistics}, \textbf{36}(2)\string:\penalty500\relax 423--439, 1965.

\bibitem[Wel77]{wells1977some}
D.R. \bgroup\fonteauteurs\bgroup Wells\egroup\egroup{} :
\newblock {\em Some moment inequalities and a result on multivariate unimodality.}
\newblock Indiana University, 1977.



\end{thebibliography}
\end{document}